\documentclass{amsart}
\usepackage{amsmath ,amssymb ,amsthm, mathrsfs, stmaryrd, dsfont, extarrows}
\usepackage{bm}
\usepackage[all]{xy}
\usepackage[pdftex]{graphicx} % replace `dvipdfmx' by `pdftex' when updating on arXiv, which is available only on arXiv. 
\usepackage[pdftex ,pdfborder={0 0 0}]{hyperref} % samely
\usepackage{tikz, tikz-cd}
\usepackage{shadethm}
\usepackage{graphicx, xcolor}
\usepackage{url}
\DeclareMathAlphabet\mathbfcal{OMS}{cmsy}{b}{n}
\DeclareMathAlphabet{\mathpzc}{OT1}{pzc}{m}{it}

\DeclareMathOperator*{\supp}{supp}

\theoremstyle{definition}
\newtheorem{defin}{Definition}[section]

\newtheorem{thm}[defin]{Theorem}
\newtheorem{conj}[defin]{Conjecture}
\newtheorem{lem}[defin]{Lemma}
\newtheorem{prop}[defin]{Proposition}
\newtheorem{cor}[defin]{Corollary}

\theoremstyle{remark}
\newtheorem{rem}[defin]{Remark}
\newtheorem{eg}[defin]{Example}
\newtheorem{quest}[defin]{Question}

\newcommand{\cFut}{\check{\mathrm{F}}\mathrm{ut}}
\newcommand{\Fut}{\mathrm{Fut}}

\newcommand{\cmu}{\bm{\check{\mu}}_{\mathrm{ch}}}
\newcommand{\NAmu}{\bm{\check{\mu}}_{\mathrm{NA}}}

\newcommand{\nH}{\mathcal{H}_{\mathrm{NA}}}
\newcommand{\PSH}{\mathrm{PSH}_{\mathrm{NA}}}

\newcommand{\pcH}{\mathcal{H}^{\mathbb{R}}_{\mathrm{NA}}}
\newcommand{\E}{\mathcal{E}_{\mathrm{NA}}}
\newcommand{\DHm}{\mathrm{DH}}
\newcommand{\Proj}{\operatorname{\mathrm{Proj}}}

\title[II, Non-archimedean $\mu$-entropy and $\mu$K-stability]{Entropies in $\mu$-framework of canonical metrics and K-stability, II -- Non-archimedean aspect: non-archimedean $\mu$-entropy and $\mu$K-semistability}

\author{Eiji Inoue}

\address{RIKEN, iTHEMS, 2-1 Hirosawa, Wako, Saitama 351-0198, Japan. }

\email{eiji.inoue@riken.jp}

\begin{document}

\begin{abstract}
This is the second in a series of two papers studying $\mu$-cscK metrics and $\mu$K-stability from a new perspective, inspired by observations on $\mu$-character in \cite{Ino3} and on Perelman's $W$-entropy in the first paper \cite{Ino4}. 

This second paper is devoted to studying a non-archimedean counterpart of Perelman's $\mu$-entropy. 
The concept originally appeared as \textit{$\mu$-character} of polarized family in the previous research \cite{Ino3}, where we used it to introduce an analogue of CM line bundle adapted to $\mu$K-stability. 

We firstly show some differential of the \textit{characteristic $\mu$-entropy} $\cmu^\lambda$ is the minus of $\mu^\lambda$-Futaki invariant, which connects $\mu^\lambda$K-semistability to the maximization of characteristic $\mu^\lambda$-entropy. 
It in particular provides us a criterion for $\mu^\lambda$K-semistability working without detecting the vector $\xi$ involved in the $\mu^\lambda_\xi$-Futaki invariant. 
We observe a family of filtrations $\{ \mathcal{F}_{\xi + \tau (\mathcal{X}, \mathcal{L})} \}_{\tau \in [0,\infty)}$ associated to a test configuration $(\mathcal{X}, \mathcal{L})$ and a vector field $\xi$ acting on $(X, L)$ to consider the differential. 
We conceptualize such family of filtrations as \textit{polyhedral configuration} and study its generalities. 
The concept implicitly appeared in many literatures involved in $\mathbb{R}$-test configuration. 

In the latter part, we propose a non-archimedean pluripotential approach to the maximization problem. 
In order to adjust the characteristic $\mu$-entropy $\cmu^\lambda$ to Boucksom--Jonsson's non-archimedean framework, we introduce a natural modification $\NAmu^\lambda$ which we call \textit{non-archimedean $\mu$-entropy}. 
We extend the non-archimedean $\mu$-entropy from the set of test configurations to a space $\E^{\exp} (X, L)$ of non-archimedean psh metrics on the Berkovich space $X^{\mathrm{NA}}$, which is endowed with a complete metric structure. 
We introduce \textit{moment measure} $\int \chi \mathcal{D}_\varphi$ on Berkovich space for this sake, which can be considered as a hybrid of Monge--Amp\`ere measure and Duistermaat--Heckman measure. 

We also compare our $\mu$-framework with other frameworks: $H$-entropy framework in the context of K\"ahler--Ricci flow and Calabi energy framework in the context of Calabi flow. 
Some illustrations by toric examples are attached in Appendix. 
\end{abstract}

\maketitle

\tableofcontents

\section{Main results}

In this second paper of the series, we explore an invariant for test configuration called $\mu$-entropy in the first paper \cite{Ino4} and its connection to $\mu$K-semistability introduced in \cite{Ino3} (see also \cite{Lah1, Ino2}). 
This paper consists of three parts. 

In the first part, section 2, we study family of filtrations associated to equivariant family of polarized schemes over affine toric variety, which we call polyhedral configuration, and variation of $\mu$-entropy along such family of filtrations. 
In particular, we observe $\mu$-entropy maximization implies $\mu$K-semistability, which motivates us to study the existence and the uniqueness of maximizers of the $\mu$-entropy. 

To find a maximizer of a functional defined on an infinite dimensional space, it is often effective to study suitable completions and extension of the functional to the completions. 
The rest two parts are devoted to this attempt. 
We adapt our framework to Boucksom--Jonsson's non-archimedean pluripotential theory \cite{BJ1, BJ2, BJ3, BJ4}, which provides completions $\PSH (X, L), \E^1 (X, L), \ldots$ of the space of equivalence classes of filtrations/test configurations. 

Similarly as in the case of the non-archimedean Mabuchi invariant (a variant of Donaldson--Futaki invariant which fits into the non-archimedean formalism), to extend the $\mu$-entropy, we would express equivariant intersections by some integration on Berkovich space. 
Different from the case of Mabuchi invariant, the non-archimedean Monge--Amp\`ere measure $\mathrm{MA} (\varphi)$ is not suitable for our purpose, and there is another measure on Berkovich space concerned with equivariant intersections of higher moments. 
For a normal test configuration $(\mathcal{X}, \mathcal{L})$ and a Borel measurable function $\chi$ on $\mathbb{R}$, the measure is expressed as 
\[ \int \chi \mathcal{D}_{(\mathcal{X}, \mathcal{L})} = \sum_{E \subset \mathcal{X}_0} \mathrm{ord}_E \mathcal{X}_0 \int_\mathbb{R} \chi \DHm_{(E, \mathcal{L}|_E)} . \delta_{v_E}. \]
We call it \textit{moment measure}. 
For $\chi = 1_{\mathbb{R}}$, it gives the non-archimedean Monge--Amp\`ere measure. 
For $\chi = e^{-t}$, it provides a measure appropriate to the $\mu$-entropy. 

In the second part, section 3, we construct the measure $\int \chi \mathcal{D}_\varphi$ for general non-archimedean psh metric $\varphi \in \E^1 (X, L)$ of finite energy class. 
To construct the measure, we firstly observe the non-archimedean Monge--Amp\`ere measure $\mathrm{MA} (\varphi \wedge \tau)$ of the rooftop $\varphi \wedge \tau$ for non-archimedean metric $\varphi = \varphi_{(\mathcal{X}, \mathcal{L})} \in \nH (X, L)$ associated to test configuration and $\tau \in \mathbb{R}$. 
Similarly as the above expression of the moment measure, it is concerned with the primary decomposition of the Duistermaat--Heckman measure: $\DHm_{(\mathcal{X}, \mathcal{L})} = \sum_{E \subset \mathcal{X}_0} \mathrm{ord}_E \mathcal{X}_0 \cdot \DHm_{(E, \mathcal{L}|_E)}$. 
Secondly, we study a generalization of the Duistermaat--Heckman measure for test configuration to general non-archimedean psh metric $\varphi \in \PSH (X, L)$, based on the monotonic continuity of $\int_{[\tau, \infty)} \DHm_{(\mathcal{X}_i, \mathcal{L}_i)}$ along decreasing nets $\varphi_{(\mathcal{X}_i, \mathcal{L}_i)} \in \nH (X, L)$. 
Then the measure $\int \chi \mathcal{D}_\varphi$ is constructed in a measure theoretic way based on these observations. 
The total mass $\iint_{X^{\mathrm{NA}}} \chi \mathcal{D}_\varphi := \int_{X^{\mathrm{NA}}} \int \chi \mathcal{D}_\varphi$ is equal to $\int_{\mathbb{R}} \chi \DHm_\varphi$. 

In the last part, section 4, we firstly study a metric space $(\E^{\exp} (X, L), d_{\exp})$ consisting of non-archimedean psh metrics of finite \textit{exponential moment energy} $E_{\exp}$. 
The metric $d_{\exp}$ originates from Orlicz norm, a generalization of $L^p$-norm. 
The topology induced from $d_{\exp}$ is stronger than any $d_p$-topology for $1 \le p < \infty$ and is weaker than $d_\infty$-topology (uniform convergence). 
Under the hypothesis on the continuity of envelopes, which is valid for smooth $X$ (see section \ref{continuity of envelopes}), we show the completeness of the metric space. 
We then show the $\mu$-entropy $\NAmu^\lambda$ has a natural upper semi-continuous extension to $\E^{\exp} (X, L)$, which is finally expressed as 
\[ \NAmu^\lambda (\varphi) = - \frac{\int_{X^{\mathrm{NA}}} (2\pi A_X + \lambda \varphi) \int e^{-t} \mathcal{D}_\varphi + E_{\exp}^{2\pi K_X +\lambda L} (\varphi)}{\iint_{X^{\mathrm{NA}}} e^{-t} \mathcal{D}_\varphi} - \lambda \log \iint_{X^{\mathrm{NA}}} e^{-t} \mathcal{D}_\varphi. \]
% By our sign convention, which is synchronized with the sign convention on Perelman's $W$-entropy, the non-archimedean $\mu$-entropy becomes an upper semi-continuous functional on $\E^{\exp} (X, L)$. 
At last, we discuss maximization problem for the non-archimedean $\mu$-entropy and relation to other works. 
In Appendix, we observe some toric examples to illustrate our theory. 

In this article, we restrict our interest to schemes of locally finite type over the field $\mathbb{C}$, for which we have convergence results on equivariant intersections as we proved in \cite{Ino3}. 
A \textit{polarized scheme} (resp. \textit{polarized variety}) $(X, L)$ is a pair of a pure $n$-dimensional projective scheme (resp. projective integral scheme) $X$ over $\mathbb{C}$ and an ample $\mathbb{Q}$-line bundle $L$ over $X$: $L$ is a pair $L = (l, \hat{L}) =  (1/l) \hat{L}$ of a positive integer $l$ and an ample line bundle $\hat{L}$ on $X$. 
For instance, for a $\mathbb{Q}$-Fano variety $X$, we consider $L = -K_X = (l, (\omega_X^{\otimes l})^\vee)$ for sufficiently divisible $l$ so that $(\omega_X^{\otimes l})^\vee$ is invertible. 
We denote by $\mathbb{A}^1$ the affine space $\mathbb{C}$ and by $\mathbb{G}_m$ the multiplicative group $\mathbb{C}^\times$. 
Torus means algebraic torus $T = N \otimes \mathbb{G}_m$, where $N \cong \mathbb{Z}^r$ is a finite rank lattice. 
We denote by $\mathfrak{t} = N \otimes \mathbb{R}$ the associated real Lie algebra. 
We identify it with the Lie algebra of the maximal compactum $T_{\mathrm{cpt}} := N \otimes U (1)$ via $\xi \mapsto \frac{d}{dt}|_{t=0} \exp (t 2\pi \sqrt{-1} \xi)$. 
% We note the equivariant cohomology $H_{\mathbb{G}_m} (X) = H (X \times_{\mathbb{G}_m} E \mathbb{G}_m)$ is concerned with the usual Lie group topology on $\mathbb{G}_m$ (not Zariski topology). 

Now we explain the main results for each section. 

\subsection{Characteristic $\mu$-entropy and $\mu$K-semistability}

\subsubsection{Introduction to characteristic $\mu$-entropy}

In \cite{Ino3}, we introduced an equivariant characteristic class $\bm{\check{\mu}}_G^\lambda (\mathcal{X}/B, \mathcal{L}) \in \hat{H}_G (B, \mathbb{Q})$ for $G$-equivariant family of polarized schemes with intent to construct an analogy of CM line bundle in the context of $\mu$K-stability. 
For a test configuration $(\mathcal{X}, \mathcal{L})$, the base $B= \mathbb{A}^1$ is $\mathbb{G}_m$-equivariantly homotopic to a point with the trivial $\mathbb{G}_m$-action, so that we can identify the characteristic class $\bm{\check{\mu}}_{\mathbb{G}_m}^\lambda (\mathcal{X}/B, \mathcal{L}) \in \hat{H}_{\mathbb{G}_m} (\mathbb{A}^1, \mathbb{Q}) = \hat{H} (\mathbb{C}P^\infty, \mathbb{Q})$ with an infinite power series $\sum_{i=1} a_i x^i \in \mathbb{Q} \llbracket x \rrbracket$ by identifying $x$ with $c_1 (\mathcal{O} (-1)) \in H^2 (\mathbb{C}P^\infty)$, which is the equivariant first Chern class $c_{1, \mathbb{G}_m} (\mathbb{C}_1) \in H^2_{\mathbb{G}_m} (\mathrm{pt})$ of the weight one representation $\mathbb{C}_1$ of $\mathbb{G}_m$. 
It is shown in \cite{Ino3} that this infinite series is compactly absolutely-convergent to a real analytic function on $\mathbb{R}$. 
For a normal test configuration $(\mathcal{X}, \mathcal{L})$, the functional can be expressed as 
\begin{align*}
\cmu (\mathcal{X}, \mathcal{L}; \rho) 
&= \frac{(K_X. e^L) - \rho (K_{\bar{\mathcal{X}}/\mathbb{P}^1}^{\mathbb{G}_m}. e^{\bar{\mathcal{L}}_{\mathbb{G}_m}}; \rho)}{\int_\mathbb{R} e^{-\rho t} \DHm_{(\mathcal{X}, \mathcal{L})}}, 
\\
\bm{\check{\sigma}} (\mathcal{X}, \mathcal{L}; \rho) 
&= \frac{\int_\mathbb{R} (n-\rho t) e^{-\rho t} \DHm_{(\mathcal{X}, \mathcal{L})}}{\int_\mathbb{R} e^{-\rho t} \DHm_{(\mathcal{X}, \mathcal{L})}} - \log \int_\mathbb{R} e^{-\rho t} \DHm_{(\mathcal{X}, \mathcal{L})}, 
\\
\cmu^\lambda (\mathcal{X}, \mathcal{L}; \rho) 
&= \cmu (\mathcal{X}, \mathcal{L}; \rho) + \lambda \bm{\check{\sigma}} (\mathcal{X}, \mathcal{L}; \rho)
\end{align*}
for $\rho \in \mathbb{R}$. 
Here 
\begin{itemize}
\item we put $(e^L) = (L^{\cdot n})/n!$ and $(K_X. e^L) = (K_X. L^{\cdot n-1})/(n-1)!$, 

\item the $\mathbb{G}_m$-equivariant intersection $(K_{\bar{\mathcal{X}}/\mathbb{A}^1}^{\mathbb{G}_m}. e^{\bar{\mathcal{L}}_{\mathbb{G}_m}}; \rho)$ on the compactification $(\bar{\mathcal{X}}, \bar{\mathcal{L}})$ is defined by some infinite sum of cup products of equivariant cohomology classes (see section \ref{equivariant intersection} for the precise definition), 

\item the Duistermaat--Heckman measure $\DHm_{(\mathcal{X}, \mathcal{L})}$ is associated to the $\mathbb{G}_m$-action on the central fibre and is normalized by $\int_\mathbb{R} \DHm_{(\mathcal{X}, \mathcal{L})} = (e^L)$. 
\end{itemize}
We call this $\cmu^\lambda (\mathcal{X}, \mathcal{L}; \rho)$ the \textit{characteristic $\mu$-entropy}, distinguishing it with the non-archimedean $\mu$-entropy we later introduce. 

By the equivariant localization and the equivariant Grothendieck--Riemann--Roch theorem, we can localize the equivariant intersection to the central fibre (cf. Definition \ref{mu-entropy of polyhedral configuration} and Proposition \ref{Localization}): 
\[ (K_X. e^L) - \rho (K_{\bar{\mathcal{X}}/\mathbb{P}^1}^{\mathbb{G}_m}. e^{\bar{\mathcal{L}}_{\mathbb{G}_m}}; \rho) = (\kappa_{\mathcal{X}_0}^{\mathbb{G}_m}. e^{\mathcal{L}_{\mathbb{G}_m}}; \rho), \]
using a $\mathbb{G}_m$-equivariant homology class $\kappa_{\mathcal{X}_0}^{\mathbb{G}_m}$ derived from the equivariant homology Todd class $\tau_{\mathcal{X}_0}^{\mathbb{G}_m} (\mathcal{O}_{\mathcal{X}_0})$ (cf. \cite{EG2, Ino3}). 
Or conversely, we can express the integrations by equivariant intersections on the compactification $\bar{\mathcal{X}}$: 
\begin{align*} 
\int_\mathbb{R} e^{-\rho t} \DHm_{(\mathcal{X}, \mathcal{L})} 
&= (\mathcal{X}_0^{\mathbb{G}_m}. e^{\mathcal{L}_{\mathbb{G}_m}}; \rho) =  (e^L) - \rho (e^{\bar{\mathcal{L}}_{\mathbb{G}_m}}; \rho), 
\\
\int_\mathbb{R} (n-\rho t) e^{-\rho t} \DHm_{(\mathcal{X}, \mathcal{L})}
&= (\mathcal{X}_0^{\mathbb{G}_m}. \mathcal{L}_{\mathbb{G}_m}. e^{\mathcal{L}_{\mathbb{G}_m}}; \rho) = (L. e^L) - \rho (\bar{\mathcal{L}}_{\mathbb{G}_m}. e^{\bar{\mathcal{L}}_{\mathbb{G}_m}}; \rho). 
\end{align*}
Here we note $(L. e^L) = (L^{\cdot n})/(n-1)!$. 
We remind 
\begin{equation}
\label{DH measure and equivariant intersection} 
\frac{1}{k!} \int_\mathbb{R} (-\rho t)^k \DHm_{(\mathcal{X}, \mathcal{L})} = \frac{(\mathcal{L}^{\cdot n+k}_{\mathbb{G}_m}|_{\mathcal{X}_0}; \rho)}{(n+k)!} = \frac{(L^{\cdot n+k})}{(n+k)!} -\frac{\rho (\bar{\mathcal{L}}_{\mathbb{G}_m}^{\cdot n+k}; \rho)}{(n+k)!}, 
\end{equation}
where $\frac{(L^{\cdot n+k})}{(n+k)!} = 0$ for $k > 0$ and $\frac{\rho (\bar{\mathcal{L}}_{\mathbb{G}_m}^{\cdot n+k}; \rho)}{(n+k)!} = 0$ for $k = 0$. 
We explain the precise definition of these equivariant intersections in section \ref{equivariant intersection}. 

A \textit{proper vector} (or often called just vector in this article) on $(X, L)$ is an element $\xi \in \mathfrak{t}$ of the real Lie algebra of a torus $T$ acting on $(X, L)$. 
For integral $\xi \in N \subset \mathfrak{t}$, we can assign a product configuration $X \times \mathbb{A}^1$ endowed with the diagonal action of $\mathbb{G}_m$ derived from the group homomorphism $\mathbb{G}_m \to T$ associated to $\xi$. 
We can define the characteristic $\mu$-entropy $\cmu^\lambda (X, L; \xi)$ for a proper vector $\xi$, using the expression localized to the central fibre. 
This is compatible with the above characteristic $\mu$-entropy for product configuration associated to integral (or rational) $\xi$. 
This is the original form of characteristic $\mu$-entropy appeared in \cite{Ino2} as $\mu$-volume functional, in which we generalize Tian--Zhu's volume minimization argument for K\"ahler--Ricci soliton to $\mu$-cscK metric. 
It is proved there that there exists a proper vector maximizing $\cmu (X, L; \bullet)$ among all proper vectors when $X$ is smooth. 
We will see the properness can be extended to $X$ with klt singularities. 

As explained in section \ref{Filtration}, we can assign filtrations for test configuration and proper vector. 
This leads to the following study. 

\subsubsection{Characteristic $\mu$-entropy and $\mu$K-semistability}

We firstly study the characteristic $\mu$-entropy for finitely generated filtrations, generalizing the above description. 
We begin with studying the characteristic $\mu$-entropy of some geometric object $(\mathcal{X}/B_\sigma, \mathcal{L}; \xi)$, which we call \textit{polyhedral configuration} (see Definition \ref{polyhedral configuration} and Definition \ref{mu-entropy of polyhedral configuration}). 
It then turns out that any finitely generated filtration is associated to some polyhedral configuration (no unique choice) in Proposition \ref{polyhedral configuration and finite generation of filtration}, and the characteristic $\mu$-entropy is indeed an invariant for filtrations in Proposition \ref{chmu is an invariant for filtration}. 
Studying a variation of the characteristic $\mu$-entropy along geometric family $\{ \mathcal{F}_{(\mathcal{X}/B_\sigma, \mathcal{L}; \xi)} \}_{\xi \in \sigma}$ of filtrations, we obtain the following criterion for $\mu$K-semistability. 

\begin{thm}[Summary of section \ref{section: mu-entropy of polyhedral configuration} and section \ref{maximizing non-archimedean mu-entropy}]
\label{characteristic mu maximization implies muK-semistability}
Let $(X, L)$ be a polarized scheme. 
Assume one of the following: 
\begin{enumerate}
\item There exists a proper vector $\xi_{\mathrm{opt}}$ on $(X, L)$ such that 
\[ \cmu^\lambda (X, L; \xi_{\mathrm{opt}}) \ge \cmu^\lambda (\mathcal{X}, \mathcal{L}; \rho) \]
for every test configuration $(\mathcal{X}, \mathcal{L})$ and $\rho \in \mathbb{Q}_{\ge 0}$. 

\item $X$ is a normal variety with only klt singularities, and for each test configuration $(\mathcal{X}, \mathcal{L}; \rho)$ there exists a proper vector $\xi$ on $(X, L)$ such that 
\[ \cmu^\lambda (X, L; \xi) \ge \cmu^\lambda (\mathcal{X}, \mathcal{L}; \rho). \]
\end{enumerate}
Then $(X, L)$ is $\check{\mu}^\lambda$K-semistable for any proper vector $\xi_{\mathrm{opt}}$ which maximizes the characteristic $\mu$-entropy among all proper vectors/test configurations/finitely generated filtrations. 
\end{thm}

This result is an algebraic counterpart of Theorem 1.4 in \cite{Ino4} on Perelman's $\mu$-entropy. 
We will reinterpret this theorem as Theorem \ref{NAmu entropy and muK-stability} in terms of non-archimedean formalism. 
The latter criterion is more pragmatic: we no longer need to detect the (transcendental) vector $\xi$ for which $(X, L)$ must be $\check{\mu}^\lambda_\xi$K-semistable to check its $\mu^\lambda$K-semistability. 
We only need to find a vector $\xi$ \textit{for each test configuration} $(\mathcal{X}, \mathcal{L}; \rho)$ so that the above inequality holds. 
To show the latter claim, we prove the properness of the $\mu$-entropy $\cmu^\lambda (X, L; \bullet)$ for proper vectors. 
This is proved for smooth $X$ in \cite{Ino2} in a differential geometric way and will be proved for klt $X$ in section \ref{maximizing non-archimedean mu-entropy} after establishing some non-archimedean pluripotential theoretic formulae.

The following gives an extension of the above theorem. 
Analogous results for other frameworks are known by \cite{Der2} and \cite{HL2}: the central fibre of an `optimal degeneration' is `semistable' in a suitable sense depending on the framework.  

\begin{thm}[Summary of section \ref{section: mu-entropy of polyhedral configuration} and section \ref{mu-entropy via associated filtration}]
\label{NAmu maximizer}
Let $(X, L)$ be a polarized scheme. 
If a finitely generated filtration $\mathcal{F}$ maximizes $\cmu^\lambda$ among all finitely generated filtrations (or test configurations), then the central fibre $(\mathcal{X}_o (\mathcal{F}), \mathcal{L}_o (\mathcal{F}))$ of $\mathcal{F}$ is $\check{\mu}^\lambda$K-semistable with respect to the proper vector $\xi_o$ on $\mathcal{X}_o (\mathcal{F})$ induced by the filtration $\mathcal{F}$. 
\end{thm}

See section \ref{Filtration} and Definition \ref{f.g. filtration} for finitely generated filtration. 
The central fibre $(\mathcal{X}_o (\mathcal{F}), \mathcal{L}_o (\mathcal{F}))$ of a finitely generated filtration $\mathcal{F}$ is defined by 
\begin{align} 
\mathcal{R}_o (\mathcal{F}) 
&:= \bigoplus_{m \in \mathbb{N}} \bigoplus_{\lambda \in \mathbb{R}} \varpi^{-\lambda} \mathcal{F}^\lambda R_m/\mathcal{F}^{\lambda+} R_m, 
\\
(\mathcal{X}_o (\mathcal{F}), \mathcal{L}_o (\mathcal{F})) 
&:= \Proj \mathcal{R}_o (\mathcal{F}). 
\end{align}
By Proposition \ref{central fibre via filtration}, the central fibre $(\mathcal{X}_o (\mathcal{F}), \mathcal{L}_o (\mathcal{F}))$ can be identified with the central fibre of a polyhedral configuration which represents the filtration.

%\begin{thm}
%If a polyhedral configuration $(\mathcal{X}/B_\sigma, \mathcal{L}; \xi)$ maximizes $\cmu^\lambda$, then the central fibre $(\underline{X}, \underline{L}) = (\mathcal{X}_o, \mathcal{L}|_{\mathcal{X}_o})$ is $\check{\mu}^\lambda_\xi$K-semistable with respect to all $T$-equivariant test configurations for $T = (\overline{\exp \mathbb{R} \xi})_\mathbb{C} \subset \mathrm{Aut} (\underline{X}, \underline{L})$. 
%\end{thm}

\subsubsection{Non-archimedean $\mu$-entropy of test configuration}
\label{non-archimedean mu-entropy of test configuration}

The characteristic $\mu$-entropy $\cmu^\lambda$ gives a right concept in view of GIT on Hilbert scheme as studied in \cite{Ino3}. 
However, similarly as Donaldson--Futaki invariant, this notion does not fits into Boucksom--Jonsson's non-archimedean pluripotential theory because it is not well-behaved with respect to the normalized base change along $z^d: \mathbb{A}^1 \to \mathbb{A}^1$ due to bad behavior of the canonical divisor: $\cmu^\lambda (\mathcal{X}_d, \mathcal{L}_d; \rho) \neq \cmu^\lambda (\mathcal{X}, \mathcal{L}; d\rho)$. 
This prevents us to interpret $\cmu^\lambda$ as a functional for non-archimedean metrics since the attempt $\cmu^\lambda (\varphi_{(\mathcal{X}, \mathcal{L}; \rho)}) := \cmu^\lambda (\mathcal{X}, \mathcal{L}; \rho)$ is not well-defined for the associated non-archimedean metric $\varphi_{(\mathcal{X}, \mathcal{L}; \rho)} = \varphi_{(\mathcal{X}_d, \mathcal{L}_d; d^{-1} \rho)}$. 

Similarly as Mabuchi invariant, we can refine this by using the equivariant log canonical divisor: 
\[ K_{\bar{\mathcal{X}}/\mathbb{P}^1}^{\log, \mathbb{G}_m} := (K_{\bar{\mathcal{X}}}^{\mathbb{G}_m} + [\mathcal{X}_0^{\mathrm{red}, \mathbb{G}_m}]) - \varpi^* (K_{\mathbb{P}^1}^{\mathbb{G}_m} + [0^{\mathbb{G}_m}]) \in H_{2n}^{\mathbb{G}_m} (\bar{\mathcal{X}}, \mathbb{Z}). \]
The following variant fits into the non-archimedean formalism: 
\begin{align}
\NAmu (\mathcal{X}, \mathcal{L}; \rho) 
&:= 2 \pi \frac{(K_X. e^L) - \rho (K_{\bar{\mathcal{X}}/\mathbb{P}^1}^{\log, \mathbb{G}_m}. e^{\bar{\mathcal{L}}_{\mathbb{G}_m}}; \rho)}{(e^L) - \rho (e^{\bar{\mathcal{L}}_{\mathbb{G}_m}}; \rho)}, 
\\
\NAmu^\lambda (\mathcal{X}, \mathcal{L}; \rho) 
&:= \NAmu^\lambda (\mathcal{X}, \mathcal{L}; \rho) + \lambda \bm{\check{\sigma}} (\mathcal{X}, \mathcal{L}; \rho) 
\end{align}
for a normal test configuration $(\mathcal{X}, \mathcal{L}; \rho)$. 
Since $((\mathcal{X}_0 - \mathcal{X}_0^{\mathrm{red}}). e^{\mathcal{L}}; \rho) \ge 0$, we have $\NAmu^\lambda (\mathcal{X}, \mathcal{L}; \rho) \ge \cmu^\lambda (\mathcal{X}, \mathcal{L}; \rho)$ in general, where the equality holds when the central fibre is reduced. 
Similarly as non-archimedean Mabuchi invariant (cf. \cite[Proposition 7.14]{BHJ1}), we have $\NAmu^\lambda (\mathcal{X}_d, \mathcal{L}_d; \rho) = \NAmu^\lambda (\mathcal{X}, \mathcal{L}; d\rho)$ as we explain below. 

We observe the normalized base change behavior of the $\mathbb{G}_m$-equivariant log canonical homology class $K^{\log, \mathbb{G}_m}_{\bar{\mathcal{X}}} \in H_{2n-2}^{\mathbb{G}_m} (\bar{\mathcal{X}})$. 
In the non-equivariant case, it is observed in \cite[section 3]{LX}. 
(Note $K^{\log, \mathbb{G}_m}_{\bar{\mathcal{X}}}$ is not a divisor on $\bar{\mathcal{X}}$. 
It is a divisor/reflexive sheaf on $\bar{\mathcal{X}} \times_{\mathbb{G}_m} E_l \mathbb{G}_m$. 
See section \ref{equivariant intersection}. )
We write the effective divisor $\mathcal{X}_0$ as $\sum_i d_i E_i$ and put $Z := \mathcal{X}^{\mathrm{sing}} \cup \bigcup_i E_i^{\mathrm{sing}} \cup (\bigcup_{i \neq j} E_i \cap E_j) \subset \mathcal{X}$. 
Around $E_i \setminus Z$, the test configuration $\mathcal{X} \to \mathbb{A}^1$ is locally expressed as $\Delta^{n+1} \to \mathbb{A}^1: (z_i) \mapsto z_i^{d_i}$ (implicit function theorem), so that the normalized base change is locally expressed in the following diagram: 
\begin{equation} 
\label{normalized base change diagram}
\begin{tikzcd}
\sqcup^{(d, d_i)} \Delta^{n+1} \ar{rr}{(z_0, \ldots, z_i^{\frac{d}{(d, d_i)}}, \ldots, z_n)} \ar{d}[swap]{z_i^{\frac{d_i}{(d, d_i)}}} 
&~ 
& \Delta^{n+1} \ar{d}{z_i^{d_i}}
\\
\mathbb{A}^1 \ar{rr}{w^d}
&~ 
& \mathbb{A}^1 
\end{tikzcd} 
\end{equation}
We put $\mathcal{X}^\circ := \mathcal{X} \setminus Z$ and $\mathcal{X}_d^\circ := \mathcal{X}_d \setminus \nu_d^{-1} Z$, where $\nu_d: \mathcal{X}_d \to \mathcal{X}$ is the normalized base change morphism. 
The $\mathbb{G}_m$-equivariant Chow class $K^{\log, \mathbb{G}_m}_{\mathcal{X}}$ (resp. $K^{\log, \mathbb{G}_m}_{\mathcal{X}_d}$) is the push-forward of the $\mathbb{G}_m$-equivariant Chern class of the log cotangent bundle $T^{\log, *} \mathcal{X}^\circ$ (resp. $T^{\log, *} \mathcal{X}_d^\circ$), which is locally spanned by $dz_1, \ldots, z_i^{-1} dz_i, \ldots, dz_n$ around the boundary $E_i \setminus Z$. 
By the above local expression, we deduce that the derivative of $\nu_d$ induce the isomorphism of log tangent bundles $\nu_{d, *}: T^{\log} \mathcal{X}_d^\circ \cong \nu_d^* T^{\log} \mathcal{X}^\circ$. 
(The derivative $\nu_{d, *}$ does not induce an isomorphism of the usual tangent bundles $T \mathcal{X}_d^\circ, T \mathcal{X}^\circ$ as $\nu_d$ ramifies along the central fibre. )
Since the derivative is functorially given, $\nu_{d, *}$ is equivariant with respect to the $d$-times scaled $\mathbb{G}_m$-action on $\nu_d^* T^{\log} \mathcal{X}^\circ$. 
(Note $\nu_d$ is equivariant with respect to $t^d: \mathbb{G}_m \to \mathbb{G}_m$. ) 
As $Z$ has codimension greater than one, we get $(\nu_d)_* K_{\bar{\mathcal{X}}_d}^{\log, \mathbb{G}_m} = d K_{\bar{\mathcal{X}}}^{\log, \mathbb{G}_m}$ as $\mathbb{G}_m$-equivariant Chow classes, with $d$-times scaled $\mathbb{G}_m$-action on $\bar{\mathcal{X}}$. 
Therefore, we obtain 
\[ (K_{\bar{\mathcal{X}}_d/\mathbb{P}^1}^{\log, \mathbb{G}_m}. e^{\bar{\mathcal{L}}_d}; \rho) = (d K_{\bar{\mathcal{X}}/\mathbb{P}^1}^{\log, \mathbb{G}_m}. e^{\bar{\mathcal{L}}}; d\rho) \]
by the equivariant projection formula. 
This shows $\NAmu^\lambda (\mathcal{X}_d, \mathcal{L}_d; \rho) = \NAmu^\lambda (\mathcal{X}, \mathcal{L}; d\rho)$ as desired.

\subsubsection{Towards non-archimedean formalism: moment measure of test configuration}
\label{Towards non-archimedean formalism: moment measure of test configuration}

Now we explain how we use the moment measure $\int \chi \mathcal{D}_{(\mathcal{X}, \mathcal{L}; \rho)}$. 
We put 
\begin{equation}
\label{moment measure for test configuration} 
\int \chi \mathcal{D}_{(\mathcal{X}, \mathcal{L}; \rho)} := \sum_{E \subset \mathcal{X}_0} \mathrm{ord}_E \mathcal{X}_0 \int_\mathbb{R} \chi (\rho t) \DHm_{(E, \mathcal{L}|_E)} (t) . \delta_{\rho. v_E}
\end{equation}
for a normal test configuration $(\mathcal{X}, \mathcal{L})$. 
Here $v_E$ denotes the valuation on $X$ associated to the prime divisor $E \subset \mathcal{X}$ as in Example \ref{valuation associated to prime divisor of test configuration}, and $\delta_{v_E}$ denotes the Dirac measure on the space of valuations $\mathrm{Val} (X)$ charging $v_E$. 

Assume $X$ is log canonical, in particular $K_X$ is $\mathbb{Q}$-Cartier. 
Take a resolution $\beta: \tilde{\mathcal{X}} \to \bar{\mathcal{X}}$ so that the canonical rational map $p_X: \tilde{\mathcal{X}} \dashrightarrow X \times \mathbb{P}^1 \to X$ extends uniquely to a morphism of schemes. 
Then as in \cite[Proposition 4.11]{BHJ1}, we have 
\[ K_{\tilde{\mathcal{X}}/\mathbb{P}^1}^{\log, \mathbb{G}_m} - p_X^* K_X^{\mathbb{G}_m} = \sum_{E \subset \tilde{\mathcal{X}}_0} A_X (v_E) \mathrm{ord}_E \mathcal{X}_0 .E^{\mathbb{G}_m} \]
as $\mathbb{G}_m$-equivariant classes, where $E^{\mathbb{G}_m}$ denotes the $\mathbb{G}_m$-equivariant class on $\tilde{\mathcal{X}}$ associated to the $\mathbb{G}_m$-invariant irreducible component $E$ of $\tilde{\mathcal{X}}_0$. 
It follows that we can express $\NAmu (\mathcal{X}, \mathcal{L}; \rho)$ as 
\[ -2\pi \frac{\sum_E A_X (\rho. v_E) \mathrm{ord}_E \mathcal{X}_0 \cdot (E^{\mathbb{G}_m}. e^{\mathcal{L}_{\mathbb{G}_m}}; \rho)}{(e^{\mathcal{L}_{\mathbb{G}_m}|_{\mathcal{X}_0}}; \rho)} + 2 \pi \frac{(K_X. e^L) - \rho (p_X^* K_X^{\mathbb{G}_m}. e^{\beta^* \bar{\mathcal{L}}_{\mathbb{G}_m}}; \rho)}{(e^{\mathcal{L}_{\mathbb{G}_m}|_{\mathcal{X}_0}}; \rho)}. \]

Now since 
\[ (E^{\mathbb{G}_m}. e^{\mathcal{L}_{\mathbb{G}_m}}; \rho) = \int_\mathbb{R} e^{-\rho t} \DHm_{(E, \mathcal{L}|_E)} (t), \]
we have 
\begin{gather*} 
\sum_E A_X (\rho. v_E) \mathrm{ord}_E \mathcal{X}_0 \cdot (E^{\mathbb{G}_m}. e^{\mathcal{L}_{\mathbb{G}_m}}; \rho) = \int_{\mathrm{Val} (X)} A_X \int e^{-t} \mathcal{D}_{(\mathcal{X}, \mathcal{L}; \rho)}, 
\\
(\mathcal{X}_0^{\mathbb{G}_m}. e^{\mathcal{L}_{\mathbb{G}_m}}; \rho) = \int_{\mathrm{Val} (X)} \int e^{-t} \mathcal{D}_{(\mathcal{X}, \mathcal{L}; \rho)} =: \iint_{\mathrm{Val} (X)}  e^{-t} \mathcal{D}_{(\mathcal{X}, \mathcal{L}; \rho)}. 
\end{gather*}
Thus we get the following expression of $\NAmu (\mathcal{X}, \mathcal{L}; \rho)$: 
\[ -2\pi \frac{\int_{\mathrm{Val} (X)} A_X \int e^{-t} \mathcal{D}_{(\mathcal{X}, \mathcal{L}; \rho)}}{\iint_{\mathrm{Val} (X)} e^{-t} \mathcal{D}_{(\mathcal{X}, \mathcal{L}; \rho)}} - 2\pi \frac{E^{K_X}_{\exp} (\mathcal{X}, \mathcal{L}; \rho)}{\iint_{\mathrm{Val} (X)} e^{-t} \mathcal{D}_{(\mathcal{X}, \mathcal{L}; \rho)}}, \]
where we put 
\begin{equation} 
\label{moment M-energy}
E_{\exp}^M (\mathcal{X}, \mathcal{L}; \rho) := -\big{(} (M. e^L) - \rho (p_X^* M_{\mathbb{G}_m}. e^{\beta^* \bar{\mathcal{L}}_{\mathbb{G}_m}}; \rho) \big{)}
\end{equation}
for a $\mathbb{Q}$-line bundle $M$ on $X$. 
We have a similar expression on $\bm{\check{\sigma}}$ as we will observe in Corollary \ref{sigma formula}. 

\begin{rem}
We can find an analogy to moment map in the moment measure $\mathcal{D}_{(\mathcal{X}, \mathcal{L}; \rho)}$ as follows. 
For a moment map $\mu: X \to \mathbb{R}$ of a $U (1)$-invariant K\"ahler metric $\omega$, consider the measure $\mathcal{D}_{\omega, \mu} = (\mathrm{id}_X \times \mu)_* (\omega^n/n!)$ on $X \times \mathbb{R}$. 
Since the support of the measure is the graph of the moment map, we can recover the moment map from this measure. 
For continuous functions $\chi$ on $\mathbb{R}$ and $g$ on $X$, we have 
\[ \int_X g \int_\mathbb{R} \chi \mathcal{D}_{\omega, \mu} = \int_X g \cdot \chi (\mu) ~ \omega^n/n!. \]

We speculate for subgeodesic rays $\{ \phi_s \}_{s \in [0, \infty)}, \{ \psi_s \}_{s \in [0, \infty)} \subset \mathcal{E}^1 (X, \omega)$ subordinate to $\varphi, \psi \in \E^1 (X, L)$, the following holds 
\[ \int_{X^{\mathrm{NA}}} \psi \int \frac{(-t)^k}{k!} \mathcal{D}_\varphi = \lim_{s \to \infty} - \int_X (dd^c \psi - \pi \dot{\psi}_s) \wedge \frac{(dd_\omega^c \phi_s -\pi \dot{\phi}_s)^{n+k}}{(n+k)!}, \]
in particular 
\[ \int_{X^{\mathrm{NA}}} \psi \int e^{-t} \mathcal{D}_\varphi = \lim_{s \to \infty} - \int_X (dd^c \psi - \pi \dot{\psi}_s) \wedge e^{dd_\omega^c \phi_s -\pi \dot{\phi}_s}. \]
Compare the formula (\ref{DH measure and equivariant intersection}) and \cite[Proposition 3.19]{Ino4}

%There is a proper continuous map $\pi: X^{\mathrm{hyb}} \to [0,1]$ from a Hausdorff space $X^{\mathrm{hyb}}$ with a homeomorphism $j: X \times (0, 1] \to \pi^{-1} ((0,1])$ over $(0,1]$ and a homeomorphism $j_0: X^{\mathrm{NA}} \to \pi^{-1} (0)$ (cf. \cite{BJ}). 
%We can check the fibrewise measure 
%\[ (j_t)_* (f_t \times \rho)_* \mathcal{D}_{\omega, \mu} = (j_t)_* (\mathrm{id}_X \times (\rho \mu))_* (f_t)_* (\omega^n/n!) \] 
%on $X^{\mathrm{hyb}} \times \mathbb{R}$ converges to $\mathcal{D}_{(\mathcal{X}_\Lambda, \mathcal{L}_\Lambda; \rho)}$ for the product test configuration $(\mathcal{X}_\Lambda, \mathcal{L}_\Lambda)$ associated to the $U (1)$-action. 
\end{rem}

The moment measure should be not confused with the weighted non-archimedean Monge--Ampere measure constructed in \cite{HL1}. 
The former reflects the higher moments of $\mathbb{G}_m$-action on test configuration and the latter reflects the higher moments of $T$-action on $T$-equivariant test configuration. 
The crucial difference is that $\mathbb{G}_m$ acts non-trivially even on the base $\mathbb{A}^1$ while $T$ acts only on $\mathcal{X}$ fibrewisely over $\mathbb{A}^1$. 

\subsection{Tomography of non-archimedean Monge--Amp\`ere measure}

The statements here are described based on Boucksom--Jonsson's global non-archimedean pluripotential theory \cite{BJ1, BJ2, BJ3, BJ4}. 
We review various terminologies and notations in section \ref{Non-archimedean pluripotential theory}. 
Here we just recall $X^{\mathrm{NA}}$ denotes the Berkovich space associated to $X$, $\PSH (X, L)$ denotes the set of non-archimedean psh metrics, $\E^1 (X, L)$ denotes the finite energy class, $\nH^\mathbb{R} (X, L)$ denotes the set of non-archimedean psh metrics assigned to finitely generated filtrations and $\nH (X, L)$ denotes the set of non-archimedean psh metrics assigned to test configurations: 
\[ \nH (X, L) \subset \nH^\mathbb{R} (X, L) \subset \E^1 (X, L) \subset \PSH (X, L), \]
which are all subsets of 
\[ \{ \text{ upper semi-continuous functions on } X^{\mathrm{NA}} \}. \]

\subsubsection{Duistermaat--Heckman measure}

In section \ref{Moment energy and Duistermaat--Heckman measure}, we study a generalization of the Duistermaat--Heckman measure for test configurations to non-archimedean psh metrics. 
This is used in the construction of moment measure. 

\begin{rem}
Recently, M. Xia \cite{Xia2} also constructed the Duistermaat--Heckman measure for $\varphi \in \E^1 (X, L)$. 
Though the construction is different from ours, both constructions give the same measure as these are continuous along decreasing nets $\varphi_i \searrow \varphi$. 
Our construction is based on the monotonic continuity of moment energy $E_\chi (\varphi)$ along decreasing nets $\varphi_i \searrow \varphi$, which is observed in section \ref{subsection: Moment energy}. 
Xia's construction is based on observation on Okounkov body and is concerned with the associated (archimedean) geodesic ray via $\E^1 (X, L) \hookrightarrow \mathcal{R}^1 (X, \omega)$ (cf. \cite{BBJ}). 
\end{rem}

\begin{thm}[Summary of section \ref{Duistermaat--Heckman measure of non-archimedean psh metric}, \ref{Duistermaat--Heckman measure and moment energy} and section \ref{The continuity of Duistermaat--Heckman measure with respect to d1-topology}]
Let $(X, L)$ be a polarized variety. 
For a non-archimedean psh metric $\varphi \in \PSH (X, L)$, we can assign a finite Borel measure $\DHm_\varphi$ on $\mathbb{R}$ with total mass $\int_\mathbb{R} \DHm_\varphi \le (e^L)$ which is characterized by the following properties: 
\begin{itemize}
\item For $\varphi_{(\mathcal{X}, \mathcal{L})} \in \nH (X, L)$, we have $\DHm_{\varphi_{(\mathcal{X}, \mathcal{L})}} = \DHm_{(\mathcal{X}, \mathcal{L})}$. 

\item For $\varphi_i \searrow \varphi \in \PSH (X, L)$, we have $\int_{[\tau, \infty)} \DHm_{\varphi_i} \to \int_{[\tau, \infty)} \DHm_\varphi$. 
\end{itemize}

Moreover, we have $\int_\mathbb{R} \chi \DHm_{\varphi_i} \to \int_\mathbb{R} \chi \DHm_\varphi$ in the following cases: 
\begin{enumerate}
\item $\chi$ is tame in the sense of Definition \ref{tame} and $\varphi_i \searrow \varphi \in \PSH (X, L)$. 

\item $\chi$ is moderate in the sense of Definition \ref{moderate} and $\varphi_i \searrow \varphi \in \E^1 (X, L)$. 

\item $\chi$ is continuous and has left bounded support, and a net $\{ \varphi_i \}_{i \in I} \in \E^1 (X, L)$ converges to $\varphi \in \E^1 (X, L)$ in the strong topology.   
\end{enumerate}
\end{thm}

For general $\varphi \in \PSH (X, L)$, we may have $\int_\mathbb{R} \DHm_\varphi < (e^L)$ as $\chi = 1_{\mathbb{R}}$ is not tame, while $\int_\mathbb{R} \DHm_\varphi = (e^L)$ for $\varphi \in \E^1 (X, L)$. 
This is reminiscent of the fact that the (archimedean) Monge--Amp\`ere measure for general psh metric may lose mass (cf. \cite[Section 10]{GZ}). 

We also introduce moment energy $E_\chi (\varphi)$ and subspaces $\E^\chi (X, L) \subset \E^1 (X, L)$ (see (\ref{moment energy}) and Definition \ref{finite moment energy class}) for non-constant increasing concave function $\chi$ on $\mathbb{R}$. 
In this article, increasing (resp. decreasing) means $\chi (t) \le \chi (t')$ for $t \le t'$ (resp. $\chi (t) \ge \chi(t')$ for $t \le t'$). 

\subsubsection{Moment measure}

As we observe, the following construction is the key for the extension of non-archimedean $\mu$-entropy. 

\begin{thm}[Summary of section \ref{Moment measure}]
Let $(X, L)$ be a polarized variety. 
For $\varphi \in \mathcal{E}^1_{\mathrm{NA}} (X, L)$ and a Borel measurable function $\chi$ on $\mathbb{R}$ with $\int_{\mathbb{R}} |\chi| \DHm_\varphi < \infty$, we can assign a signed Radon measure $\int \chi \mathcal{D}_\varphi$ on the Berkovich space $X^{\mathrm{NA}}$ which enjoys the following properties: 
\begin{enumerate}
\item For $\varphi_{(\mathcal{X}, \mathcal{L})} \in \nH (X, L)$ represented by a normal test configuration $(\mathcal{X}, \mathcal{L})$, we have 
\[ \int \chi \mathcal{D}_{\varphi_{(\mathcal{X}, \mathcal{L})}} = \sum_{E \subset \mathcal{X}_0} \mathrm{ord}_E \mathcal{X}_0 \int_{\mathbb{R}} \chi \DHm_{(E, \mathcal{L}|_E)}. \delta_{v_E}. \]

\item $\int \chi \mathcal{D}_\varphi$ is linear on $\chi$. 
If $\chi \ge 0$, the measure $\int \chi \mathcal{D}_\varphi$ is non-negative. 

\item For any pointwise convergent increasing sequence $0 \le \chi_i \nearrow \chi$, we have the weak convergence of measures
\[ \int \chi_i \mathcal{D}_\varphi \nearrow \int \chi \mathcal{D}_\varphi. \]

\item We have $\iint_{X^{\mathrm{NA}}} \chi \mathcal{D}_\varphi := \int_{X^{\mathrm{NA}}} \int \chi \mathcal{D}_\varphi = \int_{\mathbb{R}} \chi \DHm_\varphi$. 

\item We have $\int 1_\mathbb{R} \mathcal{D}_\varphi = \mathrm{MA} (\varphi)$ as measures. 

\item Suppose $\chi$ is moderate in the sense of Definition \ref{moderate}. 
Then for a convergent decreasing net $\varphi_i \searrow \varphi \in \E^1 (X, L)$, we have the weak convergence of measures
\[ \int \chi \mathcal{D}_{\varphi_i} \to \int \chi \mathcal{D}_\varphi. \]
\end{enumerate}
These properties characterize the measure $\int \chi \mathcal{D}_\varphi$. 
\end{thm}

The following formula on the Monge--Amp\`ere measure of the rooftop $\varphi \wedge \tau$ (the non-archimedean psh envelope of $\min \{ \varphi, \tau \}$, see section \ref{existence of rooftop}) is the key in the construction, which we call tomography of non-archimedean Monge--Amp\`ere measure. 

\begin{prop}[Summary of section \ref{Tomography of non-archimedean Monge--Ampere measure}]
For any $\varphi = \varphi_{(\mathcal{X}, \mathcal{L})} \in \nH (X, L)$ and $\tau \in \mathbb{R}$, we have 
\[ \mathrm{MA} (\varphi \wedge \tau) = \sum_{E \subset \mathcal{X}_0} \mathrm{ord}_E \mathcal{X}_0 \int_{(-\infty, \tau)} \DHm_{(E, \mathcal{L}|_E)}. \delta_{v_E} + \int_{[\tau, \infty)} \DHm_{(\mathcal{X}, \mathcal{L})}. \delta_{v_{\mathrm{triv}}}. \]
\end{prop}

To show this, we observe in section \ref{Primary decomposition via filtration} the primary decomposition of the Duistermaat--Heckman measure $\DHm_{(\mathcal{X}, \mathcal{L})} = \sum_{E \subset \mathcal{X}_0} \mathrm{ord}_E \mathcal{X}_0 \cdot \DHm_{(E, \mathcal{L}|_E)}$ in terms of the filtration $\mathcal{F}_\varphi$ associated to non-archimedean psh metric $\varphi = \varphi_{(\mathcal{X}, \mathcal{L})}$. 
We note this filtration is different from the filtration $\mathcal{F}_{(\mathcal{X}, \mathcal{L})}$ associated to test configuration in general: $\mathcal{F}_\varphi = \widehat{\mathcal{F}}_{(\mathcal{X}, \mathcal{L})}$ is stabilized along normalized base change. 
The filtration $\mathcal{F}_\varphi$ is suitable for our purpose as we have $\mathcal{F}_{\varphi \wedge \tau} = \mathcal{F}_\varphi \cap \mathcal{F}_{v_\mathrm{triv}} [\tau]$ (see section \ref{existence of rooftop}). 

\subsection{Non-archimedean $\mu$-entropy on $\E^{\exp} (X, L)$}

\subsubsection{The metric space $\E^{\exp} (X, L)$ and the non-archimedean $\mu$-entropy}

The non-archimedean $\mu$-entropy is concerned with exponential weight, so we have a special interest in $E_{\exp} (\varphi) := E_{- e^{-t}} (\varphi)$ (see (\ref{moment energy})). 
The following is a natural class to consider the non-archimedean $\mu$-entropy: 
\begin{equation} 
\E^{\exp} (X, L) := \{ \varphi \in \PSH (X, L) ~|~ E_{\exp} (\varphi_{;\rho}) > -\infty \text{ for } \forall \rho > 0 \}. 
\end{equation}
For $\varphi \in \E^{\exp} (X, L)$, we have $E_{\exp} (\varphi_{;\rho}) = - \int_\mathbb{R} e^{-\rho t} \DHm_\varphi > - \infty$ (see Proposition \ref{moment energy and DH measure} and Corollary \ref{finite moment energy class and full mass class}), so we can consider the moment measure $\int e^{-t} \mathcal{D}_\varphi$. 

We study two topologies on $\E^{\exp} (X, L)$. 
One is the coarsest refinement of the strong topology ($d_1$-topology) which makes $E_{\exp} (\varphi_{;\rho})$ continuous for every $\rho > 0$ and the other is the metric topology induced from a metric $d_{\exp}$ which is modeled on Orlicz norm for exponential weight. 
We refer to these topologies as $E_{\exp}$-topology and $d_{\exp}$-topology, respectively. 

\begin{thm}[Summary of section \ref{A metric structure on the space Eexp}, \ref{Intermediates}, \ref{Completeness} and section \ref{Continuity of exponential moment energy}]
\label{dexp metric}
Let $(X, L)$ be a polarized variety. 
There exists a metric $d_{\exp}$ on $\E^{\exp} (X, L)$ for which we have the following. 
\begin{enumerate}
\item For every $\rho > 0$, $E_{\exp} (\varphi_{;\rho})$ is continuous with respect to $d_{\exp}$-topology. 
Namely, $d_{\exp}$-topology is finer than $E_{\exp}$-topology. 

\item Pointwisely convergent decreasing nets are $d_{\exp}$-convergent. 
In particular, $\nH (X, L) \subset \E^{\exp} (X, L)$ is dense. 

\item The metric space $(\E^{\exp} (X, L), d_{\exp})$ is complete when $X$ is smooth. 
\end{enumerate}
\end{thm}

We may replace the smoothness assumption in the last claim with the continuity of envelopes for $(X, L)$ (see section \ref{continuity of envelopes}). 

\begin{thm}[Summary of section \ref{Eexp-topology}, \ref{Continuous extension of the functional EexpM}, \ref{Continuity of exponential moment measure} and section \ref{Non-archimedean mu-entropy}]
\label{NAmu entropy extension}
Let $(X, L)$ be a polarized variety. 
We have the following continuity results for $E_{\exp}$-topology on $\E^{\exp} (X, L)$. 
\begin{enumerate}
\item For every strongly convergent sequence $\psi_i \to \psi \in \E^1 (X, L)$ (or uniformly convergent sequence $\psi_i \to \psi \in C^0 (X^{\mathrm{NA}})$) and every $E_{\exp}$-convergent sequence $\varphi_i \to \varphi \in \E^{\exp} (X, L)$, we have 
\[ \int_{X^{\mathrm{NA}}} \psi_i \int e^{-t} \mathcal{D}_{\varphi_i} \to \int_{X^{\mathrm{NA}}} \psi \int e^{-t} \mathcal{D}_\varphi. \] 

\item For a $\mathbb{Q}$-line bundle $M$ on $X$, the functional $E_{\exp}^M$ on $\nH (X, L)$ defined by (\ref{moment M-energy}) extends continuously to $\E^{\exp} (X, L)$. 

\item Suppose $X$ has only klt singularities, then we have the greatest lower semi-continuous extension of the log discrepancy $A_X$ to $X^{\mathrm{NA}}$. 
In this case, the functional 
\[ \nH (X, L) \to \mathbb{R}: \NAmu^\lambda (\varphi) + 2\pi \int_{X^{\mathrm{NA}}} A_X \int e^{-t} \mathcal{D}_\varphi \]
extends continuously to $\E^{\exp} (X, L)$. 
As a consequence, $\NAmu^\lambda$ on $\nH (X, L)$ extends to $\E^{\exp} (X, L)$ as an upper semi-continuous function. 
\end{enumerate}
\end{thm}

As for (3), we must assume $X$ is log canonical for the existence of lsc extension of the log discrepancy $A_X$ to $X^{\mathrm{NA}}$. 
Indeed, since $X^{\mathrm{NA}}$ is compact, any lsc extension of $A_X$ must be bounded from below, but if $X$ is not log canonical, then we have $A_X (v) < 0$ for some $v \in X^{\mathrm{NA}}$ and hence $\inf_{v \in X^{\mathrm{NA}}} A_X (v) \le \varliminf_{\rho \to \infty} A_X (\rho. v) = -\infty$, which is a contradiction. 
At the moment, we assume $X$ is klt to ensure the lsc extension of $A_X$: in this case, the lsc extension is given by putting $A_X (v) = \infty$ for $v \in X^{\mathrm{NA}} \setminus X^{\mathrm{val}}$. 
See the proof in \cite{BJ2} for the lower semi-continutiy. 
The author speculates the lsc extension of $A_X$ exists for general log canonical $X$. 
We note in the log canonical case, we must put $A_X = 0$ on the closure of $\{ v \in \mathrm{Val} (X) ~|~ A_X (v) \}$, which makes the lower semi-continuity nontrivial. 

\subsubsection{Non-archimedean $\mu$-entropy and optimal degeneration}

To reformulate Theorem \ref{characteristic mu maximization implies muK-semistability} and Theorem \ref{NAmu maximizer} in the non-archimedean formalism, we must compare the characteristic $\mu$-entropy $\cmu^\lambda (\mathcal{F}_\varphi)$ and the non-archimedean $\mu$-entropy $\NAmu^\lambda (\varphi)$ for $\varphi \in \nH^\mathbb{R} (X, L)$ (see section \ref{Filtration associated to continuous psh metric} for $\mathcal{F}_\varphi$). 
The problem can be reduced to the continuity of the log discrepancy $A_X$ along some geometric family of valuations. 
We will check this for valuations associated to proper vectors. 
As a consequence, we get $\NAmu^\lambda (\varphi_\xi) = \cmu^\lambda (X, L; \xi)$ for proper vectors (Proposition \ref{comparison of NAmu and chmu}) and thus obtain the following reformulation. 

\begin{thm}[Summary of section \ref{section: mu-entropy of polyhedral configuration} and section \ref{maximizing non-archimedean mu-entropy}]
\label{NAmu entropy and muK-stability}
Let $(X, L)$ be a polarized normal variety with only klt singularities. 
If for each $\varphi \in \nH (X, L)$ there exists a proper vector $\xi$ on $(X, L)$ such that 
\[ \NAmu^\lambda (\varphi_\xi) \ge \NAmu^\lambda (\varphi), \]
then $(X, L)$ is $\check{\mu}^\lambda$K-semistable for some proper vector $\xi_{\mathrm{opt}}$ which maximizes the non-archimedean $\mu$-entropy among all proper vectors (or among all $\varphi \in \nH (X, L)$). 
\end{thm}

Now let us combine this with analytic results in the first paper \cite{Ino4}. 
Here $\bm{\check{\mu}}_{\mathrm{Per}}^\lambda$ denotes the Perelman's $\mu$-entropy. 
We use the convention in \cite{Ino4}, which has the reverse sign compared to the original one \cite{Per}. 

\begin{cor}
\label{muK-stability characterization}
Let $(X, L)$ be a polarized smooth variety. 
For $\lambda \in \mathbb{R}$, the following (c) implies (b), and (b) implies (a). 
\begin{enumerate}
\item[(a)] $(X, L)$ is $\check{\mu}^\lambda_\xi$K-semistable.   

\item[(b)] $\NAmu^\lambda (\varphi_\xi) = \sup_{\varphi \in \nH (X, L)} \NAmu^\lambda (\varphi)$. 

\item[(c)] $\NAmu^\lambda (\varphi_\xi) = \inf_{\omega_\phi \in \mathcal{H} (X, L)} \bm{\check{\mu}}_{\mathrm{Per}}^\lambda (\omega_\phi)$. 
\end{enumerate}
\end{cor}

\begin{proof}
Theorem 1.5 in \cite{Ino4} directly implies (c) $\Rightarrow$ (b): 
\[ \NAmu^\lambda (\varphi_\xi) = \inf_{\omega_\phi \in \mathcal{H} (X, L)} \bm{\check{\mu}}_{\mathrm{Per}}^\lambda (\omega_\phi) \ge \sup_{\varphi \in \nH (X, L)} \NAmu^\lambda (\varphi) \ge \NAmu^\lambda (\varphi_\xi). \]
Here note the last inequality is not trivial as $\varphi_\xi \in \nH^\mathbb{R} (X, L) \setminus \nH (X, L)$ for irrational $\xi \in \mathfrak{t}$, but it follows by the continuity on $\xi$ thanks to Proposition \ref{comparison of NAmu and chmu}. 
The implication (b) $\Rightarrow$ (a) is nothing but Theorem \ref{NAmu entropy and muK-stability}. 
\end{proof}

The author speculates these conditions are actually equivalent. 

\begin{cor}
Let $(X, L)$ be a polarized smooth variety. 
If there is a $\mu^\lambda_\xi$-cscK metric $\omega$ on $(X, L)$ for $\lambda \le 0$, then we have 
\[ \NAmu^\lambda (\varphi_\xi) = \sup_{\varphi \in \nH (X, L)} \NAmu^\lambda (\varphi). \]
In particular, $(X, L)$ is $\mu^\lambda_\xi$K-semistable. 
\end{cor}

\begin{proof}
This is a consequence of the above corollary and Theorem 1.4 in \cite{Ino4} which states 
\[ \NAmu^\lambda (\varphi_\xi) = \bm{\check{\mu}}_{\mathrm{Per}}^\lambda (\omega) = \inf_{\omega_\phi \in \mathcal{H} (X, L)} \bm{\check{\mu}}_{\mathrm{Per}}^\lambda (\omega_\phi). \]
\end{proof}

Though the $\mu$K-semistability of $\mu$-cscK manifold is not new (cf. \cite{Lah1}, \cite{Ino3} and \cite{AJL}), the proof relies on a completely different perspective from the previous one concerned with the boundedness and the slope of $\mu$/weighted-Mabuchi functional. 

Since $\NAmu^\lambda$ is only upper semi-continuous, we cannot immediately deduce 
\[ \sup_{\varphi \in \E^{\exp} (X, L)} \NAmu^\lambda (\varphi) = \sup_{\varphi \in \nH (X, L)} \NAmu^\lambda (\varphi) \]
from the density of $\nH (X, L) \subset \E^{\exp} (X, L)$. 
This problem can be reduced to the following conjecture, which is analogous to Conjecture \ref{regularization of entropy} for the usual non-archimedean entropy. 

\begin{conj}[Regularization of exponential entropy]
\label{Regularization of exponential entropy}
Let $(X, L)$ be a polarized normal variety with only klt singularities. 
For any $\varphi \in \E^{\exp} (X, L)$, there exists a sequence $\{ \varphi_i \} \subset \nH (X, L)$ converging to $\varphi$ in $E_{\exp}/d_{\exp}$-topology (or prefarably it is a convergent decreasing sequence) such that 
\[ \lim_{i \to \infty} \int_{X^{\mathrm{NA}}} A_X \int e^{-t} \mathcal{D}_{\varphi_i} = \int_{X^{\mathrm{NA}}} A_X \int e^{-t} \mathcal{D}_\varphi. \]
\end{conj}

In \cite{Ino4}, we proved 
\[ \sup_{\varphi \in \nH (X, L)} \NAmu^\lambda (\varphi) \le \inf_{\omega_\phi \in \mathcal{H} (X, L)} \bm{\check{\mu}}_{\mathrm{Per}}^\lambda (\omega_\phi) \]
and for $\lambda \le 0$ the equality holds at least when there exists a $\mu^\lambda$-cscK metric. 
(At the moment, we cannot replace $\nH (X, L)$ in the left hand side with a slightly large $\nH^\mathbb{R} (X, L)$, which is plausible in view of maximization problem. 
However, it would be a consequence of a result in coming \cite{BJ5} as we remark below. )
The author speculates the equality holds in general, even for extended $\mu$-entropies. 
See also analogous Conjecture \ref{Minimax conjecture for Calabi energy} in Calabi energy framework. 

\begin{conj}[Minimax conjecture for $\mu$-entropy]
Let $(X, L)$ be a polarized normal variety with only klt singularities. 
Then for $\lambda \le 0$ we have 
\[ \sup_{\varphi \in \E^{\exp} (L)} \NAmu^\lambda (\varphi) = \sup_{\varphi \in \nH^\mathbb{R} (L)} \NAmu^\lambda (\varphi) = \inf_{\omega_\phi \in \mathcal{H} (L)} \bm{\check{\mu}}_{\mathrm{Per}}^\lambda (\omega_\phi) = \inf_{\omega_\phi \in \mathcal{E}^{\exp} (L)} \bm{\check{\mu}}_{\mathrm{Per}}^\lambda (\omega_\phi) \]
with an appropriate definition of $\mathcal{E}^{\exp} (X, L)$ and Perelman's $\mu$-entropy $\bm{\check{\mu}}_{\mathrm{Per}}^\lambda (\omega_\phi)$ for $\omega_\phi \in \mathcal{E}^{\exp} (X, L)$. 
\end{conj}

Compared to Theorem \ref{NAmu entropy and muK-stability}, the following theorem just rephrases Theorem \ref{NAmu maximizer} as we currently assume $\NAmu^\lambda (\varphi) = \cmu^\lambda (\mathcal{F}_\varphi)$. 

\begin{thm}
Let $(X, L)$ be a polarized normal variety with only klt singularities. 
If $\NAmu^\lambda$ is maximized by $\varphi \in \nH^\mathbb{R} (X, L)$ on $\nH^\mathbb{R} (X, L)$ and $\NAmu^\lambda (\varphi) = \cmu^\lambda (\mathcal{F}_\varphi)$, then the central fibre $(\mathcal{X}_o (\varphi), \mathcal{L}_o (\varphi)) = \Proj \mathcal{R}_o (\mathcal{F}_\varphi)$ is reduced and $\check{\mu}^\lambda$K-semistable with respect to the proper vector $\xi^\varphi_o$ induced by the filtration $\mathcal{F}_\varphi$. 
\end{thm}

The reducedness of the central fibre is a general phenomenon for the filtration $\mathcal{F}_\varphi$ associated to $\varphi \in \pcH (X, L)$ (see section \ref{Filtration associated to continuous psh metric}). 
In this article, we only check the equality for $\varphi \in \nH (X, L)$ and $\varphi_\xi$ for proper vectors. 
However, it would be proved in coming \cite{BJ5} that the log discrepancy $A_X$ has a desired property we discuss in section \ref{maximizing non-archimedean mu-entropy}, which implies the condition $\NAmu^\lambda (\varphi) = \cmu^\lambda (\mathcal{F}_\varphi)$ follows immediately by the assumption $\varphi \in \nH^\mathbb{R} (X, L)$. 

\subsubsection{Maximization problem for non-archimedean $\mu$-entropy}

We are now interested in finding a maximizer of $\NAmu^\lambda$ in $\pcH (X, L)$. 
What we benefit from the non-archimedean formalism is some sort of completeness of the domain $\E^{\exp} (X, L)$ and the semi-continuity of $\mu$-entropy as we stated in Theorem \ref{dexp metric} and Theorem \ref{NAmu entropy extension}. 
With this in mind, we would split the maximization problem into two parts: 
\begin{enumerate}
\item (Existence) Firstly, find a maximizer in $\E^{\exp} (X, L)$. 

\item (Regularity) Then show the maximizer is actually in $\pcH (X, L)$. 
\end{enumerate}

This kind of reduction is often effective for maximization problem. 
Indeed, we employed such reduction in the proof of \cite[Theorem 2.2]{Ino4} (cf. \cite{Rot}) which shows the existence of maximizing momentum $f_\omega$ for Perelman's W-entropy $\check{W}^\lambda (\omega, f)$. 
The author speculates for $\lambda \le 0$ the maximizing momentum $f_\omega$ for $\check{W}^\lambda (\omega, f)$ would define a flow $\omega_t$ of K\"ahler metrics: \textit{$\mu$-flow equation} $\dot{\omega}_t = dd^c f_{\omega_t}$. 
Then it is likely that a maximizing non-archimedean psh metric $\varphi \in \E^{\exp} (X, L)$ for $\NAmu^\lambda$ could be interpreted as the limit of such flow. 
This is still a far-off dream, but seems plausible as indeed such picture for K\"ahler--Ricci flow on Fano manifold, the Hamilton--Tian conjecture, is completely realized in $H$-entropy formalism (cf. \cite{He, CSW, DS, HL2, BLXZ}). 
We also refer to analytic result for Calabi flow \cite{Xia}. 
We will discuss these works from our non-archimedean perspective in section \ref{Relation to H-entropy} and section \ref{Maximizing non-archimedean Calabi energy}. 

The following conjecture seems more down to earth. 
By the upper semi-continuity of $\NAmu$, the subset is closed in $E_{\exp}/d_{\exp}$-topology. 
See also analogous Conjecture \ref{properness conjecture for Calabi energy} in Calabi energy framework. 

\begin{conj}[Properness of $\mu$-entropy]
\label{properness conjecture}
Let $(X, L)$ be a polarized normal variety with only klt singularities. 
Then the subset 
\[ \Big{\{} \varphi \in \E^{\exp} (X, L) ~\Big{|}~ \sup \varphi = 0, ~\NAmu (\varphi) \ge C \Big{\}} \]
is compact with respect to $E_{\exp}$-topology. 
\end{conj}

The author speculates the klt assumption cannot be simply replaced with log canonical assumption in this conjecture (cf. \cite{Hat}). 

\begin{prop}
\label{maximizer under properness conjecture}
Assuming Conjecture \ref{properness conjecture}, there exists a maximizer $\varphi^\lambda_{\mathrm{opt}} \in \E^{\exp} (X, L)$ of $\NAmu^\lambda$ for each $\lambda \le 0$. 
\end{prop}

\begin{proof}
We recall $\NAmu^\lambda = \NAmu + \lambda \bm{\check{\sigma}}$ and $\bm{\check{\sigma}}$ is bounded from below by Remark \ref{sigma bounded}. 
It follows that for $\lambda \le 0$ $\NAmu = \NAmu^\lambda - \lambda \bm{\check{\sigma}}$ is bounded from below along a maximizing sequence $\varphi_i \in \E^{\exp} (X, L)$: $\NAmu^\lambda (\varphi_i) \searrow \sup \NAmu^\lambda$. 
Since $\NAmu^\lambda$ is normalization free, we may normalize $\varphi_i$ so that $\sup \varphi_i = 0$. 
Then by the above conjecture, we have a convergent subsequence $\varphi_j \to \varphi$ in $\E^{\exp} (X, L)$ in $E_{\exp}$-topology. 
Since $\NAmu^\lambda$ is upper semi-continuous with respect to $E_{\exp}$-topology, the limit $\varphi$ attains the maximum of $\NAmu^\lambda$. 
\end{proof}

For Calabi--Yau variety and canonically polarized variety, we can show the trivial metric maximizes the non-archimedean $\mu$-entropy. 
This can be regarded as a reformulation of Odaka's theorem \cite{Oda1} in our $\mu$-entropy formalism. 
Here we assume $X$ has only klt singularieties at the moment in order to ensure the lsc extension of the log discrepancy to $X^{\mathrm{NA}}$. 
The claim can be extended to the log canonical case as soon as the lsc extension is realized for log canonical varieties. 

\begin{thm}[Summary of section \ref{Odaka's theorem in mu-entropy formalism}]
Let $(X, L)$ be one of the following: 
\begin{itemize}
\item $X$ has only klt singularities and $K_X \equiv 0$, 

\item or $X$ has only klt singularities and $K_X > 0$ and $L = K_X$. 
\end{itemize}
Then for every $\lambda \le 0$, the trivial metric $\varphi_{\mathrm{triv}}$ maximizes the non-archimedean $\mu$-entropy $\NAmu^\lambda$ on $\E^{\exp} (X, L)$. 
In particular, $(X, L)$ is K-semistable (actually K-stable under the klt assumption). 
\end{thm}

\subsubsection{Relation to other works}

For a $\mathbb{Q}$-Fano variety $(X, L) = (X, -K_X)$ (a $\mathbb{Q}$-Gorenstein variety with only klt singularities whose anti-canonical sheaf has ample reflexive power), the non-archimedean $\mu$-entropy $\NAmu^{2\pi}$ is related to the following $H$-entropy: 
\begin{equation} 
\check{H}_{\mathrm{NA}} (\varphi) := -\inf_{x \in X^{\mathrm{div}}} (A_X (x) + \varphi (x)) - \log \iint_{X^{\mathrm{NA}}} e^{-t} \mathcal{D}_\varphi. 
\end{equation}
%This gives a lower semi-continuous functional on $\E^{\exp} (X, L)$ with respect to $E_{\exp}/d_{\exp}$-topology. 
%It would be shown in coming \cite{BJ5} that the first term is actually continuous with respect to the strong topology ($d_1$-topology), and hence $\check{H}_{\mathrm{NA}}$ is in fact continuous on $\E^{\exp} (X, L)$ with respect to $E_{\exp}/d_{\exp}$-topology. 
We have $\NAmu^{2\pi} (\varphi) \le 2\pi \check{H}_{\mathrm{NA}} (\varphi)$ in general. 

The following theorem can be understood as a non-archimedean counterpart of \cite{DS}: 
\[ \inf_{\omega_\phi \in \mathcal{H} (X, L)} \bm{\check{\mu}}_{\mathrm{Per}}^{2\pi} (\omega_\phi) = \inf_{\omega_\phi \in \mathcal{H} (X, L)} 2\pi H (\omega_\phi). \]
The existence part is due to \cite{HL2, BLXZ}. 

\begin{thm}[Summary of section \ref{Relation to H-entropy}]
Let $(X, L) = (X, -K_X)$ be a $\mathbb{Q}$-Fano variety. 
Then we have 
\[ \sup_{\varphi \in \E^{\exp} (L)} \NAmu^{2\pi} (\varphi) = \sup_{\varphi \in \pcH (L)} \NAmu^{2\pi} (\varphi) = \sup_{\varphi \in \nH^\mathbb{R} (L)} 2\pi \check{H}_{\mathrm{NA}} (\varphi) = \sup_{\varphi \in \E^{\exp} (L)} 2\pi \check{H}_{\mathrm{NA}} (\varphi). \]
The maximums are attained by some common $\varphi_\mathcal{F} \in \pcH (X, L)$ associated to a finitely generated filtration $\mathcal{F}$ with $\mu^{2\pi}$K-semistable $\mathbb{Q}$-Fano central fibre. 
The maximizers are unique up to the addition of constants. 
\end{thm}

It is shown by \cite{HL2} that for a $\mathbb{Q}$-Fano variety which is modified K-semistable with respect to a proper vector $\xi$, the maximum of $\check{H}_{\mathrm{NA}}$ is achieved by $\varphi_\xi$, so that we conclude the following. 

\begin{cor}
For a $\mathbb{Q}$-Fano variety $(X, -K_X)$, the $\mu^{2\pi}$K-semistability with respect to general test configuration is equivalent to the modified K-semistability with respect to equivariant special degenerations. 
\end{cor}

We return to the case of general polarized varieties. 
We can show the following extension of the extremal limit observation in \cite{Ino2}. 

\begin{thm}[Summary of section \ref{Relation to normalized Donaldson--Futaki invariant}]
Let $(X, L)$ be a polarized normal variety with only klt singularieties. 
For $\varphi \in \PSH^{\mathrm{bdd}} (X, L)$, we have 
\[ \lim_{\rho \to +0} \rho^{-1} (\NAmu^{-\rho^{-1}} (\varphi_{; \rho}) - \NAmu^{-\rho^{-1}} (\varphi_{\mathrm{triv}})) = C_{\mathrm{NA}} (\varphi). \]
\end{thm}

Here $C_{\mathrm{NA}}$ is a non-archimedean counterpart of Calabi energy, which is introduced in \cite{Ino4}. 
The functional $C_{\mathrm{NA}}$ is not scaling invariant $C_{\mathrm{NA}} (\varphi_{;\rho}) \neq C_{\mathrm{NA}} (\varphi)$: it is rather quadratic on $\rho$. 
Its maximizer along $\rho > 0$ gives the non-archimedean variant of normalized Donaldson--Futaki invariant $\frac{1}{2 (e^L)} (2\pi M_{\mathrm{NA}} (\varphi) / \| \bar{\varphi} \|)^2$ when $\varphi$ destabilizes $(X, L)$. 
There is an analogous story for $C_{\mathrm{NA}}$ as $\NAmu^\lambda$ and $\check{H}_{\mathrm{NA}}$. 
We will discuss this and its Ding version in section \ref{Relation to other works}. 

\subsubsection{Toric illustration}

For toric test configurations, we can compute the non-archimedean $\mu$-entropy by integrations on the moment polytope. 

\begin{prop}
Let $(X, L)$ be a polarized toric (normal) variety and $P$ be the associated moment polytope. 
For a toric test configuration $(\mathcal{X}, \mathcal{L})$ with ample $\bar{\mathcal{L}}$, take the piecewise affine convex function $q$ on $P$ so that $Q = \{ (\mu, t) \in P \times \mathbb{R} ~|~ 0 \le t \le - q (\mu) \}$ denotes the moment polytope of $(\bar{\mathcal{X}}, \bar{\mathcal{L}})$, then we have 
\begin{align} 
\NAmu (\mathcal{X}, \mathcal{L}; \rho) 
&= -2\pi \frac{\int_{\partial P} e^{\rho q} d\sigma}{\int_P e^{\rho q} d\mu}, 
\\
\bm{\check{\sigma}} (\mathcal{X}, \mathcal{L}; \rho)
&= \frac{\int_P (n+\rho q) e^{\rho q} d\mu}{\int_P e^{\rho q} d\mu} - \log \int_P e^{\rho q} d \mu. 
\end{align}
\end{prop}

More generally, for a $T$-invariant non-archimedean psh metric $\varphi \in \E^{\exp} (X, L)$, we can assign a lower semi-continuous convex function $q_\varphi$ on $P$ which enjoys the integrability condition $\int_P e^{\rho q} d\mu < \infty$ for every $\rho > 0$. 
Then assuming Conjecture \ref{Regularization of exponential entropy}, we have the same formula for $\NAmu^\lambda (\varphi)$. 
These are explained in Appendix. 

\subsubsection*{Acknowledgements}

I wish to thank Sebastien Boucksom and Mattias Jonsson, Ruadha\'i Dervan, Masafumi Hattori, Tomoyuki Hisamoto, Yaxiong Liu and Mingchen Xia for their interest, helpful comments and discussions on various occasions. 
This work is supported by RIKEN iTHEMS Program. 

\section{Characteristic $\mu$-entropy and $\mu$K-semistability}

\subsection{Algebraic preliminaries}

\subsubsection{Filtration}
\label{Filtration}

We firstly recall some notions and terminologies related to filtration used throughout this article. 
We put $\mathbb{N} := \{ 0, 1, 2, \ldots, \}$ and $\mathbb{N}_+ = \{ 1, 2, \ldots \}$. 
For $d \in \mathbb{N}_+$, we put $\mathbb{N}^{(d)} := \{ m \in \mathbb{N} ~|~ d \text{ divides } m \}$. 

Let $(X, L)$ be a polarized scheme. 
We put $R_m := H^0 (X, L^{\otimes m})$, $R := \bigoplus_{m \in \mathbb{N}} R_m$ and $R^{(d)} := \bigoplus_{m \in \mathbb{N}^{(d)}} R_m$. 
We also put $N_m := \dim_{\mathbb{C}} R_m$. 

A \textit{filtration} $\mathcal{F}$ of $(X, L)$ is a collection of linear subspaces $\{ \mathcal{F}^\lambda R_m \subset R_m \}_{\lambda \in \mathbb{R}, m \in \mathbb{N}^{(d)}}$ for some $d \in \mathbb{N}_+$ which satisfies the following:  for every $\lambda, \lambda' \in \mathbb{R}$ and $m, m' \in \mathbb{N}^{(d)}$
\begin{enumerate}
\item $\mathcal{F}^{\lambda-} R_m = \mathcal{F}^\lambda R_m$ for $\mathcal{F}^{\lambda-} R_m := \bigcap_{\lambda' < \lambda} \mathcal{F}^{\lambda'} R_m$. 

\item $\sum_{\lambda \in \mathbb{R}} \mathcal{F}^\lambda R_m = R_m $, $\bigcap_{\lambda \in \mathbb{R}} \mathcal{F}^\lambda R_m = 0$. 

\item $\mathcal{F}^\lambda R_m \cdot \mathcal{F}^{\lambda'} R_{m'} \subset \mathcal{F}^{\lambda + \lambda'} R_{m + m'}$. 
\end{enumerate}
In particular, we have $\mathcal{F}^\lambda R_m \subset \mathcal{F}^{\lambda'} R_m$ for $\lambda' \le \lambda$. 
We use the notation $\mathcal{F}_1 \subset \mathcal{F}_2$ when $\mathcal{F}_1^\lambda R_m \subset \mathcal{F}_2^\lambda R_m$ for every $\lambda \in \mathbb{R}$ and $m \in \mathbb{N}^{(d)}$ for \textit{some} sufficiently divisible $d \in \mathbb{N}_+$. 
We identify two filtrations $\mathcal{F}_1, \mathcal{F}_2$ if $\mathcal{F}_1 \subset \mathcal{F}_2$ and $\mathcal{F}_2 \subset \mathcal{F}_1$. 

We put 
\[ \mathcal{F}_{\mathrm{triv}}^\lambda R_m := 
\begin{cases} 
R_m
& \lambda \le 0
\\
0
& \lambda > 0
\end{cases}. \]

\begin{eg}
Recall a \textit{test configuration} $(\mathcal{X}, \mathcal{L})$ of a polarized scheme $(X, L)$ is a $\mathbb{G}_m$-equivariant proper flat family $\varpi: (\mathcal{X}, \mathcal{L}) \to \mathbb{A}^1$ of polarized schemes endowed with a $\mathbb{G}_m$-equivariant relatively ample $\mathbb{Q}$-line bundle $\mathcal{L}$ and an isomorphism $(X, L) \cong (\mathcal{X}_1, \mathcal{L}|_{\mathcal{X}_1})$. 
For a test configuration $(\mathcal{X}, \mathcal{L})$ (not necessarily normal), we assign the following filtration: 
\begin{equation} 
\mathcal{F}_{(\mathcal{X}, \mathcal{L})}^\lambda R_m := \{ s \in H^0 (X, L^{\otimes m}) ~|~ \varpi^{-\lceil \lambda \rceil} \bar{s}  \text{ extends to a section of } \mathcal{L}^{\otimes m} \}. 
\end{equation}
This is $\mathbb{Z}$-graded: $\mathcal{F}^\lambda = \mathcal{F}^{\lceil \lambda \rceil}$. 
Later, in section \ref{Primary decomposition via filtration} and the subsequent section, we observe a variant $\widehat{\mathcal{F}}_{(\mathcal{X}, \mathcal{L})}$ for a normal test configuration $(\mathcal{X}, \mathcal{L})$. 
It is $\mathbb{Q}$-graded and is identified with a filtration $\mathcal{F}_{\varphi_{(\mathcal{X}, \mathcal{L})}}$ recovered from the non-archimedean psh metric $\varphi_{(\mathcal{X}, \mathcal{L})}$ associated to $(\mathcal{X}, \mathcal{L})$. 
\end{eg}

\begin{eg}
\label{filtration for proper vector}
When we have a torus action $(X, L) \circlearrowleft T$, for a vector $\xi \in \mathfrak{t} = N \otimes \mathbb{R}$, we assign
\begin{equation}
\mathcal{F}_\xi^\lambda R_m := \{ s \in H^0 (X, L^{\otimes m}) ~|~ \langle \mu, \xi \rangle \ge \lambda \text{ for every } \mu \in M \text{ with } s_\mu \neq 0 \}, 
\end{equation}
using the weight decomposition $s = \sum_{\mu \in M} s_\mu$. 
For irrational $\xi \in \mathfrak{t}$, the filtration is not $\mathbb{Q}$-graded. 

For integral $\eta \in N$, we can assign a product configuration $(X_{\mathbb{A}^1}^\eta, L_{\mathbb{A}^1}^\eta)$ endowed with the diagonal action induced by $\eta$. 
Then we have $\mathcal{F}_\eta = \mathcal{F}_{(X_{\mathbb{A}^1}^\eta, L_{\mathbb{A}^1}^\eta)}$ as $\bar{s} (x, t) = (s.t) (x) = \sum_{\mu \in M} t^{\langle \mu, \eta \rangle} s_\mu (x)$. 
\end{eg}

In some arguments, it is convenient to use non-archimedean norm instead of filtration. 
A \textit{non-archimedean norm} on a vector space $V$ over (the trivially valued non-archimedean field) $\mathbb{C}$ is a map $\| \cdot \|: V \to [0, \infty)$ satisfying 
\begin{itemize}
\item $\| v \| = 0$ iff $v = 0$. 

\item $\| a u + b v \| \le \max \{ \| u \|, \| v \| \}$ for $u, v \in V$ and $a, b \in \mathbb{C}$. 
\end{itemize}
We note the second condition implies $\| a v \| = \| v \|$ for $a \neq 0$. 
A basis $\{ e_i \}_{i=1}^r$ of $V$ is called \textit{diagonal} with respect to $\| \cdot \|$ if it satisfies 
\[ \| \sum_i a_i e_i \| = \max \{ \| e_i \| ~|~ a_i \neq 0 \}. \]
It is known by \cite{BE} there always exists such a basis. 
For a quotient space $V/W$, we can induce a non-archimedean norm on $V/W$ by 
\[ \| [v] \|_{V/W} := \inf \{ \| v' \|_V ~|~ v' \in [v] \}, \]
which is equivalent to declare 
\[ -\log \| [v] \|_{V/W} \ge \lambda \iff \exists v' \in [v] \text{ s.t. } - \log \| v' \|_V \ge \lambda. \]

For a filtration $\mathcal{F}$ of $(X, L)$, we associate a non-archiemedean norm $\| \cdot \|^{\mathcal{F}}_m$ on $R_m$ by 
\begin{equation} 
\label{non-archimedean norm associated to filtration}
\| s \|^{\mathcal{F}}_m := \inf \{ e^{- \lambda} ~|~ s \in \mathcal{F}^\lambda R_m \}. 
\end{equation}
Since the filtration is left continuous $\mathcal{F}^{\lambda-} = \mathcal{F}^\lambda$, we have 
\begin{align*} 
-\log \| s \|^{\mathcal{F}}_m \ge \lambda
&\iff s \in \mathcal{F}^\lambda R_m 
\\
-\log \| s \|^{\mathcal{F}}_m = \lambda
&\iff s \in \mathcal{F}^\lambda R_m \setminus \mathcal{F}^{\lambda +} R_m. 
\end{align*}

For a filtration $\mathcal{F}$ and $d \in \mathbb{N}_+$, we put 
\begin{align} 
\sigma_{\min, d} (\mathcal{F}) 
&:= \inf_{m \in \mathbb{N}^{(d)}} m^{-1} \sup \{ \lambda \in \mathbb{R} ~|~ \mathcal{F}^\lambda R_m = R_m \}, 
\\ 
\sigma_{\max, d} (\mathcal{F}) 
&:= \sup_{m \in \mathbb{N}^{(d)}} m^{-1} \inf \{ \lambda \in \mathbb{R} ~|~ \mathcal{F}^\lambda R_m = 0 \}. 
\end{align}
A filtration is \textit{linearly bounded} if $\sigma_{\min, d} (\mathcal{F})$ and $\sigma_{\max, d} (\mathcal{F})$ is finite for some $d$. 

\subsubsection{Valuation}

A \textit{valuation} on an irreducible variety $X$ is a map $v: \mathbb{C} (X) \to (-\infty, \infty]$, where $\mathbb{C} (X)$ denotes the field of rational functions, satisfying 
\begin{itemize}
\item $v (f) = \infty$ if and only if $f = 0$, 

\item $v (f) = 0$ for $f \in \mathbb{C}^\times$, 

\item $v (fg) = v (f) + v (g)$, 
 
\item $v (f+ g) \ge \min \{ v (f), v (g) \}$. 
\end{itemize}
We in particular have the trivial valuation $v_{\mathrm{triv}}$: $v_{\mathrm{triv}} (f) = 0$ iff $f \neq 0$. 

\begin{eg}
For a birational map $X' \dashrightarrow X$ from a normal variety $X'$ and a prime divisor $E \subset X'$ and $c \in \mathbb{Q}_{\ge 0}$, we can assign a valuation $v = c. \mathrm{ord}_E$. 
We call a valuation \textit{divisorial} if it is of this form or trivial. 
\end{eg}

\begin{eg}[(cf. \cite{BHJ1})]
\label{valuation associated to prime divisor of test configuration}
For a normal test configuration $\mathcal{X}$ of $X$ and an irreducible component $E \subset \mathcal{X}_0$ of the central fibre, we assign the following valuation 
\begin{equation} 
v_{\mathcal{X}, E} = \frac{\mathrm{ord}_E \circ p_X^*}{\mathrm{ord}_E \mathcal{X}_0}, 
\end{equation}
using the canonical rational map $p_X: \mathcal{X} \dashrightarrow X \times \mathbb{A}^1 \to X$. 
We often abbreviate $v_{\mathcal{X}, E}$ as $v_E$. 

These valuations $v_E$ are divisorial by \cite[Lemma 4.1]{BHJ1}. 
Conversely, by \cite[Theorem 4.6]{BHJ1}, any divisorial valuation $v$ can be written as $v = v_E$ for some irreducible component $E \subset \mathcal{X}_0$ of the central fibre of some test configuration $\mathcal{X}$ of $X$. 
\end{eg}

\begin{eg}
When we have a torus action $X \circlearrowleft T$, for a vector $\xi \in \mathfrak{t} = N \otimes \mathbb{R}$, we assign the following valuation
\begin{equation}
v_\xi (f) := \inf \{ \langle \mu, \xi \rangle ~|~ f_\mu \neq 0 \} = \inf \{ \lambda \in \mathbb{R} ~|~ \sum_{\langle \mu, \xi \rangle = \lambda} f_\mu \neq 0 \}, 
\end{equation}
using the weight decomposition $f = \sum_{\mu \in M} f_\mu$. 

Similarly as Example \ref{filtration for proper vector}, we have $v_\eta = v_{X_{\mathbb{A}^1}^\eta, X}$ for $\eta \in N$. 
% as $p_X^* f (x, t) = (f. t) (x) = \sum_{\mu \in M} t^{\langle \mu, \eta \rangle} f_\mu (x)$ by $p_X (x, t) = (x. t^{-1}, t)$. 
\end{eg}

\begin{eg}[Quasi-monomial valuation (cf. \cite{JM})]
Consider a proper birational morphism $X' \to X$ from a smooth variety $X'$ and an snc divisor $D = \sum_{i=1}^r E_i$ on a Zariski neighbourhood of a schematic point $\eta \in X'$. 
Take a regular system of parameters $z_1, \ldots, z_r$ of the regular local ring $\mathcal{O}_{X', \eta}$ of a schematic point $\eta \in X'$ so that $z_i$ defines $E_i$. 
By Cohen's structure theorem, the $\mathfrak{m}$-adic completion $\widehat{\mathcal{O}}_{X', \eta}$ is isomorphic to $(\widehat{\mathcal{O}}_{X', \eta}/\mathfrak{m}) \llbracket z_1, \ldots, z_r \rrbracket$. 
Thus $f \in \mathcal{O}_{X', \eta}$ can be written as $f = \sum_{\beta \in \mathbb{N}^r} c_\beta z^\beta$ in $\widehat{\mathcal{O}}_{X', \eta}$. 
For $\alpha \in [0, \infty)^r$, we assign the following valuation 
\[ v_\alpha (f) := \min \{ \langle \alpha, \beta \rangle ~|~ c_\beta \neq 0 \}. \]
This is independent of the choice of the regular system. 
We call such a valuation \textit{quasi-monomial} and denote by $\mathrm{QM}_\eta (X', D) \cong [0, \infty)^r$ the set of valuations given as above. 
\end{eg}

For a valuation $v$ on a projective variety $X$ with a polarization $L$, we define a filtration $\mathcal{F}_v [\sigma]$ by 
\begin{equation} 
\mathcal{F}_v^\lambda [\sigma] R_m := \{ s \in R_m ~|~ v (s) + m\sigma \ge \lambda \} = \mathcal{F}^{\lambda - m \sigma}_v R_m. 
\end{equation}
Here we put $v (s) := v (s/e)$ for $s \in R_m = H^0 (X, L^{\otimes m})$ by taking a Zariski local generator $e$ of $L^{\otimes m}$ around the center of the valuation $v$, which exists by the properness of $X$. 
We obviously have $\mathcal{F}_v [\sigma'] \subset \mathcal{F}_v [\sigma]$ for $\sigma' \le \sigma$. 
Since $v (s) \ge 0$, we have $\mathcal{F}_v^\lambda [\sigma] R_m = R_m$ for $\lambda \le m \sigma$, so that $\sigma_{\min, d} (\mathcal{F}_v [\sigma]) \ge \sigma$. 
We call a valuation $v$ \textit{linear growth} if $\mathcal{F}_v = \mathcal{F}_v [0]$ is linearly bounded, i.e. $\sigma_{\max, d} (\mathcal{F}_v) < \infty$ for some $d$. 
It is known that every quasi-monomial valuation including divisorial valuation is linearly bounded. 

\begin{eg}
Let $\xi$ be a proper vector on $(X, L)$ and take sufficiently divisible $d$ so that there is a section $e \in H^0 (X, L^{\otimes d})$ which does not vanish around the center of $v_\xi$. 
For $m \in \mathbb{N}^{(d)}_+$ and $s \in H^0 (X, L^{\otimes m})$, we have 
\begin{align*} 
0 \le v_\xi (s) 
&= v_\xi (s/e^{m/d})
\\
&= \inf \{ \lambda \in \mathbb{R} ~|~ \sum_{\langle \mu, \xi \rangle = \lambda} (s/e^{m/d})_\mu \neq 0 \} 
\\
&= \inf \{ \lambda \in \mathbb{R} ~|~ \sum_{\langle \mu', \xi \rangle = \lambda} s_{\mu'} \neq 0 \} - \frac{m}{d} \inf \{ \lambda \in \mathbb{R} ~|~ \sum_{\langle \mu'', \xi \rangle = \lambda} e_{\mu''} \neq 0 \}, 
\end{align*}
so we get 
\begin{align*} 
\inf \{ \lambda \in \mathbb{R} ~|~ \sum_{\langle \mu'', \xi \rangle = \lambda} e_{\mu''} \neq 0 \} 
&= \inf_{s \in H^0 (X, L^{\otimes m})} \inf \{ \lambda \in \mathbb{R} ~|~ \sum_{\langle \mu', \xi \rangle = \lambda} s_{\mu'} \neq 0 \}
\\
&= \sup \{ \lambda \in \mathbb{R} ~|~ \mathcal{F}_\xi^\lambda R_m = R_m \} = d \sigma_{\min, d} (\mathcal{F}_\xi). 
\end{align*}
It follows that $\mathcal{F}_{v_\xi} [\sigma_{\min, d} (\mathcal{F}_\xi)] = \mathcal{F}_\xi$. 
\end{eg}

For a filtration $\mathcal{F}$ and a valuation $v$, we put 
\begin{equation} 
\sigma_v = \sigma_v (\mathcal{F}) := \inf \{ \sigma \in \mathbb{R} ~|~ \mathcal{F} \subset \mathcal{F}_v [\sigma] \}. 
\end{equation}
As we see later, we have $\sigma_v (\mathcal{F}_{(\mathcal{X}, \mathcal{L})}) = \varphi_{(\mathcal{X}, \mathcal{L})} (v)$ for the filtration $\mathcal{F}_{(\mathcal{X}, \mathcal{L})}$ and the non-archimedean metric $\varphi_{(\mathcal{X}, \mathcal{L})}$ associated to a test configuration $(\mathcal{X}, \mathcal{L})$. 
More generally, assuming the continuity of envelopes, we can assign a non-archimedean psh metric $\varphi_{\mathcal{F}}$ for a general linearly bounded filtration, and then $\sigma_v$ coincides with the value $\varphi_{\mathcal{F}} (v)$ of the associated non-archimedean metric $\varphi_{\mathcal{F}}$ for $v$ of linear growth. 

We use $\sigma_v$ to describe the Duistermaat--Heckman measure of an irreducible component $E$ of the central fibre of a normal test configuration in terms of the filtrations $\widehat{\mathcal{F}}_{(\mathcal{X}, \mathcal{L})}, \mathcal{F}_{v_E} [\sigma]$. 
In the argument, we need to compute $\sigma_v$ along normalized base change. 
Approving the fact $\sigma_v = \varphi (v)$ (see Proposition \ref{non-archimedean metric associated to filtration}), this is explained in \cite{BHJ1, BJ1, BJ2, BJ3, BJ4} based on non-archimedean perspective, however, it is also possible to explain this in a more direct algebraic perspective. 
Multiple viewpoints would be good for readers, so we display the proofs. 
This observation also provides a way to access Boucksom--Jonsson's non-archimedean pluripotential theory, which we really need from section \ref{Non-archimedean pluripotential theory}. 

For any $\sigma < \sigma_{\min, d} (\mathcal{F})$, we have $\mathcal{F}^{m\sigma} R_m = R_m$, so that we have $\mathcal{F}^{m \sigma}_v [\sigma'] R_m = R_m$ for any $\sigma' > \sigma_v (\mathcal{F})$. 
Taking $s \in R_m$ which does not vanish at the center of $v$, we have $v (s) = 0$. 
Then $\mathcal{F}^{m \sigma}_v [\sigma'] R_m = R_m$ implies $\sigma' \ge \sigma$. 
Thus we get $\sigma_v \ge \sigma_{\min, d}$. 
On the other hand, for any $\sigma > \sigma_{\max, d} (\mathcal{F})$, we have $\mathcal{F}^{m\sigma} R_m = 0$, so that we have $\mathcal{F} \subset \mathcal{F}_{\mathrm{triv}} [\sigma]$. 
Meanwhile, $\mathcal{F}_{\mathrm{triv}} [\sigma] \subset \mathcal{F}_v [\sigma]$ for any valuation $v$, so we get $\sigma_v \le \sigma$. 
Thus we get also $\sigma_v \le \sigma_{\max, d}$. 
Therefore, we have 
\begin{equation} 
\sigma_{\min, d} (\mathcal{F}) \le \sigma_v (\mathcal{F}) \le \sigma_{\max, d} (\mathcal{F}). 
\end{equation}
In particular, $\sigma_v$ is finite for linearly bounded filtration. 

For a filtration $\mathcal{F}$ for $(X, L)$ and $\rho \in \mathbb{R}_+$, we put 
\begin{equation} 
\mathcal{F}_{;\rho}^\lambda R_m := \mathcal{F}^{\rho^{-1} \lambda} R_m. 
\end{equation}
We have $\mathcal{F}_{v; \rho} [\sigma] = \mathcal{F}_{\rho v} [\rho \sigma]$, so $\sigma_{\rho v} (\mathcal{F}_{;\rho}) = \rho \sigma_v (\mathcal{F})$. 

\subsubsection{Spectral measure}
\label{Spectral measure}

We recall 
\begin{equation}
\mathcal{F}^{\lambda+} R_m := \sum_{\lambda' > \lambda} \mathcal{F}^{\lambda'} R_m = \bigcup_{\lambda' > \lambda} \mathcal{F}^{\lambda'} R_m
\end{equation} 
may differ from $\mathcal{F}^\lambda R_m$. 
For a linearly bounded filtration $\mathcal{F}$ on $R$ and for each $m \in \mathbb{N}^{(d)}$, we associate the following measure $\nu_m (\mathcal{F})$ on $\mathbb{R}$: 
\begin{equation} 
\nu_m (\mathcal{F}) := \frac{1}{m^n} \sum_{\lambda \in \mathbb{R}} (\dim \mathcal{F}^\lambda R_m / \mathcal{F}^{\lambda +} R_m) . \delta_{\lambda/m}. 
\end{equation}
Note our normalization constant $1/m^n$ is different from that $1/N_m$ in \cite{BHJ1}. 
It is shown by \cite{CM} (cf. \cite{BC}) that there exists a compactly supported measure $\nu_\infty (\mathcal{F})$ on $\mathbb{R}$ such that 
\[ \int_\mathbb{R} \chi \nu_m (\mathcal{F}) \to \int_\mathbb{R} \chi \nu_\infty (\mathcal{F}) \]
for every continuous $\chi$ on $\mathbb{R}$ (note the supports of measures are bounded). 
We call $\nu_\infty (\mathcal{F})$ the \textit{spectral measure} of $\mathcal{F}$. 
For $\mathcal{F} = \mathcal{F}_{(\mathcal{X}, \mathcal{L})}$ associated to a test configuration $(\mathcal{X}, \mathcal{L})$, we denote the spectral measure by $\DHm_{(\mathcal{X}, \mathcal{L})}$ and call it the \textit{Duistermaat--Heckman measure}. 
There is a relative construction introduced in \cite{BJ2}, which we review in section \ref{Strong topology, d1-topology and dp-topology}.

\subsubsection{Affine toric variety}

Here we review affine toric geometry in order to clarify our notations in polyhedral configuration. 
Let $N$ be a finite rank lattice and $M$ be the dual lattice over $\mathbb{Z}$. 
Let $T = N \otimes \mathbb{G}_m$ be the algebraic torus associated to $N$. 
For $\xi \in N$ and $\mu \in M$, we denote by $\chi_\xi: \mathbb{G}_m \to T$ the one parameter subgroup associated to $\xi$ and by $\chi^\mu: T \to \mathbb{G}_m$ the character associated to $\mu$. 
We have $\chi^\mu \circ \chi_\xi (z) = z^{\langle \mu, \xi \rangle}$. 

The exponential map $\exp: \mathbb{C} \to \mathbb{G}_m$ induces a group homomorphism $\exp: N \otimes \mathbb{C} \to T$. 
We have $\chi^\mu (\exp (\xi)) = e^{\langle \mu, \xi \rangle}$ for $\xi \in N \otimes \mathbb{C}$ and $\mu \in M$. 
The kernel of this map is $2 \pi \sqrt{-1} N \subset N \otimes \mathbb{C}$. 

A cone $\sigma \subset \mathfrak{t}$ is called (i) strictly convex, (ii) full-dimensional and (iii) rational polyhedral if (i) it is convex and it does not contain a line, (ii) its interior is non-empty and (iii) it is the intersection of finitely many half spaces $H^+_i = \{ \xi \in \mathfrak{t} ~|~ \langle \mu_i, \xi \rangle \ge 0 \}$ for $\mu_i \in M$. 
We call a cone satisfying  (i)--(iii) \textit{toric cone}. 
(Usually (i) is not assumed. )

For a toric cone $\sigma \subset \mathfrak{t}$, the affine toric variety $B_\sigma$ associated to $\sigma$ is given by 
\[ B_\sigma := \mathrm{Spec} (\mathbb{C} [\sigma^\vee \cap M]). \]
Closed points of $B_\sigma$ correspond to semigroup homomorphisms $x: (\sigma^\vee \cap M, +) \to (\mathbb{C}, \times)$, where the latter is the multiplicative semigroup, not a group (cf. \cite{CLS}). 
We denote by $o \in B_\sigma$ the point corresponding to the homomorphism 
\[ o: \sigma^\vee \cap M \to \mathbb{C}: \mu \mapsto \begin{cases} 1 & \mu=0 \\ 0 & \mu \neq 0 \end{cases} \]
and denote by $1 \in B_\sigma$ the point corresponding to the homomorphism 
\[ 1: \sigma^\vee \cap M \to \mathbb{C}: \mu \mapsto 1. \]

The $T$-action on $B_\sigma$ is given by $(x. t) (\mu) = \chi^\mu (t) x (\mu)$ for $x \in B_\sigma$, $t \in T$ and $\mu \in \sigma^\vee \cap M$. 
Since $(x. \exp (-\rho \xi)) (\mu) = e^{-\rho \langle \mu, \xi \rangle} x (\mu)$, we have 
\[ \lim_{\rho \to \infty} x. \exp (-\rho \xi) = o \in B_\sigma \]
for every interior point $\xi \in \sigma^\circ \subset \mathfrak{t}$. 

\subsubsection{Toric vector bundle and weight filtration}

Let $E$ be a $T$-equivariant vector bundle over an affine toric variety $B_\sigma$. 
For $e \in E_1$, we denote by $\bar{e}$ the rational section of $E$ given by $\bar{e} (\tau) = e. \tau$ for $\tau \in T \subset B_\sigma$. 
For a rational section $s$ of $E$ and $t \in T$, we define a rational section $s. t$ of $E$ by $(s. t) (b) := s (b. t^{-1}). t$. 
Finally, for a rational section $s$ of $E$ and $\mu \in M$, we define a rational section $\chi^{-\mu} s$ by $(\chi^{-\mu} s) (\tau) = \chi^{-\mu} (\tau) s (\tau)$ for $\tau \in T$. 

For $e \in E_1$ and $t \in T$, we have $\bar{e}. t = \bar{e}$ as 
\[ (\bar{e}. t) (\tau) = \bar{e} (\tau t^{-1}). t = (e. \tau t^{-1}). t = e. \tau = \bar{e} (\tau). \]
Conversely, any rational section $s$ satisfying $s. t = s$ can be written as $s = \bar{e}$ for $e = s_1$. 

For a rational section $s$ of $E$ and $\mu \in M$, we have $(\chi^{-\mu} s). t = \chi^\mu (t) (\chi^{-\mu} (s.t))$ as 
\begin{align*} 
((\chi^{-\mu} s). t) (b)
&= (\chi^{-\mu} s) (b t^{-1}). t = (\chi^{-\mu} (b) \chi^{-\mu} (t^{-1}) s (b t^{-1})). t 
\\
&= \chi^\mu (t) (\chi^{-\mu} (b) (s. t) (b)) = \chi^\mu (t) (\chi^{-\mu} (s.t) (b)). 
\end{align*}
It follows that $(\chi^{-\mu} \bar{e}). t = \chi^\mu (t) (\chi^{-\mu} \bar{e})$ for $e \in E_1$ and $\mu \in M$ and conversely, any rational section $s$ satisfying $s. t = \chi^\mu (t) s$ can be written as $s = \chi^{-\mu} \bar{e}$ for $e = s_1$. 

For $\mu \in M$, we consider the following subspaces of $E_1$ 
\begin{align} 
\mathcal{F}^\mu_E 
&:= \{ e \in E_1 ~|~ \chi^{-\mu} \bar{e} \text{ extends to a global section of } E \}, 
\\
\mathcal{F}^{\mu+}_E 
&:= \sum_{\mu' \gneq_\sigma \mu} \mathcal{F}^{\mu'}_E. 
\end{align}

\begin{prop}
\label{toric vector bundle via weight filtration}
Let $\mathcal{I}_o$ be the defining ideal of $o \in B_\sigma$. 
Then we have 
\begin{gather*} 
H^0 (B_\sigma, E) = \bigoplus_{\mu \in M} \chi^{-\mu} \mathcal{F}^\mu_E, 
\\
H^0 (B_\sigma, \mathcal{I}_o \otimes E) = \bigoplus_{\mu \in M} \chi^{-\mu} \mathcal{F}^{\mu+}_E. 
\end{gather*}
\end{prop}

\begin{proof}
We note $E|_T$ is trivial as the $T$-action is free on $T \subset B_\sigma$, so that $H^0 (T, E)$ decomposes into eigenspaces $H^0 (T, E) \cong \mathbb{C} [M]^{\oplus r} = \bigoplus_{\mu \in M} (\mathbb{C} \chi^\mu)^{\oplus r}$, even though the representation $H^0 (T, E)$ is infinite dimensional. 
Since $H^0 (B_\sigma, E) \to H^0 (T, E)$ is injective, the weight decomposition of $H^0 (T, E)$ inherits to $H^0 (B_\sigma, E)$: 
\[ H^0 (B_\sigma, E) = \bigoplus_{\mu \in M} H^0 (B_\sigma, E)_\mu, \]
where $H^0 (B_\sigma, E)_\mu$ consists of $s \in H^0 (B_\sigma, E)$ with $s. t = \chi^\mu (t) s$. 

As explained, we can write $s \in H^0 (B_\sigma, E)_\mu$ as $s = \chi^{-\mu} \bar{e}$ for some $e \in E_1$. 
Since $s$ is a global section, we have $e \in \mathcal{F}^\mu E_1$. 
Thus we get 
\[ H^0 (B_\sigma, E)_\mu = \chi^{-\mu} \mathcal{F}^\mu_E, \]
which proves the first claim. 

Since $B_\sigma$ is affine, we have 
\[ H^0 (B_\sigma, \mathcal{I}_o \otimes E) = H^0 (B_\sigma, \mathcal{I}_o) \otimes_{H^0 (B_\sigma, \mathcal{O})} H^0 (B_\sigma, E). \]
On the other hand, we have 
\[ H^0 (B_\sigma, \mathcal{O}) = \mathbb{C} [\sigma^\vee \cap M], \quad H^0 (B_\sigma, \mathcal{I}_o) = \bigoplus_{0 \neq \mu \in \sigma^\vee \cap M} \mathbb{C}. \chi^\mu. \]
Then the last claim follows by 
\begin{align*} 
H^0 (B_\sigma, \mathcal{I}_o \otimes E) 
&= \bigoplus_{0 \neq \tilde{\mu} \in \sigma^\vee \cap M} \mathbb{C}. \chi^{\tilde{\mu}} \otimes_{\mathbb{C} [\sigma^\vee \cap M]} \bigoplus_{\mu \in M} \chi^{-\mu'} \mathcal{F}^{\mu'}_E 
\\
&= \sum_{0 \neq \tilde{\mu} \in \sigma^\vee \cap M} \bigoplus_{\mu' \in M} \chi^{-(\mu' - \tilde{\mu})} \mathcal{F}^{\mu'}_E = \bigoplus_{\mu \in M} \chi^{-\mu} \sum_{\mu' \gneq_\sigma \mu} \mathcal{F}^{\mu'}_E = \bigoplus_{\mu \in M} \chi^{-\mu} \mathcal{F}^{\mu+}_E. 
\end{align*}
\end{proof}

\begin{prop}
\label{weight filtration of quotient bundle}
Let $E, E'$ be a $T$-equivariant vector bundle over $B_\sigma$ and $\phi: E \twoheadrightarrow E'$ be a $T$-equivariant surjective map of vector bundles. 
Then $\mathcal{F}^\mu_{E'}$ is the image of $\mathcal{F}^\mu_E$ along $\phi_1: E_1 \twoheadrightarrow E'_1$. 
\end{prop}

\begin{proof}
By the above proposition, we can identify $\mathcal{F}^\mu_E$ with $H^0 (B_\sigma, E)_\mu$ and $\mathcal{F}^\mu_{E'}$ with $H^0 (B_\sigma, E')_\mu$. 
Since the induced map $H^0 (B_\sigma, E) \to H^0 (B_\sigma, E')$ is $T$-equivariant and surjective, $H^0 (B_\sigma, E')_\mu$ is the image of $H^0 (B_\sigma, E)_\mu$ by Schur's lemma, which shows the claim. 
\end{proof}

We introduce a partial order $\le_\sigma$ on $M$: 
\begin{equation} 
\mu \le_\sigma \mu' \iff \mu' - \mu \in \sigma^\vee \cap M. 
\end{equation}
Then we have 
\[ \mu \le_\sigma \sigma' \Rightarrow \mathcal{F}^{\mu'}_E \subset \mathcal{F}^\mu_E. \]
We also define 
\[ \mu \lneq_\sigma \mu' \iff \mu \le_\sigma \mu' \text{ and } \mu \neq \mu'. \]

We can describe $\mathcal{F}^\mu_E$ using a `diagonal' basis of $E_1$ as follows. 

\begin{lem}
\label{equivariant trivialization}
For any $T$-equivariant vector bundle $E$ over an affine toric variety $B_\sigma$, there exists a basis $\{ e_i \}_{i=1}^r$ of $E_1$ and a collection of characters $\{ \mu_i \in M \}_{i=1}^r$ such that each rational section $\chi^{-\mu_i} \bar{e}_i$ extends to a regular section of $E$ and $\{ \chi^{-\mu_i} \bar{e}_i \}_{i=1}^r$ gives a trivialization of $E$. 
For such basis, we have 
\[ \mathcal{F}^\mu_E = \langle e_i ~|~ \mu_i \ge_\sigma \mu \rangle. \]
\end{lem}

The first claim is equivalent to the following: $E$ is $T$-equivariantly isomorphic to the pull-back $p^* E_o = E_o$ of the fibre $E_o$ over the fixed point $o \in B_\sigma$ along $p: B_\sigma \to o$: the trivial bundle $B_\sigma \times E_o$ endowed with the $T$-action $(b, v). t = (b.t, v.t)$. 

\begin{proof}
As $B_\sigma$ is affine, for any basis $\{ e_{o, i} \}_{i=1}^r$ of $E_o$, we can take sections $\{ s_i \}_{i=1}^r$ of $E$ so that each $s_i (o) = e_{o, i}$. 
Take $\{ e_{o, i} \}_{i=1}^r$ so that each $e_{o, i}$ is an eigenvector of a character $\mu_i \in M$. 
In this case, we can take each $s_i$ as an eigensection of $\mu_i$, hence $s_i = \chi^{-\mu_i} \bar{e}_i$ for $e_i = s_i (1) \in E_1$, by replacing the original $s_i$ with the eigencomponent $a_{i \mu_i} s_{i \mu_i}$ appearing in the weight decomposition $s_i = \sum_\mu a_{i \mu} s_{i \mu}$: $s_i (o) = a_{i \mu_i} s_{i \mu_i} (o)$. 
Since $o \in B_\sigma$ is in the closure of any $T$-orbit, the set of points of $B_\sigma$ such that $\{ s_i (b) \}$ does not form a basis of $E_b$ is empty, so that $\{ s_i \}_{i=1}^r$ gives a trivialization of $E$. 

Finally, we compute 
\begin{align*}
\mathcal{F}^\mu_E 
&= \{ \sum_i a_i e_i ~|~ a_i \neq 0 \Rightarrow \chi^{\mu_i} \bar{e}_i = \chi^{\mu_i - \mu} \chi^\mu \bar{e}_i \in H^0 (B_\sigma, E) \} 
\\
&= \{ \sum_i a_i e_i ~|~ a_i \neq 0 \Rightarrow \chi^{\mu_i - \mu} \in H^0 (B_\sigma, \mathcal{O}) \}
\\
&= \langle e_i ~|~ \mu_i \ge_\sigma \mu \rangle. 
\end{align*}

%Equivalence of the claims: 
%Let $\phi: E \to p^* E_o$ be a $T$-equivariant isomorphism. 
%Let $\varepsilon \in E_o$ be an eigenvector of a character $\mu \in M$: $\varepsilon. t = \chi^\mu (t) \varepsilon$. 
%Then we have $(\phi^{-1} p^* \varepsilon). t = \phi^{-1} p^* (\varepsilon. t) = \phi^{-1} p^* (\chi^\mu (t) \varepsilon) = \chi^\mu (t) (\phi^{-1} p^* \varepsilon)$, so that $(\chi^\mu \phi^{-1} p^* \varepsilon). t = \chi^{-\mu} (t) (\chi^\mu (\chi^\mu (t) \phi^{-1} p^* \varepsilon)) = \chi^\mu \phi^{-1} p^* \varepsilon$. 
%It follows that $\phi^{-1} p^* \varepsilon = \chi^{-\mu} \bar{e}$ for some $e \in E_1$. 
\end{proof}

\begin{prop}
We have 
\[ \mathcal{F}^\mu_E \cap \mathcal{F}^{\mu'}_E = \sum_{\mu'' \ge_\sigma \mu, \mu'} \mathcal{F}^{\mu''}_E. \]
\end{prop}

\begin{proof}
For $\mu'' \ge_\sigma \mu, \mu'$, we have $\mathcal{F}^{\mu''}_E \subset \mathcal{F}^\mu_E \cap \mathcal{F}^{\mu'}_E$. 
Take a basis $\{ e_i \}_{i=1}^r$ of $E_1$ as in the above lemma. 
Then the reverse inclusion follows by 
\[ \mathcal{F}^\mu_E \cap \mathcal{F}^{\mu'}_E = \langle e_i ~|~ \mu_i \ge_\sigma \mu, \mu' \rangle \subset \sum_{\mu'' \ge_\sigma \mu, \mu'} \mathcal{F}^{\mu'}_E. \]
\end{proof}

%\begin{rem}
%Suppose $\sigma$ is a smooth toric cone, i.e. a cone spanned by a basis of $N$. 
%Then for $\mu, \mu' \in M$, we have a unique $\mu'' \in M$ satisfying 
%\[ \mu'' + \sigma^\vee = (\mu + \sigma^\vee) \cap (\mu' + \sigma^\vee) \]
%and so we have 
%\[ \mathcal{F}^{\mu''} = \mathcal{F}^\mu \cap \mathcal{F}^{\mu'}. \]
%We denote $\mu''$ by $\mu \vee \mu'$. 
%
%Proof: fix an isomorphism $M \cong \mathbb{Z}^r$ so that $\sigma^\vee$ maps to $[0, \infty)^r$. 
%Then since 
%\[ ((m_i)_{i=1}^r + [0, \infty)^r) \cap ((m'_i)_{i=1}^r + [0, \infty)^r) = (\max \{ m_i, m'_i \})_{i=1}^r + [0, \infty)^r, \]
%$\mu'' \in M$ corresponding to $(\max \{ m_i, m'_i \})_{i=1}^r$ enjoys the desired property. 
%\end{rem}

For $\xi \in \sigma$ and $\lambda \in \mathbb{R}$, we put 
\begin{equation} 
\mathcal{F}^\lambda_{E, \xi} := \sum_{\langle \mu, \xi \rangle \ge \lambda} \mathcal{F}^\mu_E \subset E_1. 
\end{equation}
Taking a basis $\{ e_i \}_{i=1}^r$ of $E_1$ as in the above lemma, we can express it as 
\begin{equation} 
\mathcal{F}^\lambda_{E, \xi} = \langle e_i ~|~ \langle \mu_i, \xi \rangle \ge \lambda \rangle. 
\end{equation}
Then we can easily check the following. 
\begin{itemize}
\item $\mathcal{F}^{\lambda'}_{E, \xi} = \bigcap_{\lambda < \lambda'} \mathcal{F}^\lambda_{E, \xi}$, 

\item $\mathcal{F}^\lambda_{E, \xi} = 0$ for $\lambda \ge \max_i \langle \mu_i, \xi \rangle$ and $\mathcal{F}^\lambda_{E, \xi} = E_1$ for $\lambda \le \min_i \langle \mu_i, \xi \rangle$. 
\end{itemize}
The family of filtrations $\{ \mathcal{F}_{E, \xi} \}_{\xi \in \sigma}$ recover the weight filtration $\mathcal{F}^\mu_E$ by 
\[ \mathcal{F}^\mu_E = \bigcap_{\langle \mu, \xi \rangle \ge \lambda, \xi \in \sigma} \mathcal{F}^\lambda_{E, \xi}. \]

We put 
\begin{equation} 
\mathcal{F}^{\lambda+}_{E, \xi} := \sum_{\lambda < \lambda'} \mathcal{F}^\lambda_{E, \xi}. 
\end{equation}
For $\mu \lneq_\sigma \mu'$ and $\xi \in \sigma^\circ$, we have $\langle \mu, \xi \rangle < \langle \mu', \xi \rangle$, so that we have an inclusion $\mathcal{F}^{\mu+}_E \subset \mathcal{F}^{\langle \mu, \xi \rangle+}_{E, \xi}$ for $\xi \in \sigma^\circ$. 
It induces a map 
\[ \mathcal{F}^\mu_E/\mathcal{F}^{\mu+}_E \to \mathcal{F}^{\langle \mu, \xi \rangle}_{E, \xi}/\mathcal{F}^{\langle \mu, \xi \rangle+}_{E, \xi}. \]

\begin{prop}
\label{fibre over the fixed point via filtration and weight filtration}
For $\xi \in \sigma^\circ$, we have a canonical isomorphisms
\[ \mathcal{F}^\lambda_{E, \xi}/\mathcal{F}^{\lambda+}_{E, \xi} \cong \bigoplus_{\langle \mu, \xi \rangle = \lambda} \mathcal{F}^\mu_E/\mathcal{F}^{\mu+}_E \cong E_{o, \lambda} := \bigoplus_{\langle \mu, \xi \rangle = \lambda} E_{o, \mu}, \]
where $E_{o, \mu}$ denotes the eigenspace of the character $\mu \in M$: $E_o = \bigoplus_{\mu \in M} E_{o, \mu}$. 
\end{prop}

\begin{proof}
We remark for $\mu \neq \mu'$ with $\langle \mu, \xi \rangle \le \langle \mu', \xi \rangle$, we have $\mu'' \gneq_\sigma \mu$ for every $\mu'' \ge_\sigma \mu, \mu'$, so that we get 
\[ \mathcal{F}^\mu_E \cap \mathcal{F}^{\mu'}_E = \sum_{\mu'' \ge_\sigma \mu, \mu'} \mathcal{F}^{\mu''}_E \subset \sum_{\mu'' \gneq_\sigma \mu} \mathcal{F}^{\mu''}_E \subset \mathcal{F}^{\mu+}_E. \]

We firstly see the map $\mathcal{F}^\mu_E/\mathcal{F}^{\mu+}_E \to \mathcal{F}^\lambda_{E, \xi}/\mathcal{F}^{\lambda+}_{E, \xi}$ is injective for $\mu$ with $\langle \mu, \xi \rangle = \lambda$. 
The kernel consists of the image of $\mathcal{F}^\mu_E \cap \mathcal{F}^{\lambda +}_{E, \xi} = \sum_{\langle \mu', \xi \rangle > \lambda} \mathcal{F}^\mu_E \cap \mathcal{F}^{\mu'}_E$. 
By the above remark, we have $\mathcal{F}^\mu_E \cap \mathcal{F}^{\lambda +}_{E, \xi} \subset \mathcal{F}^{\mu+}_E$, which proves the injectivity. 

For $\mu \neq \mu'$ with $\langle \mu, \lambda \rangle = \langle \mu, \xi \rangle$, we have $\mathcal{F}^\mu_E \cap \mathcal{F}^{\mu'}_E \subset \mathcal{F}^{\mu+}_E \cap \mathcal{F}^{\mu'+}_E$. 
It follows that $\mathcal{F}^\mu_E/\mathcal{F}^{\mu+}_E \cap \mathcal{F}^{\mu'}_E/\mathcal{F}^{\mu'+}_E = 0$, hence the sum of the subspaces $\mathcal{F}^\mu_E/\mathcal{F}^{\mu+}_E \subset \mathcal{F}^\lambda_{E, \xi}/\mathcal{F}^{\lambda+}_{E, \xi}$ is a direct product. 

Since 
\[ \mathcal{F}^\lambda_{E, \xi}/\mathcal{F}^{\lambda+}_{E, \xi} = \sum_{\langle \mu, \xi \rangle = \lambda} \mathcal{F}^\mu_E/\mathcal{F}^\mu_E \cap \mathcal{F}^{\lambda+}_{E, \xi}, \]
the map $\bigoplus_{\langle \mu, \xi \rangle = \lambda} \mathcal{F}^\mu_E/\mathcal{F}^{\mu+}_E \to \mathcal{F}^\lambda_{E, \xi}/\mathcal{F}^{\lambda+}_{E, \xi}$ is surjective. 
Therefore, we obtain the first isomorphism. 

To see the second isomorphism, we construct an isomorphism $E_{o, \mu} \to \mathcal{F}^\mu_E/\mathcal{F}^{\mu+}_E$. 
As in the previous lemma, for each $e_o \in E_{o, \mu}$, we can find $e \in \mathcal{F}^\mu_E$ satisfying $(\chi^{-\mu} \bar{e}) (o) = e_o$. 
The element $[e] \in \mathcal{F}^\mu_E/\mathcal{F}^{\mu+}_E$ is independent of the choice of such $e$. 
Indeed, take another $e' \in \mathcal{F}^\mu_E$ with $(\chi^{-\mu} \bar{e}') (o) = e_o$, then since $(\chi^{-\mu} \overline{e-e'}) (o) = 0$, we have $\chi^{-\mu} \overline{e-e'} \in H^0 (B_\sigma, \mathcal{I}_o \otimes E)$, hence $e - e' \in \mathcal{F}^{\mu+}_E$ by Proposition \ref{toric vector bundle via weight filtration}. 
Thus we get a well-defined linear map $E_{o, \mu} \to \mathcal{F}^\mu_E/\mathcal{F}^{\mu+}_E: e_o \mapsto [e]$. 
The inverse map is given by $[e] \mapsto (\chi^{-\mu} \bar{e}) (o)$, which is well-defined as $(\chi^{-\mu} \bar{e}) (o) = 0$ for $e \in \mathcal{F}^{\mu+}_E$. 
\end{proof}

Let $\sigma \subset N_\mathbb{R}$ and $\sigma' \subset N'_\mathbb{R}$ be toric cones and $\phi_*: N' \to N$ be a morphism of lattices which induces a linear map $\phi_*: N_\mathbb{R} \to N'_\mathbb{R}$ mapping $\sigma'$ into $\sigma$: $\phi_* (\sigma') \subset \sigma$. 
Then we have the induced morphism $\phi: B_{\sigma'} \to B_\sigma$ of toric varieties, which maps $o$ to $o$ and $1$ to $1$. 
Since $(\phi^* E)_1 = E_1$, we can compare $\mathcal{F}_E$ and $\mathcal{F}_{\phi^* E}$. 

\begin{prop}
\label{base change of weight filtration}
We have 
\begin{gather*} 
\mathcal{F}^{\mu'}_{\phi^* E} = \sum_{\phi^* \mu \ge_{\sigma'} \mu'} \mathcal{F}^{\mu}_E, 
\\
\mathcal{F}^\lambda_{\phi^* E, \xi'} = \mathcal{F}^\lambda_{E, \phi_* \xi'}
\end{gather*}
for $\xi' \in \sigma'$. 
\end{prop}

\begin{proof}
Take a basis $\{ e_i \}_{i=1}^r$ of $E_1$ as in the previous lemma. 
Then the basis $\{ \phi^* e_i = e_i \}_{i=1}^r$ of $\phi^* E_1 = E_1$ enjoys the same property for the collection $\{ \phi^* \mu_i \in M' \}_{i=1}^r$ of weights. 
Thus we have 
\[ \mathcal{F}^{\mu'}_{\phi^* E} = \langle e_i ~|~ \phi^* \mu_i \ge_{\sigma'} \mu' \rangle = \sum_{\phi^* \mu \ge_{\sigma'} \mu'} \langle e_i ~|~ \mu_i \ge_\sigma \mu \rangle = \sum_{\phi^* \mu \ge_{\sigma'} \mu'} \mathcal{F}^{\mu}_E. \]
As a consequence, we get 
\[ \mathcal{F}^\lambda_{\phi^* E, \xi'} = \sum_{\langle \mu', \xi' \rangle \ge \lambda} \mathcal{F}^{\mu'}_{\phi^* E} = \sum_{\langle \mu', \xi' \rangle \ge \lambda} \sum_{\phi^* \mu \ge_{\sigma'} \mu'} \mathcal{F}^{\mu}_E = \sum_{\langle \phi^* \mu, \xi' \rangle \ge \lambda} \mathcal{F}^{\mu}_E = \mathcal{F}^\lambda_{E, \phi_* \xi'}. \]
\end{proof}

\subsection{Polyhedral configuration and family of filtrations}

\subsubsection{Polyhedral configuration}

Now we introduce polyhedral configuration. 
The notion is a simple generalization of test configuration to general affine toric base $B_\sigma$. 
Let $(X, L)$ be a polarized scheme. 

\begin{defin}[polyhedral configuration]
\label{polyhedral configuration}
Let $\sigma \subset \mathfrak{t}$ be a toric cone. 
A \textit{$\sigma$-configuration} of $(X, L)$ is a $T$-equivariant proper flat family of polarized schemes $\pi: (\mathcal{X}, \mathcal{L}) \to B_\sigma$ endowed with a $T$-equivariant relatively ample $\mathbb{Q}$-line bundle $\mathcal{L}$ and an isomorphism $\iota: (X, L) \xrightarrow{\sim} (\mathcal{X}_1, \mathcal{L}|_{\mathcal{X}_1})$. 

We call the fibre $\mathcal{X}_o$ over the point $o \in B_\sigma$ the \textit{central fibre} and the fibre $\mathcal{X}_1$ over the point $1 \in B_\sigma$ the \textit{general fibre}. 
\end{defin}

A $T$-equivariant relatively ample $\mathbb{Q}$-line bundle is a pair $\mathcal{L} = (l, \hat{\mathcal{L}})$ of a positive integer $l$ and a $T$-equivariant relatively ample line bundle $\hat{\mathcal{L}}$. 
We often denote $\hat{\mathcal{L}}$ by $l \mathcal{L}$. 
An isomorphism of $T$-equivariant $\mathbb{Q}$-line bundles $\mathcal{L}, \mathcal{L}'$ is a $T$-equivariant isomorphism of line bundles $\mathcal{L}^{\otimes m} := (l \mathcal{L})^{\otimes m/l } \cong (l' \mathcal{L}')^{\otimes m/l'} =: (\mathcal{L}')^{\otimes m}$ for some $m \in \mathbb{Z}$ dividing $l, l'$. 

We often identify $\mathcal{L}$ with its equivariant first Chern class $c_{1, T} (\mathcal{L}) = l^{-1} c_{1, T} (l \mathcal{L}) \in H^2_T (\mathcal{X}, \mathbb{R})$, whereas we later make use of the $\mathbb{Q}$-line bundle structure of $\mathcal{L}$ in order to assign a family of filtrations $\{ \mathcal{F}_{(\mathcal{X}, \mathcal{L}; \xi)} \}_{\xi \in \sigma}$. 

For an element $\xi \in \sigma \cap N$, we associate a $\mathbb{G}_m$-equivariant morphism $\overline{\exp}_\xi: \mathbb{A}^1 \to B_\sigma$ by 
\[ t \mapsto  \begin{cases} 1. \chi_\xi (t) = (\mu \mapsto t^{\langle \mu, \xi \rangle}) & t \neq 0 \\ 0 & t =0 \end{cases}. \] 
For a $\sigma$-configuration $(\mathcal{X}/B_\sigma, \mathcal{L})$, we denote by $(\mathcal{X}_\xi, \mathcal{L}_\xi)$ the pull-back along $\overline{\chi}_\xi$. 
It gives a (non-normal) test configuration of $(X, L)$. 

%A $[0, \infty)$-configuration is nothing but a test configuration, however, we regard the object as a family of non-archimedean psh metrics $\{ \varphi_\rho \}$ rescaled by $\rho \in [0, \infty)$, rather than as a single metric $\varphi_1$, when specifically using the former name. 
%
%\begin{rem}
%For a general $\xi \in \sigma$, we can associate an analytic map $\exp: \mathbb{A}^1 \to B_\sigma$ by 
%\[ \tau \mapsto 1. \exp (\tau \xi) = (\mu \mapsto e^{\langle \mu, \tau \xi \rangle}), \]
%which is equivariant with respect to the group homomorphism $\mathbb{G}_a \to T: \tau \mapsto \exp (\tau \xi)$, where $\mathbb{G}_a = (\mathbb{C}, +)$. 
%The map continuously extends to a map from ${^b \mathbb{C}} := [ -\infty, \infty) \times \mathbb{R}$ by assigning $0$ to the boundary $\partial {^b \mathbb{C}} = \{ -\infty \} \times \mathbb{R}$. 
%We suspect that we can consider a category of ``complex analytic spaces with boundary at infinity'' in which we obtain a pull-back $\mathcal{X}_\xi \to {^b \mathbb{C}}$ along the above morphism. 
%\end{rem}

\begin{eg}
\label{product sigma-configuration}
Let $(X, L)$ be a polarized scheme with a $T$-action. 
For any toric cone $\sigma \subset \mathfrak{t}$, we can construct a $\sigma$-configuration 
\begin{equation} 
(X_\sigma, L_\sigma) = (X \times B_\sigma, L \times B_\sigma) 
\end{equation}
whose $T$-action is given by $(x, b). t = (x.t, b.t)$. 

A similar construction can be applied to test configuration. 
Let $(\mathcal{X}, \mathcal{L})$ be a $T$-equivariant test configuration of a $T$-equivariant polarized scheme $(X, L)$. 
For a toric cone $\sigma \subset \mathfrak{t}$, we get a $[0, \infty) \times \sigma$-configuration 
\begin{equation}
(\mathcal{X}_\sigma, \mathcal{L}_\sigma) = (\mathcal{X} \times B_\sigma, \mathcal{L} \times B_\sigma) 
\end{equation} 
whose $\mathbb{G}_m \times T$-action is given by $(x, b). (s, t) = (x.s.t, b.t)$. 
\end{eg}

%Later, we study a filtration $\mathcal{F}_{(\mathcal{X}, \mathcal{L}; \xi)}$ associated to the pair $(\mathcal{X}, \mathcal{L}; \xi)$ of a $\sigma$-configuration $(\mathcal{X}, \mathcal{L})$ and a vector $\xi \in \sigma$. 
%For the above example, we observe that the filtrations $\mathcal{F}_{(X_\sigma, L_\sigma; \xi)}, \mathcal{F}_{(\mathcal{X}_\sigma, \mathcal{L}_\sigma; (\rho, \xi))}$ are independent of the choice of the cone $\sigma$ with $\xi \in \sigma$: $\mathcal{F}_{(X_\sigma, L_\sigma; \xi)} = \mathcal{F}_{(X_{\sigma'}, L_{\sigma'}; \xi)}$, $\mathcal{F}_{(\mathcal{X}_\sigma, \mathcal{L}_\sigma; (\rho, \xi))} = \mathcal{F}_{(\mathcal{X}_{\sigma'}, \mathcal{L}_{\sigma'}; (\rho, \xi))}$ for $\sigma, \sigma'$ with $\xi \in \sigma, \sigma'$. 

\begin{eg}
\label{polyhedral configuration via Hilbert scheme}
Let $X \hookrightarrow \mathbb{C}P^N$ be the Kodaira embedding by the linear system $|mL|$. 
It defines a point $[X]$ of Hilbert scheme $\mathrm{Hilb} (\mathbb{C}P^N)$. 
For a torus $T'$-action on $\mathbb{C}P^N$, the $T'$-equivariant morphism $f: T' \to \mathrm{Hilb} (\mathbb{C}P^N): t \mapsto [X]. t$ of schemes gives a rational map $f: B' \dashrightarrow \mathrm{Hilb} (\mathbb{C}P^N)$ from a projective toric variety $B' \supset T'$. 
By resolving the indeterminancy, we obtain a $T'$-equivariant morphism $\tilde{f}: B_\Sigma \to \mathrm{Hilb} (\mathbb{C}P^N)$ from a normal projective toric variety $B_\Sigma$ associated to a fan $\Sigma$. 
Blowing up further if necessary, we may assume each $\sigma \in \Sigma$ is strictly convex. 
Fix $\sigma \in \Sigma$ and put $\mathfrak{t} := \mathbb{R} \sigma$, then $\sigma$ is full-dimensional in $\mathfrak{t}$. 
Let $T \subset T'$ be the subtorus associated to $\mathfrak{t}$. 
Then $B_{\sigma \subset \mathfrak{t}} = B_\sigma/(T'/T)$ is a toric $T$-variety and we have a $T$-equivariant morphism $B_{\sigma \subset \mathfrak{t}} \to B_\Sigma$. 
We obtain a $\sigma$-configuration $(\mathcal{X}/B_\sigma, \mathcal{L})$ by pulling back the universal family $(\mathcal{U}, \mathcal{O} (1)|_\mathcal{U}) \to \mathrm{Hilb} (\mathbb{C}P^N)$ along $\tilde{f}|_{B_{\sigma \subset \mathfrak{t}}}$. 
\end{eg}

\subsubsection{Polyhedral configuration and family of filtrations}

We assign a family of filtrations $\{ \mathcal{F}_{(\mathcal{X}, \mathcal{L}; \xi)} \}_{\xi \in \sigma}$ to a polyhedral configuration $(\mathcal{X}/B_\sigma, \mathcal{L})$. 
As we will see, each filtration in the family is finitely generated, and conversely, every finitely generated filtration can be obtained in this way (but not canonical). 
In the study of K-stability, the practical use of general finitely generated filtration is firstly unveiled in \cite{CSW} (see also \cite{Sze2}). 
It is called $\mathbb{R}$-degeneration in \cite{DS} and $\mathbb{R}$-test configuration in \cite{BJ4}. 
A gemstone of polyhedral configuration is appeared in \cite{HL2}, where polyhedral configuration is recognized as a geometric realization of a single finitely generated filtration. 
Here we strengthen polyhedral configuration is useful because it realizes an intuitive treatment of \textit{family of filtrations} rather than a single filtration. 

Let $(\mathcal{X}/B_\sigma, \mathcal{L})$ be a polyhedral configuration with $\mathcal{L} = (l, \hat{\mathcal{L}})$. 
In what follows, we take $d \in \mathbb{N}_+$ so that $R^{(d)} = \bigoplus_{m \in \mathbb{N}^{(d)}} R_m$ is generated in $R_d$, $l$ divides $d$ and $R^i \pi_* \mathcal{L}^{\otimes m} = 0$ for every $i \ge 1$ and $m \in \mathbb{N}^{(d)}$. 
Then $\pi_* \mathcal{L}^{\otimes m}$ is locally free for $m \in \mathbb{N}^{(d)}$, hence gives a $T$-equivariant vector bundle. 
We can identify the fibre $(\pi_* \mathcal{L}^{\otimes m})_1$ of $1 \in B_\sigma$ with $R_m = H^0 (X, L^{\otimes m})$ via the given isomorphism $\iota: (X, L) \cong (\mathcal{X}_1, \mathcal{L}|_{\mathcal{X}_1})$. 
We also have
\[ H^i (\mathcal{X}, \mathcal{L}^{\otimes m}) = 
\begin{cases} 
H^0 (B_\sigma, \pi_* \mathcal{L}^{\otimes m}) 
& i=0 
\\ 0 
& i > 0 
\end{cases} \]
and 
\[ H^i (\mathcal{X}, \mathcal{I}_{\mathcal{X}_o} \otimes \mathcal{L}^{\otimes m}) = 
\begin{cases} 
H^0 (B_\sigma, \mathcal{I}_o \otimes \pi_* \mathcal{L}^{\otimes m}) 
& i=0 
\\ 0 
& i > 0 
\end{cases} \]
by Leray spectral sequence. 
(Note $B_\sigma$ is affine. )

For $\mu \in M$ and $m \in \mathbb{N}^{(d)}$, we put 
\begin{align}
\mathcal{F}^\mu_{(\mathcal{X}, \mathcal{L})} R_m 
&:= \{ s \in H^0 (X, L^{\otimes m}) ~|~ \chi^{-\mu} \bar{s}  \text{ extends to a section of } \mathcal{L}^{\otimes m} \}
\\ \notag
&= \mathcal{F}^\mu_{\pi_* \mathcal{L}^{\otimes m}} (\pi_* \mathcal{L}^{\otimes m})_1. 
\end{align}
We obviously have 
\begin{equation} 
\mathcal{F}_{(\mathcal{X}, \mathcal{L})}^\mu R_m \cdot \mathcal{F}_{(\mathcal{X}, \mathcal{L})}^{\mu'} R_{m'} \subset \mathcal{F}_{(\mathcal{X}, \mathcal{L})}^{\mu+\mu'} R_{m+m'}. 
\end{equation}

We can recover the $\sigma$-configuration $(\mathcal{X}/B_\sigma, \mathcal{L})$ from $\{ \mathcal{F}^\mu_{(\mathcal{X}, \mathcal{L})} \}_{\mu \in M}$ as follows. 

\begin{prop}
There are natural isomorphisms of rings: 
\begin{gather}
S (B_\sigma) := \mathbb{C} [\sigma^\vee \cap M] \cong \bigoplus_{\mu \in M} \chi^{- \mu} \mathcal{F}^\mu_{(\mathcal{X}, \mathcal{L})} R_0, 
\\
\mathcal{R} (\mathcal{X}, \mathcal{L}) := \bigoplus_{m \in \mathbb{N}^{(d)}} H^0 (\mathcal{X}, \mathcal{L}^{\otimes m}) \cong \bigoplus_{m \in \mathbb{N}^{(d)}} \bigoplus_{\mu \in M} \chi^{- \mu} \mathcal{F}^\mu_{(\mathcal{X}, \mathcal{L})} R_m, 
\\
\mathcal{R} (\mathcal{X}_o, \mathcal{L}|_{\mathcal{X}_o}) := \bigoplus_{m \in \mathbb{N}^{(d)}} H^0 (\mathcal{X}_o, \mathcal{L}|_{\mathcal{X}_o}^{\otimes m}) \cong \bigoplus_{m \in \mathbb{N}^{(d)}} \bigoplus_{\mu \in M} \chi^{- \mu} \frac{\mathcal{F}^\mu_{(\mathcal{X}, \mathcal{L})} R_m}{\mathcal{F}^{\mu+}_{(\mathcal{X}, \mathcal{L})} R_m}. 
\end{gather}
where we put 
\[ \mathcal{F}^{\mu+}_{(\mathcal{X}, \mathcal{L})} R_m := \sum_{\mu' \gneq_\sigma \mu} \mathcal{F}^{\mu'}_{(\mathcal{X}, \mathcal{L})} R_m. \]
\end{prop}

\begin{proof}
The first line follows by 
\[ \mathcal{F}_{(\mathcal{X}, \mathcal{L})}^\mu R_0 = 
\begin{cases}
\mathbb{C}
& -\mu \in \sigma^\vee \cap M
\\
0 
& \text{otherwise}
\end{cases}. \]

By Proposition \ref{toric vector bundle via weight filtration}, we have 
\[ H^0 (\mathcal{X}, \mathcal{L}^{\otimes m}) = H^0 (B_\sigma, \pi_* \mathcal{L}^{\otimes m}) = \bigoplus_{\mu \in M} \chi^{-\mu} \mathcal{F}^\mu_{(\mathcal{X}, \mathcal{L})} R_m. \]
We obviously have $(\chi^{-\mu} \bar{s}) (\chi^{-\mu'} \bar{s}') = \chi^{-(\mu +\mu')} \overline{s s'}$ for $s \in R_m, s' \in R_{m'}$, so the isomorphism preserves the ring structures, hence we get the second line. 

Again by Proposition \ref{toric vector bundle via weight filtration}, we have 
\[ H^0 (\mathcal{X}, \mathcal{I}_{\mathcal{X}_o} \otimes \mathcal{L}^{\otimes m}) = \bigoplus_{\mu \in M} \chi^{-\mu} \mathcal{F}^{\mu+}_{(\mathcal{X}, \mathcal{L})} R_m. \]
Thanks to the cohomology vanishing, we compute 
\begin{align*} 
H^0 (\mathcal{X}_o, \mathcal{L}|_{\mathcal{X}_o}^{\otimes m}) 
&= H^0 (\mathcal{X}, \mathcal{L}^{\otimes m})/ H^0 (\mathcal{X}, \mathcal{I}_{\mathcal{X}_o} \otimes \mathcal{L}^{\otimes m}) 
\\
&= \bigoplus_{\mu \in M} \chi^{-\mu} (\mathcal{F}^\mu_{(\mathcal{X}, \mathcal{L})} R_m / \mathcal{F}^{\mu+}_{(\mathcal{X}, \mathcal{L})} R_m), 
\end{align*}
which proves the last line. 
\end{proof}

For $\xi \in \sigma$ and $\lambda \in \mathbb{R}$, we put 
\begin{equation}
\mathcal{F}_{(\mathcal{X}, \mathcal{L}; \xi)}^\lambda R_m := \sum_{\langle \mu, \xi \rangle \ge \lambda} \mathcal{F}_{(\mathcal{X}, \mathcal{L})}^\mu R_m \subset R_m. 
\end{equation}
%for $m \ge 1$ and 
%\begin{equation}
%\mathcal{F}_{(\mathcal{X}, \mathcal{L}; \xi)}^\lambda R_0 := 
%\begin{cases}
%\mathbb{C} 
%& -\lambda \in \langle \sigma^\vee \cap M, \xi \rangle
%\\
%0 
%& \text{otherwise}
%\end{cases}.
%\end{equation}

We obviously have $\mathcal{F}^\lambda_{(\mathcal{X}, \mathcal{L}; \rho \xi)} R_m = \mathcal{F}^{\rho^{-1} \lambda}_{(\mathcal{X}, \mathcal{L}; \xi)} R_m$. 
As in the previous section, we can check that $\mathcal{F}_{(\mathcal{X}, \mathcal{L}; \xi)}$ gives a filtration of $(X, L)$. 
The linearly boundedness is less obvious. 
It is a consequence of the finite generation of the filtration, which we will see later. 

Similarly as in the previous section, the family of filtrations $\{ \mathcal{F}_{(\mathcal{X}, \mathcal{L}; \xi)} \}_{\xi \in \sigma}$ recovers the weight filtration $\{ \mathcal{F}^\mu_{(\mathcal{X}, \mathcal{L})} \}$, hence also the polyhedral configuration $(\mathcal{X}/B_\sigma, \mathcal{L})$. 

We can also recover the central fibre $(\mathcal{X}_o, \mathcal{L}_o)$ from any $\mathcal{F}_{(\mathcal{X}, \mathcal{L}; \xi)}$ with $\xi \in \sigma^\circ$. 

\begin{prop}
\label{central fibre via filtration}
For a $\sigma$-configuration $(\mathcal{X}/B_\sigma, \mathcal{L})$ and $\xi \in \sigma^\circ$, we have canonical isomorphisms 
\[ \mathcal{F}^\lambda_{(\mathcal{X}, \mathcal{L}; \xi)} R_m/\mathcal{F}^{\lambda+}_{(\mathcal{X}, \mathcal{L}; \xi)} R_m \cong \bigoplus_{\langle \mu, \xi \rangle = \lambda} \chi^{-\mu} (\mathcal{F}_{(\mathcal{X}, \mathcal{L})}^\mu R_m/\mathcal{F}_{(\mathcal{X}, \mathcal{L})}^{\mu+} R_m) \cong \bigoplus_{\langle \mu, \xi \rangle = \lambda} H^0 (\mathcal{X}_o, \mathcal{L}|_{\mathcal{X}_o}^{\otimes m})_\mu, \]
which induces an isomorphism of graded rings
\[ \mathcal{R} (\mathcal{X}_o, \mathcal{L}|_{\mathcal{X}_o}) \cong \bigoplus_{m \in \mathbb{N}^{(d)}} \bigoplus_{\lambda \in \mathbb{R}} \varpi^{-\lambda} (\mathcal{F}^\lambda_{(\mathcal{X}, \mathcal{L}; \xi)} R_m/\mathcal{F}^{\lambda+}_{(\mathcal{X}, \mathcal{L}; \xi)} R_m). \]
\end{prop}

\begin{proof}
The isomorphisms are given by Proposition \ref{fibre over the fixed point via filtration and weight filtration}. 
\end{proof}

Let $\sigma \subset N_\mathbb{R}$ and $\sigma' \subset N'_\mathbb{R}$ be toric cones and $\phi_*: N' \to N$ be a morphism of lattices which induces a linear map $\phi_*: N_\mathbb{R} \to N'_\mathbb{R}$ mapping $\sigma'$ into $\sigma$: $\phi_* (\sigma') \subset \sigma$. 
Pulling back a $\sigma$-configuration $(\mathcal{X}/B_\sigma, \mathcal{L})$ along $\phi: B_{\sigma'} \to B_\sigma$, we obtain a $\sigma'$-configuration $(\mathcal{X}'/B_{\sigma'}, \mathcal{L}')$. 
Since $\phi^* \pi_* \mathcal{L}^{\otimes m} = \pi'_* (\mathcal{L}')^{\otimes m}$ by cohomology vanishing, we have 
\[ \mathcal{F}^\lambda_{(\mathcal{X}'/B_{\sigma'}, \mathcal{L}'; \xi')} = \mathcal{F}^\lambda_{(\mathcal{X}/B_\sigma, \mathcal{L}; \phi_* \xi')} \]
for $\xi' \in \sigma'$ by Proposition \ref{base change of weight filtration}. 

\subsubsection{Finitely generated filtration}

For $\mathcal{F} = \mathcal{F}_{(\mathcal{X}, \mathcal{L}; \xi)}$, we denote $\| \cdot \|^\mathcal{F}_m$ by $\| \cdot \|^{(\mathcal{X}, \mathcal{L}; \xi)}_m$. 

Take $d$ so that $R^{(d)} = \bigoplus_{m \in \mathbb{N}^{(d)}} R_m$ is generated in $R_d$. 
For a non-archimedean norm $\| \cdot \|$ on $R_d$, we can associate the following associated filtration $\mathcal{F}_{\| \cdot \|}$ on $R^{(d)}$: for $m \in \mathbb{N}^{(d)}$, 
\begin{equation} 
\label{filtration associated to non-archimedean norm}
\mathcal{F}^\lambda_{\| \cdot \|} R_m = \langle \prod_{i =1}^{m/d} s_i ~|~ s_i \in R_d, ~ - \log \prod_{i=1}^{m/d} \| s_i \| \ge \lambda \rangle. 
\end{equation}

Taking a diagonal basis $x^1, \ldots, x^{N_m}$ of $R_m$ with respect to $\| \cdot \|$, we can also write it as
\[ \mathcal{F}^\lambda_{\| \cdot \|} R_m = \Big{\{} \sum_{|I| = m/d} a_I x^I ~\Big{|}~ \sum_{i \in I} \lambda_i \ge \lambda \text{ for } a_I \neq 0. \Big{\}}, \]
where we put $\lambda_i := - \log \| x^i \|$. 

\begin{defin}[finitely generated filtration]
\label{f.g. filtration}
We call $\mathcal{F}$ \textit{finitely generated} if $\mathcal{F}|_{R^{(d)}} = \mathcal{F}_{\| \cdot \|}$ for some $d$ and some non-archimedean norm $\| \cdot \|$ on $R_d$. 
This is equivalent to 
\[ \mathcal{F}^\lambda R_m = \sum_{\substack{|I|= m/d, \\ \sum_{i \in I} \lambda_i \ge \lambda}} \prod_{i \in I} \mathcal{F}^{\lambda_i} R_d \]
for some $d$ and every $m \in \mathbb{N}^{(d)}$. 
\end{defin}

%\begin{lem}
%A filtration $\mathcal{F}$ is finitely generated if and only if the following ring is finitely generated. 
%\[ \mathcal{R} (\mathcal{F}) = \bigoplus_{k \ge 0} \bigoplus_{\lambda \in \Gamma (\mathcal{F})} t^{-\lambda} \mathcal{F}^\lambda_{\min}} R_k. \]
%
%\begin{align*} 
%\bigoplus_{k \ge 0} \bigoplus_{\mu \in M} t^{-\mu} \mathcal{F}^\mu R_k 
%&\to \bigoplus_{k \ge 0} \bigoplus_{\lambda \in \langle M, \xi \rangle} t^{- \lambda} \mathcal{F}_\xi^\lambda R_k 
%\\
%\sum_i t^{-\mu_i} x^i 
%&\mapsto \sum_i t^{-\langle \mu_i, \xi \rangle} x^i
%\end{align*}
%Take a basis $r_1, \ldots, r_n$ of $\langle M, \xi \rangle \subset \mathbb{R}$, we get an isomorphism $\mathbb{Z}^n \cong \langle M, \xi \rangle: m \mapsto mr$. 
%Then putting $\mathcal{F}_\xi^m R_k := \mathcal{F}_\xi^{m r} R_k$, we have an isomorphism 
%\[ \bigoplus_{k \ge 0} \bigoplus_{m \in \mathbb{Z}^n} t^{- m} \mathcal{F}_\xi^m R_k \xrightarrow{\sim} \bigoplus_{k \ge 0} \bigoplus_{\lambda \in \langle M, \xi \rangle} t^{-\lambda} \mathcal{F}_\xi^\lambda R_k. \]
%\[ \mathbb{C} [x^1, \ldots, x^n] \to \bigoplus_{k \ge 0} \bigoplus_{\lambda \in \langle M, \xi \rangle} t^{-\lambda} \mathcal{F}_\xi^\lambda R_k: \sum_m a_m x^m \mapsto \sum_m a_m t^{-r m} \]
%\[ \mathbb{C} [\sigma^\vee \cap M] \to \bigoplus_{k \ge 0} \bigoplus_{\mu \in M} t^{-\mu} \mathcal{F}^\mu R_k: \sum_i a_i \exp^{\mu_i} \mapsto \sum_i a_i t^{-\mu_i} \]
%\end{lem}

The following is observed in \cite{HL2}. 

\begin{prop}
\label{polyhedral configuration and finite generation of filtration}
The filtration $\mathcal{F}_{(\mathcal{X}, \mathcal{L}; \xi)}$ is finitely generated. 
Conversely, every finitely generated filtration is obtained in this way (for $\xi \in \sigma^\circ$). 
\end{prop}

\begin{proof}
Since the ring $\mathcal{R} = \bigoplus_{m \ge 0} \bigoplus_{\mu \in M} \varpi^{- \mu} \mathcal{F}^\mu R_m$ is finitely generated, we can take finite collections $x^1, \ldots, x^{N_d} \in R_d$ and $\mu_1, \ldots, \mu_{N_d} \in M$ with $x^i \in \mathcal{F}^{\mu_i} R_d$ ($x^i$ may overlap) so that for $m \in \mathbb{N}^{(d)}$ every element $s \in \mathcal{F}^\mu R_m$ can be written as 
\[ s = \sum_{\substack{|I| = m/d, \\ \sum_{i \in I} \mu_i = \mu}} a_I x^I. \] 
By the definition of the filtration, every element $s \in \mathcal{F}^\lambda R_m$ can be written as $\sum_{\langle \mu, \xi \rangle \ge \lambda} s^\mu$ with $s^\mu \in \mathcal{F}^\mu R_m$. 
On the other hand, we can write $s^\mu$ as $\sum_{|I| = m/d, \sum_{i \in I} \mu_i = \mu} a^\mu_I x^I$. 
As a consequence, $s \in \mathcal{F}^\lambda R_m$ can be written as 
\[ s = \sum_{\substack{|I| = m/d, \\ \sum_{i \in I} \langle \mu_i, \xi \rangle \ge \lambda}} a_I x^I. \]
It follows that 
\begin{align*} 
\mathcal{F}^\lambda R_m 
&= \Big{\{} \sum_{\substack{|I| =m/d, \\ \sum_{i \in I} \langle \mu_i, \xi \rangle \ge \lambda}} a_I x^I \Big{\}} 
\\
&= \sum_{\substack{|I|=m/d, \\ \sum_{i \in I} \langle \mu_i, \xi \rangle \ge \lambda}} \prod_{i \in I} \mathcal{F}^{\langle \mu_i, \xi \rangle} R_d, 
\end{align*}
which shows the finite generation. 

Conversely, assume $\mathcal{F}$ is finitely generated. 
Take $d$ so that $R^{(d)}$ and $\mathcal{F}|_{R^{(d)}}$ is generated in $R_d$. 
Take a diagonal basis $x^1, \ldots, x^{N_d} \in R_d$ with respect to $\| \cdot \| = \| \cdot \|^\mathcal{F}_d$ and put $\lambda_i := - \log \| x^i \|$. 
Consider the torus $T' = (\mathbb{G}_m)^{N_d}$ action on $R_d$ given by $x^i. (t_1, \ldots, t_{N_d}) = t_i x^i$ and put $\xi := (\lambda_1, \ldots, \lambda_{N_d}) \in \mathbb{R}^{N_d} = \mathfrak{t}'$. 
Then $\mathcal{F}^\lambda R_{dl}$ is the image of 
\[ \mathcal{F}^\lambda_{\| \cdot \|} S^l R_d = \{ \sum_{|I| =l} a_I x^{\otimes I} ~|~ \sum_{i \in I} \lambda_i \ge \lambda \text{ for } a_I \neq 0 \}= \bigoplus_{\langle \mu, \xi \rangle \ge \lambda} (S^l R_d)_\mu \]
along $S^l R_d \twoheadrightarrow R_{dl}$. 

Now we embed $X$ into the projective space $\mathbb{C}P^{N_d-1} = \mathbb{P} (R_d)$, using the diagonal basis $x^1, \ldots, x^{N_d}$. 
As in Example \ref{polyhedral configuration via Hilbert scheme}, we can construct a $T'$-equivariant morphism $f': B' \to \mathrm{Hilb} (\mathbb{C}P^{N_d-1})$ from a proper normal toric variety $B$. 
For $\xi = (\lambda_1, \ldots, \lambda_{N_d}) \in \mathfrak{t}'$, take the minimal cone $\sigma \in \Sigma$ with $\xi \in \sigma$ and put $\mathfrak{t} := \mathbb{R} \sigma, T := (\mathfrak{t} \cap N') \otimes \mathbb{G}_m$. 
Regarding $\sigma$ as a cone in $\mathfrak{t}$, we have $\xi \in \sigma^\circ$. 
For the affine toric variety $B_\sigma \supset T$, we get a $\sigma$-configuration $(\mathcal{X}/B_\sigma, \mathcal{L})$ by pulling back the universal family $(\mathcal{U}, \mathcal{O} (1)^{1/d}|_{\mathcal{U}})$ over the Hilbert scheme along $f: B_\sigma \to B' \to \mathrm{Hilb} (\mathbb{C}P^{N_d -1})$. 
By the construction, $\mathcal{X}$ is a closed subscheme of $B_\sigma \times \mathbb{C}P^{N_d -1}$ with the flat projection $\pi: \mathcal{X} \to B_\sigma$. 

Now we claim $\mathcal{F}|_{R^{(dl)}} = \mathcal{F}_{(\mathcal{X}, \mathcal{L}; \xi)}|_{R^{(dl)}}$ for sufficiently large $l$. 
For the relatively ample $\mathcal{O} (1)$ on $B_\sigma \times \mathbb{C}P^{N_d -1}/B_\sigma$, take large $l$ so that $R^i \pi_* (\mathcal{I}_\mathcal{X} \otimes \mathcal{O} (m/d)) = 0$ and $R^i \pi_* \mathcal{O} (m/d)|_{\mathcal{X}} = 0$ for every $i \ge 1$ and $m \in \mathbb{N}^{(dl)}$. 
Then $\pi_* \mathcal{O} (m/d)|_{\mathcal{X}}$ is locally free and we have a $T$-equivariant surjection $\pi_* \mathcal{O} (m/d) \twoheadrightarrow \pi_* \mathcal{O} (m/d)|_{\mathcal{X}}$ of vector bundles. 
By Proposition \ref{weight filtration of quotient bundle}, $\mathcal{F}^\mu_{(\mathcal{X}, \mathcal{L})} R_m$ for $m \in \mathbb{N}^{(dl)}$ is the image of 
\[ \mathcal{F}^\mu_{(B_\sigma \times \mathbb{C}P^{N_d -1}/B_\sigma, \mathcal{O} (1))} H^0 (\mathbb{C}P^{N_d-1}, \mathcal{O} (m/d)) = (S^{m/d} R_d)_\mu \] 
along $H^0 (\mathbb{C}P^{N_d-1}, \mathcal{O} (m/d)) = S^{m/d} R_d \twoheadrightarrow R_m$. 
Thus the filtration $\mathcal{F}^\lambda_{(\mathcal{X}, \mathcal{L}; \xi)} R_m$ is the image of 
\[ \mathcal{F}^\lambda_{(B_\sigma \times \mathbb{C}P^{N_d -1}/B_\sigma, \mathcal{O} (1); \xi)} H^0 (\mathbb{C}P^{N_d-1}, \mathcal{O} (m/d)) = \bigoplus_{\langle \mu, \xi \rangle \ge \lambda} (S^{m/d} R_d)_\mu = \mathcal{F}^\lambda_{\| \cdot \|} S^{m/d} R_d. \]
Therefore, we get $\mathcal{F}^\lambda_{(\mathcal{X}, \mathcal{L}; \xi)} R_m = \mathcal{F}^\lambda_{\| \cdot \|} R_m$ for every $m \in \mathbb{N}^{(dl)}$. 
\end{proof}

We note the graded ring 
\begin{equation} 
\mathcal{R}_{(\mathcal{X}, \mathcal{L}; \xi)} := \bigoplus_{m \ge 0} \bigoplus_{\lambda \in \langle M, \xi \rangle} \varpi^{- \lambda} \mathcal{F}^\lambda_{(\mathcal{X}, \mathcal{L}; \xi)} R_m 
\end{equation}
is not finitely generated over $\mathbb{C}$ as the base ring 
\begin{equation} 
S_{B_\sigma, \xi} := \bigoplus_{\lambda \in \langle M, \xi \rangle} \varpi^{- \lambda} \mathcal{F}^\lambda_{(\mathcal{X}, \mathcal{L}; \xi)} R_0, 
\end{equation}
is not finitely generated over $\mathbb{C}$. 
We can see this in the following example. 

\begin{eg}
Consider $M = N = \mathbb{Z}^2$, $\sigma = [0, \infty)^2 \subset N \otimes \mathbb{R}$. 
For irrational $\xi = (1, \sqrt{2}) \in \sigma$, we have $\mu = \mu' \in M$ iff $\langle \mu, \xi \rangle = \langle \mu', \xi \rangle$. 
It follows that 
\[ S_{B_\sigma, \xi} \cong \mathbb{C} [P_\xi] \]
for the monoid 
\[ P_\xi = \{ \mu \in M ~|~ \langle \mu, \xi \rangle \le 0 \}. \]
The ring $\mathbb{C} [P]$ is finitely generated over $\mathbb{C}$ iff the monoid $P$ is finitely generated. 
However, the monoid $P_\xi = \{ (m_1, m_2) \in \mathbb{Z}^2 ~|~ m_1 + m_2 \sqrt{2} \le 0 \}$ is not finitely generated. 
\end{eg}

\subsection{Variational formula on characteristic $\mu$-entropy}

\subsubsection{Equivariant intersection}
\label{equivariant intersection}

Equivariant intersection is a basic language for describing $\mu$K-stability and our $\mu$-entropy of test configurations. 
%It also brings us transparent understanding on some known results in K-stability (cf. \cite{Ino3, Leg}). 
%The author realized it in the study \cite{Ino3} and wondered why people had forgotten or disregarded this language. 
%The idea of using localization can be traced back to Futaki \cite{Fut}, although we apply localization in a slightly different way from his original work. 
We briefly explain the concept below. 
The readers can find further information in \cite{EG1, GS, GGK} and \cite{Ino3}. 

Let $X$ be a complex $n$-dimensional compact complex space with a holomorphic right $T$-action. 
For a $T$-equivariant homology class $D^T \in H^{T}_{2n-2} (X; \mathbb{R})$ and a $T$-equivariant cohomology class $L_T \in H^2_T (X; \mathbb{R})$, we define the \textit{equivariant intersection} $(D^T. L_T^{\cdot n+k-1}) \in S^k \mathfrak{t}^\vee$ by the equivariant push-forward to the point: 
\[ (D^T. L_T^{\cdot n+k-1}) := p_* (D^T \frown L_T^{\smile n+k-1}) \in H^{T}_{-2k} (\mathrm{pt}). \]

We can identify $(D^T. L_T^{\cdot n+k-1})$ with a polynomial function on $\mathfrak{t}$ of degree $k$ via the Poincar\'e duality $H^{T}_{-2k} (\mathrm{pt}) = H^{2k}_T (\mathrm{pt})$ and the Chern--Weil isomorphism 
\[ \Phi: H^{2k}_T (\mathrm{pt}, \mathbb{R}) \cong S^k \mathfrak{t}^\vee \] 
which maps $c^T_1 (\mathbb{C}_\chi)$ to $-\chi$ for $\chi \in M \subset \mathfrak{t}^\vee$, where $\mathbb{C}_\chi$ is the $T$-equivariant line bundle over the point whose right action is given by $z. t = \chi (t) z$. 
(Alternatively, consider $\mathbb{C}^\chi$ endowed with the right action $z. t = \chi (t)^{-1} z$. 
Then $c^T_1 (\mathbb{C}^\chi)$ is mapped to $\chi$ via $\Phi$. 
The associated left action on $\mathbb{C}^\chi$ is given by $t. z = \chi (t) z$. )
This sign convention is equivalent to choosing $\hbar = 1$ for the Chern--Weil isomorphism ${^\hbar \Phi}$ in \cite{Ino3}. 
We denote by $(D^T. L_T^{\cdot n+k-1}; \xi)$ the value of the polynomial at $\xi \in \mathfrak{t}$. 
When $T = \mathbb{G}_m$, we write $(D^{\mathbb{G}_m}. L_{\mathbb{G}_m}^{\cdot n+k-1}; \rho. \eta)$ as $(D^{\mathbb{G}_m}. L_{\mathbb{G}_m}^{\cdot n+k-1}; \rho)$ for $\rho \in \mathbb{R}$ and the generating vector $\eta \in N \subset \mathfrak{t}$ corresponding to the identity $\mathrm{id}: \mathbb{G}_m \to \mathbb{G}_m$. 

By definition, equivariant intersection is the intersection on the infinite dimensional Borel construction $X \times_T ET$. 
We can identify the equivariant intersection with the usual finite dimensional (relative) intersection in the following way. 
Take a basis $\{ \chi_i \}$ of the character lattice $M$ of $T$ and put $E_l T := \prod_{i=1}^{\mathrm{rk} T} (\mathbb{C}_{\chi_i}^{l+1} \setminus 0)$ for $l \gg k$, which is a finite dimensional approximation of the classifying space $ET$. 
We consider the right $T$-action on $E_l T$ induced by $\chi_i$, and denote by $X \times_T E_l T$ the quotient $(X \times E_l T)/T$ with respect to the diagonal action. 
By the construction of the equivariant homology, $D^T$ is identified with a homology class of degree $(2n-2) + 2 l \mathrm{rk} T$ on $X \times_T E_l T$. 
On the other hand, $L_T$ gives a degree $2$ cohomology class on $X \times_T E_l T$. 
Therefore we may identify $D^T \frown L_T^{\cdot n+k-1}$ with a homology class of degree $-2k + 2 l \mathrm{rk} T$ on $X \times_T E_l T$ and $(D^T. L_T^{\cdot n+k-1})$ with its push-forward along $X \times_T E_l T \to E_l T / T$, which lives in $H^{2k} (E_l T/T) = H^{2k}_T (\mathrm{pt})$ by the Poincare duality. 
This construction is independent of the choice of $E_l T$ by \cite{EG1}. 

Let $h = \sum_{k=0}^\infty a_k x^k$ be a real analytic function on $\mathbb{R}$ which extends to an entire holomorphic function on $\mathbb{C}$. 
Using the equivariant resolution of $X$ and the Cartan model of equivariant cohomology, we showed in \cite{Ino3} that the following infinite series is compactly absolutely-convergent and hence gives a real analytic function on $\mathfrak{t}$: 
\begin{equation} 
(D^T. h (L_T); \xi) := \sum_{k=0}^\infty a_k (D^T. L_T^{\cdot k}; \xi). 
\end{equation}
In the study of $\mu$K-stability, we applied this to the case $h (x) = e^x$. 
We often abbreviate $(D^T. h (L_T); \xi)$ as $(D. h (L); \xi)$. 

In the first article, we computed equivariant intersection using equivariant differential form. 
See \cite{Ino3} and \cite{GS, GGK} for the equivalence of two calculations. 

\subsubsection{$\mu$-Futaki invariant}

Let $\xi$ be a proper vector on $(X, L)$ and $T$ be the torus generated by $\xi$. 
For $\lambda \in \mathbb{R}$, a polarized scheme $(X, L)$ is called \textit{$\check{\mu}^\lambda_\xi$K-semistable} if $\cFut^\lambda_\xi (\mathcal{X}, \mathcal{L}) \ge 0$ for every $T$-equivariant test configuration $(\mathcal{X}, \mathcal{L})$. 
It is shown in \cite{Lah1} and is reformulated in \cite{Ino3} that if a smooth polarized manifold $(X, L)$ admits a $\check{\mu}^\lambda_\xi$-cscK metric ($=$ $\mu^\lambda_{-\xi/2}$-cscK metric), then $(X, L)$ is $\check{\mu}^\lambda_\xi$K-semistable. 
Here the \textit{$\mu$-Futaki invariant} $\cFut^\lambda_\xi (\mathcal{X}, \mathcal{L}) := D_\xi \check{\mu} (\mathcal{X}, \mathcal{L}) + \lambda D_\xi \check{\sigma} (\mathcal{X}, \mathcal{L})$ for a normal test configuration $(\mathcal{X}, \mathcal{L})$ is defined by the following equivariant intersection: 
\begin{align*}
D_\xi \check{\mu} (\mathcal{X}, \mathcal{L}) 
&:= 2 \pi \frac{(K_{\bar{\mathcal{X}}/\mathbb{C}P^1}^T. e^{\bar{\mathcal{L}}_T}; \xi) \cdot (e^{L_T}; \xi)  - (K_X^T. e^{L_T}; \xi) \cdot (e^{\bar{\mathcal{L}}_T}; \xi) }{(e^{L_T}; \xi)^2}, 
\\
D_\xi \check{\sigma} (\mathcal{X}, \mathcal{L})
&:= \frac{(\bar{\mathcal{L}}_T. e^{\bar{\mathcal{L}}_T}; \xi) \cdot (e^{L_T}; \xi) - (L_T. e^{L_T}; \xi) \cdot (e^{\bar{\mathcal{L}}_T}; \xi) }{(e^{L_T}; \xi)^2} - \frac{(e^{\bar{\mathcal{L}}_T}; \xi)}{(e^{L_T}; \xi)}. 
\end{align*}
For general non-normal test configuration, we replace the equivariant canonical divisor class $K_{\bar{\mathcal{X}}/\mathbb{C}P^1}^T$ with an equivariant Chow class $\kappa_{\bar{\mathcal{X}}/\mathbb{C}P^1}^T$ as explained in \cite{Ino3}, which fits into equivariant Grothendieck--Riemann--Roch theorem for general scheme (cf. \cite{EG2}). 
For a polarized normal variety, the $\check{\mu}^\lambda_\xi$K-semistability is equivalent to $\cFut^\lambda_\xi (\mathcal{X}, \mathcal{L}) \ge 0$ for every $T$-equivariant \textit{normal} test configuration $(\mathcal{X}, \mathcal{L})$. 

\subsubsection{Characteristic $\mu$-entropy of polyhedral configuration}
\label{section: mu-entropy of polyhedral configuration}

\begin{defin}[$\mu$-entropy of polyhedral configuration]
\label{mu-entropy of polyhedral configuration}
For a $\sigma$-configuration $(\mathcal{X}/B_\sigma, \mathcal{L})$, we put 
\begin{align} 
\cmu (\mathcal{X}, \mathcal{L}; \xi) 
&:= 2 \pi \frac{(\kappa_{\mathcal{X}_0}^T. e^{\mathcal{L}_T|_{\mathcal{X}_0}}; \xi)}{(e^{\mathcal{L}_T|_{\mathcal{X}_0}}; \xi)} \in \mathbb{R}, 
\\
\bm{\check{\sigma}} (\mathcal{X}, \mathcal{L}; \xi) 
&:= \frac{(\mathcal{L}_T|_{\mathcal{X}_0}. e^{\mathcal{L}_T|_{\mathcal{X}_0}}; \xi)}{(e^{\mathcal{L}_T|_{\mathcal{X}_0}}; \xi)} - \log (e^{\mathcal{L}_T|_{\mathcal{X}_0}}; \xi) \in \mathbb{R}
\end{align}
and 
\begin{equation} 
\cmu^\lambda (\mathcal{X}, \mathcal{L}; \xi) := \cmu (\mathcal{X}, \mathcal{L}; \xi) + \lambda \bm{\check{\sigma}} (\mathcal{X}, \mathcal{L}; \xi). 
\end{equation}
\end{defin}

\begin{prop}
The map $\cmu^\lambda (\mathcal{X}, \mathcal{L}; \bullet): \sigma \to \mathbb{R}$ is real analytic in the sense that it extends to a real analytic function on $\mathfrak{t}$. 
\end{prop}

\begin{proof}
This is a consequence of \cite[Proposition 3.8]{Ino3}. 
\end{proof}

\begin{thm}
\label{variation of mu-entropy}
Let $(X, L)$ be a $T$-polarized scheme. 
For $\xi \in \mathfrak{t}$ and a $T$-equivariant test configuration $(\mathcal{X}, \mathcal{L})$, we have 
\[ \frac{d}{d\rho}\Big{|}_{\rho=0} \cmu^\lambda (\mathcal{X}_\sigma, \mathcal{L}_\sigma; \xi + \rho. \eta) = - \cFut^\lambda_\xi (\mathcal{X}, \mathcal{L}) \]
for arbitrary $\sigma \subset \mathfrak{t}$ with $\xi \in \sigma$. 
\end{thm}

\begin{rem}
This is essentially observed in \cite{Ino3}, where we introduced the $\mu$-Futaki invariant as a derivative of $\mu$-character. 
The Taylor series of the $\mu$-entropy is identified with the $\mu$-character 
\[ \bm{\check{\mu}}^\lambda_{T \times \mathbb{G}_m} (\mathcal{X}_\sigma, \mathcal{L}_\sigma) \in \hat{H}_{T \times \mathbb{G}_m} (\mathbb{A}^1 \times B_\sigma, \mathbb{R}) \]
under the Chern--Weil isomorphism 
\[ \hat{H}_{T \times \mathbb{G}_m} (\mathbb{A}^1 \times B_\sigma, \mathbb{R}) \xrightarrow{|_{\{ o \}}} \hat{H}_{T \times \mathbb{G}_m} (\{ o \}, \mathbb{R}) \xrightarrow{{^1 \Phi}} \hat{S} (\mathfrak{t}^\vee \times \mathbb{R}). \]
It follows by \cite[Proposition 3.3, Corollary 3.20]{Ino3} that 
\[ \frac{d}{d\rho}\Big{|}_{\rho = 0} \cmu^\lambda (\mathcal{X}_\sigma, \mathcal{L}_\sigma; \xi + \rho. \eta) = \langle {^1 \mathfrak{D}_\xi} \bm{\check{\mu}}^\lambda_{T \times \mathbb{G}_m} (\mathcal{X}, \mathcal{L}), \eta \rangle = - \cFut^\lambda_\xi (\mathcal{X}, \mathcal{L}), \]
which proves the result. 
\end{rem}

We give a more direct proof in the following. 

\begin{proof}
By the localization formula for $T \times \mathbb{G}_m$-equivariant cohomology class on $\mathbb{C}P^1$ with the trivial $T$-action, we have 
\begin{gather*} 
(e^{\mathcal{L}_{T \times \mathbb{G}_m}|_{\mathcal{X}_0}}; (\rho, \xi)) = - \rho (e^{\bar{\mathcal{L}}_{T \times \mathbb{G}_m}}; (\rho, \xi)) + (e^{L_T}; \xi), 
\\
(e^{\mathcal{L}_{T \times \mathbb{G}_m}|_{\mathcal{X}_0}}; (\rho, \xi)) = - \rho (\bar{\mathcal{L}}_T. e^{\bar{\mathcal{L}}_{T \times \mathbb{G}_m}}; (\rho, \xi)) + (L_T. e^{L_T}; \xi), 
\\
(\kappa_{\mathcal{X}_0}. e^{\mathcal{L}_{T \times \mathbb{G}_m}|_{\mathcal{X}_0}}; (\rho, \xi)) = - \rho (K_{\bar{\mathcal{X}}/\mathbb{P}^1}. e^{\bar{\mathcal{L}}_{T \times \mathbb{G}_m}}; (\rho, \xi)) + (K_X^T. e^{L_T}; \xi), 
\end{gather*}
which shows 
\begin{gather*} 
\frac{d}{d\rho}\Big{|}_{\rho=0} (e^{\mathcal{L}_{T \times \mathbb{G}_m}|_{\mathcal{X}_0}}; (\rho, \xi)) = - (e^{\bar{\mathcal{L}}_T}; \xi), 
\\
\frac{d}{d\rho}\Big{|}_{\rho=0} (e^{\mathcal{L}_{T \times \mathbb{G}_m}|_{\mathcal{X}_0}}; (\rho, \xi)) = - (\bar{\mathcal{L}}_T. e^{\bar{\mathcal{L}}_T}; \xi), 
\\
\frac{d}{d\rho}\Big{|}_{\rho=0} (\kappa_{\mathcal{X}_0}. e^{\mathcal{L}_{T \times \mathbb{G}_m}|_{\mathcal{X}_0}}; (\rho, \xi)) = - (K_{\bar{\mathcal{X}}/\mathbb{P}^1}. e^{\bar{\mathcal{L}}_T}; \xi). 
\end{gather*}

It follows that 
\begin{align*}
\frac{d}{d\rho}\Big{|}_{\rho=0} \cmu (\mathcal{X}_\sigma, \mathcal{L}_\sigma; (\rho, \xi))
&= -2 \pi \frac{ (K_{\bar{\mathcal{X}}/\mathbb{P}^1}^T. e^{\bar{\mathcal{L}}_T}; \xi) \cdot (e^{L_T}; \xi)  - (\kappa_X^T. e^{L_T}; \xi) \cdot (e^{\bar{\mathcal{L}}_T}; \xi) }{(e^{L_T}; \xi)^2}, 
\\
\frac{d}{d\rho}\Big{|}_{\rho=0} \bm{\check{\sigma}} (\mathcal{X}_\sigma, \mathcal{L}_\sigma; (\rho, \xi))
&= -\frac{(\bar{\mathcal{L}}_T. e^{\bar{\mathcal{L}}_T}; \xi) \cdot (e^{L_T}; \xi) - (L_T. e^{L_T}; \xi) \cdot (e^{\bar{\mathcal{L}}_T}; \xi) }{(e^{L_T}; \xi)^2} + \frac{(e^{\bar{\mathcal{L}}_T}; \xi)}{(e^{L_T}; \xi)}. 
\end{align*}
Thus we get 
\[ \frac{d}{d\rho}\Big{|}_{\rho=0} \cmu^\lambda (\mathcal{X}_\sigma, \mathcal{L}_\sigma; \xi + \rho. \eta) = - \cFut^\lambda_\xi (\mathcal{X}, \mathcal{L}) \]
\end{proof}

\begin{cor}
Let $(X, L)$ be a polarized scheme. 
If there exists a proper vector $\xi$ on $(X, L)$ such that 
\[ \cmu^\lambda (X, L; \xi) = \sup_{(\mathcal{X}/B_\sigma, \mathcal{L}; \zeta)} \cmu^\lambda (\mathcal{X}, \mathcal{L}; \zeta), \]
then $(X, L)$ is $\check{\mu}^\lambda_\xi$K-semistable. 
\end{cor}

\begin{proof}
Since $\cmu^\lambda (\mathcal{X}_\sigma, \mathcal{L}_\sigma; \xi + \rho. \eta) \le \cmu^\lambda (X, L; \xi) = \cmu^\lambda (\mathcal{X}_\sigma, \mathcal{L}_\sigma; \xi)$, we get 
\[ \cFut^\lambda_\xi (\mathcal{X}, \mathcal{L}) = - \frac{d}{d\rho}\Big{|}_{\rho=0} \cmu^\lambda (\mathcal{X}_\sigma, \mathcal{L}_\sigma; \xi + \rho. \eta) \ge 0. \]
\end{proof}

Since $\cmu^\lambda (\mathcal{X}, \mathcal{L}; \zeta)$ is continuous on $\zeta$, we may restrict the range of the supremum to test configurations $(\mathcal{X}, \mathcal{L}; \rho)$ with $\rho \in \mathbb{Q}_{\ge 0}$ as claimed in Theorem \ref{characteristic mu maximization implies muK-semistability}. 
%By observations in section \ref{mu-entropy of test configuration}, we may further replace it to normal test configurations with reduced central fibre when $X$ is normal. 
%On the other hand, similarly as \cite{Ino3}, we can easily see $\cmu^\lambda (\mathcal{X}, \mathcal{L}; \xi)$ is continuous with respect to perturbation of the $\mathbb{G}_m$-equivariant class $\mathcal{L} \in H^2_{\mathbb{G}_m} (\mathcal{X}; \mathbb{R})$. 
%It follows that when $X$ is smooth we may replace the range to smooth test configurations with reduced snc central fibre by \cite[Chapter IV, Section 3]{KKMS}. 

\begin{thm}
\label{forall exists}
Let $(X, L)$ be a polarized normal variety with only klt singularities. 
Suppose for every test configuration $(\mathcal{X}, \mathcal{L})$ and $\rho \in \mathbb{Q}_{\ge 0}$, there exists a proper vector $\xi$ on $(X, L)$ such that 
\[ \cmu^\lambda (X, L; \xi) \ge \cmu^\lambda (\mathcal{X}, \mathcal{L}; \rho), \]
then $(X, L)$ is $\mu^\lambda$K-semistable with respect to some $\xi_{\mathrm{opt}}$ maximizing the characteristic $\mu$-entropy. 
\end{thm}

\begin{proof}
This is a consequence of the above corollary and the properness of the $\mu$-entropy for proper vectors, which is proved for $X$ with klt singularities in Corollary \ref{properness of mu-entropy}. 
We note the properness is proved for smooth $X$ also in \cite{Ino2} by a differential geometric method. 
\end{proof}

A slight modification of the proof yields the following. 

\begin{thm}
\label{mu-entropy maximization for polyhedral configuration}
If a polyhedral configuration $(\mathcal{X}/B_\sigma, \mathcal{L}; \xi)$ maximizes $\cmu^\lambda$ among all polyhedral configurations, then the central fibre $(\underline{X}, \underline{L}) = (\mathcal{X}_o, \mathcal{L}|_{\mathcal{X}_o})$ is $\check{\mu}^\lambda_\xi$K-semistable with respect to all $T$-equivariant test configurations for $T = (\overline{\exp \mathbb{R} \xi})_\mathbb{C} \subset \mathrm{Aut} (\underline{X}, \underline{L})$. 
\end{thm}

\begin{proof}
Let $(\underline{\mathcal{X}}, \underline{\mathcal{L}})$ be a $T$-equivariant test configuration of $(\underline{X}, \underline{L})$. 
Fixing isomorphisms $H^0 (\underline{\mathcal{X}}_o, \underline{\mathcal{L}}_o^{\otimes k}) \cong H^0 (\underline{X}, \underline{L}^{\otimes k}) \cong H^0 (X, L^{\otimes k})$, we endow a $T \times \mathbb{G}_m$-action on $H^0 (X, L^{\otimes k})$. 
This gives a $T \times \mathbb{G}_m$-action on $\mathrm{Hilb}$ of $\mathbb{P} (H^0 (X, L^{\otimes k})^\vee)$. 
Fixing a $T$-equivariant relative embedding of $(\mathcal{X}/B_\sigma, \mathcal{L})$ into $B_\sigma \times \mathbb{P} (H^0 (X, L^{\otimes k})^\vee)$, we get a morphism $B_\sigma \to \mathrm{Hilb}$. 

Consider the $T \times \mathbb{G}_m$-equivariant morphism $f: B_\sigma \times (\mathbb{A}^1 \setminus 0) \to \mathrm{Hilb}: (b, t) \mapsto [\mathcal{X}_b, \mathcal{L}_b]. t$. 
By the construction, we have $\lim_{t \to 0} f (o, t) = [\underline{\mathcal{X}}_o, \underline{\mathcal{L}}_o]$. 
Take a $T \times \mathbb{G}_m$-equivariant resolution $\tilde{B} \to B_\sigma \times \mathbb{A}^1$ of the indeterminancy of the rational map $f: B_\sigma \times \mathbb{A}^1 \dashrightarrow \mathrm{Hilb}$. 
Let $\tilde{\Sigma}$ be the fan associated to the toric variety $\tilde{B}$. 
Since any maximal torus equivairant resolution of a toric variety is a refinement of the associated fan, we can find a cone $\tilde{\sigma} \in \tilde{\Sigma}$ such that $\tilde{\sigma} \subset \sigma \times [0, \infty)$ and $(\xi, \rho) \in \tilde{\sigma}^\circ$ for any small $\rho > 0$ by the irrationality of $\xi$. 
Let $\varphi: B_{\tilde{\sigma}} \to \tilde{B} \to B_\sigma \times \mathbb{A}^1$ be the associated affine toric variety and $(\tilde{\mathcal{X}}, \tilde{\mathcal{L}}) \to B_{\tilde{\sigma}}$ be the pull-back of the universal family on $\mathrm{Hilb}$. 
Since we have $\lim_{t \to \infty} \exp (-t. (\xi, \rho)) = o \in B_{\tilde{\sigma}}$, $\lim_{t \to \infty} \varphi (\exp (-t. (\xi, \rho)))$ is mapped to the unique $T \times \mathbb{G}_m$-fixed point $o \in B_\sigma \times \mathbb{A}^1$. 
Thus we get a natural $T \times \mathbb{G}_m$-equivariant isomorphism $(\tilde{\mathcal{X}}_o, \tilde{\mathcal{L}}_o) \cong (\underline{\mathcal{X}}_o, \underline{\mathcal{L}}_o)$. 

It follows that 
\[ \bm{\check{\mu}}^\lambda_{\mathrm{ch}, (\underline{X}, \underline{L})} (\underline{\mathcal{X}}_\sigma, \underline{\mathcal{L}}_\sigma; (\xi, \rho)) = \cmu^\lambda (\underline{\mathcal{X}}_o, \underline{\mathcal{L}}_o; (\xi, \rho)) = \cmu^\lambda (\tilde{\mathcal{X}}_o, \tilde{\mathcal{L}}_o; (\xi, \rho))  = \bm{\check{\mu}}^\lambda_{\mathrm{ch}, (X, L)} (\tilde{\mathcal{X}}, \tilde{\mathcal{L}}; (\xi, \rho)). \]
Since 
\[ \bm{\check{\mu}}^\lambda_{\mathrm{ch}, (X, L)} (\tilde{\mathcal{X}}, \tilde{\mathcal{L}}; (\xi, \rho)) \le \bm{\check{\mu}}^\lambda_{\mathrm{ch}, (X, L)} (\mathcal{X}, \mathcal{L}; \xi) = \bm{\check{\mu}}^\lambda_{\mathrm{ch}, (\underline{X}, \underline{L})} (\xi) = \bm{\check{\mu}}^\lambda_{\mathrm{ch}, (\underline{X}, \underline{L})} (\underline{\mathcal{X}}_\sigma, \underline{\mathcal{L}}_\sigma; (\xi, 0)), \]
we get 
\[ - \Fut^\lambda_{(\underline{X}, \underline{L})} (\underline{\mathcal{X}}, \underline{\mathcal{L}}) = \frac{d}{dt}\Big{|}_{\rho =0} \bm{\check{\mu}}^\lambda_{\mathrm{ch}, (\underline{X}, \underline{L})} (\underline{\mathcal{X}}_\sigma, \underline{\mathcal{L}}_\sigma; (\xi, \rho)) \le 0, \]
which proves the theorem. 
\end{proof}

\subsubsection{Characteristic $\mu$-entropy via associated filtration}
\label{mu-entropy via associated filtration}

We observe the $\mu$-entropy $\cmu^\lambda (\mathcal{X}, \mathcal{L}; \xi)$ can be recovered from the associated filtration $\mathcal{F}_{(\mathcal{X}, \mathcal{L}; \xi)}$. 

Let $(\mathcal{X}/B_\sigma, \mathcal{L})$ be a $\sigma$-configuration. 
For $\xi \in \sigma$, we consider the following endomorphism $H_{m, \xi}$ of $H^0 (\mathcal{X}_o, \mathcal{L}^{\otimes m}|_{\mathcal{X}_o})$: 
\begin{equation} 
H_{m, \xi} (s) := \frac{d}{d\rho}\Big{|}_{\rho=0} s. \exp (\rho. \xi) = H_{m, \xi} (s) = \sum_{\mu \in M} \langle \mu, \xi \rangle s_\mu, 
\end{equation}
where in the latter expression we use the weight decomposition $s = \sum_{\mu \in M} s_\mu$. 
By Proposition \ref{central fibre via filtration}, we have 
\[ \dim \mathrm{Ker} (H_{m, \xi} -\lambda) = \dim \mathcal{F}_{(\mathcal{X}, \mathcal{L}; \xi)}^\lambda R_m / \mathcal{F}_{(\mathcal{X}, \mathcal{L}; \xi)}^{\lambda +} R_m. \]

Now we can express the $\mu$-entropy via the associated filtration as follows. 

\begin{prop}
\label{chmu is an invariant for filtration}
Let $(\mathcal{X}/B_\sigma, \mathcal{L})$ be a $\sigma$-configuration. 
For the measures $\nu_m$ associated to the filtration $\mathcal{F}_{(\mathcal{X}, \mathcal{L}; \xi)}$ (see section \ref{Spectral measure}), we have 
\begin{align*} 
\int_{\mathbb{R}} e^{-t} \nu_m
&= \frac{1}{m^n} \mathrm{Tr} (e^{-m^{-1} H_{m, \xi}}) 
\\
&= (e^{\mathcal{L}|_{\mathcal{X}_0}}; \xi) - \frac{m^{-1}}{2} (\kappa_{\mathcal{X}_0}. e^{\mathcal{L}|_{\mathcal{X}_0}}; \xi) + O (m^{-2}). 
\end{align*}
In particular, we have 
\begin{equation}
\label{cmu via filtration}
\cmu (\mathcal{X}, \mathcal{L}; \xi) = -4\pi \lim_{m \to \infty} m \log \frac{\int_\mathbb{R} e^{-t} \nu_m (\mathcal{F}_{(\mathcal{X}; \mathcal{L}; \xi)})}{\int_\mathbb{R} e^{-t} \nu_\infty (\mathcal{F}_{(\mathcal{X}; \mathcal{L}; \xi)})}, 
\end{equation}
hence the characteristic $\mu$-entropy of polyhedral configuration is an invariant of finitely generated filtration. 
\end{prop}

\begin{proof}
Let $\lambda_1 < \lambda_2 < \dotsb < \lambda_p$ be the eigenvalues of $H_{m, \xi}$. 
We have 
\begin{align*} 
\int_{\mathbb{R}} e^{-t} \nu_m 
&= \frac{1}{m^n} \sum_{i=1}^p \dim \mathrm{Ker} (H_{m, \xi} - \lambda_i). e^{-\lambda_i/m} 
\\
&= \frac{1}{m^n} \mathrm{Tr} (e^{-m^{-1} H_{m, \xi}})= \frac{1}{m^n}  \chi_T (\mathcal{X}_0, \mathcal{L}^{\otimes m}|_{\mathcal{X}_0}; m^{-1} \xi). 
\end{align*}
By the equivariant Riemann--Roch theorem \cite{EG2, Ino3}, we can compute the last term as 
\begin{align*} 
\frac{1}{m^n} \chi_T (\mathcal{X}_0, \mathcal{L}^{\otimes m}|_{\mathcal{X}_0}; m^{-1} \xi)
&=\frac{1}{m^n} (\tau_{\mathcal{X}_0} (\mathcal{O}_{\mathcal{X}_0}). e^{m \mathcal{L}|_{\mathcal{X}_0}}; m^{-1} \xi) 
\\
&= \frac{1}{m^n} \Big{(} ([\mathcal{X}_0]. e^{m \mathcal{L}|_{\mathcal{X}_0}}; m^{-1} \xi) - \frac{1}{2} (\kappa_{\mathcal{X}_0}. e^{m \mathcal{L}|_{\mathcal{X}_0}}; m^{-1} \xi) + \dotsb \Big{)}
\\
&= (e^{\mathcal{L}|_{\mathcal{X}_0}}; \xi) - \frac{m^{-1}}{2} (\kappa_{\mathcal{X}_0}. e^{\mathcal{L}|_{\mathcal{X}_0}}; \xi) + O (m^{-2}), 
\end{align*}
where the last equality follows by the following lemma. 
\end{proof}

\begin{lem}
Let $X$ be a compact topological space. 
For $\alpha \in H^T_{2 k} (X, \mathbb{R})$, $L \in H^2_T (X, \mathbb{R})$ and $\beta \in \mathbb{R}^\times$, we have 
\[ (\alpha. e^{\beta L}; \beta^{-1} \xi) = \beta^k (\alpha. e^L; \xi). \]
\end{lem}

\begin{proof}
The claim follows from
\[ (\alpha. e^{\beta L}; \beta^{-1} \xi) = \sum_{l=0}^\infty \frac{1}{l!} \beta^l (\alpha. L^l; \beta^{-1} \xi) = \sum_{l=0}^\infty \frac{1}{l!} \beta^l \beta^{k-l} (\alpha. L^l; \xi) = \beta^k (\alpha. e^L; \xi). \]
\end{proof}

\begin{eg}
We check our sign convention in the above proposition. 
Consider the simplest case $\mathcal{X} = B_\sigma = \mathbb{A}^1$ with $\sigma = [0,\infty).\eta$. 
The line bundle $\mathcal{L} = \mathcal{X} \times \mathbb{C}$ is endowed with the $\mathbb{G}_m$-action $(x, s). t = (x.t, s.t)$. 
In this case, we have 
\[ \tau_{\mathcal{X}_0}^{\mathbb{G}_m} (\mathcal{O}_{\mathcal{X}_0}) =1, \qquad \mathcal{L}_{\mathbb{G}_m}|_{\mathcal{X}_0} = -\eta^\vee, \]
so that we get 
\[ (\tau_{\mathcal{X}_0} (\mathcal{O}_{\mathcal{X}_0}).e^{\mathcal{L}|_{\mathcal{X}_0}}; x. \eta) = e^{-x}. \]
On the other hand, we have 
\[ \mathrm{Tr} (e^{-H_{x. \eta}}) = e^{-x} \]
as $H_{x. \eta} (s) = (d/d\rho)|_{\rho=0} (e^{\rho x} s) = x. s$. 
\end{eg}

Using the formula (\ref{cmu via filtration}), we can define the characteristic $\mu$-entropy for general linearly bounded filtration as follows: 
\begin{align*} 
\bm{\check{\sigma}} (\mathcal{F}) 
&:= \frac{\int_\mathbb{R} (n-t) e^{-t} \nu_\infty (\mathcal{F})}{\int_\mathbb{R} e^{-t} \nu_\infty (\mathcal{F})} - \log \int_\mathbb{R} e^{-t} \nu_\infty (\mathcal{F}), 
\\ 
\cmu (\mathcal{F}) 
&:= -4\pi \varliminf_{m \to \infty} m \log \frac{\int_\mathbb{R} e^{-t} \nu_m (\mathcal{F})}{\int_\mathbb{R} e^{-t} \nu_\infty (\mathcal{F})}, 
\\
\cmu^\lambda (\mathcal{F}) 
&:= \cmu (\mathcal{F}) + \lambda \bm{\check{\sigma}} (\mathcal{F}). 
\end{align*}
Alternatively, imitating \cite{Sze2}, we put 
\[ {^\flat }\cmu^\lambda (\mathcal{F}) := \varlimsup_{d \to \infty} \cmu^\lambda (\mathcal{F}_d) \]
where $\mathcal{F}_d$ is the finitely generated filtration associated to the norm $\| \cdot \|^{\mathcal{F}}_d$. 
Both definitions seems not so tractable compared to the non-archimedean $\mu$-entropy, so we leave a further exploration and just put the following example. 

\begin{eg}[Example 3.8 in \cite{BJ4}] 
We consider a graded norm $\| \cdot \|_\bullet$ on $R = \bigoplus_{m \in \mathbb{N}} H^0 (\mathbb{C}P^1, \mathcal{O} (m)) = \mathbb{C} [x, y]$ defined by 
\[ \| f \|_m = 
\begin{cases} 
e^m 
& x \nmid f
\\
1
& x \mid f
\end{cases}. \]
Let $\mathcal{F}$ denote the associated filtration. 
We have $\nu_m (\mathcal{F}) = \frac{1}{m+1} \delta_{-1} + \frac{m}{m+1} \delta_0$ and $\nu_m (\mathcal{F}_d) = \sum_{i=1}^{m/d} \frac{1}{m+1} \delta_{-1 + (i-1) d/m} + \frac{m+1 - m/d}{m+1} \delta_0$ for $m \in \mathbb{N}^{(d)}_+$. 
We have $\nu_\infty (\mathcal{F}) = \delta_0$ and $\nu_\infty (\mathcal{F}_d) = \frac{1}{d} d\mu|_{[-1, 0]} + \frac{d-1}{d} \delta_0$. 
In particular, $\mathcal{F}$ defines the trivial metric $\varphi_{\mathrm{triv}}$ by \cite[Theorem 4.16]{BJ2}. 
However, we have 
\[ \cmu (\mathcal{F}) = -4\pi \varliminf_{m \to \infty} m \log \frac{\int_\mathbb{R} e^{-t} \nu_m (\mathcal{F})}{\int_\mathbb{R} e^{-t} \nu_\infty (\mathcal{F})} = - 4\pi (e-1) < \cmu (\mathcal{F}_{\mathrm{triv}}) = 0. \]

As for ${^\flat} \cmu$, we note the filtration $\mathcal{F}_d$ is the filtration associated to a toric test configuration whose associated convex function $q_d$ on the interval $[0,1]$ is given by $q_d (t) = \max \{ 0, d (t-1) +1 \}$. 
The central fibre of the test configuration is reduced by \cite[Proposition 4.1.1]{CLS}. 
Then we have $\NAmu (\varphi_{\mathcal{F}_d}) = \cmu (\mathcal{F}_d)$, so by Proposition \ref{Toric formula}, we compute 
\[ \cmu (\mathcal{F}_d) = - 2\pi \frac{e^{q_d (0)} + e^{q_d (1)}}{\int_\mathbb{R} \nu_\infty (\mathcal{F}_d)} = - 2\pi \frac{1+e}{\frac{1}{d} (e-1) + \frac{d-1}{d}}. \]
It follows that 
\[ {^\flat} \cmu (\mathcal{F}) = \varlimsup_{d \to \infty} \cmu (\mathcal{F}_d) = -2\pi (e+1) < \cmu (\mathcal{F}) < \cmu (\mathcal{F}_{\mathrm{triv}}). \]
This in particular shows the non-archimedean $\mu$-entropy is not continuous along the convergent increasing sequence $\varphi_{\mathcal{F}_d} \nearrow \varphi_{\mathrm{triv}}$ of non-archimedean psh metrics, but only upper semi-continuous. 
%\[ \cmu (\mathcal{F}_d) = -4\pi \lim_{m \to \infty} m \log \frac{\sum_{i=1}^{m/d} \frac{1}{m+1} e^{1 - (i-1) d/m} + \frac{m+1 - m/d}{m+1} }{\frac{1}{d} (e-1) + \frac{d-1}{d}} \]
\end{eg}

\subsubsection{$\mu$-entropy of test configuration}
\label{mu-entropy of test configuration}

We observe the characteristic $\mu$-entropy is equivalent to the $\mu$-entropy of test configuration used in \cite{Ino4}, where we express it by the equivariant intersection formula on the compactified total space $\bar{\mathcal{X}}$. 
This enables us to compare the $\mu$-entropies of the normalized base change $\mathcal{X}_d \to \mathcal{X}$. 

\begin{prop}
\label{Localization}
For a test configuration $(\mathcal{X}, \mathcal{L})$, we have 
\begin{align*} 
\cmu (\mathcal{X}, \mathcal{L}; \rho) 
&= 2 \pi \frac{(K_X. e^L) - \rho (\kappa_{\bar{\mathcal{X}}/\mathbb{C}P^1}^{\mathbb{G}_m}. e^{\bar{\mathcal{L}}_{\mathbb{G}_m}}; \rho)}{(e^L) - \rho (e^{\bar{\mathcal{L}}_{\mathbb{G}_m}}; \rho)}, 
\\
\bm{\check{\sigma}} (\mathcal{X}, \mathcal{L}; \rho) 
&= \frac{(L. e^L) - \rho (\bar{\mathcal{L}}_{\mathbb{G}_m}. e^{\bar{\mathcal{L}}_{\mathbb{G}_m}}; \rho)}{(e^L) - \rho (e^{\bar{\mathcal{L}}_{\mathbb{G}_m}}; \rho)} - \log \Big{(} (e^L) - \rho (e^{\bar{\mathcal{L}}_{\mathbb{G}_m}}; \rho) \Big{)}. 
\end{align*}
\end{prop}

\begin{proof}
The claim follows by the localization of equivariant intersection on $\mathbb{C}P^1$ (cf. \cite[Appendix C.7]{GGK} or \cite[Example 2.6]{Ino3}) and the equivariant Grothendieck--Riemann--Roch theorem \cite[Corollary 2.19]{Ino3}. 
\end{proof}

We recall $\kappa_{\bar{\mathcal{X}}/\mathbb{C}P^1}^{\mathbb{G}_m} = K_{\bar{\mathcal{X}}/\mathbb{C}P^1}^{\mathbb{G}_m}$ for normal test configuration. 

\begin{prop}
Let $(X, L)$ be a polarized scheme and $(\mathcal{X}, \mathcal{L})$ be a test configuration of $(X, L)$. 
Let $(\mathcal{X}', \mathcal{L}' = \beta^* \mathcal{L})$ be another test configuration of $(X, L)$ dominating $(\mathcal{X}, \mathcal{L})$ by the canonical rational morphism $\beta: \mathcal{X}' \to \mathcal{X}$. 
If $\beta$ is finite away from a codimension two subscheme on the target, then we have 
\[ \cmu^\lambda (\mathcal{X}', \mathcal{L}'; \rho) \ge \cmu^\lambda (\mathcal{X}, \mathcal{L}; \rho). \]
If moreover $\beta$ is an isomorphism away from a codimension two subscheme on the target $X$, then 
\[ \cmu^\lambda (\mathcal{X}', \mathcal{L}'; \rho) = \cmu^\lambda (\mathcal{X}, \mathcal{L}; \rho). \]
% \[ S (\mathcal{X}', \beta^* \mathcal{L}) = S (\mathcal{X}, \mathcal{L}), \quad S_{\mathrm{pol}} (\mathcal{X}', \beta^* \mathcal{L}) = S_{\mathrm{pol}} (\mathcal{X}, \mathcal{L}). \] 
% \[ S_{\mathrm{can}} (\mathcal{X}', \beta^* \mathcal{L}) = S_{\mathrm{can}} (\mathcal{X}, \mathcal{L}) \quad (\text{resp.} ~ S_{\mathrm{can}} (\mathcal{X}', \beta^* \mathcal{L}) \ge S_{\mathrm{can}} (\mathcal{X}, \mathcal{L}) ). \] 
\end{prop}

\begin{proof}
The claim follows by the above proposition and \cite[Proposition 3.14]{Ino3}. 
\end{proof}

\begin{cor}
For any two test configurations $(\mathcal{X}, \mathcal{L}), (\mathcal{X}', \mathcal{L}')$ of $(X, L)$ which are isomorphic to each other in codimension one (i.e. there is an isomorphism away from codimension two subschemes of both $\mathcal{X}$ and $\mathcal{X}'$), we have 
\[ \cmu^\lambda (\mathcal{X}, \mathcal{L}; \rho) = \cmu^\lambda (\mathcal{X}', \mathcal{L}'; \rho). \]
% \[ S (\mathcal{X}, \mathcal{L}) = S (\mathcal{X}', \mathcal{L}'), \quad S_{\mathrm{pol}/\mathrm{can}} (\mathcal{X}, \mathcal{L}) = S_{\mathrm{pol}/\mathrm{can}} (\mathcal{X}', \mathcal{L}') \]
\end{cor}

\begin{cor}
If $X$ is normal, then 
\[ \cmu^\lambda (\mathcal{X}^\nu, \nu^* \mathcal{L}; \rho) \ge \cmu^\lambda (\mathcal{X}, \mathcal{L}; \rho) \]
for the normalization $\nu: \mathcal{X}^\nu \to \mathcal{X}$. 
% \[ S (\mathcal{X}^\nu, \nu^* \mathcal{L}) = S (\mathcal{X}, \mathcal{L}), \quad S_{\mathrm{pol}} (\mathcal{X}^\nu, \nu^* \mathcal{L}) = S_{\mathrm{pol}} (\mathcal{X}, \mathcal{L}), \] 
% \[ S_{\mathrm{can}} (\mathcal{X}^\nu, \nu^* \mathcal{L}) \ge S_{\mathrm{can}} (\mathcal{X}, \mathcal{L}) \]
 \end{cor}
 
Let $(\mathcal{X}, \mathcal{L})$ be a normal test configuration. 
We denote by $\nu_d: (\mathcal{X}_d, \mathcal{L}_d) \to (\mathcal{X}, \mathcal{L})$ the normalized base change of a test configuration $(\mathcal{X}, \mathcal{L})$ along the finite morphism $z^d: \mathbb{A}^1 \to \mathbb{A}^1$. 
The morphism $\nu_d$ is $\mathbb{G}_m$-equivariant with respect to the $d$-times scaled action on $\mathcal{X}$. 
Let $(\mathcal{X}'_d, \mathcal{L}'_d)$ denote the (non-normalized) base change of $(\mathcal{X}, \mathcal{L})$, then we have $\cmu^\lambda (\mathcal{X}'_d, \mathcal{L}'_d; \rho) = \cmu^\lambda (\mathcal{X}, \mathcal{L}; d\rho)$ from the definition of $\cmu^\lambda$. 
Thus we get the following. 
Compare $\NAmu^\lambda (\mathcal{X}_d, \mathcal{L}_d; \rho) = \NAmu^\lambda (\mathcal{X}, \mathcal{L}; d\rho)$ explained in section \ref{non-archimedean mu-entropy of test configuration}. 

\begin{cor}
If $X$ is normal, then we have 
\[ \cmu^\lambda (\mathcal{X}_d, \mathcal{L}_d; \rho) \ge \cmu^\lambda (\mathcal{X}, \mathcal{L}; d\rho). \]
\end{cor}

\section{Tomography of non-archimedean Monge--Amp\`ere measure}
\label{section: NAmu}

In the rest of this article, we assume $(X, L)$ is a polarized \textit{normal} variety, for simplicity. 

\subsection{Primary decomposition via filtration}
\label{Primary decomposition via filtration}

In this section, we study the primary decomposition of the Duistermaat--Heckman measure 
\[ \DHm_{(\mathcal{X}, \mathcal{L})} = \sum_{E \subset \mathcal{X}_0} \mathrm{ord}_E \mathcal{X}_0 \cdot \DHm_{(E, \mathcal{L}|_E)}. \] 
More precisely, we recover the measure from the associated filtration $\widehat{\mathcal{F}}_{(\mathcal{X}, \mathcal{L})}$, which will be identified with the filtration $\mathcal{F}_\varphi$ associated to the non-archimedean psh metric $\varphi = \varphi_{(\mathcal{X}, \mathcal{L})}$. 
This is the key observation in the construction of moment measure for general non-archimeden psh metrics. 

\subsubsection{Primary ideal associated to valuation}

Recall for a (not necessarily normal) test configuration $(\mathcal{X}, \mathcal{L})$, we associate the following ($\mathbb{Z}$-graded) filtration: 
\begin{align*} 
\mathcal{F}_{(\mathcal{X}, \mathcal{L})}^\lambda R_m 
&:= \{ s \in H^0 (X, L^{\otimes m}) ~|~ \varpi^{-\lceil \lambda \rceil} \bar{s} \text{ extends to a section of } \mathcal{L}^{\otimes m} \} 
\\
&= \mathcal{F}^{\lceil \lambda \rceil}_{(\mathcal{X}, \mathcal{L})} R_m. 
\end{align*}
As it is $\mathbb{Z}$-graded, we have 
\[ \sigma_v (\mathcal{F}_{(\mathcal{X}, \mathcal{L})}) = \inf \{ \sigma \in \mathbb{R} ~|~ \mathcal{F}_{(\mathcal{X}, \mathcal{L})}^\lambda \subset \mathcal{F}^\lambda_v [\sigma] \text{ for } \forall \lambda \in \mathbb{Z} \}. \]

\begin{prop}[Lemma A.5 in \cite{BJ4}]
\label{sigma for normalization}
For a test configuration $(\mathcal{X}, \mathcal{L})$, we have 
\[ \sigma_v (\mathcal{F}_{(\mathcal{X}, \mathcal{L})}) = \sigma_v (\mathcal{F}_{(\mathcal{X}^\nu, \nu^* \mathcal{L})}). \]
\end{prop}

\begin{proof}
Since $\mathcal{F}_{(\mathcal{X}, \mathcal{L})} \subset \mathcal{F}_{(\mathcal{X}^\nu, \nu^* \mathcal{L})}$, we have $\sigma_v (\mathcal{F}_{(\mathcal{X}, \mathcal{L})}) \le \sigma_v (\mathcal{F}_{(\mathcal{X}^\nu, \nu^* \mathcal{L})})$. 
To see the reverse inequality, we must show that $\mathcal{F}_{(\mathcal{X}, \mathcal{L})} \subset \mathcal{F}_v [\sigma]$ implies $\mathcal{F}_{(\mathcal{X}^\nu, \nu^* \mathcal{L})} \subset \mathcal{F}_v [\sigma]$ for every $v$ and $\sigma$. 

Take sufficiently divisible $m$ so that all the higher cohomology vanish. 
By the definition of normalization, every element $f \in H^0 (\mathcal{X}^\nu, \nu^* \mathcal{L}^{\otimes m})$ is integral over $\mathcal{X}$, hence we have $d \ge 1$ and $\sigma_i \in H^0 (\mathcal{X}, \mathcal{L}^{\otimes im})$ such that $f^d + f_1 f^{d-1} + \dotsb + f_d = 0$. 
We expand $f_i = \sum_{\lambda \in \mathbb{Z}} \varpi^{-\lambda} s_{i, \lambda}$ using $s_{i, \lambda} \in \mathcal{F}^\lambda_{(\mathcal{X}, \mathcal{L})} R_{im}$. 
For $s \in \mathcal{F}_{(\mathcal{X}^\nu, \nu^* \mathcal{L})}^\mu R_m$, putting $f = \varpi^{-\mu} s$, we obtain $\varpi^{-d \mu} (s^d + s_{1, \mu} s^{d-1} + s_{2, 2\mu} s^{d-2} + \dotsb + s_{d, d\mu}) = 0$. 
Now assume $\mathcal{F}_{(\mathcal{X}, \mathcal{L})} \subset \mathcal{F}_v [\sigma]$, then we have $- \log \| s_{i, i\mu} \|^{\mathcal{F}_v [\sigma]} \ge - \log \| s_{i, i\mu} \|_{im}^{(\mathcal{X}, \mathcal{L})} \ge i\mu$. 
Since $- \log \| t^d \|^{\mathcal{F}_v [\sigma]} = - d \log \| t \|^{\mathcal{F}_v [\sigma]}$, we get 
\begin{align*} 
- d \log \| s \|^{\mathcal{F}_v [\sigma]} 
&= -\log \| s_{1, \mu} s^{d-1} + s_{2, 2\mu} s^{d-2} + \dotsb + s_{d, d\mu} \|^{\mathcal{F}_v [\sigma]} \\
&\ge \min_{1 \le i \le d} \{ i \mu - (d-i) \log \| s \|^{\mathcal{F}_v [\sigma]} \}, 
\end{align*}
hence $- \log \| s \|^{\mathcal{F}_v [\sigma]} \ge \mu$. 
Therefore we get $s \in \mathcal{F}_v [\sigma]^\mu$, which shows the claim. 
\end{proof}

We note the following weight decomposition: 
\[ \mathcal{R} (\mathcal{X}, \mathcal{L}) = \bigoplus_{m \ge 0} \mathcal{R}_m := \bigoplus_{m \ge 0} H^0 (\mathcal{X}, \mathcal{L}^{\otimes m}) = \bigoplus_{m \ge 0} \bigoplus_{\lambda \in \mathbb{Z}} \varpi^{-\lambda} \mathcal{F}_{(\mathcal{X}, \mathcal{L})}^\lambda R_m. \]

\begin{lem}
\label{primary ideal}
Let $(\mathcal{X}, \mathcal{L})$ be a (not necessarily normal) test configuration. 
Then we have the following. 
\begin{enumerate}
\item The subset 
\[ \mathcal{I}_v := \bigoplus_{m \ge 0} \bigoplus_{\lambda \in \mathbb{Z}} \varpi^{-\lambda} (\mathcal{F}_{(\mathcal{X}, \mathcal{L})}^\lambda \cap \mathcal{F}_v^{\lambda+} [\sigma_v]) R_m \subset \mathcal{R} (\mathcal{X}, \mathcal{L}) \]
is a (homogeneous) prime ideal with $\mathcal{R}_+ \not\subset \mathcal{I}_v$ and the subset 
\[ \mathcal{I}^{+1}_v := \bigoplus_{m \ge 0} \bigoplus_{\lambda \in \mathbb{Z}} \varpi^{-\lambda} (\mathcal{F}_{(\mathcal{X}, \mathcal{L})}^\lambda \cap \mathcal{F}_v^{\lambda+1} [\sigma_v]) R_m \subset \mathcal{R} (\mathcal{X}, \mathcal{L}) \]
is a (homogeneous) primary ideal with $\sqrt{\mathcal{I}_v^{+1}} = \mathcal{I}_v$. 

\item The schematic point of $\mathcal{X} = \Proj_{\mathbb{C} [t]} \mathcal{R} (\mathcal{X}, \mathcal{L})$ corresponding to $\mathcal{I}_v$ is the center of the Gauss extension $G (v)$: 
\[ G (v) (\sum_{\lambda \in \mathbb{Z}} \varpi^\lambda h_\lambda) := \min_{\lambda \in \mathbb{Z}} (v (h_\lambda) + \lambda). \]
\end{enumerate}
\end{lem}

Note for $\sigma > \sigma_v$, we have 
\[ \bigoplus_{m \ge 0} \bigoplus_{\lambda \in \mathbb{Z}} \varpi^{-\lambda} (\mathcal{F}^\lambda_{(\mathcal{X}, \mathcal{L})} \cap \mathcal{F}_v^{\lambda+} [\sigma]) R_m = \mathcal{R} (\mathcal{X}, \mathcal{L}) \]
as $\mathcal{F}^\lambda_{(\mathcal{X}, \mathcal{L})} R_m \subset \mathcal{F}_v^\lambda [\sigma_v] R_m \subset \mathcal{F}_v^{\lambda+} [\sigma] R_m$. 

\begin{proof}
Take two elements $f_1, f_2 \in \mathcal{R}_{(\mathcal{X}, \mathcal{L})}$. 
We can write these as 
\[ f_i = \sum_{j \in I_i} \varpi^{-\lambda_i^j} s_i^j, \]
where $I_i$ is a finite index set and $s_i^j \in \mathcal{F}^{\lambda_i^j} R_{m_i^j}$. 
Since $\mathcal{F}^{\lambda_i^j} R_{m_i^j} \subset \mathcal{F}_v^{\lambda_i^j} [\sigma_v] R_{m_i^j}$, we have $v (s_i^j) + m_i^j \sigma_v \ge \lambda_i^j$. 

We have $f_i \in \mathcal{I}_v$ iff $v (s_i^j) + m_i^j \sigma_v > \lambda_i^j$. 
Note 
\[ f_1 f_2 = \sum_{j \in I_1, k \in I_2} \varpi^{-(\lambda_1^j + \lambda_2^k)} s_1^j s_2^k. \]

The subset $\mathcal{I}_v$ is obviously closed under sum of two elements and is homogeneous. 
Now suppose $f_1 \in \mathcal{I}_v$, then since 
\[ v (s_1^j s_2^k) + (m_1^j + m_2^k) \sigma_v = (v (s_1^j) + m_1^j \sigma_v) + (v (s_2^k) +m_2^k \sigma_v) > \lambda_1^j+ \lambda_2^k, \]
we have $f_1 f_2 \in \mathcal{I}_v$. 
Thus $\mathcal{I}_v$ is ideal. 
We can similarly show that $\mathcal{I}_v^{+1}$ is ideal. 

We can check $\sqrt{\mathcal{I}_v^{+1}} = \mathcal{I}_v$ as follows. 
For $f_1 \in \mathcal{I}_v$, take $\epsilon > 0$ so that $v (s_1^j) + m_1^j \sigma_v \ge \lambda_1^j + \epsilon$ for every $j \in I_1$. 
Consider a multiple 
\[ f_1^d = \sum_{j_1, \ldots, j_d \in I_1} \varpi^{-\sum_{r=1}^d \lambda_1^{j_r}} \prod_{r=1}^d s_1^{j_r}. \]
Then since 
\[ v (\prod_{r=1}^d s_1^{j_r}) + (\sum_{r=1}^d m_1^{j_r}) \sigma_v = \sum_{r=1}^d (v(s_1^{j_r}) + m_1^{j_r} \sigma_v) \ge \sum_{r=1}^d \lambda_1^{j_r} + d \epsilon \]
we have $f_1^d \in \mathcal{I}_v$ for $d$ with $d \ge \epsilon^{-1}$. 

Next suppose $f_i \notin \mathcal{I}_v$, then we have $v (s_i^j) + m_i^j \sigma_v = \lambda_i^j$. 
It follows that $v (s_1^j s_2^k) + (m_1^j + m_2^k) \sigma_v = \lambda_1^j+ \lambda_2^k$, which shows $f_1 f_2 \notin \mathcal{I}_v$. 
Thus $\mathcal{I}_v$ is prime ideal, hence $\mathcal{I}_v^{+1}$ is primary. 

By our choice of $\sigma_v$, there is $\lambda \in \mathbb{Z}$ and $s \in \mathcal{F}^\lambda_{(\mathcal{X}, \mathcal{L})} R_m$ for $m \ge 1$ such that $v (s) + m \sigma_v = \lambda$, which is equivalent to $\varpi^{-\lambda} s \notin \mathcal{I}_v$. 
Thus $\mathcal{R}_+ \notin \mathcal{I}_v$. 

Take $f_2 \notin \mathcal{I}_v \cap \mathcal{R}_m$. 
By the construction of $\Proj$, $\mathcal{X} \setminus f_2^{-1} (0)$ is naturally identified with $\mathrm{Spec} \mathcal{R}_{(f_2)}$ where $\mathcal{R}_{(f_2)} = \{ f_1/f_2^n \in \mathcal{R}_{f_2} ~|~ f_1 \in \mathcal{R}_{nm}, n \in \mathbb{N} \}$, where $\mathcal{I}_v$ is identified with $\{ f_1/f_2^n \in \mathcal{R}_{f_2} ~|~ f_1 \in \mathcal{I}_v \cap \mathcal{R}_{nm}, n \in \mathbb{N} \}$. 
Since 
\begin{align*} 
G (v) (f_1/f_2^n) 
&= G (v) (\sum_{j \in I_1} \varpi^{-\lambda_1^j} s_1^j) - n G (v) (\sum_{j \in I_2} \varpi^{-\lambda_2^j} s_2^j) 
\\
&= \min_{j \in I_1} (v (s_1^j) - \lambda_1^j) - n \min_{j \in I_2} (v (s_2^j) - \lambda_2^j)
\end{align*}
By our choice of $f_2$, we have $v (s_2^j) - \lambda_2^j = - m \sigma_v$. 
Since $v (s_1^j) - \lambda_1^j \ge - m \sigma_v$, we have $G (v) (f_1/f_2^n) \ge 0$. 
If $f_1 \in \mathcal{I}_v$, we have $G (v) (f_1/f_2^n) > 0$. 
This proves the center of $v$ is $\mathcal{I}_v$. 
\end{proof}

\subsubsection{Primary decomposition of the central fibre via filtration}

Recall for a normal test configuration $\mathcal{X}$ and an irreducible component $E \subset \mathcal{X}_0$, we have the following associated valuation on $X$: 
\[ v_E = \frac{\mathrm{ord}_E \circ p_X^*}{\mathrm{ord}_E \mathcal{X}_0}, \]
where $p_X: \mathcal{X} \dashrightarrow X \times \mathbb{A}^1 \to X$ denotes the canonical rational map. 

For another normal test configuration $\mathcal{X}'$ and an irreducible component $E' \subset \mathcal{X}'_0$, we have $v_{E'} = v_E$ iff $\mathrm{ord}_{E'} = \mathrm{ord}_E$ as a valuation on $X \times \mathbb{A}^1$ (\cite[Lemma 4.5]{BHJ1}). 
For an irreducible component $\tilde{E} \subset \mathcal{X}_0$ of a normal test configuration $\tilde{\mathcal{X}}$ dominating $\mathcal{X}$ via $\beta$, we put 
\begin{equation}
\sigma_{\tilde{E}} (\mathcal{X}, \mathcal{L}) := \sigma_{v_{\tilde{E}}} (\mathcal{F}_{(\mathcal{X}, \mathcal{L})}). 
\end{equation}
For each irreducible component $E \subset \mathcal{X}_0$, we can find an irreducible component $\tilde{E} \subset \tilde{\mathcal{X}}$ with $\beta (\tilde{E}) = E$, for which we have $\mathrm{ord}_{\tilde{E}} = \mathrm{ord}_E$ and $v_{\tilde{E}} = v_E$. 

We have the following explicit formula on $\sigma_{\tilde{E}} (\mathcal{X}, \mathcal{L})$. 
See also \cite[Lemma 5.17]{BHJ1}, which shows $\sigma_E \le \frac{\mathrm{ord}_E (\mathcal{L} - L)}{\mathrm{ord}_E \mathcal{X}_0}$ for $E \subset \mathcal{X}_0$. 

% Suppose $\mathcal{X}$ is normal. 
% Then for the dominant morphism $\beta: \tilde{\mathcal{X}} \to \mathcal{X}$, we have $\beta_* \mathcal{O}_{\tilde{\mathcal{X}}} = \mathcal{O}_{\mathcal{X}}$, so that $H^0 (\tilde{\mathcal{X}}, \tilde{\mathcal{L}}^{\otimes m}) = H^0 (\mathcal{X}, \mathcal{L}^{\otimes m})$ by the projection formula. 

\begin{prop}[Lemma A.6 in \cite{BJ4}]
\label{explicit formula for sigma}
Let $(\mathcal{X}, \mathcal{L})$ be a test configuration and $\tilde{\mathcal{X}}$ be a normal test configuration dominating both $\mathcal{X}$ and $X \times \mathbb{A}^1$ over $\mathbb{A}^1$. 
\[ \begin{tikzcd}[row sep=tiny, column sep=tiny]
&
\tilde{\mathcal{X}} \ar{dr}{p_X \times \tilde{\varpi}} \ar{dl}[swap]{\beta}
&
\\
\mathcal{X}
&
& X \times \mathbb{A}^1
\end{tikzcd} \]
Let $\tilde{\mathcal{L}} := \beta^* \mathcal{L}$ and $ \tilde{L}_{\mathbb{A}^1} := p_X^* L$ be the pull-backs of the $\mathbb{Q}$-line bundles. 
Then for any irreducible component $\tilde{E} \subset \tilde{\mathcal{X}}_0$, we have 
\[ \sigma_{\tilde{E}} (\mathcal{X}, \mathcal{L}) = \frac{\mathrm{ord}_{\tilde{E}} (\tilde{\mathcal{L}} - \tilde{L}_{\mathbb{A}^1})}{\mathrm{ord}_{\tilde{E}} \tilde{\mathcal{X}}_0}. \]
\end{prop}

\begin{proof}
We note for $s \in R_m$
\begin{align*} 
v_{\tilde{E}} (s) = v_{\tilde{E}} (s/e_{L^{\otimes m}}) 
&= \frac{\mathrm{ord}_{\tilde{E}} (\bar{s}/e_{\tilde{L}_{\mathbb{A}^1}^{\otimes m}})}{\mathrm{ord}_{\tilde{E}} \tilde{\mathcal{X}}_0} 
\\
&= \frac{\mathrm{ord}_{\tilde{E}} (\bar{s}/e_{\tilde{\mathcal{L}}^{\otimes m}})}{\mathrm{ord}_{\tilde{E}} \tilde{\mathcal{X}}_0} - m\frac{\mathrm{ord}_{\tilde{E}} (\tilde{\mathcal{L}} - \tilde{L}_{\mathbb{A}^1})}{\mathrm{ord}_{\tilde{E}} \tilde{\mathcal{X}}_0}. 
\end{align*}
Suppose $s \in \mathcal{F}_{(\mathcal{X}, \mathcal{L})}^\lambda R_m$, we have $\frac{\mathrm{ord}_{\tilde{E}} (\bar{s}/e_{\tilde{\mathcal{L}}^{\otimes m}})}{\mathrm{ord}_{\tilde{E}} \tilde{\mathcal{X}}_0} \ge \lambda$, so $v_{\tilde{E}} (s) + m\frac{\mathrm{ord}_{\tilde{E}} (\tilde{\mathcal{L}} - \tilde{L}_{\mathbb{A}^1})}{\mathrm{ord}_{\tilde{E}} \tilde{\mathcal{X}}_0} \ge \lambda$. 
Thus we obtain $\sigma_{\tilde{E}} \le \frac{\mathrm{ord}_{\tilde{E}} (\tilde{\mathcal{L}} - \tilde{L}_{\mathbb{A}^1})}{\mathrm{ord}_{\tilde{E}} \tilde{\mathcal{X}}_0}$. 

To see the reverse inequality, it suffices to show that there exists $\lambda \in \mathbb{Z}, m \in \mathbb{N}^{(d)}$ and $s \in \mathcal{F}^\lambda R_m$ satisfying 
\[ v_{\tilde{E}} (s) + m \frac{\mathrm{ord}_{\tilde{E}} (\tilde{\mathcal{L}} -\tilde{L}_{\mathbb{A}^1})}{\mathrm{ord}_{\tilde{E}} \tilde{\mathcal{X}}_0} = \lambda. \]
Indeed for such $s \in \mathcal{F}^\lambda R_m$, we have $s \notin \mathcal{F}^\lambda_{v_{\tilde{E}}} [\sigma]$ for $\sigma < \frac{\mathrm{ord}_{\tilde{E}} (\tilde{\mathcal{L}} -\tilde{L}_{\mathbb{A}^1})}{\mathrm{ord}_{\tilde{E}} \tilde{\mathcal{X}}_0}$. 
For $s \in R_m$, put 
\[ \lambda (s) := - \log \| s \|^{(\mathcal{X}, \mathcal{L})} = \sup \{ \lambda ~|~ \varpi^{-\lceil \lambda \rceil} \bar{s} \in H^0 (\mathcal{X}, \mathcal{L}^{\otimes m}) \}. \]
Since $\{ \varpi^{-\lambda} \bar{s} ~|~ \lambda \in \mathbb{Z}, s \in \mathcal{F}^\lambda R_m \}$ generates $H^0 (\mathcal{X}, \mathcal{L}^{\otimes m})$ over $\mathbb{C}$, $\{ \varpi^{-\lambda (s)} \bar{s}  ~|~ s \in R_m \}$ generates $H^0 (\mathcal{X}, \mathcal{L}^{\otimes m})$ over $\mathbb{C} [t]$. 
It follows that we get $s \in R_m$ such that $\varpi^{-\lambda (s)} \bar{s}|_{\tilde{E}} \neq 0$ for sufficiently large $m$ with globally generated $\mathcal{L}^{\otimes m}|_{\mathcal{X}_0}$. 
For this $s$, we have $\mathrm{ord}_{\tilde{E}} (\bar{s}/e_{\tilde{\mathcal{L}}^{\otimes m}}) = \mathrm{ord}_{\tilde{E}} (\varpi^{\lambda (s)}) = \lambda (s) \mathrm{ord}_{\tilde{E}} \tilde{\mathcal{X}}_0$, hence we obtain $v_{\tilde{E}} (s) = \lambda (s) - m \frac{\mathrm{ord}_{\tilde{E}} (\tilde{\mathcal{L}} - \tilde{L}_{\mathbb{A}^1})}{\mathrm{ord}_{\tilde{E}} \tilde{\mathcal{X}}_0}$, which shows the claim. 
\end{proof}

Let $\mathcal{X}$ be a normal test configuration. 
By \cite[Lemma 4.5]{BHJ1}, we have 
\[ G (v_E) = \frac{\mathrm{ord}_E}{\mathrm{ord}_E \mathcal{X}_0}, \] 
which in particular shows $\mathcal{I}_{v_E} = \mathcal{I}_E$ thanks to Lemma \ref{primary ideal}. 
Thus for irreducible components $E, E' \subset \mathcal{X}_0$, $v_E = v_{E'}$ implies $E=E'$. 
We have $\mathcal{I}_{\mathcal{X}_0^{\mathrm{red}}} = \bigcap_{E \subset \mathcal{X}_0} \mathcal{I}_{v_E}$. 

Since $\mathcal{X}$ is normal, the central fibre does not have embedded points (\cite[Proposition 2.6]{BHJ1}, so that primary decomposition of the ideal $\mathcal{I}_{\mathcal{X}_0}$ is unique. 
It is given by the primary ideals $\mathcal{I}_{v_E}^{+1}$. 

\begin{cor}
We have 
\[ \mathcal{I}_{v_E}^{+1} = \{ \sum_m \varpi^{-\lambda_m}. \bar{s}_m ~|~ \frac{\mathrm{ord}_E (\varpi^{-\lambda_m}. \bar{s}_m)}{\mathrm{ord}_E \mathcal{X}_0} \ge 1 \}. \]
In particular, we obtain the following primary decomposition: 
\[ \mathcal{I}_{\mathcal{X}_0} = \bigcap_{E \subset \mathcal{X}_0} \mathcal{I}^{+1}_{v_E}. \]
\end{cor}

\begin{proof}
Since 
\[ \frac{\mathrm{ord}_E (\bar{s}_m)}{\mathrm{ord}_E \mathcal{X}_0} = v_E (s_m) + m \frac{\mathrm{ord}_{\tilde{E}} (\tilde{\mathcal{L}} - \tilde{L})}{\mathrm{ord}_{\tilde{E}} \tilde{\mathcal{X}}_0}, \]
we compute 
\begin{align*} 
\mathcal{I}_{v_E}^{+1} 
&= \{ \sum_m \varpi^{-\lambda_m}. \bar{s}_m ~|~ v_E (s_m) + m \frac{\mathrm{ord}_{\tilde{E}} (\tilde{\mathcal{L}} - \tilde{L})}{\mathrm{ord}_{\tilde{E}} \tilde{\mathcal{X}}_0} \ge \lambda_m+1 \} 
\\
&= \{ \sum_m \varpi^{-\lambda_m}. \bar{s}_m ~|~ \frac{\mathrm{ord}_E (\varpi^{-\lambda_m}. \bar{s}_m)}{\mathrm{ord}_E \mathcal{X}_0} \ge 1 \}. 
\end{align*}
This shows 
\[ \mathcal{I}_{\mathcal{X}_0} = \mathcal{O}_{\mathcal{X}}. \varpi = \bigcap_{E \subset \mathcal{X}_0} \mathcal{I}^{+1}_{v_E}. \]
\end{proof}

\subsubsection{Primary decomposition of the Duistermaat--Heckman measure}

In what follows, we consider normal test configuration. 
We denote by $\mathcal{X}_0^E \subset \mathcal{X}$ the (non-reduced) subscheme corresponding to the primary ideal $\mathcal{I}_{v_E}^{+1}$: 
\[ R (\mathcal{X}_0^E, \mathcal{L}|_{\mathcal{X}_0^E}) = \bigoplus_{m \in \mathbb{N}^{(d)}} H^0 (\mathcal{X}_0^E, \mathcal{L}|_{\mathcal{X}_0^E}^{\otimes m}) = \bigoplus_{m \in \mathbb{N}^{(d)}} \mathcal{R}_m/(\mathcal{R}_m \cap \mathcal{I}_{v_E}^{+1}). \]
The $\mathbb{G}_m$-action on $(\mathcal{X}_0^E, \mathcal{L}|_{\mathcal{X}_0^E})$ gives the weight decomposition 
\[ \mathcal{R}_m/(\mathcal{R}_m \cap \mathcal{I}_{v_E}^{+1}) = \bigoplus_{\lambda \in \mathbb{Z}} \varpi^{-\lambda} \frac{\mathcal{F}_{(\mathcal{X}, \mathcal{L})}^\lambda R_m}{(\mathcal{F}_{(\mathcal{X}, \mathcal{L})}^\lambda \cap \mathcal{F}^{\lambda+1}_{v_E} [\sigma_E]) R_m}. \]

Similarly, we have 
\[ R (E, \mathcal{L}|_E) = \bigoplus_{m \in \mathbb{N}^{(d)}} H^0 (E, \mathcal{L}|_E^{\otimes m}) = \bigoplus_{m \in \mathbb{N}^{(d)}} \mathcal{R}_m/(\mathcal{R}_m \cap \mathcal{I}_{v_E}) \]
and the weight decomposition 
\[ \mathcal{R}_m/(\mathcal{R}_m \cap \mathcal{I}_{v_E}) = \bigoplus_{\lambda \in \mathbb{Z}} \varpi^{-\lambda} \frac{\mathcal{F}_{(\mathcal{X}, \mathcal{L})}^\lambda R_m}{(\mathcal{F}_{(\mathcal{X}, \mathcal{L})}^\lambda \cap \mathcal{F}^{\lambda+}_{v_E} [\sigma_E] ) R_m}. \]

Now we can compute the Duistermaat--Heckman measure of primary components via the associated filtration. 
Recall 
\begin{align} 
\DHm_{(\mathcal{X}_0^E, \mathcal{L}|_{\mathcal{X}_0^E})} 
&:= \lim_{m \to \infty} \frac{1}{m^n} \sum_{\lambda \in \mathbb{Z}} \dim H^0 (\mathcal{X}_0^E, \mathcal{L}^{\otimes m}|_{\mathcal{X}_0^E})_\lambda .\delta_{\lambda/m} 
\\ \notag
&= \mathrm{ord}_E \mathcal{X}_0 \cdot \DHm_{(E, \mathcal{L}|_E)}. 
\end{align}

\begin{prop}
\label{primary decomposition of DH measure}
%\[ \lim_{m \to \infty} \frac{1}{m^n} \dim \mathcal{R}_m/(\mathcal{R}_m \cap \mathcal{I}_v) = \begin{cases} \frac{1}{n!} (E. \mathcal{L}^{\cdot n}) & v = v_E \text{ for } E \subset \mathcal{X}_0 \\ 0 & v \notin \{ v_E ~|~ E \subset \mathcal{X}_0 \} \end{cases} \]
%\[ \lim_{m \to \infty} \frac{1}{m^n} \dim \mathcal{R}_m/(\mathcal{R}_m \cap \mathcal{I}_v^{+1}) = \begin{cases} \frac{1}{n!} \mathrm{ord}_E (\mathcal{X}_0) (E. \mathcal{L}^{\cdot n}) & v = v_E \text{ for } E \subset \mathcal{X}_0 \\ 0 & v \notin \{ v_E ~|~ E \subset \mathcal{X}_0 \} \end{cases}. \]
We have 
\[ \lim_{m \to \infty} \frac{1}{m^n} \sum_{\lambda \in \mathbb{Z}} \dim \frac{\mathcal{F}_{(\mathcal{X}, \mathcal{L})}^\lambda R_m}{(\mathcal{F}_{(\mathcal{X}, \mathcal{L})}^\lambda \cap \mathcal{F}^{\lambda+}_v [\sigma_v]) R_m} . \delta_{\lambda/m} = \begin{cases} \DHm_{(E, \mathcal{L}|_E)} & v = v_E \text{ for } E \subset \mathcal{X}_0 \\ 0 & v \notin \{ v_E ~|~ E \subset \mathcal{X}_0 \} \end{cases} \]
\[ \lim_{m \to \infty} \frac{1}{m^n} \sum_{\lambda \in \mathbb{Z}} \dim \frac{\mathcal{F}_{(\mathcal{X}, \mathcal{L})}^\lambda R_m}{(\mathcal{F}_{(\mathcal{X}, \mathcal{L})}^\lambda \cap \mathcal{F}^{\lambda+1}_v [\sigma_v]) R_m}. \delta_{\lambda/m} = \begin{cases} \mathrm{ord}_E \mathcal{X}_0 \cdot \DHm_{(E, \mathcal{L}|_E)} & v = v_E \text{ for } E \subset \mathcal{X}_0 \\ 0 & v \notin \{ v_E ~|~ E \subset \mathcal{X}_0 \} \end{cases} \]
\end{prop}

\begin{proof}
As we already noted, we have 
\[ H^0 (\mathcal{X}_0^E, \mathcal{L}|_{\mathcal{X}_0^E})_\lambda = \frac{\mathcal{F}_{(\mathcal{X}, \mathcal{L})}^\lambda R_m}{(\mathcal{F}_{(\mathcal{X}, \mathcal{L})}^\lambda \cap \mathcal{F}^{\lambda+1}_{v_E} [\sigma_E] ) R_m} \]
and 
\[ H^0 (E, \mathcal{L}|_E)_\lambda = \frac{\mathcal{F}_{(\mathcal{X}, \mathcal{L})}^\lambda R_m}{(\mathcal{F}_{(\mathcal{X}, \mathcal{L})}^\lambda \cap \mathcal{F}^{\lambda+}_{v_E} [\sigma_E] ) R_m}, \]
which shows the claim for $v = v_E$. 

As for $v \notin \{ v_E ~|~ E \subset \mathcal{X}_0 \}$, thanks to the following lemma, we have $\mathcal{I}_v \neq \mathcal{I}_E$ for any irreducible component $E \subset \mathcal{X}_0$. 
It follows that the dimension of the irreducible subscheme $Z_v \subset \mathcal{X}_0$ associated to the prime ideal $\mathcal{I}_v$ is less than $n$, so that we have  
\[ \lim_{m \to \infty} \frac{1}{m^n} \dim H^0 (Z_v, \mathcal{L}^{\otimes m}|_{Z_v}) = 0. \]
Then the claim follows by 
\[ H^0 (Z_v, \mathcal{L}^{\otimes m}|_{Z_v}) = \mathcal{R}_m/(\mathcal{R}_m \cap \mathcal{I}_v) = \bigoplus_{\lambda \in \mathbb{Z}} \varpi^{-\lambda} \frac{\mathcal{F}_{(\mathcal{X}, \mathcal{L})}^\lambda R_m}{(\mathcal{F}_{(\mathcal{X}, \mathcal{L})}^\lambda \cap \mathcal{F}^{\lambda+1}_v [\sigma_v]) R_m}. \]
\end{proof}

\begin{lem}
\label{the center of valuation}
For a general valuation $v$ on $X$, we have $v = v_E$ if and only if $\mathcal{I}_v = \mathcal{I}_E$. 
\end{lem}

\begin{proof}
The stalk $\mathcal{O}_{\mathcal{X}, E}$ of the generic point of $E$ is DVR by the normality, so that we can take a uniformizer $u \in \mathfrak{m}$: any $f \neq 0 \in \mathcal{O}_{\mathcal{X}, E}$ can be written as $f = a u^n$ by unique $a \in \mathcal{O}_{\mathcal{X}, E} \setminus \mathfrak{m}$ and $n \in \mathbb{N}$. 
By Lemma \ref{primary ideal}, $E$ is the center of $G (v)$, so that we have $G (v) (f) \ge 0$ for $f \in \mathcal{O}_{\mathcal{X}, E}$ and $G (v) (f) > 0$ for $f \in \mathfrak{m} \subset \mathcal{O}_{\mathcal{X}, E}$. 
Then since $G (v) (a u^n) = n G (v) (u)$, we must have $G (v) = G (v) (u) \cdot \mathrm{ord}_E$. 
Since $G (v) (\varpi) = 1$, we must have $G (v) (u) = (\mathrm{ord}_E \mathcal{X}_0)^{-1}$. 
This shows $G (v) = G (v_E)$ and hence $v = v_E$ by restriction. 
\end{proof}

\subsubsection{Valuations along the normalized base change}

Here we recall some facts observed in \cite{BHJ1}. 

\begin{prop}
\label{non-archimedean metric of normalized base change}
Let $(\mathcal{X}, \mathcal{L})$ be a test configuration and $\tilde{\mathcal{X}}$ be a normal test configuration dominating both $\mathcal{X}$ and $X \times \mathbb{A}^1$. 
Let $\nu_d: (\mathcal{X}_d, \mathcal{L}_d) \to (\mathcal{X}, \mathcal{L})$, $\tilde{\nu}_d: \tilde{\mathcal{X}}_d \to \tilde{\mathcal{X}}$ be the normalized base changes. 
\[ \begin{tikzcd}
\tilde{\mathcal{X}}_d \ar{r} \ar{d}
& \tilde{\mathcal{X}} \ar{d}
\\
\mathcal{X}_d \ar{r}
& \mathcal{X}
\end{tikzcd} \]
Then for any irreducible component $E' \subset \tilde{\mathcal{X}}_{d, 0}$, we have 
\begin{gather*}
v_{E'} = d.v_E
\\ 
\sigma_{E'} (\mathcal{X}_d, \mathcal{L}_d) = d. \sigma_E (\mathcal{X}, \mathcal{L})
\end{gather*}
for the irreducible component $E := \tilde{\nu}_d (E') \subset \tilde{\mathcal{X}}_0$. 
\end{prop}

\begin{proof}
Since $(\tilde{\nu}_d)_* \tilde{\mathcal{X}}_{d, 0} = (\tilde{\nu}_d)_* \varpi_d^* (0) = \varpi^* (\nu_d)_* (0) = \tilde{\mathcal{X}}_0$, we compute 
\[ d. v_E (f) = d. \frac{\mathrm{ord}_E f \circ p^{\tilde{\mathcal{X}}}_X}{\mathrm{ord}_E \tilde{\mathcal{X}}_0} = \frac{\mathrm{ord}_E (\tilde{\nu}_d)_* \mathrm{div} (f \circ p^{\tilde{\mathcal{X}}_d}_X)}{\mathrm{ord}_E (\tilde{\nu}_d)_* \tilde{\mathcal{X}}_{d, 0}}
%= \frac{\mathrm{ord}_E (\nu_d)_* E' \cdot \mathrm{ord}_{E'} \nu_d^* (f \circ p_X)}{\mathrm{ord}_E (\nu_d)_* E' \cdot \mathrm{ord}_{E'} \tilde{\mathcal{X}}_{d, 0}} 
= \frac{\mathrm{ord}_{E'} f \circ p^{\tilde{\mathcal{X}}_d}_X}{\mathrm{ord}_{E'} \tilde{\mathcal{X}}_{d, 0}} = v_{E'} (f) \]
for $f \in \mathbb{C} (X)$. 

Similarly, we compute 
\begin{align*}
\sigma_{E'} (\mathcal{X}_d, \mathcal{L}_d) = \frac{\mathrm{ord}_{E'} (\tilde{\mathcal{L}}_d - \tilde{L}_d)}{\mathrm{ord}_{E'} \tilde{\mathcal{X}}_{d, 0}} 
%= \frac{\mathrm{ord}_E (\nu_d)_* E' \cdot \mathrm{ord}_{E'} \nu_d^* (\tilde{\mathcal{L}} - \tilde{L}) }{\mathrm{ord}_E (\nu_d)_* E' \cdot \mathrm{ord}_{E'} \tilde{\mathcal{X}}_{d, 0}} 
= \frac{\mathrm{ord}_E (\tilde{\nu}_d)_* \tilde{\nu}_d^* (\tilde{\mathcal{L}} - \tilde{L}) }{\mathrm{ord}_E (\tilde{\nu}_d)_* \tilde{\mathcal{X}}_{d, 0}} = \frac{d. \mathrm{ord}_E (\tilde{\mathcal{L}} - \tilde{L})}{\mathrm{ord}_E \tilde{\mathcal{X}}_0} = d. \sigma_E (\mathcal{X}, \mathcal{L}). 
\end{align*}
\end{proof}

\begin{prop}
For any normal test configuration $\mathcal{X}$ and $d$, we have the following one to one correspondence via $\nu_d$: 
\[ \{ \text{irreducible components of } \mathcal{X}_0 \} \xleftrightarrow{E' \mapsto E = \nu_d (E')} \{ \text{irreducible components of } \mathcal{X}_{d, 0} \}. \] 
Moreover, we have 
\begin{gather*}
\mathrm{ord}_E \mathcal{X}_0 = (\mathrm{ord}_E \mathcal{X}_0, d) \cdot \mathrm{ord}_{E'} \mathcal{X}_{d, 0}, 
\\
(E', \mathcal{L}_d) = (\mathrm{ord}_E \mathcal{X}_0, d) \cdot (E. \mathcal{L}^{\cdot n})
\end{gather*} 
and $(\nu_d)_* E' = (\mathrm{ord}_E \mathcal{X}_0, d). E$ as divisors, where $(\mathrm{ord}_E \mathcal{X}_0, d)$ is the gcd. 
\end{prop}

\begin{proof}
As we have already used in our arguments, $\nu_d$ is surjective. 
Suppose $\nu_d (E'_1) = \nu_d (E'_2) = E$. 
By the above proposition, we have $v_{E_1} = d. v_E = v_{E_2}$, so that $\mathcal{I}_{E_1'} = \mathcal{I}_{d. v_E} = \mathcal{I}_{E_2'}$. 
This shows the injectivity of $\nu_d$. 

Now the multiplicity can be computed using the local coordinate expression (\ref{normalized base change diagram}) of the normalized base change. 
\end{proof}

The normality of $\mathcal{X}$ is essential here: for a non-normal test configuration $\mathcal{X}$, the number of the irreducible components of the central fibre of the normalization $\mathcal{X}^\nu$ may increase from that of the original $\mathcal{X}$. 
The author learned the following example from Masafumi Hattori. 

\begin{eg}[Hattori's example]
Consider an elliptic curve $X = \{ x_0^3 + x_1^3 + x_2^3 = 0 \} \subset \mathbb{P}^2$ and its test configuration $\mathcal{X}'$ defined as the integral image of the following morphism 
\begin{align*}
X \times \mathbb{A}^1 
&\to \mathbb{P}^2 \times \mathbb{A}^1
\\
((x_0: x_1: x_2), t) 
&\mapsto ((x_0: x_1: tx_2), t). 
\end{align*}
Explicitly, $\mathcal{X}' = \{ ((z_0: z_1: z_2), t) ~|~ t^3 z_0^3 + t^3 z_1^3 + z_2^3 = 0 \}$, which is endowed with a $\mathbb{G}_m$-action $((z_0: z_1: z_2), t). \tau = ((z_0: z_1: \tau z_2), \tau t)$. 
The central fibre of $\mathcal{X}'$ is non-reduced $3 \mathbb{P}^1$, and the normalization is the trivial configuration $X \times \mathbb{A}^1$. 

On the central fibre, we have a degree $3$ morphism $\phi: X \to \mathcal{X}_0': (x_0: x_1: x_2) \mapsto (x_0: x_1: 0)$. 
Let $\mathcal{X}$ be the blowing up of $\mathcal{X}'$ at $p = (1: 1: 0) \in \mathcal{X}_0'$, over which the morphism $\phi$ is \'etale. 
The central fibre consists of two irreducible components: one is the proper transform of $3\mathbb{P}^1$ and the other is the exceptional divisor $E$. 
By our choice of $p$, $E \cap 3 \mathbb{P}^1$ consists of three distinct points. 
Consider the normalization $\nu: \mathcal{X}^\nu \to \mathcal{X}$. 
Let $\beta: \mathcal{X}^\nu \to X \times \mathbb{A}^1$ be the induced morphism. 
\[ \begin{tikzcd}
\mathcal{X}^\nu \ar{r}{\nu} \ar{d}{\beta} 
& \mathcal{X} \ar{d}
\\
X \times \mathbb{A}^1 \ar{r}
& \mathcal{X}'
\end{tikzcd} \]
Since $\beta (\nu^{-1} E) = \phi^{-1} (p)$ consists of three distinct points, we have three distinct irreducible components $E_1, E_2, E_3$ of $\nu^{-1} E \subset \mathcal{X}_0$. 
Each $E_i$ is mapped onto $E$. 
\end{eg}

\subsubsection{Stabilize (homogenize) the filtration $\mathcal{F}_{(\mathcal{X}, \mathcal{L})}$}

We continue to consider a normal test configuration $(\mathcal{X}, \mathcal{L})$. 
By \cite[Lemma 5.17]{BHJ1}, we have 
\begin{equation} 
\mathcal{F}^\lambda_{(\mathcal{X}, \mathcal{L})} = \bigcap_{E \subset \mathcal{X}_0} \mathcal{F}_{v_E}^\lambda [\sigma_E] 
\end{equation}
for $\lambda \in \mathbb{Z}$. 
This is no longer true for $\lambda \in \mathbb{R}$: the latter filtration may jump at $\lambda \in \mathbb{Q}$ as $v_E (s) \in \mathbb{Q}$ may be not integral when the central fibre $\mathcal{X}_0$ is not reduced. 

For $\lambda \in \mathbb{R}$, we put
\begin{equation} 
\widehat{\mathcal{F}}^\lambda_{(\mathcal{X}, \mathcal{L})} := \bigcap_{E \subset \mathcal{X}_0} \mathcal{F}_{v_E}^\lambda [\sigma_E]. 
\end{equation}
We obviously have $\mathcal{F}_{(\mathcal{X}, \mathcal{L})} \subset \widehat{\mathcal{F}}_{(\mathcal{X}, \mathcal{L})}$. 

Compared to the associated filtration $\mathcal{F}_{(\mathcal{X}, \mathcal{L})}$, the new filtration $\widehat{\mathcal{F}}_{(\mathcal{X}, \mathcal{L})}$ behaves well under normalized base change. 

\begin{prop}
For a normal test configuration $(\mathcal{X}, \mathcal{L})$, we have 
\begin{equation} 
\widehat{\mathcal{F}}_{(\mathcal{X}_d, \mathcal{L}_d)} = \widehat{\mathcal{F}}_{(\mathcal{X}, \mathcal{L}); d}. 
\end{equation}
\end{prop}

\begin{proof}
Using the results in the previous subsection, we compute 
\[ \mathcal{F}_{v_{E'}} [\sigma_{E'} (\mathcal{X}_d, \mathcal{L}_d)] = \mathcal{F}_{d. v_E} [d. \sigma_E (\mathcal{X}, \mathcal{L})] = (\mathcal{F}_{v_E} [\sigma_E (\mathcal{X}, \mathcal{L})])_{;d}, \]
which proves the claim. 
\end{proof}

%As a consequence, for $\lambda \in \mathbb{Z}$ and $d \in \mathbb{N}_+$, we get 
%\[ \widehat{\mathcal{F}}^{\lambda/d}_{(\mathcal{X}, \mathcal{L})} = \widehat{\mathcal{F}}^\lambda_{(\mathcal{X}, \mathcal{L}); d} = \widehat{\mathcal{F}}_{(\mathcal{X}_d, \mathcal{L}_d)}^\lambda = \mathcal{F}^\lambda_{(\mathcal{X}_d, \mathcal{L}_d)}. \]
Compared to this, we only have $\mathcal{F}_{(\mathcal{X}, \mathcal{L})} \subset \mathcal{F}_{(\mathcal{X}_d, \mathcal{L}_d); d^{-1}}$: for $\lambda \in \mathbb{R}$, 
\[ \mathcal{F}_{(\mathcal{X}, \mathcal{L})}^\lambda = \widehat{\mathcal{F}}_{(\mathcal{X}, \mathcal{L})}^{\lceil \lambda \rceil} = \widehat{\mathcal{F}}_{(\mathcal{X}_d, \mathcal{L}_d)}^{d \lceil \lambda \rceil} = \mathcal{F}_{(\mathcal{X}_d, \mathcal{L}_d)}^{d \lceil \lambda \rceil} \subset \mathcal{F}_{(\mathcal{X}_d, \mathcal{L}_d)}^{d \lambda}. \]

When $\mathcal{X}_0$ is reduced, $v_E (s) + m \sigma_E = \mathrm{ord}_E (\bar{s})/\mathrm{ord}_E \mathcal{X}_0 = \mathrm{ord}_E (\bar{s})$ is an integer for $s \in R_m$ and $m \in \mathbb{N}^{(d)}$, so that we have $\mathcal{F}_{v_E}^\lambda [\sigma_E] = \mathcal{F}_{v_E}^{\lceil \lambda \rceil} [\sigma_E]$ and hence 
\begin{equation} 
\label{filtration for tc with reduced central fibre}
\widehat{\mathcal{F}}^\lambda_{(\mathcal{X}, \mathcal{L})} = \widehat{\mathcal{F}}^{\lceil \lambda \rceil}_{(\mathcal{X}, \mathcal{L})} = \mathcal{F}^{\lceil \lambda \rceil}_{(\mathcal{X}, \mathcal{L})} = \mathcal{F}^\lambda_{(\mathcal{X}, \mathcal{L})} 
\end{equation}
for $\lambda \in \mathbb{R}$. 

By taking sufficiently divisible $d$ so that $\mathcal{X}_{d, 0}$ is reduced, we get 
\begin{equation} 
\label{Krull envelope and normalized base change}
\widehat{\mathcal{F}}^\lambda_{(\mathcal{X}, \mathcal{L})} = \widehat{\mathcal{F}}^{d\lambda}_{(\mathcal{X}_d, \mathcal{L}_d)} = \mathcal{F}^{\lambda}_{(\mathcal{X}_d, \mathcal{L}_d; d^{-1})} 
\end{equation}
for general normal $(\mathcal{X}, \mathcal{L})$ and $\lambda \in \mathbb{R}$. 
Therefore, we can understand $\widehat{\mathcal{F}}_{(\mathcal{X}, \mathcal{L})}$ as the stable limit of the replacement $\mathcal{F}_{(\mathcal{X}, \mathcal{L})} \mapsto \mathcal{F}_{(\mathcal{X}_d, \mathcal{L}_d); d^{-1}}$. 
In particular, the filtration $\widehat{\mathcal{F}}_{(\mathcal{X}, \mathcal{L})}$ is finitely generated. 

As a consequence, we get $\sigma_{v_E} (\widehat{\mathcal{F}}_{(\mathcal{X}, \mathcal{L})}) = \sigma_E (\mathcal{X}, \mathcal{L})$ for a $\mathbb{G}_m$-invairant prime divisor $E$ over $\mathcal{X}$ centered on $\mathcal{X}_0$: 
\[ \sigma_{v_E} (\widehat{\mathcal{F}}_{(\mathcal{X}, \mathcal{L})}) = \sigma_{v_E} (\mathcal{F}_{(\mathcal{X}_d, \mathcal{L}_d); d^{-1}}) = d^{-1} \sigma_{d. v_E} (\mathcal{F}_{(\mathcal{X}_d, \mathcal{L}_d)}) = d^{-1} \sigma_{E'} (\mathcal{X}_d, \mathcal{L}_d) = \sigma_E (\mathcal{X}, \mathcal{L}). \]

\subsubsection{Primary decomposition via $\widehat{\mathcal{F}}_{(\mathcal{X}, \mathcal{L})}$}

Now we compute the primary decomposition via $\widehat{\mathcal{F}}_{(\mathcal{X}, \mathcal{L})}$. 
Compare Proposition \ref{primary decomposition of DH measure}. 

\begin{prop}
\label{primary decomposition of DH measure via hat filtration}
Let $(\mathcal{X}, \mathcal{L})$ be a normal test configuration and $\tilde{\mathcal{X}}$ be a normal test configuration dominating $\mathcal{X}$. 
Then for any irreducible component $\tilde{E} \subset \tilde{\mathcal{X}}$, we have 
\[ \lim_{m \to \infty} \frac{1}{m^n} \sum_{\lambda \in \mathbb{Q}} \dim \frac{\widehat{\mathcal{F}}^\lambda_{(\mathcal{X}, \mathcal{L})} R_m}{(\widehat{\mathcal{F}}^\lambda_{(\mathcal{X}, \mathcal{L})} \cap \mathcal{F}^{\lambda+}_{v_{\tilde{E}}} [\sigma_{\tilde{E}}]) R_m}. \delta_{\lambda/m} = 
\begin{cases} 
\mathrm{ord}_E \mathcal{X}_0 \cdot \DHm_{(E, \mathcal{L}|_E)}
&
v_{\tilde{E}} = v_E \text{ for } E \subset \mathcal{X}_0
\\
0
&
v_{\tilde{E}} \notin \{ v_E ~|~ E \subset \mathcal{X}_0 \}
\end{cases}. \]
\end{prop}

\begin{proof}
By the equivariant Riemann--Roch theorem (cf. \cite[Corollary 3.4]{BHJ1}), we have 
\[ \frac{1}{k!} \int_\mathbb{R} (-\rho t)^k \DHm_{(E, \mathcal{L}|_E)} = \frac{1}{(n+k)!} (E^{\mathbb{G}_m}. \mathcal{L}_{\mathbb{G}_m}|_E^{\cdot (n+k)}; \rho) \]
for $k \in \mathbb{N}$. 
Thus it suffices to show 
\[ \mathrm{ord}_E \mathcal{X}_0\frac{(E^{\mathbb{G}_m}. \mathcal{L}_{\mathbb{G}_m}|_E^{\cdot (n+k)}; \rho)}{(n+k)!} = \lim_{m \to \infty} \frac{1}{m^n} \sum_{\lambda \in \mathbb{Q}} \dim \frac{\widehat{\mathcal{F}}^\lambda_{(\mathcal{X}, \mathcal{L})} R_m}{(\widehat{\mathcal{F}}^\lambda_{(\mathcal{X}, \mathcal{L})} \cap \mathcal{F}^{\lambda+}_{v_E} [\sigma_E]) R_m} \frac{(-\rho \lambda/m)^k}{k!}. \]

Take $d \in \mathbb{N}_+$ so that the central fibre of the normalized base change $\mathcal{X}_d$ is reduced. 
Then for the irreducible component $E' \subset \mathcal{X}_{d, 0}$ corresponding to $E \subset \mathcal{X}_0$, we have 
\[ \mathrm{ord}_E \mathcal{X}_0 \cdot (E^{\mathbb{G}_m}. \mathcal{L}_{\mathbb{G}_m}|_E^{\cdot (n+k)}; \rho) = ((E')^{\mathbb{G}_m}. \mathcal{L}_{d, \mathbb{G}_m}|_{E'}^{\cdot (n+k)}; \rho/d) \]
by the equivariant projection formula (note $\nu_d$ is equivariant with respect to the $d$-times scaling $\mathbb{G}_m$-action on $E$). 
By Proposition \ref{primary decomposition of DH measure}, we compute 
\begin{align*}
\frac{1}{(n+k)!} 
&((E')^{\mathbb{G}_m}. \mathcal{L}_{d, \mathbb{G}_m}|_{E'}^{\cdot (n+k)}; \rho/d)
\\
&= \lim_{m \to \infty} \frac{1}{m^n} \sum_{\lambda' \in \mathbb{Z}} \dim \frac{\mathcal{F}^{\lambda'}_{(\mathcal{X}_d, \mathcal{L}_d)} R_m}{(\mathcal{F}^{\lambda'}_{(\mathcal{X}_d, \mathcal{L}_d)} \cap \mathcal{F}^{\lambda'+}_{v_{E'}} [\sigma_{E'}]) R_m}  \frac{(-\rho/d. \lambda'/m)^k}{k!}
\\
&= \lim_{m \to \infty} \frac{1}{m^n} \sum_{\lambda \in \mathbb{Q}} \dim \frac{\mathcal{F}^\lambda_{(\mathcal{X}_d, \mathcal{L}_d); d^{-1}} R_m}{(\mathcal{F}^\lambda_{(\mathcal{X}_d, \mathcal{L}_d); d^{-1}} \cap (\mathcal{F}_{d. v_E} [d. \sigma_E])_{;d^{-1}}^{\lambda+}) R_m} \frac{(-\rho/d. d\lambda/m)^k}{k!}
\\
&= \lim_{m \to \infty} \frac{1}{m^n} \sum_{\lambda \in \mathbb{Q}} \dim \frac{\widehat{\mathcal{F}}^\lambda_{(\mathcal{X}, \mathcal{L})} R_m}{(\widehat{\mathcal{F}}^\lambda_{(\mathcal{X}, \mathcal{L})} \cap \mathcal{F}^{\lambda+}_{v_E} [\sigma_E]) R_m} \frac{(-\rho \lambda/m)^k}{k!}.
\end{align*}
\end{proof}

\begin{prop}
\label{tomography of primary DH measure}
For a normal test configuration $(\mathcal{X}, \mathcal{L})$, an irreducible component $E \subset \mathcal{X}_0$ and $\tau \in \mathbb{R}$, we have 
\[ \mathrm{ord}_E \mathcal{X}_0 \int_{(-\infty, \tau]} \DHm_{(E, \mathcal{L}|_E)} = \lim_{m \to \infty} \frac{1}{m^n} \sum_{\lambda \in \mathbb{Q}, \lambda \le m \tau} \dim \frac{\widehat{\mathcal{F}}^\lambda_{(\mathcal{X}, \mathcal{L})} R_m}{\widehat{\mathcal{F}}^\lambda_{(\mathcal{X}, \mathcal{L})} \cap \mathcal{F}^{\lambda+}_{v_E} [\sigma_E] R_m}. \]
\end{prop}

\begin{proof}
The measure $\DHm_{(E, \mathcal{L}|_E)}$ is compactly supported and either a Dirac mass, which is the case only when the $\mathbb{G}_m$-action on $E$ is trivial, or absolutely continuous with respect to the Lebesgue measure. 

In the former case, the claim is clear as $v_E = v_{\mathrm{triv}}$: 
\[ \mathcal{F}^{\lambda+}_{v_E} [\sigma_E] R_m = 
\begin{cases} 
R_m & \lambda < m\sigma_E
\\
0 & \lambda \ge m \sigma_E
\end{cases}. \]
In the latter case, the claim follows by 
\[ \lim_{m \to \infty} \frac{1}{m^n} \sum_{\lambda \in \mathbb{Q}, m\tau < \lambda < m \tau + \varepsilon} \dim \frac{\widehat{\mathcal{F}}^\lambda_{(\mathcal{X}, \mathcal{L})} R_m}{\widehat{\mathcal{F}}^\lambda_{(\mathcal{X}, \mathcal{L})} \cap \mathcal{F}^{\lambda+}_{v_E} [\sigma_E] R_m} \le \int_{[\tau, \tau+\varepsilon]} \DHm_{(E, \mathcal{L}|_E)} \to 0 \]
\end{proof}

We use the following in the proof of Proposition \ref{large limit of mu-entropy}. 

\begin{prop}
\label{infimum of DH measure}
For a normal test configuration $(\mathcal{X}, \mathcal{L})$ and an irreducible component $E \subset \mathcal{X}_0$, we have 
\[ \sigma_E = \inf \operatorname{\mathrm{supp}} \DHm_{(E, \mathcal{L}|_E)}. \]
\end{prop}

\begin{proof}
For $\lambda < m \sigma_E$, we have $\mathcal{F}^{\lambda+}_{v_E} [\sigma_E] R_m = R_m$, so that $\int_{(-\infty, \tau]} \DHm_{(E, \mathcal{L}|_E)} = 0$ for $\tau < \sigma_E$ by the above proposition. 
Thus we have $\sigma_E \le \inf \operatorname{\mathrm{supp}} \DHm_{(E, \mathcal{L}|_E)}$. 

On the other hand, we note for 
\[ \lambda_{\min}^{(m)} := \inf \{ \lambda \in \mathbb{R} ~|~ \frac{\widehat{\mathcal{F}}^\lambda_{(\mathcal{X}, \mathcal{L})} R_m}{\widehat{\mathcal{F}}^\lambda_{(\mathcal{X}, \mathcal{L})} \cap \mathcal{F}^{\lambda+}_{v_E} [\sigma_E] R_m} \neq 0 \}, \]
we have $\lambda_{\min}^{(m)}/m = \inf \supp \DHm_{(E, \mathcal{L}|_E)}$ for sufficiently divisible $m$. 
Indeed, for $\lambda \in d^{-1} \mathbb{Z}$, we have 
\[ \frac{\widehat{\mathcal{F}}^\lambda_{(\mathcal{X}, \mathcal{L})} R_m}{\widehat{\mathcal{F}}^\lambda_{(\mathcal{X}, \mathcal{L})} \cap \mathcal{F}^{\lambda+}_{v_E} [\sigma_E] R_m} = \frac{\mathcal{F}^{d\lambda}_{(\mathcal{X}_d, \mathcal{L}_d)} R_m}{\mathcal{F}^{d\lambda}_{(\mathcal{X}_d, \mathcal{L}_d)} \cap \mathcal{F}^{d\lambda+}_{v_{E'}} [\sigma_{E'}] R_m} = H^0 (E', \mathcal{L}_d|_{E'}^{\otimes m})_{d\lambda} \]
for $E' \subset \mathcal{X}_{d, 0}$ corresponding to $E \subset \mathcal{X}_0$. 
Take sufficiently divisible $d$ so that the central fibre of $\mathcal{X}_d$ is reduced, then we have 
\[ \lambda_{\min}^{(m)} := d^{-1} \inf \{ \lambda \in \mathbb{Z} ~|~ H^0 (E', \mathcal{L}_d|_{E'}^{\otimes m})_\lambda \neq 0 \}. \] 
Consider the product configuration of $(E', \mathcal{L}|_{E'})$ associated to the $\mathbb{G}_m$-action on $E'$. 
Then by \cite[section 5.5]{BHJ1}, we obtain 
\[ \lambda_{\min}^{(m)}/m = d^{-1} \inf \supp \DHm_{(E', \mathcal{L}_d|_{E'})} = \inf \supp \DHm_{(E, \mathcal{L}|_E)} \] 
for sufficiently divisible $m$. 

Therefore, it suffices to show for any $\varepsilon > 0$ and for every sufficiently divisible $m$, there exists $\lambda \le m (\sigma_E + \varepsilon)$ such that $\frac{\widehat{\mathcal{F}}^\lambda_{(\mathcal{X}, \mathcal{L})} R_m}{\widehat{\mathcal{F}}^\lambda_{(\mathcal{X}, \mathcal{L})} \cap \mathcal{F}^{\lambda+}_{v_E} [\sigma_E] R_m} \neq 0$. 
We can simplify this slightly: for $\varepsilon > 0$, we want to find $m'$, $\lambda' \le m' (\sigma_E + \varepsilon)$ and $s' \in R_{m'}$ so that $v_F (s') + m' \sigma_F \ge \lambda'$ for every $F \subset \mathcal{X}_0$ and $v_E (s') + m' \sigma_E = \lambda'$. 

Since $\widehat{\mathcal{F}}_{(\mathcal{X}, \mathcal{L})}$ is finitely generated, there exist $\lambda_0 \in \mathbb{R}$, $m_0 \in \mathbb{N}^{(d)}_+$ and $s_0 \in \widehat{\mathcal{F}}^{\lambda_0}_{(\mathcal{X}, \mathcal{L})} R_{m_0}$ such that 
\[ v_E (s_0) + m_0 \sigma_E = \lambda_0. \]
(See also the proof of Proposition \ref{explicit formula for sigma}. )
Since $L$ is globally generated, we can take $l \in \mathbb{N}^{(d)}_+$ and $t \in R_l$ so that $t$ does not vanish at all the centers of $v_F$ for $F \subset \mathcal{X}_0$. 
Now for $\varepsilon > 0$, take large $k$ so that
\[ v_E (s_0) \le (m_0 + kl) \varepsilon. \]
We put $m' := m_0 + kl$, $s' := s_0 t^k \in R_{m'}$ and $\lambda' := \lambda_0 + kl \sigma_F$. 
Then since $v_E (s') = v_E (s_0 t^k)$, we have 
\[ \lambda' = v_E (s') + m' \sigma_E \le m' (\sigma_E + \varepsilon). \]
On the other hand, since $s_0 \in \widehat{\mathcal{F}}^{\lambda_0}_{(\mathcal{X}, \mathcal{L})} R_{m_0}$, we have 
\[ v_F (s') + m' \sigma_F = v_F (s_0) + m_0 \sigma_F + kl \sigma_F \ge \lambda_0 + kl \sigma_F = \lambda' \]
as desired. 
\end{proof}

\subsection{Non-archimedean pluripotential theory}
\label{Non-archimedean pluripotential theory}

In this section, we recall Boucksom--Jonsson's global pluripotential theory over trivially valued non-archimedean fields developed in \cite{BJ1, BJ2, BJ3, BJ4}. 
We exhibit proofs of some known results as we would approach the theory in a slightly different manner based on observations in the last section. 
The following diagram summarizes various constructions. 

%\[ \begin{tikzcd}
%\mathsf{Filt}^{\mathrm{unif}} \ar{dr}[description]{\varphi_{\mathcal{F}}} \ar[phantom, "\cup"]{d}
%& \E^1 \ar[phantom, "\cup"]{d} \ar{dr}[description]{\mathrm{MA} (\varphi)}
%& 
%&
%X^{\mathrm{lin}} \ar{dl}[description]{\delta_v} \ar[dashed, bend right]{lll}[description]{\mathcal{F}_v} \ar[phantom, "\cup"]{d}
%\\
%\mathsf{Filt}^{\mathrm{f.g.}} \ar{r} \ar[phantom, "\cup"]{d}
%& C^0 \cap \PSH \ar{r} \ar[bend right]{ul}[description]{\mathcal{F}_\varphi}
%& \mathcal{M}_{\mathrm{NA}}^1 \ar[bend right, dashed]{ul}
%&
%X^{\mathrm{qm}} \ar{l} \ar[phantom, "\cup"]{d}
%\\
%\mathsf{TC} \ar{ur}[description]{\varphi_{(\mathcal{X}, \mathcal{L})}}
%&
%&
%&
%X^{\mathrm{div}} \ar{ul}
%\end{tikzcd} \]

\[ \begin{tikzcd}
\mathsf{Filt}^{\mathrm{unif}} \ar{dr}[description]{\varphi_{\mathcal{F}}} \ar[phantom, "\cup"]{d}
& \E^1 \ar[phantom, "\cup"]{d} \ar{dr}[description]{\mathrm{MA} (\varphi)}
& 
&
X^{\mathrm{lin}} \ar{dl}[description]{\delta_v} \ar[dashed, bend right]{lll}[description]{\mathcal{F}_v} \ar[phantom, "\cup"]{d}
\\
\mathsf{Filt}^{\mathrm{f.g.}} \ar{r} \ar[phantom, "\cup"]{d}
& C^0 \cap \PSH \ar{r} \ar[bend right]{ul}[description]{\mathcal{F}_\varphi}
& \mathcal{M}^1 (X^{\mathrm{NA}}) \ar[bend right, dashed]{ul}
&
X^{\mathrm{qm}} \ar{l} \ar[phantom, "\cup"]{d}
\\
\mathsf{TC} \ar{r}
& \mathsf{nTC} \ar[hookrightarrow]{u}[description]{\varphi_{(\mathcal{X}, \mathcal{L})}} \ar{ur}
& 
&
X^{\mathrm{div}} \ar{ul}
\end{tikzcd} \]

Here $\mathsf{TC}$ (resp. $\mathsf{nTC}$) denotes the set of isomorphism classes of test configurations (resp. normal test configurations), and $\mathsf{Filt}^{\mathrm{f.g}}$ denotes the set of finitely generated filtrations (see Definition \ref{f.g. filtration}). 
Dashed arrows exist under the continuity of envelopes (see section \ref{continuity of envelopes}). 
We will explain the rest of notations in this section. 

There are many other approaches to non-archimedean psh metrics, the details of which are beyond the author's knowledge. 
See Introduction of \cite{BJ3} for the history. 

\subsubsection{Berkovich space}

Let $X$ be a scheme of locally finite type over a field $k$. 
Every schematic point $y \in X$ (not necessarily closed) is the generic point of a unique irreducible reduced subscheme $Y \subset X$. 
Let $X_y^{\mathrm{NA}} := \mathrm{Val} (Y)$ denote the set of valuations on $Y$. 
We note we assume $v|_{k^\times} = 0$ for valuations. 

The \textit{Berkovich space} $X^{\mathrm{NA}}$ associated to $X$ is a topological space defined as follows. 
We put 
\begin{equation} 
X^{\mathrm{NA}} := \coprod_{y \in X} X_y^{\mathrm{NA}}. 
\end{equation}
The topology on $X^{\mathrm{NA}}$ is the weakest topology which satisfies the following
\begin{enumerate}
\item The forgetful map $X^{\mathrm{NA}} \to X: X_y^{\mathrm{NA}} \mapsto y$ is continuous, 

\item For every Zariski open set $U \subset X$ and $f \in \mathcal{O}_X (U)$, the following function is continuous: 
\begin{equation} 
- \log |f|: U^{\mathrm{NA}} \to (-\infty, \infty]: v \in X_y^{\mathrm{NA}} \mapsto v (f|_Y). 
\end{equation}
\end{enumerate}
It is known (cf. \cite[Thereom 3.5.3]{Berk}) that $X^{\mathrm{NA}}$ is (path-)connected, Hausdorff, compact if and only if the scheme $X$ is connected, separated, proper, respectively. 
While $X^{\mathrm{NA}}$ is not first countable (see example below), it is Fr\'echet--Urysohn space, hence sequential, by \cite[Theorem 5.3]{Poi}. 
Namely, the closure of subsets coincide with the sequential closure, hence every sequentially closed subset is closed. 
As a consequence, for proper $X$, $X^{\mathrm{NA}}$ is sequentially compact as well as compact. 

For a non-archimedean/archimedean complete field $\hat{k}$, we can apply an analogous construction which reflects the non-archimedean/archimedean norm on $\hat{k}$ to a scheme of locally finite type over $\hat{k}$: we assume $v (f) = -\log \| f \|$ for $f \in \hat{k}^\times$ in the definition of valuation. 
From this perspective, the above definition of the Berkovich space is a special case of such construction: we identify the field $k$ with the trivially valued non-archimedean field. 
The associated space is also called the Berkovich analytification of $X$ by this reason. 

We denote by $X_y^{\mathrm{lin}}, X_y^{\mathrm{qm}}$ and $X_y^{\mathrm{div}}$ the set of all linear growth, quasi-monomial and divisorial valuations on $Y$, respectively. 
We put 
\begin{align} 
X^{\mathrm{val}} 
&:= \coprod_{y \in X^{(0)}} X^{\mathrm{NA}}_y \subset X^{\mathrm{NA}}, 
\\
X^{\mathrm{div}} 
&:= \coprod_{y \in X^{(0)}} X^{\mathrm{div}}_y \subset X^{\mathrm{NA}}, 
\end{align}
where $X^{(0)}$ denotes the set of generic points of irreducible components of $X$. 
Then $X^{\mathrm{div}}$ is dense in $X^{\mathrm{NA}}$ as shown in \cite[Corollary 2.16]{BJ3}. 
We define $X^{\mathrm{lin}}, X^{\mathrm{qm}}$ in the same way. 

In the trivially valued case, there is a continuous action of the multiplicative group $\mathbb{R}_+$ on $X^{\mathrm{NA}}$ given by the scaling of valuation: $(\rho. v) (f) = \rho \cdot v (f)$. 
For a function $\psi: X^{\mathrm{NA}} \to [-\infty, \infty)$ and $\rho \in \mathbb{R}_+$, we define a rescaled function $\psi_{; \rho}: X^{\mathrm{NA}} \to [-\infty, \infty)$ by 
\begin{equation} 
\psi_{; \rho} (v) = \rho \psi (\rho^{-1}. v). 
\end{equation}

\begin{eg}
The Berkovich space $\Sigma^{\mathrm{NA}}$ of a smooth algebraic curve $\Sigma$ over (the trivially valued field) $\mathbb{C}$ is identified with 
\[ \varprojlim_{D \subset \Sigma} \mathrm{Tree}_D, \]
where $D$ runs over all finite subsets. 
Here we put 
\[ \mathrm{Tree}_D := \coprod_{z \in D} [0, \infty]_z/\sim \]
where $[0, \infty]_z$ are copies of the interval $[0, \infty]$ given for each closed point $z \in S$ and we identify all $0 \in [0, \infty]_z$ by $\sim$. 
For two finite subsets $D \subset D' \subset \Sigma$, we have the projection $\mathrm{Tree}_{D'} \to \mathrm{Tree}_D$, which makes $\{ \mathrm{Tree}_D \}$ into the inverse system. 
The point $[0]$ corresponds to the trivial valuation on $\Sigma$, each infinity $\infty \in [0, \infty]_z$ corresponds to the trivial valuation on a closed point $z \in \Sigma$ and each $t \in (0, \infty)_z$ corresponds to the valuation $t. \mathrm{ord}_z$. 
This example shows $X^{\mathrm{NA}}$ is not first countable nor separable in general. 

We note this is different from the following topological space 
\[ \varinjlim_{D \subset \Sigma} \mathrm{Tree}_D = \coprod_{z \in \Sigma} [0, \infty]_z/\sim, \]
though we have a bijective continuous map $\coprod_{z \in \Sigma} [0, \infty]_z/\sim \to \Sigma^{\mathrm{NA}}$. 

% In the one dimensional case, we have $X^{\mathrm{val}} = X^{\mathrm{lin}} = X^{\mathrm{qm}} = X^{\mathrm{div}} = \varprojlim_{D \subset \Sigma} \mathrm{Tree}_D^\circ$. 
% What is the metric topology $d_\infty$ on it? 
\end{eg}

The above abstract topological description is enough for our purpose: we can use Dini's lemma and Riesz--Markov--Kakutani representation theorem on Radon measures. 
We recall a Radon measure on a compact Hausdorff space $X$ is a finite Borel measure which is inner regular: we have 
\[ \mu (B) = \sup \{ \mu (K) ~|~ B \supset K \text{ compact } \} \]
for every Borel set $B \subset X$. 

We recall every finite Borel measure on $\mathbb{R}$ is both inner regular and outer regular: we also have 
\[ \mu (B) = \sup \{ \mu (U) ~|~ B \subset U \text{ open } \} \] 
for every Borel set $B \subset \mathbb{R}$. 
We will use this fact in several times. 

In the arhimedean case, every finite Borel measure on a manifold is known to be inner regular, even outer regular. 
It would be comfortable if we have the same property for $X^{\mathrm{NA}}$, however, it is pointed out in \cite[Example 2.7]{Jon} that this problem could be sensitive to the choice of a model of ZFC. 
Anyway, this is not a nuisance as we are only concerned with Radon measures in our theory. 

The space $X^{\mathrm{NA}}$ is endowed with a structure sheaf of local rings, which can be considered as `analytic structure' on $X^{\mathrm{NA}}$, though it only has algebraic information in the trivially valued case, in principle. 
We do not consider such structure in this article. 
Instead, the `analytic structure' on $X^{\mathrm{NA}}$ is reflected in the following: 
\begin{itemize}
\item For a test configuration $(\mathcal{X}, \mathcal{L})$, we assign a continuous function $\varphi_{(\mathcal{X}, \mathcal{L})}: X^{\mathrm{NA}} \to \mathbb{R}$. 

\item For a test configuration $(\mathcal{X}, \mathcal{L})$, we assign a Radon measure $\mathrm{MA} (\mathcal{X}, \mathcal{L})$ on $X^{\mathrm{NA}}$. 

\item Assume $X$ has only klt singularities. 
Then the log discrepancy $A_X$ defined on $X^{\mathrm{div}}$ by $A_X (c. \mathrm{ord}_E) = c (1 + \mathrm{ord}_E (K_{Y/X}))$ extends to a lower semi-continuous function $A_X: X^{\mathrm{NA}} \to [0, \infty]$. 
See the remark after Theorem \ref{NAmu entropy extension}. 
\end{itemize}

Now we are going to define non-archimedean psh metrics. 
In the archimedean case, a psh metric on an ample line bundle $L$ is a singular hermitian metric on $L$ which can be written as $h e^{-\phi}$ using a smooth hermitian metric $h$ and an upper semi-continuous function $\phi$ such that $- \log |s|_h^2 + \phi$ is pluri-subharmonic for every local holomorphic section $s$ of $L$. 
This definition relies on local notion of pluri-subharmonicity. 

Thanks to Bergman kernel approximation \cite{Tian} and Demailly approximation \cite{BK}, we have the following global characterization of psh metrics. 
Fix a reference metric $h$ on $L$. 
We call a smooth function $\phi$ on $X$ \textit{Fubini--Study potential} if $h e^{-\phi}$ is the pull-back of the canonical metric on $(\mathbb{C}P^N, \mathcal{O} (1))$ along some Kodaira embedding $X \hookrightarrow \mathbb{C}P^N$. 
Then an upper semi-continuous function $\phi$ on $X$ gives a psh metric $h e^{-\phi}$ on $L$ if and only if $\phi$ is the pointwise limit of a decreasing sequence of Fubini--Study potentials $\phi_i \searrow \phi$. 
Boucksom--Jonsson's definition of non-archimedean psh metrics on $(X, L)$ is modeled on this characterization. 

\subsubsection{Fubini--Study metrics}

Recall we assume $(X, L)$ is a polarized \textit{normal} variety. 

Since $X$ is proper, we can assign the center $c (v) \in Y$ (schematic point) for each $v \in X_y^{\mathrm{val}}$ by the valuative criterion. 
For $v \in X_y^{\mathrm{val}}$ and a section $s \in H^0 (X, L^{\otimes m})$, we put $v (s) := v ((s/e)|_Y)$ by taking a local generator $e$ of $L^{\otimes m}$ around the center $c (v)$. 
This is independent of the choice of $e$. 

For a linearly bounded filtration $\mathcal{F}$ and (sufficiently divisible) $d \in \mathbb{N}_+$, we associate a continuous function $\varphi_{\mathcal{F}}^{(d)}$ on $X^{\mathrm{NA}}$ by 
\begin{equation} 
\varphi_{\mathcal{F}}^{(d)} (v) := \frac{1}{d} \max_{i=1, \ldots, N_d} \{ - v (s_i) - \log \| s_i \|^\mathcal{F}_d \},
\end{equation}
using a diagonal basis $s_1, \ldots, s_{N_d}$ of $R_d$ with respect to $\| \cdot \|_d^\mathcal{F}$. 
This is independent of the choice of the diagonal basis $\{ s_i \}$ as shown in the proof of the proposition below. 
Since $s_1, \ldots, s_{N_d}$ have no common zeros, $\varphi_{\mathcal{F}}^{(d)}$ gives a continuous function on $X^{\mathrm{NA}}$. 

When $\mathcal{F}$ is finitely generated, $\varphi_{\mathcal{F}}^{(d)}$ is independent of the choice of sufficiently divisible $d$ as shown in \cite{BJ2}. 
We denote it by $\varphi_{\mathcal{F}}$ for finitely generated $\mathcal{F}$. 
For the finitely generated filtration $\mathcal{F} = \mathcal{F}_{(\mathcal{X}, \mathcal{L}; \xi)}$ (resp. $\mathcal{F} = \mathcal{F}_{(\mathcal{X}, \mathcal{L})}$) associated to a polyhedral configuration $(\mathcal{X}/B_\sigma, \mathcal{L}; \xi)$ (resp. a test configuration $(\mathcal{X}, \mathcal{L})$), we denote $\varphi_\mathcal{F}$ by $\varphi_{(\mathcal{X}, \mathcal{L}; \xi)}$ (resp. $\varphi_{(\mathcal{X}, \mathcal{L})}$). 
We have $\varphi_{\mathcal{F}_{\mathrm{triv}} [\tau]} = \tau$. 
We especially denote the constant function $0$ by $\varphi_{\mathrm{triv}}$. 

Now recall we studied $\sigma_v (\mathcal{F}) = \inf \{ \sigma \in \mathbb{R} ~|~ \mathcal{F}^\lambda \subset \mathcal{F}^\lambda_v [\sigma] \text{ for } \forall \lambda \in \mathbb{R} \}$ in the last section. 
It is connected to non-archimedean psh metrics as we declared. 

\begin{prop}[Proposition 2.16 in \cite{BJ4}]
\label{non-archimedean metric associated to filtration}
For a linearly bounded filtration $\mathcal{F}$, $d | d' \in \mathbb{N}_+$ and $v \in \mathrm{Val} (X)$, we have $\varphi_{\mathcal{F}}^{(d)} (v) \le \varphi_{\mathcal{F}}^{(d')} (v) \le \sigma_v (\mathcal{F})$ and $\lim_{i \to \infty} \varphi_{\mathcal{F}}^{(d_i)} (v) = \sigma_v (\mathcal{F})$ for any eventually sufficiently divisible sequence $\{ d_i \}$: for any $p \in \mathbb{N}_+$ there exists $i_p$ such that $p$ divides $d_i$ for every $i \ge i_p$. 
For a finitely generated filtration $\mathcal{F}$, we have 
\[ \varphi_{\mathcal{F}} (v) = \varphi_{\mathcal{F}}^{(d)} (v) = \sigma_v (\mathcal{F}) \]
for sufficiently divisible $d$.  
\end{prop}

\begin{proof}
Take $d \in \mathbb{N}_+$ and a diagonal basis $\{ s_i \}$ of $R_d$ as in the construction of $\varphi_{\mathcal{F}}^{(d)}$. 
Since $\{ s_i \}$ is diagonal, we compute 
\begin{align*} 
v (\sum_i a_i s_i) + \log \| \sum_i a_i s_i \|^\mathcal{F}_d 
&\ge \min \{ v (s_i) ~|~ a_i \neq 0 \} + \max \{ \log \| s_i \|^\mathcal{F}_d ~|~ a_i \neq 0 \} 
\\
&\ge \min \{ v (s_i) + \log \| s_i \|^\mathcal{F}_d ~|~ a_i \neq 0 \}. 
\end{align*}
It follows that 
\[ \varphi_{\mathcal{F}}^{(d)} (v) = \frac{1}{d} \sup \{ - v (s) - \log \| s \|^\mathcal{F}_d ~|~ 0 \neq s \in R_d \}. \]

For any $l \in \mathbb{N}_+$ and $s \in R_d$, we have 
\begin{align*} 
\frac{1}{d} (- v (s) - \log \| s \|^\mathcal{F}_{d}) 
&\le \frac{1}{d l} (- v (s^{\otimes l}) - \log \| s^{\otimes l} \|^\mathcal{F}_{dl}) 
\\
&\le \frac{1}{d l} \sup \{ - v (s') - \log \| s' \|^\mathcal{F}_{d l} ~|~ 0 \neq s' \in R_{d l} \}, 
\end{align*}
so that we get 
\[ \varphi_{\mathcal{F}}^{(d)} (v) \le \varphi_{\mathcal{F}}^{(dl)} (v). \]
Now we note $s \in \mathcal{F}^\lambda R_d$ iff $-\log \| s \|^\mathcal{F}_d \ge \lambda$ and $s \in \mathcal{F}^\lambda_v [\sigma] R_d$ iff $v (s) + d \sigma \ge \lambda$. 
It follows that $\mathcal{F} \subset \mathcal{F}_v [\sigma]$ iff $v (s) + m \sigma \ge -\log \| s \|^\mathcal{F}_m$ for every $m$, so we get 
\[ \varphi_{\mathcal{F}}^{(d)} (v) \le \sigma_v (\mathcal{F}). \]

As for $\lim_{i \to \infty} \varphi_{\mathcal{F}}^{(d_i)} (v) = \sigma_v (\mathcal{F})$, we already know $\lim_{i \to \infty} \varphi_{\mathcal{F}}^{(d_i)} (v) \le \sigma_v (\mathcal{F})$. 
Take $\sigma \ge \lim_{i \to \infty} \varphi_{\mathcal{F}}^{(d_i)} (v)$. 
Then since $\sigma \ge \lim_{i \to \infty} \varphi_{\mathcal{F}}^{(d_i)} (v) \ge \varphi_{\mathcal{F}}^{(m)} (v)$ for every $m \in \mathbb{N}_+$, we have $m \sigma \ge - v (s) - \log \| s \|^{\mathcal{F}}_m$ for every $m$ and $s \in R_m$, which implies $\mathcal{F} \subset \mathcal{F}_v [\sigma]$. 
Thus we conclude $\sigma_v (\mathcal{F}) = \inf \{ \sigma ~|~ \mathcal{F} \subset \mathcal{F}_v [\sigma] \} \le \lim_{i \to \infty} \varphi_{\mathcal{F}}^{(d_i)} (v)$. 

When $\varphi$ is finitely generated, we have $\varphi_{\mathcal{F}}^{(d)} = \lim_{i \to \infty} \varphi_{\mathcal{F}}^{(d_i)} (v) = \sigma_v (\mathcal{F})$ for sufficiently divisible $d$. 
\end{proof}

\begin{cor}[Proposition 2.9 and Proposition A.3 in \cite{BJ4}]
We have $\varphi_{(\mathcal{X}, \mathcal{L})} = \varphi_{(\mathcal{X}^\nu, \nu^* \mathcal{L})}$ for the normalization $\nu: \mathcal{X}^\nu \to \mathcal{X}$ of a (non-normal) test configuration $(\mathcal{X}, \mathcal{L})$. 
\end{cor}

\begin{proof}
By the above proposition, it suffices to show $\sigma_v (\mathcal{F}_{(\mathcal{X}, \mathcal{L})}) = \sigma_v (\mathcal{F}_{(\mathcal{X}^\nu, \nu^* \mathcal{L})})$, which we proved in Proposition \ref{sigma for normalization}. 
\end{proof}

We put 
\begin{equation} 
\nH (X, L) := \{ \varphi_{(\mathcal{X}, \mathcal{L})} ~|~ (\mathcal{X}, \mathcal{L}): \text{ test configuration } \}.
\end{equation}
As we observed in the above corollary, the map $\mathsf{TC} \to \nH (X, L): (\mathcal{X}, \mathcal{L}) \mapsto \varphi_{(\mathcal{X}, \mathcal{L})}$ is not injective. 
As we will see in Corollary \ref{normalization and non-archimedean metric}, we have $\varphi_{(\mathcal{X}_1, \mathcal{L}_1)} = \varphi_{(\mathcal{X}_2, \mathcal{L}_2)}$ iff the normalizations are isomorphic $(\mathcal{X}_1^\nu, \nu^*  \mathcal{L}_1) \cong (\mathcal{X}_2^\nu, \nu^* \mathcal{L}_2)$. 
In other words, we have 
\[ \mathsf{nTC} \cong \nH (X, L). \]

Similarly, we put 
\begin{equation}
\pcH (X, L) := \{ \varphi_\mathcal{F} ~|~ \mathcal{F}: \text{ finitely generated filtration } \}. 
\end{equation}
Polyhedral configuration $(\mathcal{X}/B_\sigma; \mathcal{L})$ gives a polyhedral structure on $\nH^{\mathbb{R}} (X, L)$: put $\Phi_\sigma := \{ \varphi_{(\mathcal{X}/B_\sigma; \mathcal{L}; \xi)} \}_{\xi \in \sigma}$, then we have 
\[ \nH^{\mathbb{R}} (X, L) = \bigcup_{(\mathcal{X}/B_\sigma; \mathcal{L})} \Phi_\sigma. \]
We can then discuss continuity, piecewise linearity and even piecewise smoothness of a functional on $\nH^{\mathbb{R}} (X, L)$ by restricting the functional to each toric cone $\Phi_\sigma$. 
Among all, the continuity of the entropy $\int_{X^{\mathrm{NA}}} A_X \mathrm{MA} (\varphi)$ with respect to such `polyhedral topology' is important. 
This would be studied in \cite{BJ5}: `polyhedral topology' is enough strong as `$C^\infty$-topology' in the archimedean analysis. 

For a polyhedral configuration $(\mathcal{X}/B_\sigma, \mathcal{L}; \xi)$ and $\rho \in \mathbb{R}_+$, we have 
\[ \varphi_{(\mathcal{X}, \mathcal{L}; \rho \xi)} = \varphi_{(\mathcal{X}, \mathcal{L}; \xi); \rho} \]
as $\mathcal{F}_{(\mathcal{X}, \mathcal{L}; \rho \xi)} = \mathcal{F}_{(\mathcal{X}, \mathcal{L}; \xi); \rho}$. 
On the other hand, $\mathcal{F}_{(\mathcal{X}_d, \mathcal{L}_d)} \neq \mathcal{F}_{(\mathcal{X}, \mathcal{L}; d)}$ in general. 
Nevertheless, these gives the same functions. 

\begin{cor}[Lemma 2.22 in \cite{BJ3}]
\label{scaling of NA psh}
For a test configuration $(\mathcal{X}, \mathcal{L})$ and the normalized base change $(\mathcal{X}_d, \mathcal{L}_d)$, we have 
\[ \varphi_{(\mathcal{X}_d, \mathcal{L}_d)} = \varphi_{(\mathcal{X}, \mathcal{L}); d} = \varphi_{(\mathcal{X}, \mathcal{L}; d)}. \]
\end{cor}

\begin{proof}
By Proposition \ref{non-archimedean metric of normalized base change}, we have $\varphi_{(\mathcal{X}_d, \mathcal{L}_d)} (v_E) = \varphi_{(\mathcal{X}, \mathcal{L}); d} (v_E)$ for the valuation $v_E$ associated to any irreducible component $E \subset \tilde{\mathcal{X}}_0$ of any normal test configuration $\tilde{\mathcal{X}}$. 
As we noted, valuations of the form $v_E$ forms a dense subset $X^{\mathrm{div}} \subset X^{\mathrm{NA}}$, so the claim follows by the continuity of $\varphi_{(\mathcal{X}, \mathcal{L})}$. 
\end{proof}

\subsubsection{Non-archimedean psh metrics}

A \textit{net of functions} on $X^{\mathrm{NA}}$ is a collection $\{ \psi_i: X^{\mathrm{NA}} \to [-\infty, \infty) \}_{i \in I}$ of functions parametrized by a directed set $I$: $I$ is endowed with a preorder $\le$ and for every $i, j \in I$ there exists $k \in I$ satisfying $i, j \le k$. 
A \textit{decreasing net of functions} is a net of functions satisfying $\psi_j (x) \le \psi_i (x)$ for every $i \le j$ and $x \in X^{\mathrm{NA}}$. 
For any decreasing net of functions, the pointwise limit $\psi: X^{\mathrm{NA}} \to [-\infty, \infty)$ exists, which we denote by $\psi_i \searrow \psi$. 
When $\psi_i$ are upper semi-continuous, the limit function is also upper semi-continuous. 
We may assume the index set $I$ has a minimum $0 \in I$ by restricting $I$ to the cofinal subset $\{ i \in I ~|~ i \ge 0 \}$ if necessary. 

A (potenital of) \textit{non-archimedean psh metric} on $(X, L)$ is an upper semi-continuous function $\varphi: X^{\mathrm{NA}} \to [-\infty, \infty)$ which is the pointwise limit of some decreasing net of functions in $\nH (X, L)$ and is not identically $-\infty$. 
When we consider non-archimedean psh metrics $\varphi, \varphi'$ on different line bundles $L, L'$, we write those as $(L, \varphi), (L', \varphi')$. 
We put 
\[ \PSH (X, L) := \{ \text{ non-archimedean psh metrics on } (X, L) \}. \]
Thanks to the compactness of $X^{\mathrm{NA}}$, the pointwise limit of any decreasing net in $\PSH (X, L)$ is either in $\PSH (X, L)$ or identically $-\infty$ (cf. \cite{BJ1, BJ3}). 
Less obviously, it is proved by \cite[Thoerem 9.11]{BJ3} that for $\varphi \in \PSH (X, L)$ there exists a decreasing \textit{sequence} $\{ \varphi_i \}_{i \in \mathbb{N}} \subset \nH (X, L)$ pointwisely converging to $\varphi$. 
We will use this fact only for $\varphi \in \E^1 (X, L)$, which is much easier: $\E^1 (X, L)$ is endowed with a metric, and pointwise convergence and the metric convergence are equivalent for decreasing nets. 
We would find the proof in section \ref{metric space Eexp}. 

What we must be careful for net is that various convergence theorems in measure theory is \textit{invalid} for net. 
As for monotone convergence theorem, we have the following. 

\begin{prop}[Proposition 7.12 in \cite{Fol}]
\label{monotone convergence for net}
Let $X$ be a compact Hausdorff space and $\mu$ be a Radon measure on $X$. 
For an increasing net $\{ f_i \}_{i \in I}$ of lower semi-continuous functions (resp. a decreasing net $\{ g_i \}_{i \in I}$ of upper semi-continuous functions), we have 
\[ \sup_{i \in I} \int_X f_i d\mu = \int_X \sup_{i \in I} f_i d\mu \quad \left(\text{ resp. } \inf_{i \in I} \int_X g_i d\mu = \int_X \inf_{i \in I} g_i d\mu \right). \]
\end{prop}

Until section \ref{metric space Eexp}, except for section \ref{Tomographic expression of moment measure}, we mainly discuss continuity along decreasing nets and rather rely on monotonicity and the following Dini's lemma. 
In section \ref{non-archimedean mu-entropy}, we show the continuity of various functionals on $\E^{\exp} (X, L)$ with respect to the $E_{\exp}/d_{\exp}$-topology, using the dominated convergence theorem to some sequence of functions on $\mathbb{R}$ associated to a sequence $\{ \varphi_i \}_{i \in \mathbb{N}}$. 
Since the dominated convergence theorem is valid only for sequences, we must check the $E_{\exp}/d_{\exp}$-topologies on $\E^{\exp} (X, L)$ are sequential. 
This is obvious for $d_{\exp}$-topology once the metric is constructed. 
As for $E_{\exp}$-topology, we check it in Proposition \ref{Eexp topology is first countable}

We frequently use Dini's lemma of the following form. 
This relies on the compactness of the Berkovich space $X^{\mathrm{NA}}$. 

\begin{lem}
\label{Dini}
Let $\{ \varphi_i \}_{i \in I} \subset \PSH (X, L)$ be a decreasing net which pointwisely converges to $\varphi \in \PSH (X, L)$. 
For any $\tilde{\varphi} \in C^0 \cap \PSH (X, L)$ with $\varphi \le \tilde{\varphi}$ and $\varepsilon > 0$, there exists $i_\varepsilon \in I$ such that $\varphi_i < \tilde{\varphi} + \varepsilon$ for every $i \ge i_\varepsilon$. 
\end{lem}

\begin{proof}
Since $\varphi_i - \tilde{\varphi}$ is a decreasing net of usc functions, the subsets 
\[ F_i := \{ x \in X^{\mathrm{NA}} ~|~ \varphi_i (x) \ge \tilde{\varphi} (x) +\varepsilon \} \] 
give a decreasing net of closed sets. 
Since $\bigcap_{i \in I} F_i = \emptyset$, we have $F_j \subset F_{i_1} \cap \dotsb \cap F_{i_k} = \emptyset$ for some $i_1, \ldots, i_k$ and every $j \ge i_1, \ldots, i_k$ by the compactness of $X^{\mathrm{NA}}$. 
\end{proof}

Recall the weak convergence of psh metrics in the archimedean case is defined to be the $L^1$-convergence with respect to the Lebesgue measure (cf. \cite[Section 8]{GZ}). 
This is equivalent to the convergence $\int_X \varphi_i \mathrm{MA} (\psi) \to \int_X \varphi \mathrm{MA} (\psi)$ for every smooth metric $\psi$ as $\mathrm{MA} (\psi)$ is absolutely continuous with respect to the Lebesgue measure and vice versa. 

In the non-archimedean case, there is no Borel measure on $X^{\mathrm{NA}}$ such that every non-archimedean $\mathrm{MA} (\psi)$ is absolutely continuous with respect to the measure. 
The latter characterization adapts the non-archimedean case (cf. \cite[Corollary 9.18]{BJ3}): a net $\{ \varphi_i \}_{i \in I}$ in $\PSH (X, L)$ is called \textit{weakly convergent} to $\varphi \in \PSH (X, L)$ if $\varphi_i (v) \to \varphi (v)$ for every divisorial valuation $v \in X^{\mathrm{div}}$. 
It is equivalent to say $\varphi_i (v) \to \varphi (v)$ for every quasi-monomial valuation $v \in X^{\mathrm{qm}}$. 
For a decreasing net $\{ \varphi_i \}_{i \in I} \subset \PSH (X, L)$, $\varphi_i$ converges to $\varphi$ pointwisely if and only if it converges to $\varphi$ weakly by \cite[Corollary 4.30]{BJ3}. 

\begin{quest} 
Is the weak topology on $\PSH (X, L)$ sequential?
\end{quest}

We note $\sup \varphi = \varphi (v_{\mathrm{triv}})$ for every $\varphi \in \PSH (X, L)$. 
Indeed, we can directly check this for $\varphi_{\mathcal{F}} \in \pcH (X, L)$ and then for $\varphi \in \PSH (X, L)$ by passing to the limit of a regularization $\varphi_i \searrow \varphi$. 
It follows that $\sup \varphi_i \to \sup \varphi$ whenever $\varphi_i \to \varphi$ weakly. 

\subsubsection{Finite energy class}

For a test configuration $(\mathcal{X}, \mathcal{L})$, we have 
\[ \frac{(\overline{\mathcal{L}}^{\cdot n+1})}{(n+1)!} = -\frac{\rho^{-1} (\mathcal{L}_{\mathbb{G}_m}|_{\mathcal{X}_0}^{\cdot n+1}; \rho)}{(n+1)!}  = \int_\mathbb{R} t \DHm_{(\mathcal{X}, \mathcal{L})}. \]
As shown in \cite{BHJ1, BJ2}, this depends only on the associated non-archimedean psh metric $\varphi_{(\mathcal{X}, \mathcal{L})}$. 
We have $\frac{1}{(n+1)!} ((\overline{\mathcal{L}}')^{\cdot n+1}) \le \frac{1}{(n+1)!} (\overline{\mathcal{L}}^{\cdot n+1})$ if $\varphi_{(\mathcal{X}', \mathcal{L}')} \le \varphi_{(\mathcal{X}, \mathcal{L})}$ (cf. Lemma \ref{monotonicity}). 
The \textit{energy functional} $E: \PSH (X, L) \to [-\infty, \infty)$ is defined by 
\begin{equation} 
E (\varphi) := \inf \Big{\{} \frac{(\overline{\mathcal{L}}^{\cdot n+1})}{(n+1)!}  ~\Big{|}~ \varphi_{(\mathcal{X}, \mathcal{L})} \ge \varphi \Big{\}}. 
\end{equation}
We put 
\begin{equation}
\E^1 (X, L) := \{ \varphi \in \PSH (X, L) ~|~ E (\varphi) > -\infty \}. 
\end{equation}
The \textit{strong topology} on $\E^1 (X, L)$ is the weakest refinement of the weak topology inherited from $\PSH (X, L)$ such that $E$ is continuous. 
As discussed in \cite{BJ4}, this is equivalent to the metric topology of a distance $d_1$, which we recall in section \ref{metric space Eexp}. 
Since the energy $E$ is continuous along decreasing nets, for a decreasing net $\{ \varphi_i \}_{i \in I} \subset \E^1 (X, L)$, $\varphi_i$ converges to $\varphi$ strongly if and only if it converges to $\varphi$ weakly. 

\subsubsection{Non-archimedean Monge--Amp\`ere operator}

The \textit{non-archimedean Monge--Amp\`ere measure} of a normal test configuration $(\mathcal{X}, \mathcal{L})$ is the measure on $X^{\mathrm{NA}}$ defined by 
\begin{equation} 
\mathrm{MA} (\mathcal{X}, \mathcal{L}) := \frac{1}{n!} \sum_{E \subset \mathcal{X}_0} \mathrm{ord}_E \mathcal{X}_0 \cdot (E. \mathcal{L}^{\cdot n}). \delta_{v_E}. 
\end{equation}
Here $\delta_v$ denotes the Dirac measure charging $v \in X^{\mathrm{NA}}$. 
Since we have $\mathsf{nTC} \cong \nH (X, L)$, we may regard $\mathrm{MA}$ as a operator on $\nH (X, L)$. 

Let $\mathcal{M} (X_{\mathrm{NA}})$ denotes the space of Radon measures on $X^{\mathrm{NA}}$ with total mass $(e^L)$. 
We put 
\begin{equation} 
E^\vee (\mu) := \sup_{\varphi \in \E^1 (X, L)} (E (\varphi) - \int_{X^{\mathrm{NA}}} \varphi d\mu) 
\end{equation}
for $\mu \in \mathcal{M} (X_{\mathrm{NA}})$ and 
\begin{equation} 
\mathcal{M}^1 (X^{\mathrm{NA}}) := \{ \mu \in \mathcal{M} (X^{\mathrm{NA}}) ~|~ E^\vee (\mu) < \infty \}. 
\end{equation}
The \textit{strong topology} of $\mathcal{M}^1 (X^{\mathrm{NA}})$ is the coarsest refinement of the weak topology induced from $\mathcal{M} (X_{\mathrm{NA}})$ which makes $E^\vee$ continuous. 

We have $\mathrm{MA} (\mathcal{X}, \mathcal{L}) \in \mathcal{M}^1 (X^{\mathrm{NA}})$. 
It is studied in \cite{BJ1, BJ3} that there is a unique extension of $\mathrm{MA}: \nH (X, L) \to \mathcal{M}^1 (X^{\mathrm{NA}})$ to 
\[ \mathrm{MA}: \E^1 (X, L) \to \mathcal{M}^1 (X^{\mathrm{NA}}) \]
which is continuous with respect to the strong topologies. 
In particular, we have the following. 
Compare Proposition \ref{double limit}. 

\begin{thm}[Lemma 5.28 or Theorem 7.3 in \cite{BJ3}]
If a net $\{ \varphi_i \}_{i \in I} \subset \E^1 (X, L)$ strongly converges to $\varphi \in \E^1 (X, L)$ and either $\psi \in C^0 (X^{\mathrm{NA}})$ or $\psi \in \E^1 (X, L)$, then we have 
\[ \lim_{i \to \infty} \int_{X^{\mathrm{NA}}} \psi \mathrm{MA} (\varphi_i) = \int_{X^{\mathrm{NA}}} \psi \mathrm{MA} (\varphi). \]
\end{thm}

\begin{eg}[Example 3.17 in \cite{BJ3}] 
Consider $(X, L) = (\mathbb{C}P^1, \mathcal{O} (1))$. 
On each branch $[0, \infty]_z$ of $X^{\mathrm{NA}}$, the function $\log |s|: v \mapsto - v (s)$ for $s \in H^0 (\mathbb{C}P^1, \mathcal{O} (d))$ is identified with the linear function $- \mathrm{ord}_z (s). t$. 
It follows that $\varphi \in \nH (\mathbb{C}P^1, \mathcal{O} (1))$ can be written as 
\[ \frac{1}{d} \max_i \{ -\mathrm{ord}_z (s_i). t + \lambda_i \}  \]
on each branch $[0, \infty]_z$, which is a piecewise affine convex function. 
The slope of each affine function takes value in $[-1, 0] \cap \mathbb{Q}$ as $\mathrm{ord}_z (s) \in [0, d] \cap \mathbb{Z}$. 

On each branch $(0, \infty]_z$, the Monge--Amp\`ere measure $\mathrm{MA} (\varphi)$ is given by the distributional derivative $(\varphi|_{(0, \infty]_z})''$, which is concentrated on the non-smooth locus of $\varphi$. 
The mass of $(0, \infty]_z$ is the minus slope $a_z$ of $\varphi$ on $[0, \varepsilon]_z$ and the mass at the trivial valuation $0 \in X^{\mathrm{NA}}$ is $1- \sum_{z \in \mathbb{C}P^1} a_z$. 
\end{eg}

\subsubsection{Filtration associated to continuous psh metric}
\label{Filtration associated to continuous psh metric}

For a continuous non-archimedean psh metric $\varphi \in C^0 \cap \PSH (X, L)$, we put 
\begin{align} 
\mathcal{F}^\lambda_\varphi R_m 
&:= \bigcap_{v \in \mathrm{Val} (X)} \mathcal{F}_v^\lambda [\varphi (v)] R_m, 
\\
- \log \| s \|_m^\varphi 
&:= \sup \{ \lambda \in \mathbb{R} ~|~ s \in \mathcal{F}^\lambda_\varphi R_m \} 
\\ \notag
&= \inf \{ m \varphi (v) + v (s) ~|~ v \in \mathrm{Val} (X) \} 
\end{align}
for $s \in R_m$. 
We note $\mathcal{F}_{\varphi_{;\rho}} = (\mathcal{F}_\varphi)_{;\rho}$. 
We have $\mathcal{F}_\varphi \subset \mathcal{F}_{\varphi'}$ if and only if $\varphi \le \varphi'$. 

\begin{prop}
\label{sigma and non-archimedean psh}
For $\varphi \in C^0 \cap \PSH (X, L)$, we have 
\[ \sigma_v (\mathcal{F}_\varphi) = \varphi (v) \]
and $\lim_{i \to \infty} \varphi_{\mathcal{F}_\varphi}^{(d_i)} = \varphi$ for any eventually sufficiently divisible sequence $\{ d_i \}$. 
\end{prop}

\begin{proof}
Since $\varphi$ is continuous, for any $\varepsilon > 0$, we can take $\tilde{\varphi} \in \nH (X, L)$ so that $\varphi \le \tilde{\varphi} \le \varphi + \varepsilon$ by Lemma \ref{Dini}. 
Then we have $\mathcal{F}_{\tilde{\varphi} - \varepsilon} \subset \mathcal{F}_\varphi \subset \mathcal{F}_{\tilde{\varphi}}$, so that we get $\sigma_v (\mathcal{F}_{\tilde{\varphi} - \varepsilon}) \le \sigma_v (\mathcal{F}_\varphi) \le \sigma_v (\mathcal{F}_{\tilde{\varphi}})$. 
By Proposition \ref{non-archimedean metric associated to filtration}, we get $\tilde{\varphi} (v) - \varepsilon \le \sigma_v (\mathcal{F}_\varphi) \le \tilde{\varphi} (v)$. 
It follows that $\varphi (v) - \varepsilon \le \sigma_v (\mathcal{F}_\varphi) \le \varphi (v) + \varepsilon$. 
As $\varepsilon > 0$ is arbitrary, we obtain $\varphi (v) = \sigma_v (\mathcal{F}_\varphi)$. 

Similarly, we have $\tilde{\varphi} = \varphi_{\mathcal{F}_{\tilde{\varphi}}^{(d)}} \le \varphi_{\mathcal{F}_{\varphi + \varepsilon}}^{(d)} = \varphi_{\mathcal{F}_\varphi}^{(d)} + \varepsilon$ for sufficiently divisible $d$ again by Proposition \ref{non-archimedean metric associated to filtration} and $\mathcal{F}_{\tilde{\varphi}} \subset \mathcal{F}_{\varphi + \varepsilon}$. 
Thus we get $\varphi_{\mathcal{F}_\varphi}^{(d)} \le \varphi \le \tilde{\varphi} \le \varphi_{\mathcal{F}_\varphi}^{(d)} + \varepsilon$, which shows the uniform convergence $\varphi_{\mathcal{F}_\varphi}^{(d)} \nearrow \varphi$. 
\end{proof}

Let $\mathcal{F}$ be a linearly bounded filtration. 
We recall $\varphi_{\mathcal{F}}^{(d)} (v) = d^{-1} \sup \{ - v (s) - \log \| s \|^\mathcal{F}_d ~|~ s \in R_d \}$, so we have 
\[ \mathcal{F}_{\varphi_{\mathcal{F}}^{(d)}} = \bigcap_{v \in \mathrm{Val} (X)} \mathcal{F}_v [d^{-1} \sup \{ - v (s) - \log \| s \|^\mathcal{F}_d ~|~ 0 \neq s \in R_d \}]. \]
We can easily check the following are equivalent conditions, using the density $\mathrm{Val} (X) \subset X^{\mathrm{NA}}$ and Dini's lemma. 
\begin{enumerate}
\item For any $\varepsilon > 0$ there exists $d \in \mathbb{N}_+$ such that $\mathcal{F} \subset \mathcal{F}_{\varphi_{\mathcal{F}}^{(d)}} [\varepsilon]$. 

\item For any $\varepsilon > 0$ there exists $d \in \mathbb{N}_+$ such that $\sigma_v (\mathcal{F}) \le \varphi_{\mathcal{F}}^{(d)} (v) + \varepsilon$ for every $v \in \mathrm{Val} (X)$. 

\item The increasing sequence $\varphi_{\mathcal{F}}^{(d_i)}$ converges uniformly for some/any eventually sufficiently divisible sequence $\{ d_i \}$. 

\item The pointwise limit $\lim_{i \to \infty} \varphi_{\mathcal{F}}^{(d_i)}$ is continuous for some/any eventually sufficiently divisible sequence $\{ d_i \}$. 
\end{enumerate}
We call $\mathcal{F}$ \textit{uniformly approximated} if one of the above equivalent conditions is satisfied. 
We denote by $\mathsf{Filt}^{\mathrm{unif}}$ the set of all uniformly approximated filtrations. 
For $\mathcal{F} \in \mathsf{Filt}^{\mathrm{unif}}$, we can assign a continuous non-archimedean psh metric $\varphi_{\mathcal{F}} = \lim_{i \to \infty} \varphi_{\mathcal{F}}^{(d_i)} \in C^0 \cap \PSH (X, L)$. 

By the above proposition, the filtration $\mathcal{F}_\varphi$ associated to $\varphi \in C^0 \cap \PSH (X, L)$ is uniformly approximated, and we have $\varphi = \varphi_{\mathcal{F}_\varphi}$. 
We note $\mathcal{F} \subsetneq \mathcal{F}_{\varphi_\mathcal{F}}$ in general. 

\begin{rem}
In \cite{BJ2}, a similar construction for general linearly bounded filtration is studied, assuming the continuity of envelopes (see section \ref{continuity of envelopes}). 
For general filtration, we have a lower semi-continuous limit $\lim_{i \to \infty} \varphi_{\mathcal{F}}^{(d_i)}$. 
Since a non-archimedean psh metric must be upper semi-continuous, we must replace this function with an upper semi-continuous envelope. 
It is then non-trivial if this envelope is a non-archimedean psh metric without the continuity of envelopes. 
\end{rem}

\begin{prop}
\label{filtration associated to Fubini--Study metric and the stabilization}
For a test configuration $(\mathcal{X}, \mathcal{L})$ and its normalization $\nu: \mathcal{X}^\nu \to \mathcal{X}$, we have 
\[ \mathcal{F}_{\varphi_{(\mathcal{X}, \mathcal{L})}} = \widehat{\mathcal{F}}_{(\mathcal{X}^\nu, \nu^* \mathcal{L})}. \]
\end{prop}

\begin{proof}
Since $\varphi_{(\mathcal{X}, \mathcal{L})} = \varphi_{(\mathcal{X}^\nu, \nu^* \mathcal{L})}$, we may assume $(\mathcal{X}, \mathcal{L})$ is normal. 
By Proposition \ref{non-archimedean metric associated to filtration}, we have $\widehat{\mathcal{F}}_{(\mathcal{X}, \mathcal{L})} = \bigcap_{E \subset \mathcal{X}_0} \mathcal{F}^\lambda_{v_E} [\varphi (v_E)]$, so that we get $\mathcal{F}_{\varphi_{(\mathcal{X}, \mathcal{L})}} \subset \widehat{\mathcal{F}}_{(\mathcal{X}, \mathcal{L})}$. 
To see the reverse inclusion, take a sufficiently divisible $d$, then by (\ref{Krull envelope and normalized base change}) and Corollary \ref{scaling of NA psh}, we get 
\[ \widehat{\mathcal{F}}_{(\mathcal{X}, \mathcal{L})} =\mathcal{F}_{(\mathcal{X}_d, \mathcal{L}_d; d^{-1})} \subset (\mathcal{F}_{\varphi_{(\mathcal{X}_d, \mathcal{L}_d)}})_{;d^{-1}} = \mathcal{F}_{\varphi_{(\mathcal{X}, \mathcal{L})}}. \]
\end{proof}

\begin{cor}
\label{normalization and non-archimedean metric}
We have $\varphi_{(\mathcal{X}, \mathcal{L})} = \varphi_{(\mathcal{X}', \mathcal{L}')}$ iff the normalizations are isomorphic. 
\end{cor}

\begin{proof}
We can recover the normalization $(\mathcal{X}^\nu, \nu^* \mathcal{L})$ from the filtration $\mathcal{F}_\varphi^{\lceil \lambda \rceil}$. 
\end{proof}

\begin{cor}
\label{DH measure via Krull envelope}
For a test configuration $(\mathcal{X}, \mathcal{L})$, we have 
\[ \nu_\infty (\mathcal{F}_{\varphi_{(\mathcal{X}, \mathcal{L})}}) = \DHm_{(\mathcal{X}, \mathcal{L})}. \]
\end{cor}

\begin{proof}
We recall $\DHm_{(\mathcal{X}, \mathcal{L})} = \DHm_{(\mathcal{X}^\nu, \nu^* \mathcal{L})}$ by \cite[Theorem 3.14]{BHJ1}. 
This implies $\DHm_{(\mathcal{X}, \mathcal{L})} = d^{-1}_* \DHm_{(\mathcal{X}_d, \mathcal{L}_d)}$ for the normalized base change $(\mathcal{X}_d, \mathcal{L}_d)$. 
Take $d$ so that $\mathcal{X}_{d, 0}$ is reduced, then we have $\mathcal{F}_{(\mathcal{X}_d, \mathcal{L}_d; d^{-1})} = \widehat{\mathcal{F}}_{(\mathcal{X}^\nu, \nu^* \mathcal{L})} = \mathcal{F}_{\varphi_{(\mathcal{X}, \mathcal{L})}}$ by (\ref{Krull envelope and normalized base change}). 
It follows that  
\[ \DHm_{(\mathcal{X}, \mathcal{L})} = d^{-1}_* \DHm_{(\mathcal{X}_d, \mathcal{L}_d)} = \nu_\infty (\mathcal{F}_{(\mathcal{X}_d, \mathcal{L}_d; d^{-1})}) = \nu_\infty (\mathcal{F}_{\varphi_{(\mathcal{X}, \mathcal{L})}}). \]
%The following is alternative proof not relying on the latter fact. 
%For a non-normal test configuration $(\mathcal{X}, \mathcal{L})$, we get 
%\[ \DHm_{(\mathcal{X}, \mathcal{L})} = \DHm_{(\mathcal{X}^\nu, \nu^* \mathcal{L})} = \nu_\infty (\mathcal{F}_{\varphi_{(\mathcal{X}^\nu, \nu^* \mathcal{L})}}) \]
%thanks to \cite[Theorem 3.14]{BHJ1}. 
%On the other hand, since $\mathcal{F}_{(\mathcal{X}, \mathcal{L})} \subset \mathcal{F}_{\varphi_{(\mathcal{X}, \mathcal{L})}} \subset \mathcal{F}_{\varphi_{(\mathcal{X}^\nu, \nu^* \mathcal{L})}}$, we have 
%\[ \int_{[\tau, \infty)} \DHm_{(\mathcal{X}, \mathcal{L})} \le \int_{[\tau, \infty)} \nu_\infty (\mathcal{F}_{\varphi_{(\mathcal{X}, \mathcal{L})}}) \le \int_{[\tau, \infty)} \nu_\infty (\mathcal{F}_{\varphi_{(\mathcal{X}^\nu, \nu^* \mathcal{L})}}). \]
%(See the proof of Lemma \ref{monotonicity}. )
%It follows that
%\[ \int_{[\tau, \infty)} \DHm_{(\mathcal{X}, \mathcal{L})} = \int_{[\tau, \infty)} \nu_\infty (\mathcal{F}_{\varphi_{(\mathcal{X}, \mathcal{L})}}), \]
%which shows the claim. 
\end{proof}

\begin{cor}
For $\varphi = \varphi_{(\mathcal{X}, \mathcal{L})} \in \nH (X, L)$, $\mathcal{F}_\varphi$ is finitely generated. 
The central fibre $\mathcal{X}_o (\varphi) := \Proj \mathcal{R}_o (\mathcal{F}_\varphi)$ is isomorphic to $\mathcal{X}_{d, 0}$ for sufficiently divisible $d$ for which $\mathcal{X}_{d, 0}$ is reduced. 
\end{cor}

\begin{proof}
This is a consequence of (\ref{Krull envelope and normalized base change}). 
\end{proof}

The following is a non-trivial fact generalizing the above corollary. 
The proof involves non-trivial results on affinoid algebra. 

\begin{thm}[Theorem 2.3 in \cite{BJ4}]
\label{finite generation of Fphi}
For $\varphi \in \nH^\mathbb{R} (X, L)$, the associated filtration $\mathcal{F}_\varphi$ is finitely generated. 
\end{thm}

\begin{proof}
By \cite[Theorem 2.3]{BJ4}, it suffices to show for a finitely generated filtration $\mathcal{F}$, $\| \cdot \|^{\varphi_\mathcal{F}}_\bullet$ is the homogenization of $\| \cdot \|^{\mathcal{F}}_\bullet$ in the sense of \cite{BJ4}. 
This is nothing but \cite[Theorem 2.11]{BJ4}. 
\end{proof}

The reducedness of the central fibre is also a general phenomenon. 

\begin{prop}
\label{Krull envelope has reduced central fibre}
For $\varphi \in C^0 (X) \cap \PSH (X, L)$, the ring $\mathcal{R}_o (\mathcal{F}_\varphi)$ is reduced. 
\end{prop}

\begin{proof}
For a linearly bounded filtration $\mathcal{F}$, we put 
\[ \widehat{\mathcal{F}} := \bigcap_{v \in \mathrm{Val} (X)} \mathcal{F}_v [\sigma_v (\mathcal{F})]. \]
By Proposition \ref{sigma and non-archimedean psh}, we have $\mathcal{F}_\varphi = \widehat{\mathcal{F}_\varphi}$, so it suffices to show $\mathcal{R}_o (\widehat{\mathcal{F}})$ is reduced. 

Take $0 \neq [f] \in \bigoplus_{m \in \mathbb{N}} \bigoplus_{\lambda \in \mathbb{R}} \varpi^{-\lambda} \widehat{\mathcal{F}}^{\lambda} R_m/\widehat{\mathcal{F}}^{\lambda+} R_m$. 
We can write it as 
\[ [f] = [\sum_{i \in I} \varpi^{-\lambda_i} s_i] \]
by a collection $\{ (\lambda_i, m_i) \in \mathbb{R} \times \mathbb{N} \}_{i \in I}$ with no overlap and $s_i \in \widehat{\mathcal{F}}^{\lambda_i} R_{m_i} \setminus \widehat{\mathcal{F}}^{\lambda_i +} R_{m_i}$. 
The condition $s_i \in \widehat{\mathcal{F}}^{\lambda_i} R_{m_i} \setminus \widehat{\mathcal{F}}^{\lambda_i +} R_{m_i}$ is equivalent to $v (s_i) + m_i \sigma_v = \lambda_i$ for all $v$. 

Now for 
\[ [f]^d = [\sum_{i_1, \ldots, i_d \in I} \varpi^{-\sum_{r=1}^d \lambda_{i_r}} \prod_{r=1}^d s_{i_r}], \]
we have $[f]^d = 0$ iff for every $(\lambda, m) \in \mathbb{R} \times \mathbb{N}$, either 
\[ J_{\lambda, m} := \{ (i_1, \ldots, i_d) \in I^d ~|~ \sum_{r=1}^d m_{i_r} = m, \sum_{r=1}^d \lambda_{i_r} = \lambda \} \] 
is an empty set or $v (\sum_{(i_1, \ldots, i_d) \in J_{\lambda, m}} \prod_{r=1}^d s_{i_r}) + m \sigma_v > \lambda$. 
Let $(\lambda_{i_0}, m_{i_0})$ be the minimum of $\{ (\lambda_i, m_i) \}_{i \in I}$ with respect to the lexicographical order. 
Then since the collection $\{ (\lambda_i, m_i) \}_{i \in I}$ has no duplication, $J_{d \lambda_{i_0}, dm_{i_0}}$ is the singleton $\{ (i_0, \ldots, i_0) \}$. 
Thus we get 
\[ v (\sum_{(i_1, \ldots, i_d) \in J_{\lambda_{i_0}, m_{i_0}}} \prod_{r=1}^d s_{i_r}) + dm_{i_0} = d v (s_{i_0}) + dm_{i_0} = d \lambda_{i_0}, \]
which proves $[f]^d \neq 0$.  
Thus $\mathcal{R}_o (\widehat{\mathcal{F}})$ is reduced. 
\end{proof}

\begin{prop}
\label{reduced central fibre implies homogeneity}
Let $\mathcal{F}$ be a finitely generated filtration. 
Then the ring $\mathcal{R}_o (\mathcal{F})$ is reduced if and only if $\mathcal{F}_{\varphi_\mathcal{F}} = \mathcal{F}$. 
\end{prop}

\begin{proof}
As we already see $\mathcal{R}_o (\mathcal{F})$ is reduced when $\mathcal{F}_{\varphi_\mathcal{F}} = \mathcal{F}$, it suffices to show the converse. 
Assume the reducedness of $\mathcal{R}_o (\mathcal{F})$. 
By Proposition \ref{polyhedral configuration and finite generation of filtration}, we can find a polyhedral configuration $(\mathcal{X}/B_\sigma, \mathcal{L})$ and $\xi \in \sigma^\circ$ such that $\mathcal{F} = \mathcal{F}_{(\mathcal{X}, \mathcal{L}; \xi)}$. 
By the assumption and Proposition \ref{central fibre via filtration}, the central fibre of $(\mathcal{X}/B_\sigma, \mathcal{L})$ is reduced. 

For $\eta \in \sigma \cap N$, $\mathcal{F}_{(\mathcal{X}, \mathcal{L}; \eta)}$ is the filtration associated to a test configuration with reduced central fibre, so we have $\mathcal{F}_{(\mathcal{X}, \mathcal{L}; \eta)} = \widehat{\mathcal{F}}_{(\mathcal{X}, \mathcal{L}; \eta)} = \mathcal{F}_{\varphi_{(\mathcal{X}, \mathcal{L}; \eta)}}$ by (\ref{filtration for tc with reduced central fibre}) and Proposition \ref{filtration associated to Fubini--Study metric and the stabilization}. 
By scaling, we get $\mathcal{F}_{(\mathcal{X}, \mathcal{L}; \eta)} = \mathcal{F}_{\varphi_{(\mathcal{X}, \mathcal{L}; \eta)}}$ also for $\eta \in \sigma \cap N_{\mathbb{Q}}$. 
To show $\mathcal{F}_{(\mathcal{X}, \mathcal{L}; \eta)} = \mathcal{F}_{\varphi_{(\mathcal{X}, \mathcal{L}; \eta)}}$ for general $\eta \in \sigma$, it suffices to show that both $\| s \|_m^{\mathcal{F}_{(\mathcal{X}, \mathcal{L}; \eta)}}$ and $\| s \|_m^{\varphi_{(\mathcal{X}, \mathcal{L}; \eta)}}$ are continuous on $\eta \in \sigma$ for each $s \in R_m$. 

Take a basis $\{ e_i \}$ of $R_m$ and $\mu_i \in M$ as in Lemma \ref{equivariant trivialization} for $\pi_* \mathcal{L}^{\otimes m}$, then we have 
\[ - \log \| s \|_m^{\mathcal{F}_{(\mathcal{X}, \mathcal{L}; \eta)}} = \min \{ \langle \mu_i, \eta \rangle ~|~ s = \sum_i a_i e_i, a_i \neq 0 \}, \]
so that it is continuous on $\eta \in \sigma$. 
This also shows that for $\varphi_\eta := \varphi_{(\mathcal{X}, \mathcal{L}; \eta)}$, $\varphi_{\eta_i}$ converges uniformly to $\varphi_\eta$ if $\eta_i \to \eta \in \sigma$. 
Therefore, 
\[ -\log \| s \|^{\varphi_{(\mathcal{X}, \mathcal{L}; \eta)}}_m = \inf \{ m \varphi_{(\mathcal{X}, \mathcal{L}; \eta)} (v) + v (s) ~|~ v \in X^{\mathrm{NA}} \} \] 
is continuous on $\eta \in \sigma$. 
\end{proof}

\subsubsection{Continuity of envelopes}
\label{continuity of envelopes}

Here we recall important consequences of the continuity of envelopes. 
The continuity of envelopes is not yet proved for general polarized variety $(X, L)$, so we would avoid using the hypothesis as possible. 
However, it is crucial for the completeness of $\E^{\exp} (X, L)$. 

\begin{defin}[Continuity of envelopes \cite{BJ1, BJ3}]
We say \textit{the continuity of envelopes holds} for $(X, L)$ if 
\[ P (f) := \sup \{ \varphi \in \PSH (X, L) ~|~ \varphi \le f \} \]
is continuous for every continuous function $f \in C^0 (X^{\mathrm{NA}})$. 
\end{defin}

The smooth case is confirmed. 

\begin{thm}[Theorem 4.52 in \cite{BJ3} (cf. \cite{BJ1})]
The continuity of envelopes holds for smooth $(X, L)$. 
\end{thm}

The following are important consequences of the continuity of envelopes. 

\begin{thm}[Corollary 4.58 in \cite{BJ3}]
Assume the continuity of envelopes for $(X, L)$. 
Then the subspace 
\[ \{ \varphi \in \PSH (X, L) ~|~ \sup \varphi = 0 \} \] 
is compact with respect to the weak topology. 
\end{thm}

\begin{thm}[Theorem 9.8 in \cite{BJ3} (cf. \cite{BJ1})]
\label{NA CY theorem}
Assume the continuity of envelopes for $(X, L)$. 
Then the Monge--Amp\`ere measure $\mathrm{MA}: \E^1 (X, L)/\mathbb{R} \to \mathcal{M}_{\mathrm{NA}}^1$ gives a homeomorphism. 
\end{thm}

\begin{thm}[Proposition 7.4 in \cite{BJ4}]
\label{NApsh associated to valuation of linear growth}
Assume the continuity of envelopes for $(X, L)$. 
Then the filtration $\mathcal{F}_v$ associated to a valuation of linear growth $v \in X^{\mathrm{lin}}$ is uniformly approximated in the sense of the previous section. 
As a consequence, $\varphi_v := \varphi_{\mathcal{F}_v}$ gives a continuous non-archimedean psh metric. 
Moreover, we have $\mathrm{MA} (\varphi_v) = (e^L). \delta_v$ and $\varphi_v (v) = 0$, which characterizes $\varphi_v$. 
\end{thm}

\begin{thm}[Theorem B in \cite{BJ4}]
Assume the continuity of envelopes for $(X, L)$. 
Then the metric space $(\E^1 (X, L), d_1)$ is complete. 
\end{thm}

\subsubsection{Existence of rooftops}
\label{existence of rooftop}

\begin{defin}
For $\varphi, \varphi' \in \PSH (X, L)$, a non-archimedean psh metric $\hat{\varphi} \in \PSH (X, L)$ is called the \textit{rooftop} of $\varphi, \varphi'$ if 
\begin{itemize}
\item $\hat{\varphi} \le \min \{ \varphi, \varphi' \}$ and 

\item for any $\varphi'' \in \PSH (X, L)$ with $\varphi'' \le \min \{ \varphi, \varphi' \}$, we have $\varphi'' \le \hat{\varphi}$. 
\end{itemize}
\end{defin}

If the rooftop exists, it is unique. 
We denote it by $\varphi \wedge \varphi'$. 
We call attention that our convention on $\wedge$ is different from \cite{BJ2, BJ3}, where they just put $\varphi \wedge \varphi' := \min \{ \varphi, \varphi' \}$: our $\varphi \wedge \varphi'$ corresponds to the envelope $P (\varphi \wedge \varphi')$ in their convention. 

\begin{lem}
If $\varphi, \varphi' \in C^0 \cap \PSH (X, L)$ and $\tilde{\varphi} \in \PSH (X, L)$ satisfies $\tilde{\varphi} \le \min \{ \varphi, \varphi' \}$ and $\varphi'' \le \tilde{\varphi}$ for every $\varphi'' \in \nH (X, L)$ with $\varphi'' \le \min \{ \varphi, \varphi' \}$, then $\tilde{\varphi}$ is the rooftop $\varphi \wedge \varphi'$. 
\end{lem}

\begin{proof}
For $\varphi'' \in \PSH (X, L)$ with $\varphi'' \le \min \{ \varphi, \varphi' \}$, take a decreasing net $\{ \varphi_i'' \}_{i \in I} \subset \nH (X, L)$ converging to $\varphi''$. 
Since $\varphi, \varphi'$ are continuous, for any $\varepsilon > 0$ we can take $i_\varepsilon$ so that $\varphi_i'' \le \min \{\varphi, \varphi' \} + \varepsilon$ for $i \ge i_\varepsilon$ by Lemma \ref{Dini}. 
Then by our assumption, we get $\varphi_i'' -\varepsilon \le \tilde{\varphi}$, so that $\varphi'' \le \tilde{\varphi} + \varepsilon$ for any rational $\varepsilon > 0$. 
Thus we obtain $\varphi'' \le \tilde{\varphi}$, which shows that $\tilde{\varphi}$ is indeed the rooftop. 
\end{proof}

\begin{prop}
\label{test configuration associated to phi wedge tau}
For $\varphi \in \nH (X, L)$ and $\tau \in \mathbb{Q}$, the rooftop $\varphi \wedge \tau$ exists in $\nH (X, L)$ without assuming the continuity of envelopes. 
\end{prop}

\begin{proof}
The path $\{ \varphi_{; \rho} \}_{[0,1]}$ connecting the trivial metric $0$ and $\varphi$ is a Fubini--Study segment in the sense of \cite[Definition 4.1.1]{Remi}: we can write $\phi_{;\rho} = \varphi_{;\rho} + \phi_{\mathrm{triv}} = m^{-1} \max \{ \log |s_i| + \rho \lambda_i \}$. 
Then by the proof of \cite[Lemma 5.2.1]{Remi}, 
\[ \hat{\varphi}^\tau := \inf_{\rho \in [0,1]} \{ \varphi_{;\rho} + (1-\rho) \tau \} \] 
defines a Fubini--Study non-archimedean psh metric: $\hat{\varphi}^\tau \in \nH (X, L)$ for $\tau \in \mathbb{Q}$. 
Here we strengthen the proof consists of a combinatorial argument and does not rely on the continuity of envelopes. 
We show that this $\hat{\varphi}^\tau$ is the rooftop of $\varphi, \tau$. 
Since $\varphi_{;0} + (1-0) \tau = \tau$ and $\varphi_{;1} + (1-1) \tau = \varphi$, we have $\hat{\varphi}^\tau \le \min \{ \varphi, \tau \}$. 

By the above lemma, it suffices to show $\varphi' \le \hat{\varphi}^\tau$ for $\varphi' \in \nH (X, L)$ satisfying $\varphi' \le \min \{ \varphi, \tau \}$. 
We note $\varphi' \le \hat{\varphi}^\tau$ iff $\| \cdot \|^{\hat{\varphi}^\tau}_m \le \| \cdot \|^{\varphi'}_m$. 
Since
\begin{align*}
\| s \|^{\hat{\varphi}^\tau}_m 
&= \sup_{v \in X^{\mathrm{val}}} e^{-v (s) - m \hat{\varphi}^\tau (v)}
= \sup_{v \in X^{\mathrm{val}}} \sup_{\rho \in [0,1]} e^{-v (s) - m \rho \varphi (\rho^{-1} v) - m (1-\rho) \tau}, 
\end{align*} 
it suffices to bound $e^{-v (s) - m \rho \varphi (\rho^{-1} v) - m (1-\rho) \tau}$. 
Putting $v' := \rho^{-1} v$, we compute 
\begin{align*}
e^{-v (s) - m \rho \varphi (\rho^{-1} v) - m (1-\rho) \tau}
&= (e^{-v' (s) - m \varphi (v')})^\rho (e^{- m \tau})^{1-\rho}
\le (\| s \|^\varphi_m)^\rho (\| s \|^\tau_m)^{1-\rho} \le \| s \|^{\varphi'}_m 
\end{align*}
as desired, where we used $\max \{ \| \cdot \|^\varphi_m, \| \cdot \|^\tau_m \} \le \| \cdot \|^{\varphi'}_m$. 
\end{proof}

By the definition of the rooftop, we have $\varphi_2 \wedge \varphi'_2 \le \varphi_1 \wedge \varphi'_1$ for $\varphi_2 \le \varphi_1, \varphi'_2 \le \varphi'_1$ if the rooftops exist. 
In particular, for decreasing nets $\{ \varphi_i \}_{i \in I}, \{ \varphi'_j \}_{j \in J}$, $\{ \varphi_i \wedge \varphi'_j \}_{(i, j) \in I \times J}$ gives a decreasing net if the rooftops exist. 
Here $(i, j) \le (k, l)$ for $(i, j), (k, l) \in I \times J$ iff $i \le k$ and $j \le l$. 

\begin{prop}
\label{decreasing limit of rooftop}
Let $\{ \varphi_i \}_{i \in I}, \{ \varphi_j' \}_{j \in J} \subset \PSH (X, L)$ be decreasing nets converging to $\varphi, \varphi' \in \PSH (X, L)$, respectively. 
Suppose the rooftop $\varphi_i \wedge \varphi_j'$ exists for each $(i, j) \in I \times J$. 
Then the rooftop $\varphi \wedge \varphi'$ exists if and only if the pointwise limit of $\varphi_i \wedge \varphi'_j$ exists in $\PSH (X, L)$ ($\Leftrightarrow$ the pointwise limit is not identically $-\infty$). 
\end{prop} 

\begin{proof}
Suppose $\varphi \wedge \varphi'$ exists. 
As $\varphi \wedge \varphi' \le \varphi_i \wedge \varphi'_j$, we have $\lim_{(i, j) \to \infty} (\varphi_i \wedge \varphi'_j) (x) \ge (\varphi \wedge \varphi') (x)$ for every $x \in X^{\mathrm{NA}}$. 
Since $\varphi \wedge \varphi'$ is not identically $-\infty$, the limit is also not identically $-\infty$, so it exists in $\PSH (X, L)$. 

Conversely, suppose the limit $\lim_{(i, j) \to \infty} \varphi_i \wedge \varphi'_j$ exists in $\PSH (X, L)$. 
We show the limit satisfies the axiom of the rooftop $\varphi, \varphi'$. 
Since $\varphi_i \wedge \varphi'_j \le \min \{ \varphi_i, \varphi'_j \}$, we have $\lim_{(i, j) \to \infty} \varphi_i \wedge \varphi'_j \le \min \{ \varphi_i, \varphi'_j \}$. 
Thus we get $\lim_{(i, j) \to \infty} \varphi_i \wedge \varphi'_j \le \min \{ \varphi, \varphi' \}$. 
Take $\varphi'' \in \PSH (X, L)$ so that $\varphi'' \le \min \{ \varphi, \varphi' \}$. 
Then since $\varphi'' \le \min \{ \varphi_i, \varphi'_j \}$, we have $\varphi'' \le \varphi_i \wedge \varphi'_j$. 
Thus we get $\varphi'' \le \lim_{(i, j) \to \infty} \varphi_i \wedge \varphi'_j$, which shows the claim. 
\end{proof}

\begin{cor}
For $\varphi \in \PSH (X, L)$ and $\tau \in \mathbb{R}$, $\varphi \wedge \tau$ exists in $\PSH (X, L)$ without assuming the continuity of envelopes. 
If $\varphi \in \E^1 (X, L)$, then $\varphi \wedge \tau \in \E^1 (X, L)$. 
\end{cor}

\begin{proof}
Since $\varphi + \min \{ 0, \tau - \sup \varphi \} \le \varphi, \tau$, we have $\varphi + \min \{ 0, \tau - \sup \varphi \} \le \varphi \wedge \tau$ for $\varphi \in \nH (X, L)$ and $\tau \in \mathbb{Q}$. 
For $\varphi \in \PSH (X, L)$ and $\tau \in \mathbb{R}$, take convergent decreasing nets $\varphi_i \searrow \varphi, \tau_i \searrow \tau$ so that $\varphi_i \in \nH (X, L)$ and $\tau_i \in \mathbb{Q}$. 
Since 
\[ \varphi + \min \{ 0, \tau - \sup \varphi_i \} \le \varphi_i + \min \{ 0, \tau - \sup \varphi_i \} \le \varphi_i \wedge \tau_i, \]
we have 
\[ \varphi + \min \{ 0, \tau - \sup \varphi \} \le \lim_{i \to \infty} \varphi_i \wedge \tau_i, \] 
so that the limit exists, hence $\varphi \wedge \tau$ exists. 

If $\varphi \in \E^1 (X, L)$, then by the above estimate, we have 
\[ - \infty < E (\varphi) + \min \{ 0, \inf_i (\tau - \sup \varphi_i) \} \le E (\varphi \wedge \tau), \]
hence $\varphi \wedge \tau \in \E^1 (X, L)$. 
\end{proof}

\begin{prop}
\label{filtration of envelope}
Suppose $\varphi, \varphi' \in C^0 \cap \PSH (X, L)$ and the rooftop $\varphi \wedge \varphi'$ exists in $C^0 \cap \PSH (X, L)$, then we have 
\[ \mathcal{F}_{\varphi \wedge \varphi'} = \mathcal{F}_\varphi \cap \mathcal{F}_{\varphi'}. \]
\end{prop}

\begin{proof}
Since $\varphi \wedge \varphi' \le \varphi, \varphi'$, we have $\mathcal{F}_{\varphi \wedge \varphi'} \subset \mathcal{F}_\varphi \cap \mathcal{F}_{\varphi'}$. 
On the other hand, consider the non-archimedean metric $\varphi_m'' \in \nH^\mathbb{R} (X, L)$ associated to the filtration $\mathcal{F}_{(m)} := \mathcal{F}_{\| \cdot \|_m^\varphi \vee \| \cdot \|_m^{\varphi'}}$ generated by the norm $\| \cdot \|_m^\varphi \vee \| \cdot \|_m^{\varphi'} = \max \{ \| \cdot \|_m^\varphi, \| \cdot \|_m^{\varphi'} \}$. 
Since $\mathcal{F}_{(m)} \subset \mathcal{F}_\varphi, \mathcal{F}_{\varphi'}$, we have $\varphi_m'' \le \varphi, \varphi'$. 
Then by the property of the rooftop, we get $\varphi_m'' \le \varphi \wedge \varphi'$. 
It follows that 
\[ \mathcal{F}_{(m)} \subset \mathcal{F}_{\varphi_m''} \subset \mathcal{F}_{\varphi \wedge \varphi'}. \]
Since $\mathcal{F}_{(m)}^\lambda R_m = (\mathcal{F}^\lambda_\varphi \cap \mathcal{F}^\lambda_{\varphi'}) R_m$, we get $\mathcal{F}^\lambda_\varphi \cap \mathcal{F}^\lambda_{\varphi'} \subset \mathcal{F}_{\varphi \wedge \varphi'}$ for each $m$. 
\end{proof}

\begin{prop}
Assume the continuity of envelopes. 
For $\varphi, \varphi' \in C^0 \cap \PSH (X, L)$, $\varphi \wedge \varphi'$ exists in $C^0 \cap \PSH (X, L)$. 
\end{prop}

\begin{proof}
Thanks to the continuity of envelopes, $P (\min \{ \varphi, \varphi' \})$ is in $C^0 \cap \PSH (X, L)$. 
The envelope clearly enjoys the property of rooftop. 
\end{proof}

\begin{quest}
Does $\varphi \wedge \varphi'$ exist in $\nH (X, L)$ for $\varphi, \varphi' \in \nH (X, L)$? 
Can we show this without assuming the continuity of envelopes? 
\end{quest}

\subsection{Moment energy and Duistermaat--Heckman measure}
\label{Moment energy and Duistermaat--Heckman measure}

\subsubsection{Well ordered diagonal basis}

For $\varphi \in C^0 \cap \PSH (X, L)$ and a basis $\bm{s} = (s_1, \ldots, s_N)$ of $R_m = H^0 (X, L^{\otimes m})$, we put 
\begin{equation} 
\lambda_i^\varphi (\bm{s}) := - \log \| s_i \|^\varphi_m = \sup \{ \lambda \in \mathbb{R} ~|~ s_i \in \mathcal{F}_\varphi^\lambda R_m \}. 
\end{equation}
For $\varphi \le \varphi'$, we have $\lambda_i^\varphi (\bm{s}) \le \lambda_i^{\varphi'} (\bm{s})$ as $\mathcal{F}_\varphi \subset \mathcal{F}_{\varphi'}$. 

We recall a basis $\bm{s}$ is called \textit{diagonal} with respect to $\varphi$ if 
\[ \| \sum_i a_i s_i \|^\varphi_m = \max_{a_i \neq 0} \| s_i \|^\varphi_m, \] 
which is equivalent to $\sup \{ \lambda \in \mathbb{R} ~|~ \sum_i a_i s_i \in \mathcal{F}_\varphi^\lambda R_m \} = \min_{a_i \neq 0} \lambda_i^\varphi (\bm{s})$. 

It is known by \cite[Proposition 1.14]{BE} that there always exists a basis which is diagonal with respect to both $\varphi, \varphi'$ (codiagonal for $\varphi, \varphi'$). 
We note, on the other hand, there are no basis which is diagonal with respect to three metrics $\varphi, \varphi', \varphi''$ in general as we can see in the following example. 

\begin{eg}
There is no basis diagonal with respect to all of the following norms on $\mathbb{C}^2$: 
\[ \| (a, b) \|_1 := \begin{cases} 0 & a= b= 0 \\ 1 & a \neq 0, b= 0 \\ 2 & b \neq 0 \end{cases} \qquad \| (a, b) \|_2 := \begin{cases} 0 & a= b= 0 \\ 1 & a = 0, b \neq 0 \\ 2 & a \neq 0 \end{cases} \]
\[ \| (a, b) \|_3 := \begin{cases} 0 & a= b= 0 \\ 1 & a = b \neq 0 \\ 2 & a \neq b \end{cases} \]
\end{eg}

When $\varphi'' = \varphi \wedge \varphi'$ exists, any basis diagonal with respect to both $\varphi, \varphi'$ is diagonal also with respect to $\varphi''$: for a basis $(s_i)$ codiagonal for $\varphi, \varphi'$, we compute 
\begin{align*} 
\| \sum_i a_i s_i \|^{\varphi''} 
&= \max \{ \| \sum_i a_i s_i \|^\varphi, \| \sum_i a_i s_i \|^{\varphi'} \} = \max \{ \max_{a_i \neq 0} \| s_i \|^\varphi, \max_{a_i \neq 0} \| s_i \|^{\varphi'} \} 
\\
&= \max_{a_i \neq 0} \max \{ \| s_i \|^\varphi, \| s_i \|^{\varphi'} \} = \max_{a_i \neq 0} \| s_i \|^{\varphi''} 
\end{align*} 
In other cases, we use the following lemma in our estimate. 

We prepare some terminologies. 
Firstly, we call a basis $\bm{s}$ \textit{well ordered} with respect to $\varphi$ if $\| s_{i+1} \|^\varphi \le \| s_i \|^\varphi $, i.e. $\lambda_i^\varphi (\bm{s}) \le \lambda_{i+1}^\varphi (\bm{s})$ for every $i = 1, \ldots, N-1$. 

We define the relative version as follows. 
Let $\bm{s}$ be a basis well ordered with respect to $\varphi'$. 
We define $0= l_0 < l_1 < \dotsb < l_p \le N$ by 
\begin{align*} 
\| s_N \|^{\varphi'} = \dotsb = \| s_{N- l_1 +1} \|^{\varphi'} 
&< \| s_{N- l_1} \|^{\varphi'} = \dotsb = \| s_{N- l_2+1} \|^{\varphi'} 
\\
&< \dotsb < \| s_{N - l_p +1} \|^{\varphi'} = \dotsb = \| s_1 \|^{\varphi'}. 
\end{align*} 
We put $W_q := \langle s_N, \ldots, s_{N - l_q +1} \rangle$. 
For another $\varphi$, consider the quotient norm 
\[ \| [s] \|_{W_q}^\varphi := \inf \{ \| s+ t \|^\varphi ~|~ t \in W_{q-1} \} \]
on $W_q/W_{q-1}$. 
Then we call $\bm{s}$ \textit{well ordered with respect to $(\varphi, \varphi')$} if $\| [s_{i+1}] \|_{W_q}^\varphi \le \| [s_i ] \|_{W_q}^\varphi$ for each $q$ and $i$ with $N - l_{q-1} + 1 \le i \le N -l_q$. 
 
For any basis $\bm{s}$ and any $\varphi, \varphi'$, we can find a permutation $\sigma$ so that $\bm{s}_\sigma = (s_{\sigma (1)}, \ldots, s_{\sigma (N)})$ is a well ordered basis with respect to $(\varphi, \varphi')$. 

\begin{lem}
\label{diagonal lemma}
Let $\bm{s}$ be a basis of $R_m$ diagonal and well ordered with respect to $\varphi$ and $\bm{s}'$ be a well ordered basis with respect to $\varphi$. 
Then we have $\lambda^\varphi_i (\bm{s}) \ge \lambda^\varphi_i (\bm{s}')$ for $i= 1, \ldots, N$. 

More generally, we have the following. 
Let $\bm{s}$ be a basis of $H^0 (X, L^{\otimes m})$ which is codiagonal for $\varphi, \varphi'$ and is well ordered with respect to $(\varphi, \varphi')$. 
Then for a basis $\bm{s}'$ which is diagonal for $\varphi'$ and is well ordered with respect to $(\varphi, \varphi')$, we have $\lambda^\varphi_i (\bm{s}) - \lambda^{\varphi'}_i (\bm{s}) \ge \lambda^\varphi_i (\bm{s}') - \lambda^{\varphi'}_i (\bm{s}')$ for $i= 1, \ldots, N$. 
\end{lem}

\begin{proof}
Let $\bm{s}$ be a basis diagonal and well ordered with respect to $\varphi$ and $\bm{s}'$ be a well ordered basis with respect to $\varphi$. 
We firstly show $\| s_1 \| = \| s_1' \|$. 
Write $s_i' = \sum_{j=0}^N a_{ij} s_j$, $s_j = \sum_{i=0}^N b_{ij} s_i'$. 
Since $\bm{s}$ is diagonal, we have $\| s_i' \| = \max_{a_{ij} \neq 0} \| s_j \|$, which in particular shows $\| s_i' \| \le \| s_1 \|$ as $\bm{s}$ is well ordered. 
Similarly, we have $\| s_j \| \le \max_{b_{ij} \neq 0} \| s_i' \|$, so that we have $\| s_j \| \le \| s_1' \|$ as $\bm{s}'$ is well ordered. 
It follows that $\| s_1 \| \le \| s_1' \| \le \| s_1 \|$, so we get $\| s_1 \| = \| s_1' \|$. 

Now assume we obtained $\| s_i \| \le \| s_i' \|$ for $1 \le i \le k-1$. 
We want to see $\| s_k \| \le \| s_k' \|$. 
If $\| s_{k-1} \| \le \| s_k' \|$, we have $\| s_k \| \le \| s_{k-1} \| \le \| s_k' \|$. 
If $\| s_{k-1} \| > \| s_k' \|$, then we have $\| s_j \| > \| s_i' \|$ for $1 \le j \le k-1$ and $k \le i \le N$, so that $a_{ij} = 0$ for $1 \le j \le k-1$ and $k \le i \le N$. 
Thus we can write $s_i' = \sum_{i=k}^N a_{ij} s_j$ for $k \le i \le N$. 
Now consider the quotient norm 
\[ \| [s] \|_/ := \inf \{ \| s + t \| ~|~ t \in \langle s_0, \ldots, s_{k-1} \rangle \} \] 
on $H^0 (X, L^{\otimes m}) /\langle s_1, \ldots, s_{k-1} \rangle$. 
For $k \le i \le N$, we have $\| [s_i] \|_/ = \| s_i \|$ and $\| [s_i'] \|_/ = \| s_i' \|$ thanks to the expression $s_i' = \sum_{i=k}^N a_{ij} s_j$, so that $\{ [s_i] \}_{i=k}^N$ and $\{ [s_i'] \}_{i=k}^N$ give well ordered basis of $H^0 (X, L^{\otimes m}) /\langle s_1, \ldots, s_{k-1} \rangle$. 
Since $\| \sum_{i=k}^N a_i [s_i] \|_/ = \| \sum_{i=k}^N a_i s_i \| = \max_{a_i \neq 0} \| s_i \| = \max_{a_i \neq 0} \| [s_i] \|_/$, $\{ [s_i] \}_{i=k}^N$ is diagonal. 
Then by the above argument, we know $\| [s_k] \|_/ = \| [s_k'] \|_/$, so that we get $\| s_k \| = \| s_k' \|$. 
Thus we obtain the first claim by induction on $k$. 

Next, we show the second claim. 
Let $\bm{s}$ be a basis codiagonal for $\varphi, \varphi'$ and $\bm{s}'$ be a basis diagonal for $\varphi'$ which are well ordered with respect to $(\varphi, \varphi')$. 
In particular, these are diagonal and well ordered with respect to $\varphi'$, so that we have $\| s_i \|^{\varphi'} = \| s_i' \|^{\varphi'}$ from just what we proved. 
Take $0 = l_0 < l_1 < l_2 < \dotsb < l_p < N$ so that 
\begin{align*} 
\| s_N \|^{\varphi'} = \dotsb = \| s_{N-l_1+1} \|^{\varphi'} 
&< \| s_{N-l_1} \|^{\varphi'} = \dotsb = \| s_{N-l_2+1} \|^{\varphi'} 
\\
&< \dotsb < \| s_{N-l_p+1} \|^{\varphi'} = \dotsb = \| s_1 \|^{\varphi'}. 
\end{align*}
As $\bm{s}$ is diagonal, we have $\| s_i \|^{\varphi'} = \| s_i' \|^{\varphi'} = \max_{a_{ij} \neq 0} \| s_j \|^{\varphi'}$. 
Then we must have $a_{ij} = 0$ for $(i, j)$ with $N-l_q+1 \le i \le N$ and $1 \le j \le N-l_q$ by our choice of $l_q$, so that we can write 
\[ s_i' = \sum_{j=N-l_q+1}^N a_{ij} s_j \]
for each $q = 1, \ldots, p$ and $i$ with $N-l_q +1 \le i \le N - l_{q-1}$. 
Thus we have $\langle s_N, \ldots, s_{N-l_q+1} \rangle = \langle s_N', \ldots, s_{N-l_q+1}' \rangle =: W_q$. 

For the quotient norm $\| \cdot \|^\varphi_{W_q}$ on $W_q/W_{q-1}$, the assumption that $\bm{s}$ is diagonal with respect to $\varphi$ implies that $\| [s_i] \|^\varphi_{W_q} = \| s_i \|^\varphi$ for $N - l_q +1 \le i \le N-l_{q-1}$ and that the basis $\{ [s_{N-l_{q-1}}], \ldots, [s_{N-l_q +1}] \}$ is diagonal with respect to $\| \cdot \|^\varphi_{W_q}$. 
Since $\bm{s}$ and $\bm{s}'$ are well ordered with respect to $(\varphi, \varphi')$, the bases $\{ [s_{N- l_{q-1}}], \ldots, [s_{N-l_q+1}] \}, \{ [s_{N- l_{q-1}}'], \ldots, [s_{N-l_q+1}'] \}$ of $W_q/W_{q-1}$ are well ordered with respect to $\| \cdot \|^\varphi_{W_q}$. 
Applying the absolute case we proved, we get $\| [s_i] \|^\varphi_{W_q} \le \| [s_i'] \|^\varphi_{W_q}$ for each $q$ and $i$ with $N-l_q+1 \le i \le N - l_{q-1}$. 
It follows that for every $i$ we have $\| s_i \|^\varphi = \| s_i \|^\varphi_{W_q} \le \| [s_i'] \|^\varphi_{W_q} \le \| s_i' \|^\varphi$ by choosing suitable $q$. 
Thus for every $i$ we get 
\[ \| s_i \|^\varphi /\| s_i \|^{\varphi'} \le \| s_i' \|^\varphi /\| s_i' \|^{\varphi'}, \]
which shows the claim. 
\end{proof}

\subsubsection{Moment energy}
\label{subsection: Moment energy}

\begin{lem}
\label{monotonicity}
Let $\chi$ be an increasing function on $\mathbb{R}$. 
Suppose $\varphi' \le \varphi$ for $\varphi, \varphi' \in C^0 \cap \PSH (X, L)$, then $\int_\mathbb{R} \chi \nu_\infty (\mathcal{F}_{\varphi'}) \le \int_\mathbb{R} \chi \nu_\infty (\mathcal{F}_\varphi)$. 
\end{lem}

\begin{proof}
Note that increasing function is Borel measurable and is bounded on the support of the finite measure $\nu_\infty (\mathcal{F}_\varphi)$, hence is integrable with respect to $\nu_\infty (\mathcal{F}_\varphi)$. 
Recall that $\nu_\infty (\mathcal{F}_\varphi) = \lim_{m \to \infty} \nu_m (\mathcal{F}_\varphi)$ for 
\[ \nu_m (\mathcal{F}_\varphi) := m^{-n} \sum_{\lambda \in \mathbb{R}} \dim \mathcal{F}_\varphi^\lambda R_m/ \mathcal{F}_\varphi^{\lambda +} R_m. \delta_{\lambda/m} = m^{-n} \sum_{i=1}^{N_m} \delta_{m^{-1} \lambda_i^\varphi (\bm{s})}, \]
where $\bm{s} = (s_i)_{i=1}^{N_m}$ is a basis of $R_m$ diagonal with respect to $\| \cdot \|^\varphi$ and $\lambda_i^\varphi (\bm{s}) = - \log \| s_i \|^\varphi_m$. 

Take a basis $\bm{s}$ of $R_m$ so that it is codiagonal for $\varphi, \varphi'$. 
By the assumption $\varphi' \le \varphi$, we have $\lambda_i^{\varphi'} (\bm{s}) \le \lambda_i^\varphi (\bm{s})$. 
It follows that for increasing $\chi$, we have 
\begin{align*} 
\int_{\mathbb{R}} \chi \nu_m (\mathcal{F}_{\varphi'}) 
&= m^{-n} \sum_{i=1}^{N_m} \chi (m^{-1} \lambda_i^{\varphi'} (\bm{s})) 
\\
&\le m^{-n} \sum_{i=1}^{N_m} \chi (m^{-1} \lambda_i^\varphi (\bm{s})) = \int_{\mathbb{R}} \chi \nu_m (\mathcal{F}_\varphi). 
\end{align*}

Now suppose $\chi$ is continuous (not necessarily compactly supported). 
Since the supports of $\nu_m, \nu_\infty$ are uniformly bounded, we get 
\[ \lim_{m \to \infty} \int_{\mathbb{R}} \chi \nu_m (\mathcal{F}_\varphi) = \int_{\mathbb{R}} \chi \nu_\infty (\mathcal{F}_\varphi). \]
Then the above inequality on $\nu_m$ shows the claim for continuous increasing $\chi$. 

For $\chi = 1_{[\tau, \infty)}$ or $1_{(\tau, \infty)}$, we can easily find a bounded sequence $\chi_j$ of continuous increasing functions which pointwisely converges to $\chi$. 
By the bounded convergence theorem, we have 
\[ \lim_{j \to \infty} \int_\mathbb{R} \chi_j \nu_\infty = \int_\mathbb{R} \chi \nu_\infty \]
for both $\nu_\infty = \nu_\infty (\mathcal{F}_\varphi), \nu_\infty (\mathcal{F}_{\varphi'})$, so that the claim for $\chi = 1_{[\tau, \infty)}$ or $1_{(\tau, \infty)}$ follows by that for continuous functions. 

Now let $\chi$ be a general increasing function. 
We pick $\lambda_0 \in \mathbb{R}$ so that $\lambda_0 < \inf \supp \nu_\infty (\mathcal{F}_\varphi), \inf \supp \nu_\infty (\mathcal{F}_{\varphi'})$. 
We define $\chi_j$ by 
\[ \chi_j := \chi(\lambda_0) + \sum_{i=1}^{j 2^j} 2^{-j} 1_{\chi^{-1} ([\chi (\lambda_0) +i/2^j, \infty))}. \]
Then $\chi_j$ is an increasing sequence of increasing functions which converges to $\chi$ pointwisely on $[\lambda_0, \infty)$. 
We can check this as follows. 
By the monotonicity, we have 
\[ \bigcup_{j=1}^\infty \chi^{-1} ([\chi (\lambda_0), \chi (\lambda_0)+j)) \supset [\lambda_0, \infty), \]
for each $t \in [\lambda_0, \infty)$, we can take $j$ sufficiently large so that $t \in \chi^{-1} ([\chi (\lambda_0), \chi (\lambda_0)+j))$. 
Put $i_{j, t} := \max \{ i = 1, \ldots, j 2^j  ~|~ \chi (\lambda_0) +i/2^j \le \chi (t) \}$. 
Then we have $\chi (\lambda_0) +i_{j, t}/2^j \le \chi (t) < \chi (\lambda_0) +(i_{j, t}+1)/2^j$ and $\chi_j (t) = \chi (\lambda_0) + i_{j, t}/2^j$. 
It follows that $|\chi(t) - \chi_j (t)| \le 2^{-j}$, hence $\chi_j$ converges to $\chi$ pointwisely on $[\lambda_0, \infty)$. 
We can also see that the sequence is increasing by the formula $\chi_j (t) = \chi (\lambda_0) + i_{j, t}/2^j$. 
Since $\chi_j \nearrow \chi$ is bounded from below by $\chi (\lambda_0)$, we get 
\[ \lim_{j \to \infty} \int_\mathbb{R} \chi_j \nu_\infty (\mathcal{F}_\varphi) = \int_\mathbb{R} \chi \nu_\infty (\mathcal{F}_\varphi) \]
by the monotone convergence theorem. 
Thus it suffices to show the claim for $\chi_j$. 
By the monotonicity of $\chi$, we have $\chi^{-1} ([\sigma, \infty)) = [\tau, \infty)$ or $(\tau, \infty)$. 
Since $\chi_j$ is a linear combination of such functions with positive coefficients, the claim follows by that for $\chi = 1_{[\tau, \infty)}, 1_{(\tau, \infty)}$, which we already know. 
\end{proof}

Recall we have $\DHm_{(\mathcal{X}, \mathcal{L})} = \nu_\infty (\mathcal{F}_{\varphi_{(\mathcal{X}, \mathcal{L})}})$ by Corollary \ref{DH measure via Krull envelope}. 
For an increasing right continuous function $\chi$ on $\mathbb{R}$ and a NA psh metric $\varphi$ on $(X, L)$, we put 
\begin{equation} 
\label{moment energy}
E_\chi (\varphi) := \inf \Big{\{} \int_\mathbb{R} \chi \DHm_{(\mathcal{X}, \mathcal{L})} ~\Big{|}~ \varphi \le \varphi_{(\mathcal{X}, \mathcal{L})} \in \nH (X, L) \Big{\}}. 
\end{equation}
It may take value $-\infty$ when $\lim_{t \to -\infty} \chi (t) = -\infty$. 
In this article, we are especially interested in $E (\varphi) := E_t (\varphi)$ and $E_{\exp} (\varphi) := E_{- e^{-t}} (\varphi)$. 
We also use $F_\varphi (\tau) := E_{1_{[\tau, \infty)}} (\varphi)$ to define Duistermaat--Heckman measure for general $\varphi \in \PSH (X, L)$. 

\begin{prop}
\label{continuity of moment energy along decreasing nets}
For an increasing right continuous function $\chi$ on $\mathbb{R}$, the functional $E_\chi$ is monotonic and continuous along decreasing nets of NA psh metrics. 
If $\chi$ is moreover concave, $E_\chi$ is concave. 
\end{prop}

\begin{proof}
We firstly show the following general claim: if $F$ is a functional on $\nH (X, L)$ which is monotonic ($F (\varphi) \le F (\varphi')$ for $\varphi \le \varphi'$) and $\lim_{\varepsilon \searrow 0} F (\varphi + \varepsilon) = F (\varphi)$ for every $\varphi \in \nH (X, L)$, then the functional $\bar{F}$ on $\PSH (X, L)$ defined by 
\[ \bar{F} (\varphi) := \inf \{ F (\tilde{\varphi}) ~|~ \varphi \le \tilde{\varphi} \in \nH (X, L) \} \] 
gives an extension of $F$ which is monotonic and continuous along decreasing nets. 

The monotonicity of $\bar{F}$ is obvious from the definition. 
In particular, we have $\lim_{i \to \infty} \bar{F} (\varphi_i) \ge \bar{F} (\varphi)$ for a convergent decreasing net $\varphi_i \searrow \varphi \in \PSH (X, L)$. 
To see the reverse inequality, pick $\varepsilon > 0$ and $\tilde{\varphi} \in \nH (X, L)$ so that $\varphi + \varepsilon \le \tilde{\varphi}$. 
Then by Lemma \ref{Dini} we have $\varphi_i < \tilde{\varphi}$ for sufficiently large $i$, so that $\lim_{i \to \infty} \bar{F} (\varphi_i) \le F (\tilde{\varphi})$. 
Taking the infimum of $\tilde{\varphi}$, we get $\lim_{i \to \infty} \bar{F} (\varphi_i) \le F (\varphi + \varepsilon)$. 
Then taking the limit $\varepsilon \searrow 0$, we get $\lim_{i \to \infty} \bar{F} (\varphi_i) \le F (\varphi)$. 

To apply this to $F (\varphi_{(\mathcal{X}, \mathcal{L})}) = \int_\mathbb{R} \chi \DHm_{(\mathcal{X}, \mathcal{L})}$, it suffices to show $\lim_{\varepsilon \searrow 0} F (\varphi + \varepsilon) = F (\varphi)$ for $\varphi \in \nH (X, L)$. 
We note $\varphi_{(\mathcal{X}, \mathcal{L})} + \varepsilon = \varphi_{(\mathcal{X}, \mathcal{L} + \varepsilon \mathcal{X}_0)}$ and $\DHm_{(\mathcal{X}, \mathcal{L} + \varepsilon \mathcal{X}_0)} = (t \mapsto t + \varepsilon)_* \DHm_{(\mathcal{X}, \mathcal{L})}$ by \cite[Proposition 3.12 (i)]{BHJ1}. 
By the monotonicity and the right continuity, we have $\chi (t + \varepsilon) \searrow \chi (t)$ as $\varepsilon \searrow 0$ and the sequence is bounded from above on the support of $\DHm_{(\mathcal{X}, \mathcal{L})}$. 
Thus we get 
\[ \lim_{\varepsilon \searrow 0} \int_\mathbb{R} \chi \DHm_{(\mathcal{X}, \mathcal{L} + \varepsilon \mathcal{X}_0)} = \lim_{\varepsilon \searrow 0} \int_\mathbb{R} \chi (t + \varepsilon) \DHm_{(\mathcal{X}, \mathcal{L})} = \int_\mathbb{R} \chi \DHm_{(\mathcal{X}, \mathcal{L})} \]
by the monotone convergence theorem. 

To see the concavity, we firstly consider the case $\varphi_0 = \varphi_{(\mathcal{X}_0, \mathcal{L}_0)}, \varphi_1 = \varphi_{(\mathcal{X}_1, \mathcal{L}_1)}$. 
Fix $t \in [0,1]$. 
Take a well ordered diagonal basis $\bm{s}$ with respect to $(1-t) \varphi_0 + t \varphi_1$ and a basis $\bm{s}'$ which is codiagonal for $\varphi_0, \varphi_1$ and is well ordered with respect to $(1-t) \varphi_0 + t \varphi_1$. 
Thanks to $\| \cdot \|^{(1-t) \varphi_0 + t \varphi_1} \le (\| \cdot \|^{\varphi_0})^{1-t} (\| \cdot \|^{\varphi_1})^t$, we can compute 
\begin{align*}
\int_\mathbb{R} \chi \nu_m ((1-t) \varphi_0 + t \varphi_1) 
&= m^{-n} \sum_{i=1}^{N_m} \chi (m^{-1} \lambda_i^{(1-t) \varphi_0 + t \varphi_1} (\bm{s}))  
\\
&\ge m^{-n} \sum_{i=1}^{N_m} \chi (m^{-1} \lambda_i^{(1-t) \varphi_0 + t \varphi_1} (\bm{s}')) 
\\
&\ge m^{-n} \sum_{i=1}^{N_m} \chi ((1-t) m^{-1} \lambda_i^{\varphi_0} (\bm{s}') +t m^{-1} \lambda_i^{\varphi_1} (\bm{s}'))
\\
&\ge (1-t) m^{-n} \sum_{i=1}^{N_m} \chi (m^{-1} \lambda_i^{\varphi_0} (\bm{s}')) + t m^{-n} \sum_{i=1}^{N_m} \chi (m^{-1} \lambda_i^{\varphi_1} (\bm{s}'))
\\
&= (1-t) \int_\mathbb{R} \chi \nu_m (\varphi_0) + t \int_\mathbb{R} \chi \nu_m (\varphi_1), 
\end{align*}
using Lemma \ref{diagonal lemma}, the monotonicity and the concavity of $\chi$ for the respective inequalities. 
Since the concavity of $\chi$ implies the continuity, we get the claim for these $\varphi_0, \varphi_1$ by taking the limit $m \to \infty$. 
For general NA psh metrics $\varphi_0, \varphi_1$, the claim follows by the continuity of $E_\chi$ along decreasing nets. 
\end{proof}

\begin{prop}
\label{DH vs spectral} 
Let $\chi$ be a right continuous increasing function $\chi$ on $\mathbb{R}$. 
For $\varphi \in C^0 \cap \PSH (X, L)$, we have 
\[ E_\chi (\varphi) = \int_\mathbb{R} \chi \nu_\infty (\mathcal{F}_\varphi). \]
As a consequence, we have 
\[ E_\chi (\varphi) = \inf \Big{\{} \int_\mathbb{R} \chi \nu_\infty (\mathcal{F}_{\tilde{\varphi}}) ~\Big{|}~ \varphi \le \tilde{\varphi} \in C^0 \cap \PSH (X, L) \Big{\}}. \]
\end{prop}

\begin{proof}
By Lemma \ref{monotonicity}, we have 
\[ E_\chi (\varphi) \ge \int_\mathbb{R} \chi \nu_\infty (\mathcal{F}_\varphi). \]

By the continuity of $\varphi$, for any $\varepsilon > 0$ we can find $\varphi_{(\mathcal{X}, \mathcal{L})} \in \nH (X, L)$ such that $\varphi \le \varphi_{(\mathcal{X}, \mathcal{L})} \le \varphi + \varepsilon$ thanks to Lemma \ref{Dini}. 
Then we get 
\[ E_\chi (\varphi) \le \int_\mathbb{R} \chi \DHm_{(\mathcal{X}, \mathcal{L})} = \int_\mathbb{R} \chi \nu_\infty (\mathcal{F}_{\varphi_{(\mathcal{X}, \mathcal{L})}}) \le \int_\mathbb{R} \chi \nu_\infty (\mathcal{F}_{\varphi+\varepsilon}) = \int_\mathbb{R} \chi (t + \varepsilon) \nu_\infty (\mathcal{F}_\varphi) \]
again by Lemma \ref{monotonicity}. 
As $\chi (t+ \varepsilon) \searrow \chi (t)$ is bounded from above on the support of $\nu_\infty (\mathcal{F}_\varphi)$, the monotone convergence theorem shows 
\[ E_\chi (\varphi) \le \int_\mathbb{R} \chi \nu_\infty (\mathcal{F}_\varphi). \]
\end{proof}

\begin{eg}
The right continuity of $\chi$ is essential. 
For instance, consider $\chi = 1_{(\tau, \infty)}$ and $\varphi_i = \tau + 1/i \searrow \varphi = \tau$ for $i \in \mathbb{N}_+$, then we have 
\[ \int_\mathbb{R} 1_{(\tau, \infty)} \DHm_{\varphi_i} = \int_\mathbb{R} 1_{(\tau, \infty)} \delta_{\tau + 1/i} = 1 \nrightarrow 0 = \int_\mathbb{R} 1_{(\tau, \infty)} \delta_\tau = \int_\mathbb{R} 1_{(\tau, \infty)} \DHm_\varphi. \]
%\[ \int_\mathbb{R} 1_{(-\infty, 0]} \DHm_{\varphi_i} = \int_\mathbb{R} 1_{(-\infty, 0]} \delta_{1/i} = 0 \nrightarrow 1 = \int_\mathbb{R} 1_{(-\infty, 0]} \delta_0 = \int_\mathbb{R} 1_{(-\infty, 0]} \DHm_\varphi. \]
\end{eg}

\begin{defin}
\label{finite moment energy class}
For a non-constant increasing concave function $\chi$ on $\mathbb{R}$, we put 
\begin{equation} 
\E^\chi (X, L) := \{ \varphi \in \PSH (X, L) ~|~ E_\chi (\varphi_{; \rho}) > -\infty \text{ for } \forall \rho > 0 \}. 
\end{equation}
Note $\chi$ is automatically continuous and $\chi (-\infty) = -\infty$ by the assumption. 
\end{defin}

We can easily check the following for $\varphi \in \E^\chi (X, L)$. 
\begin{itemize}
\item If $\varphi' \ge \varphi$ for $\varphi' \in \PSH (X, L)$, then $\varphi' \in \E^\chi (X, L)$. 

%\item $\varphi + c \in \E^\chi (X, L)$ for any $c \in \mathbb{R}$. 

\item $\varphi_{; \rho} \in \E^\chi (X, L)$ for any $\rho \in \mathbb{R}_+$

\item $\varphi \wedge \tau \in \E^\chi (X, L)$ for any $\tau \in \mathbb{R}$. 

\item If $\varphi_0, \varphi_1 \in \E^\chi (X, L)$, then $(1-t) \varphi_0 + t \varphi_1 \in \E^\chi (X, L)$ for any $t \in [0,1]$. 

\item If $\chi \le \chi'$, we have $E_\chi \le E_{\chi'}$, so that $\E^\chi (X, L) \subset \E^{\chi'} (X, L)$. 
\end{itemize}
Since $E_\chi (c) = \chi (c) > -\infty$ for $c \in \mathbb{R}$, we have $\PSH^{\mathrm{bdd}} (X, L) \subset \E^\chi (X, L)$ by the first property. 

We recall 
\begin{align*}
\mathcal{E}^1_{\mathrm{NA}} (X, L) 
&:= \{ \varphi \in \PSH (X, L) ~|~ E (\varphi) > -\infty \}, 
\\
\mathcal{E}^{\exp}_{\mathrm{NA}} (X, L) 
&:= \{ \varphi \in \PSH (X, L) ~|~ E_{\exp} (\varphi_{;\rho}) > -\infty \text{ for } \forall \rho > 0 \}. 
\end{align*}
We have $\E^{\exp} (X, L) \subset \E^1 (X, L)$. 
Since $E (\varphi +c) = E (\varphi) +c (e^L)$, $E_{\exp} (\varphi + c) = e^{-c} E_{\exp} (\varphi)$ for $c \in \mathbb{R}$, we have 
\begin{itemize}
\item $\varphi+ c \in \E^1 (X, L)$ (resp. $\E^{\exp} (X, L)$) if $\varphi \in \E^1 (X, L)$ (resp. $\E^{\exp} (X, L)$). 
\end{itemize}

\begin{quest}
In general for any $\varphi \in \E^\chi (X, L)$, we know $\varphi + c \in \E^\chi (X, L)$ for $c \ge 0$ and $(1-\varepsilon) \varphi + c \in \E^\chi (X, L)$ for $0 < \varepsilon < 1$ and $c \in \mathbb{R}$. 
For general $\chi$, does $\varphi + c \in \E^\chi (X, L)$ hold for $\varphi \in \E^\chi (X, L)$ and $c \in \mathbb{R}$? 
\end{quest}

\begin{eg}
As explained in section \ref{NAmu via Legendre dual}, for a lower semi-continuous convex function $q$ on the interval $[0,1]$, we can assign a non-archimedean psh metric $\varphi_q$ on $(\mathbb{C}P^1, \mathcal{O} (1))$ and have 
\[ E_{\exp} (\varphi_{q; \rho}) = - \int_{[0,1]} e^{\rho q (t)} dt, \quad \sup |\varphi_q| = |q|. \]
It follows that the unbounded convex function $q= \log (-\log t(1-t))$ gives an unbounded example of $\varphi \in \E^{\exp} (X, L)$. 
\end{eg}

\subsubsection{Duistermaat--Heckman measure of non-archimedean psh metric}
\label{Duistermaat--Heckman measure of non-archimedean psh metric}

\begin{prop}
For $\varphi \in \PSH (X, L)$, we put 
\[ F_\varphi (\tau) := E_{1_{[\tau, \infty)}} (\varphi) = \inf \Big{\{} \int_{[\tau, \infty)} \DHm_{(\mathcal{X}, \mathcal{L})} ~\Big{|}~ \varphi \le \varphi_{(\mathcal{X}, \mathcal{L})} \Big{\}}. \]
Then we have the following. 
\begin{enumerate}
\item The function $F_\varphi$ is decreasing, left continuous and satisfies $F_\varphi (\tau) = 0$ for $\tau > \sup \varphi$ and $\lim_{\tau \to -\infty} F_\varphi (\tau) \le (e^L)$. 

\item Suppose $\varphi_1 \le \varphi_2$, then we have $F_{\varphi_1} (\tau) \le F_{\varphi_2} (\tau)$. 

\item For a convergent decreasing net $\varphi_i \searrow \varphi \in \PSH (X, L)$, we have $F_{\varphi_i} (\tau) \searrow F_\varphi (\tau)$. 
\end{enumerate}
\end{prop}

\begin{proof}
The properties (2) and (3) are the consequence of the propositions in the last subsection. 

For any $(\mathcal{X}, \mathcal{L})$ with $\varphi_{(\mathcal{X}, \mathcal{L})} \ge \varphi$, we have 
\[ \int_{[\tau', \infty)} \DHm_{(\mathcal{X}, \mathcal{L})} \ge \int_{[\tau, \infty)} \DHm_{(\mathcal{X}, \mathcal{L})} \ge F_\varphi (\tau) \]
for $\tau' \le \tau$. 
This shows the monotonicity $F_\varphi (\tau') \ge F_\varphi (\tau)$ for $\tau' \le \tau$. 

Next we check the left continuity. 
Take an increasing sequence $\tau_j \nearrow \tau \in \mathbb{R}$. 
Since we know $F_\varphi (\tau_j) \ge F_\varphi (\tau)$, it suffices to show $\lim_{j \to \infty} F_\varphi (\tau_j) \le F_\varphi (\tau)$, which is equivalent to show $\lim_{i \to \infty} F_\varphi (\tau_i) \le \int_{[\tau, \infty)} \DHm_{(\mathcal{X}, \mathcal{L})}$ for every $(\mathcal{X}, \mathcal{L})$ with $\varphi_{(\mathcal{X}, \mathcal{L})} \ge \varphi$. 
Since $1_{[\tau_j, \infty)}$ is a bounded sequence which converges to $1_{[\tau, \infty)}$ pointwisely (not only a.e. with respect to the Lebesgue measure), we have $\lim_{j \to \infty} \int_{[\tau_j, \infty)} \DHm_{(\mathcal{X}, \mathcal{L})} = \int_{[\tau, \infty)} \DHm_{(\mathcal{X}, \mathcal{L})}$ by the bounded convergence theorem. 
Thus we get $\lim_{i \to \infty} F_\varphi (\tau_i) \le \int_{[\tau, \infty)} \DHm_{(\mathcal{X}, \mathcal{L})}$. 
(Note $1_{[\tau_j, \infty)}$ converges to $1_{(\tau, \infty)}$ if we approximate $\tau$ from the right $\tau_j \searrow \tau$. 
Since $\DHm$ may have a singular component in the Lebesgue decomposition, $F_\varphi$ is not right continuous. )

Note $\int_{[\tau, \infty)} \DHm_{(\mathcal{X}, \mathcal{L})} = 0$ for $\tau > \sup \varphi_{(\mathcal{X}, \mathcal{L})}$. 
It follows that $F_\varphi (\tau) = 0$ for $\tau > \sup \varphi_{(\mathcal{X}, \mathcal{L})}$ for any $\varphi_{(\mathcal{X}, \mathcal{L})} \ge \varphi$. 
Since $\sup \varphi_i \searrow \sup \varphi$ for $\varphi_i \searrow \varphi$, we get $F_\varphi (\tau) = 0$ for $\tau > \sup \varphi$. 

The property $\int_\mathbb{R} \DHm_\varphi \le (e^L)$ follows immediately from 
\[ \int_\mathbb{R} \DHm_\varphi = \lim_{j \to \infty} \int_{[-j, \infty)} \DHm_\varphi \le \lim_{j \to \infty} \int_{[-j, \infty)} \DHm_{(\mathcal{X}, \mathcal{L})} = \int_\mathbb{R} \DHm_{(\mathcal{X}, \mathcal{L})} = (e^L) \]
for any $(\mathcal{X}, \mathcal{L})$ with $\varphi_{(\mathcal{X}, \mathcal{L})} \ge \varphi$. 
\end{proof}

Thanks to this proposition, we obtain the following extension of the Duistermaat--Heckman measure of test configuration. 

\begin{defin}[Duistermaat--Heckman measure of NA psh metric]
\label{DH measure}
The \textit{Duistermaat--Heckman measure} of a non-archimedean psh metric $\varphi \in \PSH (X, L)$ is a finite Borel measure $\DHm_\varphi$ on $\mathbb{R}$ which is uniquely characterized by 
\[ \int 1_{[\tau, \infty)} \DHm_\varphi = F_\varphi (\tau) \]
for every $\tau \in \mathbb{R}$. 
We have $\supp \DHm_\varphi \subset (-\infty, \sup \varphi]$ and $\int_\mathbb{R} \DHm_\varphi \le (e^L)$. 
\end{defin}

By Proposition \ref{DH vs spectral}, we have 
\begin{equation} 
\DHm_\varphi = \nu_\infty (\mathcal{F}_\varphi) 
\end{equation}
for $\varphi \in C^0 \cap \PSH (X, L)$. 

\begin{lem}
\label{DH measure shift}
For a non-negative Borel measurable function $\chi$ on $\mathbb{R}$, we have the following basic rules. 
\begin{enumerate}
\item $\int_\mathbb{R} \chi (t) \DHm_{\varphi_{;\rho}} = \int_\mathbb{R} \chi (\rho t) \DHm_{\varphi}$ for any $\rho \in \mathbb{R}_+$. 

\item $\int_\mathbb{R} \chi (t) \DHm_{\varphi + c} = \int_\mathbb{R} \chi (t+c) \DHm_{\varphi}$ for any $c \in \mathbb{R}$. 

\item $\int_\mathbb{R} \chi \DHm_{\varphi \wedge \tau} = \int_{(-\infty, \tau)} \chi \DHm_{\varphi} + \chi (\tau) \int_{[\tau, \infty)} \DHm_\varphi$ for any $\tau \in \mathbb{R}$. 
\end{enumerate}
\end{lem}

\begin{proof}
Since $\DHm_\varphi$ is outer regular measure, we can reduce the claim to the case $\chi = 1_{[\tau', \tau)}$, for which we have $\int_\mathbb{R} \chi \DHm_\varphi = F_\varphi (\tau') - F_\varphi (\tau)$. 
It suffices to show $F_{\varphi_{;\rho}} (\tau) = F_\varphi (\rho^{-1} \tau)$, $F_{\varphi + c} (\tau) = F_\varphi (\tau-c)$ and $F_{\varphi \wedge \tau'} (\tau) = 1_{[\tau, \infty)} (\tau') F_\varphi (\tau)$. 
Since both sides are continuous along decreasing nets, the claim is reduced to the case $\varphi \in C^0 \cap \PSH (X, L)$. 
Since $\mathcal{F}_{\varphi_{; \rho}} = (\mathcal{F}_{\varphi})_{; \rho}$, $\mathcal{F}_{\varphi + c} = \mathcal{F}_\varphi [c]$ and $\mathcal{F}_{\varphi \wedge \tau'} = \mathcal{F}_\varphi \cap \mathcal{F}_{\tau'}$, the claim follows by \cite[Proposition 3.4]{BJ2}: $\nu_\infty ((\mathcal{F}_{\varphi})_{; \rho}) = (t \mapsto \rho t)_* \nu_\infty (\mathcal{F}_\varphi)$, $\DHm_{\mathcal{F}_\varphi [c]} = (t \mapsto t + c)_* \nu_\infty (\mathcal{F}_\varphi)$ and $\nu_\infty (\mathcal{F}_\varphi \cap \mathcal{F}_{\mathrm{triv}} [\tau']) = (t \mapsto \min \{ t, \tau' \})_* \nu_\infty (\mathcal{F}_\varphi)$. 
%\[ \int_{[\tau, \infty)} \DHm_{(\mathcal{X}, \mathcal{L}; \rho)} = \int_\mathbb{R} 1_{[\tau, \infty)} (\rho t) \DHm_{(\mathcal{X}, \mathcal{L})} = \int_{[\rho^{-1} \tau, \infty)} \DHm_{(\mathcal{X}, \mathcal{L})} \]
%\[ \int_{[\tau, \infty)} \DHm_{(\mathcal{X}, \mathcal{L}+ c \mathcal{X}_0)} = \int_\mathbb{R} 1_{[\tau, \infty)} (t+ c) \DHm_{(\mathcal{X}, \mathcal{L})} = \int_{[\tau-c, \infty)} \DHm_{(\mathcal{X}, \mathcal{L})} \]
%\begin{align*} 
%\int_{[\tau, \infty)} \DHm_{(\hat{\mathcal{X}}^{\tau'}, \hat{\mathcal{L}}^{\tau'})} 
%&= \int_{(-\infty, \tau')} 1_{[\tau, \infty)} \DHm_{(\mathcal{X}, \mathcal{L})} + 1_{[\tau, \infty)} (\tau') \int_{[\tau', \infty)} \DHm_{(\mathcal{X}, \mathcal{L})} 
%\\
%& = 1_{[\tau, \infty)} (\tau') \int_{[\tau, \infty)} \DHm_{(\mathcal{X}, \mathcal{L})}
%\end{align*}
\end{proof}

We observe sufficient conditions for the continuity of $\int_\mathbb{R} \chi \DHm_\varphi$ along decreasing nets. 

\begin{defin}
\label{tame}
We call a function $\chi$ on $\mathbb{R}$ is \textit{tame} if for any $t \in \mathbb{R}$ and $\varepsilon > 0$ there exists a function $\tilde{\chi}$ of the form $\tilde{\chi} = \sum_{j=1}^k a_j 1_{[\tau'_j, \tau_j)}$ ($\tau_j = \infty$ is allowed) such that $\sup_{(-\infty, t)} |\chi - \tilde{\chi}| < \varepsilon$. 
\end{defin}

\begin{lem}
Any tame function is right continuous, locally bounded and converges to zero at $-\infty$, hence is integrable with respect to $\DHm_\varphi$. 

The following functions are tame. 
\begin{itemize}
\item Continuous function $\chi$ converging to zero at $-\infty$. 

\item Monotonic right continuous function $\chi$ converging to zero at $-\infty$. 
\end{itemize}
Here convergence to zero at $-\infty$ means $\varlimsup_{t \to -\infty} |\chi (t)| = 0$. 
\end{lem}

\begin{proof}
The right continuity is equivalent to the continuity with respect to the lower limit topology on the domain $\mathbb{R}$. 
Then the uniform limit theorem shows the right continuity of tame function. 
Local boundedness and convergence to zero at $-\infty$ is obvious. 

Suppose $\chi$ is a continuous function converging to zero at $-\infty$. 
For any $\varepsilon > 0$, we have $t_\varepsilon \in \mathbb{R}$ such that $\sup_{(-\infty, t_\varepsilon]} |\chi| < \varepsilon$. 
On the other hand, $\chi$ is uniformly continuous on $[t_\varepsilon, t+\varepsilon]$, so that we can find a step function $\tilde{\chi}$ supported on $[t_\varepsilon, t+\varepsilon)$ with $\sup_{[t_\varepsilon, t]} |\chi - \tilde{\chi}| < \varepsilon$. 
It follows that $\sup_{(-\infty, t]} |\chi - \tilde{\chi}| < \varepsilon$. 

Suppose $\chi$ is a right continuous increasing function converging to zero at $-\infty$. 
We take $t_\varepsilon$ similarly as above. 
We divide the interval $[\chi (t_\varepsilon), \chi (t))$ into finitely many disjoint intervals $[a'_i, a_i)$ with length $< \varepsilon$. 
By the right continuity and the monotonicity, $\chi^{-1} ([a'_i, a_i))$ are of the form $[\tau'_i, \tau_i)$, which are disjoint and cover $[t_\varepsilon, t)$. 
Since $\sup_{[\tau'_i, \tau_i)} |\chi - a'_i 1_{[\tau'_i, \tau_i)}| < \varepsilon$, we get $\sup_{(-\infty, t]} |\chi - \sum_i a'_i 1_{[\tau'_i, \tau_i)}| < \varepsilon$. 
\end{proof}

\begin{prop}
\label{tame}
Let $\chi$ be a tame function. 
For a convergent decreasing net $\varphi_i \searrow \varphi \in \PSH (X, L)$, we have $\int_{\mathbb{R}} \chi \DHm_{\varphi_i} \to \int_{\mathbb{R}} \chi \DHm_\varphi$. 
\end{prop}

\begin{proof}
We remark the claim holds for step functions $\tilde{\chi} = \sum_{j=1}^k a_j 1_{[\tau'_j, \tau_j)}$ as 
\[ \int_\mathbb{R} \tilde{\chi} \DHm_\varphi = \sum_{j=1}^k a_j (F_\varphi (\tau'_j) - F_\varphi (\tau_j)). \]
Since $\supp \DHm_{\varphi_i}, \supp \DHm_\varphi \subset (-\infty, \sup \varphi_0]$, we compute 
\begin{align*}
\varlimsup_{i \to \infty} |\int_\mathbb{R} \chi \DHm_\varphi - \int_\mathbb{R} \chi \DHm_{\varphi_i}| 
&\le \int_\mathbb{R} |\chi - \tilde{\chi}| \DHm_\varphi + \lim_{i \to \infty} |\int_\mathbb{R} \tilde{\chi} \DHm_\varphi - \int_\mathbb{R} \tilde{\chi} \DHm_{\varphi_i}| 
\\
& \qquad \qquad + \varlimsup_{i \to \infty} \int_\mathbb{R} |\tilde{\chi} - \chi| \DHm_{\varphi_i} 
\\
&\le 2\varepsilon (e^L)
\end{align*}
for any $\varepsilon > 0$, by taking $\tilde{\chi}$ so that $\sup_{(-\infty, \sup \varphi_0+ 1)} |\chi - \tilde{\chi}| < \varepsilon$. 
Taking the limit $\varepsilon \to 0$, we obtain the claim. 
\end{proof}

\subsubsection{$\E (X, L)$ and moment energy}
\label{Duistermaat--Heckman measure and moment energy}

For a general $\varphi \in \PSH (X, L)$, we may have $\int_\mathbb{R} \DHm_\varphi < (e^L)$ as we see in the following remark. 
In particular, the convergence $\int_\mathbb{R} \DHm_{\varphi_i} \to \int_\mathbb{R} \DHm_\varphi$ for $\varphi_i \searrow \varphi \in \PSH (X, L)$ fails in general. 
(Note $1_\mathbb{R}$ is not tame. )

\begin{rem}
Consider $\varphi_i = \max \{ \log |s_0|, \log |s_1| - i \} \in \PSH (\mathbb{C}P^1, \mathcal{O} (1))$ for $i \in \mathbb{N}$, then we have $\varphi_i \searrow \varphi = \log |s_0|$. 
Since $\DHm_{\varphi_i} = i^{-1} dt|_{[-i, 0]}$, we have $\DHm_\varphi = 0$, so that $\int_\mathbb{R} \DHm_\varphi = 0 < (e^L)$. 
On the other hand, $\int_\mathbb{R} \DHm_{\varphi_i} = (e^L)$. 

We also observe $\lim_{i \to \infty} \int_{(-\infty, \tau)} \DHm_{\varphi_i} = 1$ for every $\tau \in \mathbb{R}$. 
This illustrates the functional 
\[ G_\varphi (\tau) := \sup \Big{\{} \int_{(-\infty, \tau)} \DHm_{(\mathcal{X}, \mathcal{L})} ~\Big{|}~ \varphi_{(\mathcal{X}, \mathcal{L})} \ge \varphi \Big{\}} \] 
is not suitable for defining Duistermaat--Heckman measure: $\lim_{\tau \to -\infty} G_\varphi (\tau) \neq 0$. 
\end{rem}

For the above example, we have $\int_\mathbb{R} t \DHm_{\varphi_i} = i^{-1} \int_{-i}^0 t dt = -i/2 \searrow -\infty$, while $\int_\mathbb{R} t \DHm_\varphi = 0$. 
This in particular shows $E (\varphi) \neq \int_\mathbb{R} t \DHm_\varphi$ and the monotonicity as in Lemma \ref{monotonicity} fails for general $\varphi \in \PSH (X, L)$ and $\chi$. 

The following class is appropriate to discuss the relation to moment energy. 
\begin{equation}
\E (X, L) := \{ \varphi \in \PSH (X, L) ~|~ \int_\mathbb{R} \DHm_\varphi = (e^L) \}. 
\end{equation}

\begin{prop}
We have the following for $\varphi \in \E (X, L)$. 
\begin{enumerate}
\item If $\varphi' \ge \varphi$ for $\varphi' \in \PSH (X, L)$, then $\varphi' \in \E (X, L)$. 
In particular, we have $\PSH^{\mathrm{bdd}} (X, L) \subset \E (X, L)$. 

\item $\varphi + c \in \E (X, L)$ for any $c \in \mathbb{R}$. 

\item $\varphi_{; \rho} \in \E (X, L)$ for any $\rho \in \mathbb{R}_+$. 

\item $\varphi \wedge \tau \in \E (X, L)$ for any $\tau \in \mathbb{R}$. 

%\item If $\varphi_0, \varphi_1 \in \E (X, L)$, then $(1-t) \varphi_0 + t \varphi_1 \in \E (X, L)$ for any $t \in [0,1]$. 
\end{enumerate}
\end{prop}

\begin{proof}
The first claim follows from 
\[ (e^L) \ge \int_\mathbb{R} \DHm_{\varphi'} = \lim_{j \to \infty} \int_{[-j, \infty)} \DHm_{\varphi'} \ge \lim_{j \to \infty} \int_{[-j, \infty)} \DHm_\varphi = \int_\mathbb{R} \DHm_\varphi = (e^L). \]
The claim (2)--(4) follows from Lemma \ref{DH measure shift}. 
\end{proof}

For an increasing right continuous function $\chi$, the functional $\varphi \mapsto \int_\mathbb{R} \chi \DHm_\varphi$ behaves well on $\E (X, L)$. 

\begin{prop}
\label{moment energy and DH measure}
Let $\chi$ be an increasing right continuous function on $\mathbb{R}$. 
If $\varphi \in \E (X, L)$, we have 
\[ E_\chi (\varphi) = \int_\mathbb{R} \chi \DHm_\varphi. \]
\end{prop}

\begin{proof}
We firstly note the integral $\int_\mathbb{R} \chi \DHm_\varphi$ makes sense, which may takes $-\infty$, since $\chi$ is bounded from above on $\supp \DHm_\varphi$: we put $\int_\mathbb{R} \chi \DHm_\varphi := - \int_\mathbb{R} (c - \chi) \DHm_\varphi + c \int_\mathbb{R} \DHm_\varphi$ by taking a constant $c$ so that $c- \chi$ is non-negative on $\supp \DHm_\varphi$. 

%We have $E_{\chi +c} (\varphi) = E_\chi (\varphi) + c (L^{\cdot n})/n!$ in general. 
%On the other hand, the assumption $\varphi \in \E (X, L)$ ensures $\int_\mathbb{R} (\chi +c) \DHm_\varphi = \int_\mathbb{R} \chi \DHm_\varphi + c (L^{\cdot n})/n!$, so we may add constant on $\chi$ to show the claim. 

Take a regularization $\{ \varphi_i \}_{i \in I} \subset \nH (X, L)$ so that $\varphi_i \searrow \varphi$. 
%We may assume $\chi < 0$ on $(-\infty, \sup \varphi_1]$ by adding a sufficiently negative constant. 
%Take for each $j$ a continuous cutoff function $\beta_j$ so that $0 \le \beta_j \le 1$ and $\beta|_{(-\infty, -j-1]} = 0$ and $\beta|_{[-j, \infty)} = 1$. 
%For large $j$, we have $\chi < 0$ on $(-\infty, -j)$, so that we compute 
%\[ E_\chi (\varphi) = \lim_{i \to \infty} \int_\mathbb{R} \chi \DHm_{\varphi_i} \le \varlimsup_{i \to \infty} \int_\mathbb{R} \beta_j \chi \DHm_{\varphi_i} = \int_\mathbb{R} \beta_j \chi \DHm_\varphi \]
%by Proposition \ref{tame}. 
%Since $\beta_j \chi$ is a decreasing sequence of negative functions on $(-\infty, \sup \varphi]$, we conclude $E_\chi (\varphi) \le \int_\mathbb{R} \chi \DHm_\varphi$ by Levi's theorem. 
For each $j \in \mathbb{N}$, $\max (\chi- \chi (-j), 0)$ is an increasing right continuous function with left bounded support. 
Then by Proposition \ref{continuity of moment energy along decreasing nets}, Proposition \ref{tame} and the assumption $\varphi \in \E (X, L)$, we compute 
\begin{align*} 
E_\chi (\varphi) = \lim_{i \to \infty} \int_\mathbb{R} \chi \DHm_{\varphi_i} 
&\le \varliminf_{i \to \infty} \int_\mathbb{R} \max (\chi, \chi (-j) ) \DHm_{\varphi_i} 
\\
&= \lim_{i \to \infty} \int_\mathbb{R} \max ( \chi- \chi (-j), 0) \DHm_{\varphi_i} + \chi (-j) (e^L) 
\\
&= \int_\mathbb{R} \max ( \chi- \chi (-j), 0) \DHm_\varphi + \chi (-j) (e^L)
\\
&= \int_\mathbb{R} \max ( \chi, \chi (-j)) \DHm_\varphi. 
\end{align*}
Since $\max ( \chi, \chi (-j)) \searrow \chi$ is bounded from above on $(-\infty, \sup \varphi]$, we conclude $E_\chi (\varphi) \le \int_\mathbb{R} \chi \DHm_\varphi$ by the monotone convergence theorem. 

To see $E_\chi (\varphi) \ge \int_\mathbb{R} \chi \DHm_\varphi$, it suffices to show $\int_\mathbb{R} \chi \DHm_{\varphi_0} \ge \int_\mathbb{R} \chi \DHm_\varphi$. 
Using the assumption $\varphi \in \E (X, L)$, Proposition \ref{tame} and Lemma \ref{monotonicity}, We compute 
\begin{align*} 
\int_\mathbb{R} \chi \DHm_\varphi \le \int_\mathbb{R} \max (\chi, \chi (-j)) \DHm_\varphi 
&= \int_\mathbb{R} \max (\chi- \chi (-j), 0) \DHm_\varphi + \chi (-j) (e^L)
\\
&= \lim_{i \to \infty} \int_\mathbb{R} \max (\chi- \chi (-j), 0) \DHm_{\varphi_i} + \chi (-j) (e^L) 
\\
&\le \int_\mathbb{R} \max (\chi- \chi (-j), 0) \DHm_{\varphi_0} + \chi (-j) (e^L)
\\
&= \int_\mathbb{R} \max (\chi, \chi (-j)) \DHm_{\varphi_0}. 
\end{align*}
Taking the limit $j \to \infty$, we get $\int_\mathbb{R} \chi \DHm_\varphi \le \int_\mathbb{R} \chi \DHm_{\varphi_0}$. 
\end{proof}

\begin{defin}
\label{moderate}
A function $\chi$ on $\mathbb{R}$ is \textit{moderate} if it is the sum of a right continuous monotonic function and a tame function. 
\end{defin}

A moderate function $\chi$ is bounded from below or above on $(-\infty, \sup \varphi]$, so that the integration $\int_\mathbb{R} \chi \DHm_\varphi$ makes sense, which may take value $\pm \infty$. 
If $\varphi' \ge \varphi$ and $\int_\mathbb{R} \chi \DHm_\varphi$ is finite, then $\int_\mathbb{R} \chi \DHm_{\varphi'}$ is also finite. 

\begin{cor}
Let $\chi$ be a moderate function on $\mathbb{R}$. 
For a convergent decreasing net $\varphi_i \searrow \varphi \in \mathcal{E}_{\mathrm{NA}} (X, L)$, we have $\lim_{i \to \infty} \int_\mathbb{R} \chi \DHm_{\varphi_i} = \int_\mathbb{R} \chi \DHm_\varphi$. 
\end{cor}

\begin{prop}
Let $\chi$ be an increasing right continuous function on $\mathbb{R}$ with $\lim_{t \to -\infty} \chi (t) = -\infty$. 
Then $E_\chi (\varphi) > -\infty$ implies $\int_\mathbb{R} \DHm_\varphi = (e^L)$. 
\end{prop}

\begin{proof}
Suppose $\int_\mathbb{R} \DHm_\varphi < (e^L)$. 
Take $\varepsilon > 0$ so that $\int_\mathbb{R} \DHm_\varphi \le (e^L) - \varepsilon$. 
As $\lim_{j \to \infty} \int_{[-j, \infty)} \DHm_\varphi = \int_\mathbb{R} \DHm_\varphi$, there exists $j_\varepsilon \in \mathbb{N}$ such that 
\[ \int_{[-j, \infty)} \DHm_\varphi \le (e^L) - \varepsilon/2 \]
for every $j \ge j_\varepsilon$. 
Now we take a regularization $\{ \varphi_i \}_{i \in I} \subset \nH (X, L)$ so that $\varphi_i \searrow \varphi$. 
As $\lim_{i \to \infty} \int_{[-j, \infty)} \DHm_{\varphi_i} = \int_{[-j, \infty)} \DHm_\varphi$, for each $j \ge j_\varepsilon$ we can take $i_j \in I$ so that 
\[ \int_{[-j, \infty)} \DHm_{\varphi_i} \le \int_\mathbb{R} \DHm_{\varphi_i} - \varepsilon/4 \]
for every $i \ge i_j$, which we can rearrange as 
\[ \varepsilon/4 \le \int_{(-\infty, -j)} \DHm_{\varphi_i}. \]
For any $j \ge j_\varepsilon$ with $\chi (-j) < 0$, we compute 
\begin{align*} 
E_\chi (\varphi) \le E_\chi (\varphi_{i_j}) = \int_\mathbb{R} \chi \DHm_{\varphi_{i_j}} 
&= \int_{(-\infty, -j)} \chi \DHm_{\varphi_{i_j}} + \int_{[-j, \infty)} \chi \DHm_{\varphi_{i_j}} 
\\
&\le \chi (- j) \cdot \varepsilon/4 + \sup \varphi_0 \cdot (e^L). 
\end{align*}
Since $\chi (-j) \to -\infty$ as $j \to \infty$, we get $E_\chi (\varphi) = -\infty$, which proves the contraposition. 
\end{proof}

\begin{cor}
\label{finite moment energy class and full mass class}
For any non-constant increasing concave function $\chi$ on $\mathbb{R}$, we have 
\[ \E^\chi (X, L) \subset \E (X, L). \]
\end{cor}

\subsubsection{Brunn--Minkowski inequality}

The following will be used to bound the maximum of a $d_{\exp}$-bounded increasing non-archimedean psh metrics. 

\begin{prop}
\label{concavity}
For $\varphi \in \PSH (X, L)$, $F_\varphi (\tau)^{1/n} = (\int_{[\tau, \infty)} \DHm_\varphi)^{1/n}$ is concave on $(-\infty, \sup \varphi]$ and is zero on $(\sup \varphi, \infty)$. 
As a consequence, we have $\sup \supp \DHm_\varphi = \sup \varphi$ if $\int_\mathbb{R} \DHm_\varphi \neq 0$. 
\end{prop}

\begin{proof}
We note the claim holds for $\varphi \in \nH (X, L)$ as $F_\varphi (\tau)$ is the left continuous modification of $\mathrm{vol} (R^{(\tau)})$ described in \cite{BHJ1}, which is noted in the proof of \cite[Theorem 1.9]{BC}. 
By Brunn--Minkowski inequality, $\mathrm{vol} (R^{(\tau)})^{1/n}$ is concave on $(-\infty, \sup \varphi)$ and is zero on $(\sup \varphi, \infty)$ (cf. \cite[Theorem 5.3]{BHJ1}), hence in particular it is continuous on these intervals. 
Thus we have $F_\varphi (\tau) = \mathrm{vol} (R^{(\tau)})$ for $\tau \neq \sup \varphi$, which shows that $F_\varphi^{1/n}$ is concave on $(-\infty, \sup \varphi)$ and is zero on $(\sup \varphi, \infty)$. 
By the left continuity, $F_\varphi^{1/n}$ is also concave on $(\infty, \sup \varphi]$. 

For general $\varphi \in \PSH (X, L)$, take a regularization $\{ \varphi_i \}_{i \in I} \subset \nH (X, L)$ so that $\varphi_i \searrow \varphi$. 
Then we have $\sup \varphi_i \searrow \sup \varphi$ and $F_{\varphi_i} (\tau) \searrow F_\varphi (\tau)$. 
The concavity is preserved under pointwise limit, so that $F_\varphi^{1/n}$ is concave on $(-\infty, \sup \varphi] \subset (-\infty, \sup \varphi_i]$ and zero on $(\sup \varphi, \infty)$. 

Suppose $\int_{[\tau_0, \infty)} \DHm_\varphi = 0$ for some $\tau_0 < \sup \varphi$. 
Then $\int_{[\tau, \infty)} \DHm_\varphi = 0$ for $\tau \ge \tau_0$ by the monotonicity. 
On the other hand, for $\tau < \tau_0$ we have 
\[ F_\varphi^{1/n} (\tau) \le \frac{\sup \varphi - \tau}{\sup \varphi - \tau_0} F_\varphi^{1/n} (\tau_0) - \frac{\tau_0 - \tau}{\sup \varphi - \tau_0} F_\varphi^{1/n} (\sup \varphi) = 0 \] 
by the concavity. 
Thus $\int_\mathbb{R} \DHm_\varphi = 0$ when $\sup \supp \DHm_\varphi < \sup \varphi$. 
\end{proof}

\subsection{Moment measure on Berkovich space}

In the rest of this article, we only consider $\varphi \in \E^1 (X, L)$. 

\subsubsection{Tomography of non-archimedean Monge--Amp\`ere measure}
\label{Tomography of non-archimedean Monge--Ampere measure}

Here we compute $\mathrm{MA} (\varphi \wedge \tau)$ using results in the section \ref{Primary decomposition via filtration}. 
This is the key observation in the construction of the moment measure. 
We begin with the following lemma. 

\begin{lem}
\label{MA support via filtration}
Let $(\mathcal{X}, \mathcal{L})$ be a normal test configuration. 
Let $I$ be a finite set of valuations (with no duplication) and $\sigma': I \to \mathbb{R}$ be a map satisfying 
\[ \mathcal{F}^\lambda_{(\mathcal{X}, \mathcal{L})} = \widehat{\mathcal{F}}^\lambda_{(\mathcal{X}, \mathcal{L})} = \bigcap_{v \in I} \mathcal{F}^\lambda_v [\sigma' (v)] \]
for all $\lambda \in \mathbb{Z}$. 
Then $I$ contains $\{ v_E ~|~ E \subset \mathcal{X}_0 \}$ and $\sigma' (v_E) = \sigma_E = \varphi_{(\mathcal{X}, \mathcal{L})} (v_E)$. 
\end{lem}

\begin{proof}
By the definition of $\sigma_v$, we have $\sigma' (v) \ge \sigma_v$. 
Put $I' := \{ v \in I ~|~ \sigma ' (v) = \sigma_v \}$. 
As we already noted, we have 
\[ \bigoplus_{m \ge 0} \bigoplus_{\lambda \in \mathbb{Z}} \varpi^{-\lambda} (\mathcal{F}^\lambda_{(\mathcal{X}, \mathcal{L})} \cap \mathcal{F}^{\lambda+}_v [\sigma' (v)]) R_m = 
\begin{cases} 
\mathcal{I}_v
& \sigma' (v) = \sigma_v
\\
\mathcal{R} (\mathcal{X}, \mathcal{L})
& \sigma' (v) > \sigma_v, 
\end{cases} \]
so that we have $\mathcal{I}_{\mathcal{X}_0^{\mathrm{red}}} = \bigcap_{v \in I'} \mathcal{I}_v$ by the assumption. 
Since $\mathcal{I}_{\mathcal{X}_0^{\mathrm{red}}} = \bigcap_{E \subset \mathcal{X}_0} \mathcal{I}_{v_E}$ is the unique primary decomposition, $\{ \mathcal{I}_v ~|~ v \in I' \}$ must contain $\{ \mathcal{I}_{v_E} ~|~ E \subset \mathcal{X}_0 \}$. 
Then by Lemma \ref{the center of valuation}, we must have $v_E \in I'$. 
\end{proof}

As we observed in section \ref{Primary decomposition via filtration}, we can recover the primary decomposition of the Duistermaat--Heckman measure from $\mathcal{F}_\varphi = \widehat{\mathcal{F}}_{(\mathcal{X}, \mathcal{L})}$. 
In particular, we can recover the Monge--Amp\`ere measure from $\mathcal{F}_\varphi$. 
Since $\mathcal{F}_{\varphi \wedge \tau} = \mathcal{F}_\varphi \cap \mathcal{F}_\tau$, we can also recover $\mathrm{MA} (\varphi \wedge \tau)$ from $\mathcal{F}_\varphi$ and $\tau$, which relates $\mathrm{MA} (\varphi \wedge \tau)$ to the primary decomposition of the Duistermaat--Heckman measure. 

\begin{prop}
\label{MA of wedge}
For any $\varphi = \varphi_{(\mathcal{X}, \mathcal{L})} \in \nH (X, L)$ and $\tau \in \mathbb{R}$, we have 
\begin{align*} 
\mathrm{MA} (\varphi \wedge \tau) 
&= \sum_{E \subset \mathcal{X}_0} \mathrm{ord}_E \mathcal{X}_0 \int_{(-\infty, \tau)} \DHm_{(E, \mathcal{L}|_E)}. \delta_{v_E} + \int_{[\tau, \infty)} \DHm_{(\mathcal{X}, \mathcal{L})}. \delta_{v_{\mathrm{triv}}}
\\
&= \sum_{E \subset \mathcal{X}_0, v_E \neq v_{\mathrm{triv}}} \mathrm{ord}_E \mathcal{X}_0 \int_{(-\infty, \tau)} \DHm_{(E, \mathcal{L}|_E)}. \delta_{v_E} 
\\
&\quad+ \Big{(} \sum_{E \subset \mathcal{X}_0, v_E = v_{\mathrm{triv}}} \mathrm{ord}_E \mathcal{X}_0 \cdot (E. \mathcal{L}^{\cdot n}) + \int_{[\tau, \infty)} \DHm^\circ_{(\mathcal{X}, \mathcal{L})} \Big{)}. \delta_{v_{\mathrm{triv}}}, 
\end{align*}
where 
\[ \DHm^\circ_{(\mathcal{X}, \mathcal{L})} := \sum_{E \subset \mathcal{X}_0, v_E \neq v_{\mathrm{triv}}} \mathrm{ord}_E \mathcal{X}_0 \cdot \DHm_{(E, \mathcal{L}|_E)} \]
is the absolutely continuous part of $\DHm_{(\mathcal{X}, \mathcal{L})}$. 
Here we recall $E \subset \mathcal{X}_0$ gives the trivial valuation $v_E = v_{\mathrm{triv}}$ iff the $\mathbb{G}_m$-action on $E$ is trivial, which is equivalent to $\DHm_{(E, \mathcal{L}|_E)} = \mathrm{ord}_E \mathcal{X}_0 \cdot (E, \mathcal{L}^{\cdot n}). \delta_{\varphi (v_{\mathrm{triv}})}$. 
\end{prop}

\begin{proof}
The claim is obvious for $\tau \ge \sup \varphi = \varphi (v_{\mathrm{triv}})$. 
We may assume $\tau < \sup \varphi = \varphi (v_{\mathrm{triv}})$. 
By Proposition \ref{strong convergence of rooftops} and the absolute continuity of $\DHm_{(E, \mathcal{L}|_E)}$ for $v_E \neq v_{\mathrm{triv}}$, both sides of the equality are continuous on $\tau \in \mathbb{R}$ (with respect to the weak convergence of measures), so we may assume $\tau \in \mathbb{Q}$. 
Then $\varphi \wedge \tau \in \nH (X, L)$ by Proposition \ref{test configuration associated to phi wedge tau}. 
Let $(\hat{\mathcal{X}}^\tau, \hat{\mathcal{L}}^\tau)$ be a normal test configuration representing $\varphi \wedge \tau \in \nH (X, L)$. 
By Proposition \ref{filtration of envelope}, we have 
\begin{align*} 
\mathcal{F}^\lambda_{\varphi \wedge \tau} = \mathcal{F}_\varphi^{\lambda} \cap \mathcal{F}^\lambda_\tau 
&= \bigcap_{E \subset \mathcal{X}_0} \mathcal{F}_{v_E}^\lambda [\varphi (v_E)] \cap \mathcal{F}^\lambda_{v_{\mathrm{triv}}} [\tau] 
\\
&= \bigcap_{E \subset \mathcal{X}_0, v_E \neq v_{\mathrm{triv}}} \mathcal{F}_{v_E}^\lambda [\varphi (v_E)] \cap \mathcal{F}^\lambda_{v_{\mathrm{triv}}} [\tau]. 
\end{align*}
Here the last equality holds by $\tau < \varphi (v_{\mathrm{triv}})$. 
It follows by Lemma \ref{MA support via filtration} that we have 
\[ \supp \mathrm{MA} (\varphi \wedge \tau) = \{ v_{\hat{E}} ~|~ \hat{E} \subset \hat{\mathcal{X}}_0^\tau \} \subset \{ v_E ~|~ E \subset \mathcal{X}_0 \} \cup \{ v_{\mathrm{triv}} \} \]
and $\varphi \wedge \tau (v_{\hat{E}}) = \varphi (v_E)$ if $v_{\hat{E}} = v_E$ for $E \subset \mathcal{X}_0$ and $\varphi \wedge \tau (v_{\hat{E}}) = \tau$ if $v_{\hat{E}} = v_{\mathrm{triv}}$. 

We note 
\[ \mathcal{F}^\lambda_{\varphi \wedge \tau} R_m = \begin{cases} \mathcal{F}^\lambda_\varphi R_m & \tau \ge \lambda/m \\ 0 & \tau < \lambda/m \end{cases}. \]
If $v_{\hat{E}} \neq v_{\mathrm{triv}}$, then we have $v_{\hat{E}} = v_E$ for some $E \subset \mathcal{X}_0$. 
Using Proposition \ref{primary decomposition of DH measure via hat filtration} and Proposition \ref{tomography of primary DH measure}, we compute 
\begin{align*} 
\frac{1}{n!} \mathrm{ord}_{\hat{E}} \hat{\mathcal{X}}^\tau_0 \cdot (\hat{E}. (\hat{\mathcal{L}}^\tau)^{\cdot n}) 
&= \lim_{m \to \infty} \frac{1}{m^n} \sum_{\lambda \in \mathbb{Q}} \dim \frac{\mathcal{F}^{\lambda}_{\varphi \wedge \tau} R_m}{\mathcal{F}^{\lambda}_{\varphi \wedge \tau} \cap \mathcal{F}^{\lambda+}_{v_{\hat{E}}} [\varphi \wedge \tau (v_{\hat{E}})] R_m} 
\\
&= \lim_{m \to \infty} \frac{1}{m^n} \sum_{\lambda \in \mathbb{Q}} \dim \frac{\mathcal{F}^\lambda_{\varphi \wedge \tau} R_m}{\mathcal{F}^\lambda_{\varphi \wedge \tau} \cap \mathcal{F}^{\lambda+}_{v_E} [\varphi (v_E)] R_m}
\\
&= \lim_{m \to \infty} \frac{1}{m^n} \sum_{\lambda \in \mathbb{Q}, \lambda/m \le \tau} \dim \frac{\mathcal{F}^{\lambda}_\varphi R_m}{\mathcal{F}^{\lambda}_\varphi \cap \mathcal{F}^{\lambda+}_{v_E} [\varphi (v_E)] R_m}
\\
&= \int_\mathbb{R} \lim_{m \to \infty} \frac{1}{m^n} \sum_{\lambda \in \mathbb{Q}, \lambda/m \le \tau} \dim \frac{\widehat{\mathcal{F}}^{\lambda}_{(\mathcal{X}, \mathcal{L})} R_m}{\widehat{\mathcal{F}}^{\lambda}_{(\mathcal{X}, \mathcal{L})} \cap \mathcal{F}^{\lambda+}_{v_E} [\varphi (v_E)] R_m}. \delta_{\lambda/m}
\\
&= \mathrm{ord}_E \mathcal{X}_0 \int_{(-\infty, \tau]} \DHm_{(E, \mathcal{L}|_E)} 
\\
&= \mathrm{ord}_E \mathcal{X}_0 \int_{(-\infty, \tau)} \DHm_{(E, \mathcal{L}|_E)},
\end{align*}
where  in the last two equalities we used the fact that $\DHm_{(E, \mathcal{L}|_E)}$ is absolutely continuous with respect to the Lebesgue measure as $v_E \neq v_{\mathrm{triv}}$. 

Now we recall 
\[ \int_{X^{\mathrm{NA}}} \mathrm{MA} (\varphi_{(\mathcal{X}, \mathcal{L})} \wedge \tau) = (e^L) = \int_\mathbb{R} \DHm_{(\mathcal{X}, \mathcal{L})} = \sum_{E \subset \mathcal{X}_0} \mathrm{ord}_E \mathcal{X}_0 \int_\mathbb{R} \DHm_{(E, \mathcal{L}|_E)}. \]
It follows that for $\hat{E} \subset \hat{\mathcal{X}}^\tau_0$ with $v_{\hat{E}} = v_{\mathrm{triv}}$, we must have 
\begin{align*} 
\frac{1}{n!} \mathrm{ord}_E \hat{\mathcal{X}}^\tau_0 \cdot (\hat{E}. (\hat{\mathcal{L}}^\tau)^{\cdot n}) 
&= \sum_{E \subset \mathcal{X}_0} \mathrm{ord}_E \mathcal{X}_0 \int_\mathbb{R} \DHm_{(E, \mathcal{L}|_E)} - \sum_{E \subset \mathcal{X}_0, v_E \neq v_{\mathrm{triv}}} \mathrm{ord}_E \mathcal{X}_0 \int_{(-\infty, \tau)} \DHm_{(E, \mathcal{L}|_E)} 
\\
&= \sum_{E \subset \mathcal{X}_0, v_E = v_{\mathrm{triv}}} \mathrm{ord}_E \mathcal{X}_0 \int_{(-\infty, \tau)} \DHm_{(E, \mathcal{L}|_E)} + \int_{[\tau, \infty)} \DHm_{(\mathcal{X}, \mathcal{L})}
\\
&= \sum_{E \subset \mathcal{X}_0, v_E = v_{\mathrm{triv}}} \mathrm{ord}_E \mathcal{X}_0 \cdot (E. \mathcal{L}^{\cdot n}) + \int_{[\tau, \infty)} \DHm^\circ_{(\mathcal{X}, \mathcal{L})}. 
\end{align*}
Since $\tau < \sup \varphi$, this is also equal to $\int_{[\tau, \infty)} \DHm_{(\mathcal{X}, \mathcal{L})}$. 

Alternatively, for $\hat{E} \subset \hat{\mathcal{X}}^\tau_0$ with $v_{\hat{E}} = v_{\mathrm{triv}}$, we have $\varphi \wedge \tau (v_{\mathrm{triv}}) = \tau$ by $\tau < \varphi (v_{\mathrm{triv}})$, so that we compute 
\begin{align*} 
\frac{1}{n!} \mathrm{ord}_E \hat{\mathcal{X}}^\tau_0 \cdot (\hat{E}. (\hat{\mathcal{L}}^\tau)^{\cdot n}) 
&= \lim_{m \to \infty} \frac{1}{m^n} \sum_{\lambda \in \mathbb{Q}} \dim \frac{\mathcal{F}^\lambda_{\varphi \wedge \tau} R_m}{\mathcal{F}^\lambda_{\varphi \wedge \tau} \cap \mathcal{F}^{\lambda+}_{v_{\mathrm{triv}}} [\tau] R_m}
\\
&= \lim_{m \to \infty} \frac{1}{m^n} \dim \widehat{\mathcal{F}}^{m \tau}_{(\mathcal{X}, \mathcal{L})} R_m 
\\
&= \mathrm{vol} (R^{(\tau)}). 
\end{align*}
Recall $\int_{[\tau, \infty)} \DHm_{(\mathcal{X}, \mathcal{L})}$ is the left continuous modification of $\mathrm{vol} (R^{(\tau)})$. 
For $\tau < \sup \varphi$, $\mathrm{vol} (R^{(\tau)})$ is continuous, so we have 
\[ \mathrm{vol} (R^{(\tau)}) = \int_{[\tau, \infty)} \DHm_{(\mathcal{X}, \mathcal{L})}. \]
\end{proof}

\subsubsection{Moment measure}
\label{Moment measure}

Now we construct the moment measure $\int \chi \mathcal{D}_\varphi$. 
The key in the construction is the following formula we obtained in the previous section 
\[ \mathrm{MA} (\varphi \wedge \tau) = \int 1_{(-\infty, \tau)} \mathcal{D}_\varphi + \int_{[\tau, \infty)} \DHm_\varphi . \delta_{v_{\mathrm{triv}}}. \]

\begin{thm}[Moment measure]
For $\varphi \in \mathcal{E}^1_{\mathrm{NA}} (X, L)$ and a Borel measurable function $\chi$ on $\mathbb{R}$ with $\int_{\mathbb{R}} |\chi| \DHm_\varphi < \infty$, we can assign a signed Radon measure $\int \chi \mathcal{D}_\varphi$ on $X^{\mathrm{NA}}$ which enjoys the following properties: 
\begin{enumerate}
\item For $\varphi = \varphi_{(\mathcal{X}, \mathcal{L})} \in \nH (X, L)$, we have 
\[ \int \chi \mathcal{D}_{\varphi_{(\mathcal{X}, \mathcal{L})}} = \sum_{E \subset \mathcal{X}_0} \mathrm{ord}_E \mathcal{X}_0 \int_{\mathbb{R}} \chi \DHm_{(E, \mathcal{L}|_E)}. \delta_{v_E}. \]

\item $\int \chi \mathcal{D}_\varphi$ is linear on $\chi$. 
If $\chi \ge 0$, $\int \chi \mathcal{D}_\varphi$ is non-negative. 

\item For any pointwise convergent increasing sequence $0 \le \chi_i \nearrow \chi$, we have the weak convergence of measures
\[ \int \chi_i \mathcal{D}_\varphi \nearrow \int \chi \mathcal{D}_\varphi. \]

\item We have $\iint_{X^{\mathrm{NA}}} \chi \mathcal{D}_\varphi := \int_{X^{\mathrm{NA}}} \int \chi \mathcal{D}_\varphi = \int_{\mathbb{R}} \chi \DHm_\varphi$. 

\item We have $\int 1_\mathbb{R} \mathcal{D}_\varphi = \mathrm{MA} (\varphi)$ as measures. 

\item Suppose $\chi$ is moderate in the sense of Definition \ref{moderate}. 
Then for a convergent decreasing net $\varphi_i \searrow \varphi \in \E^1 (X, L)$, we have the weak convergence of measures
\[ \int \chi \mathcal{D}_{\varphi_i} \to \int \chi \mathcal{D}_\varphi. \]
\end{enumerate}
These properties characterize the measure $\int \chi \mathcal{D}_\varphi$. 
\end{thm}

\begin{proof}
For $\varphi \in \E^1 (X, L)$, we put 
\begin{equation} 
\label{moment measure for interval}
\int 1_{[\tau', \tau)} \mathcal{D}_\varphi := \mathrm{MA} (\varphi \wedge \tau) - \mathrm{MA} (\varphi \wedge \tau') + \int_{[\tau', \tau)} \DHm_\varphi. \delta_{v_{\mathrm{triv}}}. 
\end{equation}
This gives a non-neagitve Borel measure on $X^{\mathrm{NA}}$ as we see below. 
By Proposition \ref{MA of wedge}, we have 
\[ \int 1_{[\tau', \tau)} \mathcal{D}_{\varphi_{(\mathcal{X}, \mathcal{L})}} = \sum_{E \subset \mathcal{X}_0} \mathrm{ord}_E \mathcal{X}_0 \int_{[\tau', \tau)} \DHm_{(E, \mathcal{L}|_E)}. \delta_{v_E} \]
for $\varphi = \varphi_{(\mathcal{X}, \mathcal{L})}$, so it defines a non-negative measure in this case. 
Since $\mathrm{MA} (\varphi \wedge \tau)$ and $\int_{[\tau', \tau)} \DHm_\varphi$ are continuous along decreasing nets in $\E^1 (X, L)$ (note $1_{[\tau', \tau)}$ is moderate), we have 
\[ \lim_{i \to \infty} \int 1_{[\tau', \tau)} \mathcal{D}_{\varphi_i} = \int 1_{[\tau', \tau)} \mathcal{D}_\varphi \]
for any convergent decreasing net $\varphi_i \searrow \varphi$ in $\E^1 (X, L)$. 
In particular, we can write $\int 1_{[\tau', \tau)} \mathcal{D}_\varphi$ as the limit of non-negative measures $\int 1_{[\tau', \tau)} \mathcal{D}_{\varphi_{(\mathcal{X}_i, \mathcal{L}_i)}}$ for a regularization $\varphi_{(\mathcal{X}_i, \mathcal{L}_i)} \searrow \varphi$, so that it gives a non-negative measure on $X^{\mathrm{NA}}$ for general $\varphi \in \E^1 (X, L)$. 
We also get 
\begin{align*} 
\int_{X^{\mathrm{NA}}} \int 1_{[\tau', \tau)} \mathcal{D}_\varphi 
&= \lim_{i \to \infty} \int_{X^{\mathrm{NA}}} \int 1_{[\tau', \tau)} \mathcal{D}_{\varphi_{(\mathcal{X}_i, \mathcal{L}_i)}} 
\\
&= \lim_{i \to \infty} \int_{[\tau', \tau)} \DHm_{\varphi_{(\mathcal{X}_i, \mathcal{L}_i)}} = \int_{[\tau', \tau)} \DHm_\varphi. 
\end{align*}

Let $\mathcal{R}$ be the set consisting of subsets of $\mathbb{R}$ which can be written as a finite sum of half open intervals of finite length: $A \in \mathcal{R}$ iff $A = \bigcup_{i=1}^k [\tau'_i, \tau_i)$, where we may assume $[\tau'_i, \tau_i)$ are disjoint each other. 
For $A \in \mathcal{R}$, we write it by a disjoint sum $A = \bigcup_{i=1}^k [\tau'_i, \tau_i)$ and we put 
\begin{equation}
\label{moment measure for A} 
\int 1_A \mathcal{D}_\varphi := \sum_{i=1}^k \int 1_{[\tau'_i, \tau_i)} \mathcal{D}_\varphi. 
\end{equation}
For $\varphi = \varphi_{(\mathcal{X}, \mathcal{L})}$, we have 
\[ \int 1_A \mathcal{D}_{\varphi_{(\mathcal{X}, \mathcal{L})}} = \sum_{E \subset \mathcal{X}_0} \mathrm{ord}_E \mathcal{X}_0 \int_A \DHm_{(E, \mathcal{L}|_E)}. \delta_{v_E}. \]
We also have 
\[ \lim_{i \to \infty} \int 1_A \mathcal{D}_{\varphi_i} = \int 1_A \mathcal{D}_\varphi \]
for any convergent decreasing net $\varphi_i \searrow \varphi$ in $\E^1 (X, L)$. 

For a non-negative continuous function $g$ on $X^{\mathrm{NA}}$, we put 
\begin{equation}
\label{Ig}
 I_{\varphi, A} (g) := \int_{X^{\mathrm{NA}}} g \int 1_A \mathcal{D}_\varphi \ge 0. 
 \end{equation}
By Lemma \ref{Ig is measure} below, the assignment $A \mapsto I_{\varphi, A} (g)$ satisfies the sigma additivity. 
Thus by Carath\'eodory's extension theorem, the following outer measure defines a finite Borel measure on $\mathbb{R}$: 
\begin{equation} 
\label{IB}
\nu_{\varphi, g} (B) := I_{\varphi, B} (g) := \inf \Big{\{} \sum_{i=1}^\infty \int_{X^{\mathrm{NA}}} g \int 1_{A_i} \mathcal{D}_\varphi ~\Big{|}~ B \subset \bigcup_{i=1}^\infty A_i, A_i \in \mathcal{R} \Big{\}}. 
\end{equation}
We have $I_{\varphi, B} (1) = \int_B \nu_{\varphi, 1} = \int_B \DHm_\varphi$, as it holds for $B \in \mathcal{R}$ and $\nu_{\varphi, 1}, \DHm_\varphi$ are outer regular by the finiteness. 
(Recall any open set in $\mathbb{R}$ can be expressed as a countable sum of $A_i \in \mathcal{R}$. )

For a non-negative Borel measurable function $\chi$ on $\mathbb{R}$, we put 
\begin{equation}
\label{Ichi} 
I_{\varphi, \chi} (g) := \int_\mathbb{R} \chi \nu_{\varphi, g}. 
\end{equation}
Since $\nu_{\varphi, g} \le \sup |g| \cdot \DHm_\varphi$ by $I_{\varphi, A} (g) \le \mathrm{sup} |g| \cdot I_{\varphi, A} (1)$, we have 
\[ 0 \le I_{\varphi, \chi} (g) \le \sup |g| \int \chi \DHm_\varphi. \]
We obviously have $I_{\varphi, \chi'} (g) \le I_{\varphi, \chi} (g)$ for $\chi' \le \chi$. 
By the monotone convergence theorem, we have $I_{\varphi, \chi_i} (g) \nearrow I_{\varphi, \chi} (g)$ for any increasing pointwise convergent sequence $\chi_i \nearrow \chi$ of non-negative Borel measurable functions. 

Now suppose $\int_\mathbb{R} \chi \DHm_\varphi < \infty$, then by Lemma \ref{IB is linear} below, $I_{\varphi, \chi}$ extends to a positive bounded linear function on $C^0 (X^{\mathrm{NA}})$ in a canonical way. 
Therefore, by Riesz--Markov--Kakutani representation theorem, we get a Radon measure on $X^{\mathrm{NA}}$ which we denote by $\int \chi \mathcal{D}_\varphi$ such that 
\[ I_{\varphi, \chi} (g) = \int_{X^{\mathrm{NA}}} g \int \chi \mathcal{D}_\varphi. \]

The property (2)--(4) on the measure $\int \chi \mathcal{D}_\varphi$ follows immediately from the construction. 
The first property (1) follows by 
\begin{align*} 
\int_{X^{\mathrm{NA}}} g \int 1_B \mathcal{D}_{\varphi_{(\mathcal{X}, \mathcal{L})}} 
&= \inf \Big{\{} \sum_{i=1}^\infty \int_{X^{\mathrm{NA}}} g \int 1_{A_i} \mathcal{D}_{\varphi_{(\mathcal{X}, \mathcal{L})}} ~\Big{|}~ B \subset \bigcup_{i=1}^\infty A_i, A_i \in \mathcal{R} \Big{\}} 
\\
&= \inf \Big{\{} \sum_{E \subset \mathcal{X}_0} g (v_E) \mathrm{ord}_E \mathcal{X}_0 \sum_{i=1}^\infty \int_{A_i} \DHm_{(E, \mathcal{L}|_E)} ~\Big{|}~ B \subset \bigcup_{i=1}^\infty A_i \Big{\}}
\\
&= \sum_{E \subset \mathcal{X}_0} g (v_E) \mathrm{ord}_E \mathcal{X}_0 \int_B \DHm_{(E, \mathcal{L}|_E)} 
\end{align*}
and by the continuity with respect to increasing limit $\chi_i \nearrow \chi$. 
Here we note Lemma \ref{uniform outer regularity} for the last equality.

To see the property (5), we note $\int 1_{\mathbb{R}} \mathcal{D}_\varphi = \mathrm{MA} (\varphi)$ holds for $\varphi \in \nH (X, L)$ by the first property. 
Then the general case is reduced to the last property (6). 

Let $\varphi_i \searrow \varphi \in \E^1 (X, L)$ be a convergent decreasing net. 
For a moderate $\chi$, we must show 
\[ \int \chi \mathcal{D}_{\varphi_i} \to \int \chi \mathcal{D}_\varphi \]
in the weak sense. 
The claim is reduced to the following cases: (i) $\chi$ is tame, (ii) $\chi$ is right continuous decreasing. 
We note the following. 
Here step function means a function of the form $\sum_{j=1}^k a_j 1_{A_j}$ for $A_j \in \mathcal{R}$. 

(i) If $\chi$ is tame, we have a sequence of step functions $\chi_j$ converging to $\chi$ uniformly on $(-\infty, \sup \varphi_0]$. 

(ii) If $\chi$ is a right continuous decreasing function, in a similar way as in the proof of Lemma \ref{monotonicity}, we can find an increasing sequence of step functions $\chi_j$ pointwisely converging to $\chi$ on $(-\infty, \sup \varphi_0]$. 

Now we put
\[ \mathcal{S} := \Big{\{} \chi: \text{ Borel measurable function on } (-\infty, \sup \varphi_0] ~\Big{|}~ \int \chi \mathcal{D}_{\varphi_i} \to \int \chi \mathcal{D}_\varphi \Big{\}}. \]
As we already know step functions are in $\mathcal{S}$, the claim is reduced to the following generalities. 

(a) If $\chi$ is a uniform limit of some sequence $\chi_j \in \mathcal{S}$, then $\chi$ is in $\mathcal{S}$. 
We can easily show this similarly as the proof of Proposition \ref{tame}. 

(b) If $\chi$ is a pointwise limit of an increasing sequence $\{ \chi_j \}_{j \in \mathbb{N}} \subset \mathcal{S}$ and satisfies $\int_\mathbb{R} \chi \DHm_{\varphi_i} \to \int_\mathbb{R} \chi \DHm_\varphi$, then $\chi$ is in $\mathcal{S}$. 

Firstly, for any non-negative $g \in C^0 (X^{\mathrm{NA}})$, we have 
\[ \varliminf_{i \to \infty} \int_{X^{\mathrm{NA}}} g \int \chi \mathcal{D}_{\varphi_i} \ge \lim_{i \to \infty} \int_{X^{\mathrm{NA}}} g \int \chi_j \mathcal{D}_{\varphi_i} = \int_{X^{\mathrm{NA}}} g \int \chi_j \mathcal{D}_\varphi. \]
Thus by the property (3) of the measure (note $\chi_j$ are bounded from below on $(-\infty, \sup \varphi_0]$), we get 
\[ \varliminf_{i \to \infty} \int_{X^{\mathrm{NA}}} g \int \chi \mathcal{D}_{\varphi_i} \ge \int_{X^{\mathrm{NA}}} g \int \chi \mathcal{D}_\varphi. \]

It suffices to see the reverse inequality. 
Since $\int_\mathbb{R} \chi_j \DHm_\varphi \nearrow \int_\mathbb{R} \chi \DHm_\varphi$ by the monotone convergence theorem, for any $\varepsilon > 0$ we can take $j \in \mathbb{N}$ large so that $\int_{\mathbb{R}} \chi \DHm_\varphi \le \int \chi_j \DHm_\varphi + \varepsilon$. 
Fix such $j$ and put $\tilde{\chi} := \chi_j$. 
By the assumption, we have $\int_\mathbb{R} \chi \DHm_{\varphi_i} \to \int_\mathbb{R} \chi \DHm_\varphi$ and $\int \tilde{\chi} \DHm_{\varphi_i} \to \int \tilde{\chi} \DHm_\varphi$, so that we can take $i_\varepsilon \in I$ so that $\int \chi \DHm_{\varphi_i} \le \int \chi \DHm_\varphi + \varepsilon$ and $\int \tilde{\chi} \DHm_{\varphi_i} \le \int \tilde{\chi} \DHm_\varphi + \varepsilon$ for every $i \ge i_\varepsilon$. 
Thus we get $\int \chi \DHm_{\varphi_i} \le \int \tilde{\chi} \DHm_{\varphi_i} + 3 \varepsilon$ for $i \ge i_\varepsilon$. 

It follows that $\int (\chi - \tilde{\chi}) \mathcal{D}_{\varphi_i}$ gives a non-negative measure on $X^{\mathrm{NA}}$ whose total mass $\int_{X^{\mathrm{NA}}} \int (\chi - \tilde{\chi}) \mathcal{D}_{\varphi_i} = \int (\chi - \tilde{\chi}) \DHm_{\varphi_i}$ is no greater than $3 \varepsilon$. 
Now since 
\[ \int_{X^{\mathrm{NA}}} g \int \chi \mathcal{D}_{\varphi_i} = \int_{X^{\mathrm{NA}}} g \int \tilde{\chi} \mathcal{D}_{\varphi_i} + \int_{X^{\mathrm{NA}}} g \int (\chi - \tilde{\chi}) \mathcal{D}_{\varphi_i} \le \int_{X^{\mathrm{NA}}} g \int \tilde{\chi} \mathcal{D}_{\varphi_i} + 3 \varepsilon \cdot \sup |g|, \]
we have 
\[ \varlimsup_{i \to \infty} \int_{X^{\mathrm{NA}}} g \int \chi \mathcal{D}_{\varphi_i} \le \int_{X^{\mathrm{NA}}} g \int \tilde{\chi} \mathcal{D}_\varphi + 3 \varepsilon \cdot \sup |g| \le \int_{X^{\mathrm{NA}}} g \int \chi \mathcal{D}_\varphi + 3 \varepsilon \cdot \sup |g|. \]
Now we can take $\varepsilon$ arbitrary small and get 
\[ \varlimsup_{i \to \infty} \int_{X^{\mathrm{NA}}} g \int \chi \mathcal{D}_{\varphi_i} \le \int_{X^{\mathrm{NA}}} g \int \chi \mathcal{D}_\varphi. \]
Therefore, we conclude
\[ \lim_{i \to \infty} \int_{X^{\mathrm{NA}}} g \int \chi \mathcal{D}_{\varphi_i} = \int_{X^{\mathrm{NA}}} g \int \chi \mathcal{D}_\varphi. \]

%In the above argument, we only use $\varlimsup_{i \to \infty} \int_\mathbb{R} \chi \DHm_{\varphi_i} \le \int_\mathbb{R} \chi \DHm_\varphi$ rather than $\lim_{i \to \infty} \int_\mathbb{R} \chi \DHm_{\varphi_i} = \int_\mathbb{R} \chi \DHm_\varphi$. 
%Thus it shows that whenever $\chi$ is the pointwise limit of increasing sequence in $\mathcal{S}$ satisfying 
%\[ \varlimsup_{i \to \infty} \int_\mathbb{R} \chi \DHm_{\varphi_i} \le \int_\mathbb{R} \chi \DHm_\varphi, \]
%we have 
%\[ \int \chi \mathcal{D}_{\varphi_i} \to \int \chi \mathcal{D}_\varphi, \]
%hence actually have $\lim_{i \to \infty} \int_\mathbb{R} \chi \DHm_{\varphi_i} = \int_\mathbb{R} \chi \DHm_\varphi$. 
\end{proof}

\begin{lem}
\label{Ig is measure}
For any $\varphi \in \E^1 (X, L)$ and $g \in C^0 (X^{\mathrm{NA}})$, the assignment 
\[ \mathcal{R} \to \mathbb{R}: A \mapsto I_A := I_{\varphi, A} (g) \]
given by (\ref{Ig}) satisfies the sigma additivity. 
Namely, if $\{ A_i \}_{i=1}^\infty \subset \mathcal{R}$ is a countable disjoint collection with $A := \bigcup_{i=1}^\infty A_i \in \mathcal{R}$, then $I_A = \sum_{i=1}^\infty I_{A_i}$
\end{lem}

\begin{proof}
We firstly note that for any $A \in \mathcal{R}$ and $\epsilon > 0$, there exists $A' \in \mathcal{R}$ such that $\bar{A}' \subset A$ and $I_{A \setminus A'} < \epsilon$. 
To see this, we may assume $A = [\tau', \tau)$. 
Since
\[ I_{[\tau', \tau)} = \int_{X^{\mathrm{NA}}} g \int 1_{[\tau', \tau)} \mathcal{D}_\varphi \le \sup |g| (F_\varphi (\tau') - F_\varphi (\tau)) \]
and $F_\varphi$ is left continuous, we have $\int_{X^{\mathrm{NA}}} g \int 1_{[\tau_i, \tau)} \mathcal{D}_\varphi \to 0$ for $\tau_i \nearrow \tau$, so that $A' := [\tau_i, \tau)$ satisfies the demand for large $i$. 
We also note if $A \subset \bigcup_{i=1}^k A_i$ for $A, A_i \in \mathcal{R}$, then we have $I_A \le \sum_{i=1}^k I_{A_i}$. 

The rest of argument is just a reproduction of basic arguments in measure theory. 
We prepare the following: if $\{ A_i \}_{i=1}^\infty \in \mathcal{R}$ is a decreasing sequence with $\bigcap_{i=1}^\infty A_i = \emptyset$, then we have $I_{A_i} \to 0$. 
Indeed, suppose there exists $\delta > 0$ such that $I_{A_i} \ge \delta$ for all $i$. 
We pick $A'_i$ so that $\bar{A}_i' \subset A_i$ and $I_{A_i \setminus A'_i} < \delta/2^i$. 
Since $\bar{A}'_1 \cap \bigcap_{i=2}^\infty \bar{A}'_i \subset \bigcap_{i=1}^\infty E_i = \emptyset$ and $\bar{A}'_1$ is compact, there exists $N \in \mathbb{N}_+$ such that $\bigcap_{i=1}^N A'_i \subset \bigcap_{i=1}^N \bar{A}'_i = \emptyset$ by the finite intersection property. 
Then since 
\[ A_N = A_N \setminus \bigcap_{i=1}^N A'_i = \bigcup_{i=1}^N (A_N \setminus A'_i) \subset \bigcup_{i=1}^N (A_i \setminus A'_i), \]
we get 
\[ I_{A_N} \le \sum_{i=1}^N I_{A_i \setminus A'_i} \le \sum_{i=1}^N \delta/2^i < \delta, \]
which contradicts to the assumption $I_{A_i} \ge \delta$. 

Now if $\{ A_i \}_{i=1}^\infty \subset \mathcal{R}$ is a countable disjoint collection with $A := \bigcup_{i=1}^\infty A_i \in \mathcal{R}$, then we compute 
\begin{align*} 
I_A = I_{\bigcup_{i=1}^k A_i} + I_{A \setminus \bigcup_{i=1}^k A_i} = \sum_{i=1}^k I_{A_i} + I_{A \setminus \bigcup_{i=1}^k A_i}. 
\end{align*}
By taking the limit $k \to \infty$, we get $I_A = \sum_{i=1}^\infty I_{A_i}$. 
\end{proof}

\begin{lem}
\label{uniform outer regularity}
Let $\nu_1, \ldots, \nu_k$ be finite Borel measures on $\mathbb{R}$. 
For any Borel subset $B \subset \mathbb{R}$ and $\varepsilon > 0$, there exists a countable disjoint collection $\{ A_i \}_{i=1}^\infty \subset \mathcal{R}$ such that $B \subset \bigcup_{i=1}^\infty A_i$ and 
\[ \nu_j (B) \ge \sum_{i=1}^\infty \nu_j (A_i) - \varepsilon \]
for each $j=1, \ldots, k$. 
\end{lem}

\begin{proof}
As $\nu_1, \ldots, \nu_k$ are outer regular, we can take $\{ A^1_i \}_{i=1}^\infty, \ldots, \{ A^k_i \}_{i=1}^\infty$ so that 
\[ \nu_j (B) \ge \sum_{i=1}^\infty \nu_j (A^j_i) - \varepsilon \]
for each $j=1, \ldots, k$. 
By replacing $A^j_i$ with $A^j_i \setminus \bigcup_{l=1}^{i-1} A^j_l$, we may assume $A^j_i$ are disjoint with each other, for each $j$. 
Now consider the countable collection $\{ A^1_{i_1} \cap \dotsb \cap A^k_{i_k} \}_{i_1, \ldots, i_k = 1}^\infty$ and renumber it as $\{ A'_i \}_{i=1}^\infty$. 
We have 
\[ \bigcup_{i=1}^\infty A'_i = \bigcup_{i_1, \ldots, i_k =1}^\infty A^1_{i_1} \cap \dotsb \cap A^k_{i_k} = \bigcup_{i_1 =1}^\infty A^1_{i_1} \cap \dotsb \cap \bigcup_{i_k =1}^\infty A^k_{i_k} \supset B. \]
Put $A_i := A'_i \setminus \bigcup_{l=1}^{i-1} A'_l$, then $A_i$ are disjoint, $B \subset \bigcup_{i=1}^\infty A_i$ and $\bigcup_{i=1}^\infty A_i \subset \bigcup_{i=1}^\infty A^j_i$ for each $j$. 
Since $\nu_j$ are measures, we have 
\[ \sum_{i=1}^\infty \nu_j (A^j_i) = \nu_j (\bigcup_{i=1}^\infty A^j_i) \ge \nu_j (\bigcup_{i=1}^\infty A_i) = \sum_{i=1}^\infty \nu_j (A_i).  \]
Thus we get 
\[ \nu_j (B) \ge \sum_{i=1}^\infty \nu_j (A_i) - \varepsilon. \] 
\end{proof}

\begin{lem}
\label{IB is linear}
For any $\varphi \in \mathcal{E}^1 (X, L)$ and any Borel measurable function $\chi$ on $\mathbb{R}$, $I_{\varphi, \chi}$ defined in (\ref{Ichi}) satisfies the following. 
\begin{enumerate}
\item $I_{\varphi, \chi} (ag) = a I_{\varphi, \chi} (g)$ for any non-negative $a \in \mathbb{R}$ and $g \in C^0 (X^{\mathrm{NA}})$. 

\item $I_{\varphi, \chi} (g_1 + g_2) = I_{\varphi, \chi} (g_1) + I_{\varphi, \chi} (g_2)$ for non-negative $g_1, g_2 \in C^0 (X^{\mathrm{NA}})$. 
\end{enumerate}
\end{lem}

\begin{proof}
Since $I_{\varphi, \chi} (g) = \int \chi \nu_{\varphi, g}$, it suffices to show $\nu_{\varphi, ag} = a \nu_{\varphi, g}$ and $\nu_{\varphi, g_1 + g_2} = \nu_{\varphi, g_1} + \nu_{\varphi, g_2}$, which is equivalent to say that the claim holds for $I_{\varphi, B}$ for every Borel set $B \subset \mathbb{R}$. 
We firstly note the claim holds for $A \in \mathcal{R}$ as $I_{\varphi, A} (g)$ is defined by the integration of $g$ with respect to the measure $\int 1_A \mathcal{D}_\varphi$ in (\ref{moment measure for A}). 

Now we check $I_{\varphi, B} (g_1) + I_{\varphi, B} (g_2) \ge I_{\varphi, B} (g_1 + g_2)$. 
For $\varepsilon > 0$, we pick $\{ A_i \}_{i=1}^\infty \subset \mathcal{R}$ as in Lemma \ref{uniform outer regularity} with respect to the measures $\nu_{\varphi, g_1}, \nu_{\varphi, g_2}$: $B \subset \bigcup_{i=1}^\infty A_i$ and 
\[ \nu_{\varphi, g_j} (B) \ge \sum_{i=1}^\infty \nu_{\varphi, g_j} (A_i) - \varepsilon \]
for both $j=1, 2$. 
Then we have 
\[ \nu_{\varphi, g_1} (B) + \nu_{\varphi, g_2} (B) \ge \sum_{i=1}^\infty (\nu_{\varphi, g_1} (A_i) + \nu_{\varphi, g_2} (A_i)) - 2 \varepsilon. \]
Since the claim holds for $A_i \in \mathcal{R}$, we obtain 
\[ \nu_{\varphi, g_1} (B) + \nu_{\varphi, g_2} (B) \ge \sum_{i=1}^\infty \nu_{\varphi, g_1 + g_2} (A_i) - 2 \varepsilon \ge \nu_{\varphi, g_1+ g_2} (B) - 2 \varepsilon. \]
As we took $\varepsilon > 0$ arbitrary, we obtain 
\[ \nu_{\varphi, g_1} (B) + \nu_{\varphi, g_2} (B) \ge \nu_{\varphi, g_1+ g_2} (B). \]
The rest of the claim follows immediately from the definition of $I_{\varphi, B} (g)$. 

\end{proof}

\subsubsection{Tomographic expression of moment measure}
\label{Tomographic expression of moment measure}

For smooth $\chi$, we have the following formula. 
Here we use the dominant convergence theorem, so we employ the countable regularization \cite[Theorem 9.11]{BJ3} for $\varphi \in \E^1 (X, L)$. 

\begin{prop}
\label{moment measure via distribution}
Let $\chi$ be a non-negative compactly supported smooth function on $\mathbb{R}$ and $\varphi \in \E^1 (X, L)$. 
If either $\psi \in C^0 (X^{\mathrm{NA}})$ or $\psi \in \E^1 (X, L)$, then $\int_{X^{\mathrm{NA}}} \psi \mathrm{MA} (\varphi \wedge \tau)$ is a continuous function on $\tau$ and we have 
\[ \int_{X^{\mathrm{NA}}} \psi \int \chi \mathcal{D}_\varphi = \int_\mathbb{R} d\tau ~ \chi (\tau) \frac{d}{d\tau} \int_{X^{\mathrm{NA}}} \psi \mathrm{MA} (\varphi \wedge \tau) + \psi (v_{\mathrm{triv}}) \int_\mathbb{R} \chi \DHm_\varphi, \]
where we identify $\frac{d}{d\tau} \int_{X^{\mathrm{NA}}} \psi \mathrm{MA} (\varphi \wedge \tau)$ with the distributional derivative. 
\end{prop}

The claim includes that $\psi \in \E^1 (X, L)$ is integrable with respect to $\int \chi \mathcal{D}_\varphi$. 

\begin{proof}
We firstly note that the integrations make sense. 
The left hand side makes sense for any usc function $\psi$. 
Indeed, usc function is Borel measurable and is bounded from above by the compactness of $X^{\mathrm{NA}}$, so that we can define the integration by the integration of non-negative Borel measurable function 
\[ \int_{X^{\mathrm{NA}}} \psi \int \chi \mathcal{D}_\varphi := - \int_{X^{\mathrm{NA}}} (\sup \psi - \psi) \int \chi \mathcal{D}_\varphi + \sup \psi \int_\mathbb{R} \chi \DHm_\varphi, \]
though it may take value $-\infty$ ($\psi$ may be non-integrable). 

As for the right hand side, we note $\int_{X^{\mathrm{NA}}} \psi \mathrm{MA} (\varphi \wedge \tau)$ is continuous on $\tau$ thanks to Proposition \ref{strong convergence of rooftops}. 
In particular, $-\chi' (\tau) \int_{X^{\mathrm{NA}}} \psi \mathrm{MA} (\varphi \wedge \tau)$ is integrable with respect to $d\tau$, hence the distributional derivative makes sense. 

We check the right hand side is continuous along decreasing \textit{sequences}, applying the bounded convergence theorem. 
Recall for any convergent decreasing net $\varphi_i \searrow \varphi \in \E^1 (X, L)$, we have 
\[ \int_{X^{\mathrm{NA}}} \psi \mathrm{MA} (\varphi_i \wedge \tau) \to \int_{X^{\mathrm{NA}}} \psi \mathrm{MA} (\varphi \wedge \tau) \] 
for each $\tau \in \mathbb{R}$. 
Since the support of $-\chi'$ is compact, it suffices to get a uniform bound  
\[ |\int_{X^{\mathrm{NA}}} \psi \mathrm{MA} (\varphi_i \wedge \tau)| \le C \]
independent of $\tau$ and $i$. 

When $\psi \in C^0 (X^{\mathrm{NA}})$, we have 
\[ |\int_{X^{\mathrm{NA}}} \psi \mathrm{MA} (\varphi \wedge \tau)| \le (e^L) \cdot \sup |\psi|. \]
Suppose $\psi \in \E^1 (X, L)$. 
Then by \cite[Lemma 5.28]{BJ3} (cf. \cite[Lemma 3.23]{BJ1}), we have 
\begin{align*} 
\Big{|} \int_{X^{\mathrm{NA}}} \psi \mathrm{MA} (\varphi \wedge \tau) - \psi (v_{\mathrm{triv}}) (e^L) \Big{|} 
&= \Big{|} \int_{X^{\mathrm{NA}}} \psi (\mathrm{MA} (\varphi \wedge \tau) - \mathrm{MA} (0)) \Big{|} 
\\
&\le C_n I (\varphi \wedge \tau)^{1/2} \max \{ I (\psi), I (\varphi \wedge \tau) \}^{1/2}.
\end{align*}
Here we recall 
\[ I (\varphi) := \int_{X^{\mathrm{NA}}} (\varphi (v_{\mathrm{triv}}) - \varphi) \mathrm{MA} (\varphi). \]
We have a uniform bound $I (\varphi \wedge \tau) \le C_\varphi$ independent of $\tau$: for $\tau > \sup \varphi$, we have $I (\varphi \wedge \tau) = I (\varphi)$ for $\tau > \sup \varphi$, and for $\tau \le \sup \varphi$, using \cite[Corollary 5.27]{BJ3} and the monotonicity of $E$, we get 
\begin{align*} 
I (\varphi \wedge \tau) 
&\le C_n (E (\frac{\varphi \wedge \tau}{2}) - \frac{E (\varphi \wedge \tau)}{2}) 
\\
&\le C_n (E (\frac{\tau}{2}) - \frac{E (\varphi - \sup \varphi + \tau)}{2}) = - \frac{C_n}{2} E (\varphi -\sup \varphi) =: C_\varphi. 
\end{align*}
Along $\varphi_i \searrow \varphi$, we have a uniform bound $C_\varphi \le C$. 
Therefore, we get 
\begin{equation} 
\label{uniform bound}
\Big{|} \int_{X^{\mathrm{NA}}} \psi \mathrm{MA} (\varphi \wedge \tau) \Big{|} \le |\psi (v_{\mathrm{triv}})| (e^L) + C_n \max \{ I (\psi), C \}^{1/2} C^{1/2} 
\end{equation}
as desired. 

(i) We firstly assume $\psi \in C^0 (X^{\mathrm{NA}})$. 
In this case, we already know the left hand side is continuous along decreasing nets $\varphi_i \searrow \varphi$, so the problem is reduced to the case $\varphi = \varphi_{(\mathcal{X}, \mathcal{L})} \in \nH (X, L)$, thanks to the countable regularization \cite[Theorem 9.11]{BJ3}. 

We assume $\varphi = \varphi_{(\mathcal{X}, \mathcal{L})} \in \nH (X, L)$. 
Since 
\begin{align*} 
\int_{X^{\mathrm{NA}}} \psi \int \chi \mathcal{D}_\varphi 
&= \sum_{E \subset \mathcal{X}_0} \mathrm{ord}_E \mathcal{X}_0 \int_{\mathbb{R}} \chi \DHm_{(E, \mathcal{L}|_E)} \psi (v_E) 
\\
&= \sum_{E \subset \mathcal{X}_0} \mathrm{ord}_E \mathcal{X}_0 (\psi (v_E) -\psi (v_{\mathrm{triv}})) \int_{\mathbb{R}} \chi \DHm_{(E, \mathcal{L}|_E)} + \psi (v_{\mathrm{triv}}) \int_\mathbb{R} \chi \DHm_\varphi, 
\end{align*}
it suffices to compare 
\[ \sum_{E \subset \mathcal{X}_0} \mathrm{ord}_E \mathcal{X}_0 (\psi (v_E) -\psi (v_{\mathrm{triv}})) \int_{\mathbb{R}} \chi \DHm_{(E, \mathcal{L}|_E)} \] 
with 
\[ \int_\mathbb{R} d\tau ~ \chi (\tau) \frac{d}{d\tau} \int_{X^{\mathrm{NA}}} \psi \mathrm{MA} (\varphi \wedge \tau). \]

By Proposition \ref{MA of wedge}, we have 
\[ \mathrm{MA} (\varphi \wedge \tau) = \sum_{E \subset \mathcal{X}_0} \mathrm{ord}_E \mathcal{X}_0 \int_{(-\infty, \tau)} \DHm_{(E, \mathcal{L}|_E)}. \delta_{v_E} + \int_{[\tau, \infty)} \DHm_{(\mathcal{X}, \mathcal{L})}. \delta_{v_{\mathrm{triv}}}. \]
Then we compute 
\begin{align*} 
\int_{X^{\mathrm{NA}}} \psi \mathrm{MA} (\varphi \wedge \tau) 
&= \sum_{E \subset \mathcal{X}_0} \mathrm{ord}_E \mathcal{X}_0 \int_{(-\infty, \tau)} \DHm_{(E, \mathcal{L}|_E)} \psi (v_E) + \int_{[\tau, \infty)} \DHm_{(\mathcal{X}, \mathcal{L})} \psi (v_{\mathrm{triv}}) 
\\
&=\sum_{E \subset \mathcal{X}_0} \mathrm{ord}_E \mathcal{X}_0 (\psi (v_E) - \psi (v_{\mathrm{triv}})) \int_{(-\infty, \tau)} \DHm_{(E, \mathcal{L}|_E)} 
\\
&\qquad + \int_\mathbb{R} \DHm_{(\mathcal{X}, \mathcal{L})} \psi (v_{\mathrm{triv}}). 
\end{align*}
Now the claim follows by the following general identity 
\[ \int_\mathbb{R} \chi (\tau) \frac{d}{d\tau} \Big{(} \int_{(-\infty, \tau)} d\mu \Big{)} d\tau = \int_\mathbb{R} \chi d\mu \]
for any finite Borel measure $\mu$ on $\mathbb{R}$ which can be written as a sum of an absolutely continuous measure and finitely many Dirac masses. 

(ii) Now we show the case $\psi \in \E^1 (X, L)$. 
Take a regularization $\{ \psi_i \}_{i \in I} \subset \nH (X, L)$ so that $\psi_i \searrow \psi$. 
Thanks to the monotone convergence theorem (see Proposition \ref{monotone convergence for net}), we have 
\[ \int_{X^{\mathrm{NA}}} \psi_i \int \chi \mathcal{D}_\varphi \searrow \int_{X^{\mathrm{NA}}} \psi \int \chi \mathcal{D}_\varphi, \]
so that the left hand side is continuous along decreasing nets $\psi_i \searrow \psi$. 
(At this point, the limit may be $-\infty$. ) 
Since we already show the claim for $\psi_i \in \nH (X, L)$, it suffices show the right hand side is continuous along decreasing \textit{sequences} $\psi_i \searrow \psi$, thanks to the countable regularization \cite[Theorem 9.11]{BJ3}. 

Similarly, we have 
\[ \int_{X^{\mathrm{NA}}} \psi_i \mathrm{MA} (\varphi \wedge \tau) \searrow \int_{X^{\mathrm{NA}}} \psi \mathrm{MA} (\varphi \wedge \tau) \]
for each $\tau \in \mathbb{R}$. 
By the uniform bound (\ref{uniform bound}) and the convergences $\psi_i (v_{\mathrm{triv}}) \to \psi (v_{\mathrm{triv}})$, $I (\psi_i) \to I (\psi)$, we get 
\[ \Big{|} \int_{X^{\mathrm{NA}}} \psi_i \mathrm{MA} (\varphi \wedge \tau) \Big{|} \le C \]
independent of $\tau \in \mathbb{R}$ and $i$. 
By the bounded convergence theorem, we get 
\[ \int_\mathbb{R} d \tau ~ \chi (\tau) \frac{d}{d\tau} \int_{X^{\mathrm{NA}}} \psi_i \mathrm{MA} (\varphi \wedge \tau) \to \int_\mathbb{R} d\tau~ \chi (\tau) \frac{d}{d\tau} \int_{X^{\mathrm{NA}}} \psi \mathrm{MA} (\varphi \wedge \tau), \]
which proves the continuity of the right hand side along sequences $\psi_i \searrow \psi$. 
\end{proof}

\section{Non-archimedean $\mu$-entropy on $\E^{\exp} (X, L)$}

\subsection{The metric space $\E^{\exp} (X, L)$}
\label{metric space Eexp}

\subsubsection{Strong topology, $d_1$-topology and $d_p$-topology}
\label{Strong topology, d1-topology and dp-topology}

As in \cite[Theorem 3.2]{BJ2}, for $\varphi, \varphi' \in C^0 \cap \PSH (X, L)$, we consider the \textit{relative spectral measure} 
\begin{equation} 
\DHm_{\varphi, \varphi'} := \lim_{m \to \infty} \frac{1}{m^n} \sum_{i=1}^{N_m} \delta_{\lambda_i^\varphi (\bm{s})/m - \lambda_i^{\varphi'} (\bm{s})/m}, 
\end{equation}
using a codiagonal basis $\bm{s}$ for $\varphi, \varphi'$. 
Since $\mathcal{F}_\varphi^\lambda R_m/\mathcal{F}_\varphi^{\lambda+} R_m = \langle s_i ~|~ \lambda_i^\varphi (\bm{s}) = \lambda \rangle$, we have 
\[ \DHm_{\varphi, \varphi_{\mathrm{triv}}} = \DHm_\varphi \]
for $\varphi \in C^0 \cap \PSH (X, L)$. 

We firstly review the $L^p$-distance on $\nH (X, L)$ introduced in \cite{BJ2}: for $1 \le p < \infty$, we put 
\begin{equation} 
d_p (\varphi, \varphi') := \Big{(} \int_{\mathbb{R}} |t|^p \DHm_{\varphi, \varphi'} \Big{)}^{1/p}. 
\end{equation}
As observed in \cite{BJ2}, this defines a distance on $\nH (X, L)$ (compare Proposition \ref{dexp gives a distance on H}). 

\begin{lem}
For every $\varphi, \varphi' \in \nH (X, L)$, we have 
\begin{align*} 
(e^L)^{-1} d_1 (\varphi, \varphi') 
&\le (e^L)^{-1/p} d_p (\varphi, \varphi'),
\\ 
d_r (\varphi, \varphi')^{r (q-p)} 
&\le d_p (\varphi, \varphi')^{p (q-r)} d_q (\varphi, \varphi')^{q (r-p)} 
\end{align*}
for $1 \le p \le r \le q < \infty$. 
\end{lem}

\begin{proof}
These are consequences of H\"older's inequality: as for the second inqeuality, we put $\alpha := \frac{q-r}{q-p}$, then $p \alpha + q (1-\alpha) = r$, so 
\[ \int_{\mathbb{R}} |t|^r \DHm_{\varphi, \varphi'} \le \Big{(} \int_{\mathbb{R}} |t|^p \DHm_{\varphi, \varphi'} \Big{)}^{1/p \cdot p \alpha} \Big{(} \int_{\mathbb{R}} |t|^q \DHm_{\varphi, \varphi'} \Big{)}^{1/q \cdot q (1-\alpha)}. \]
\end{proof}

For general $\varphi, \varphi' \in \E^1 (X, L)$, take a regularization $\varphi_i \searrow \varphi, \varphi'_i \searrow \varphi$ so that $\varphi_i, \varphi' \in \nH (X, L)$ and put 
\[ d_1 (\varphi, \varphi') := \lim_{i \to \infty} d_1 (\varphi_i, \varphi'_i). \]
The limit is independent of the choice of the regularization as shown in the proof of \cite[Theorem 5.4]{BJ4}. 
We will observe the same construction for another functional $d_{\exp}$ in Proposition \ref{dexp well-defined}. 
In \cite{BJ4}, the distance $d_1$ is defined in a slightly different way: 
\[ d_1 (\varphi, \varphi') = \inf \{ (E (\varphi) - E (\tilde{\varphi})) + (E (\varphi') - E (\tilde{\varphi})) ~|~ \varphi, \varphi' \ge \tilde{\varphi} \in \nH (X, L) \}, \]
which is modeled on the formula 
\[ d_1 (\varphi, \varphi') = (E (\varphi) - E (\varphi \wedge \varphi')) + (E (\varphi') - E (\varphi \wedge \varphi')) \]
under the assumption on the existence of $\varphi \wedge \varphi'$ or the continuity of envelopes. 
The limit $d_1$ obviously defines a peudo-distance on $\E^1 (X, L)$: $d_1$ satisfies all the axiom of distance except for $d_1 (\varphi, \varphi') \Rightarrow \varphi = \varphi'$. 

To check that $d_1$ is a distance, we compare $d_1$ with the following quasi-distance $\bar{I}$ from \cite{BJ3}: for $\varphi, \varphi' \in \E^1 (X, L)$, we put 
\begin{align} 
I (\varphi, \varphi') 
&:= \int_{X^{\mathrm{NA}}} (\varphi - \varphi') (\mathrm{MA} (\varphi) - \mathrm{MA} (\varphi')), 
\\
\bar{I} (\varphi, \varphi') 
&:= I (\varphi, \varphi') + (e^L) |\sup \varphi - \sup \varphi'|. 
\end{align}
We note our $d_1$ and $\bar{I}$ are $(e^L)$ times of those in \cite{BJ3, BJ4} due to our normalization of $\mathrm{MA}$ and $\DHm$. 
We have $\varphi = \varphi'$ iff $\bar{I} (\varphi, \varphi') = 0$ by \cite[Corollary 7.4]{BJ3}. 
It is proved in \cite[Theorem 9.4]{BJ3} that the strong convergence $\varphi_i \to \varphi$ in $\E^1 (X, L)$ is equivalent to $\bar{I} (\varphi_i, \varphi) \to 0$. 

\begin{prop}[Lemma 5.5 in \cite{BJ4}]
\label{Ibar and d1}
We have a positive constant $C_n$ depending only on the dimension $n$ of $X$ such that 
\[ C_n^{-1} \bar{I} (\varphi, \varphi') \le d_1 (\varphi, \varphi') \le C_n \bar{I} (\varphi, \varphi') \]
for every $\varphi, \varphi' \in \E^1 (X, L)$. 
\end{prop}

It follows that $d_1$ is a distance on $\E^1 (X, L)$, and the $d_1$-topology is equivalent to the strong topology.

\subsubsection{Rooftops in finite energy class}

\begin{prop}
Suppose the rooftop $\varphi \wedge \psi$ exists in $C^0 \cap \PSH (X, L)$ for every $\varphi, \varphi' \in \nH (X, L)$. 
Then for $\varphi, \varphi' \in \E^1 (X, L)$, the rooftop $\varphi \wedge \varphi'$ exists in $\E^1 (X, L)$. 
\end{prop}

To show the claim, we use the metric $d_1$ on $\E^1 (X, L)$. 
Now for $\varphi, \varphi', \varphi'' \in \nH (X, L)$, suppose $\varphi \wedge \varphi'$ and $\varphi \wedge \varphi''$ exists in $C^0 \cap \PSH (X, L)$, then using $\mathcal{F}_{\varphi \wedge \varphi'} = \mathcal{F}_\varphi \cap \mathcal{F}_{\varphi'}$ (see Proposition \ref{filtration of envelope}), we easily obtain 
\begin{equation} 
\label{d1 comparison for rooftop}
d_1 (\varphi \wedge \varphi', \varphi \wedge \varphi'') \le d_1 (\varphi', \varphi'') 
\end{equation}
in the same way as in Lemma \ref{dexp comparison for rooftop}. 
Then passing to the limit along decreasing nets, we get this inequality for general $\varphi, \varphi', \varphi'' \in \E^1 (X, L)$, under the assumption of the proposition. 

\begin{proof}
Take decreasing nets $\{ \varphi_i \}_{i \in I}, \{ \varphi'_j \}_{j \in J} \in \nH (X, L)$ pointwisely converging to $\varphi, \varphi' \in \E^1 (X, L)$, respectively. 
Then for $i, k \in I$ and $j, l \in J$, we have 
\begin{align*} 
d_1 (\varphi_i \wedge \varphi'_j, \varphi_k \wedge \varphi'_l) 
&\le d_1 (\varphi_i \wedge \varphi'_j, \varphi_i \wedge \varphi'_l) + d_1 (\varphi_i \wedge \varphi'_l, \varphi_k \wedge \varphi'_l) 
\\
&\le d_1 (\varphi'_j, \varphi'_l) + d_1 (\varphi_i, \varphi_k), 
\end{align*}
so that $\varphi_i \wedge \varphi'_j$ is a Cauchy net in $\E^1 (X, L)$. 
It follows that $\varphi_i \wedge \varphi'_j$ is a decreasing $d_1$-Cauchy net, so that it converges pointwisely to some $\varphi'' \in \E^1 (X, L)$. 
(We note the completeness for general Cauchy net is equivalent to the continuity of envelopes as proved in \cite{BJ4}, however, the completeness for decreasing Cauchy net does not need the continuity of envelopes. 
Compare Proposition \ref{convergence of decreasing sequence}. )
Then by Proposition \ref{decreasing limit of rooftop}, $\varphi''$ is the rooftop $\varphi \wedge \varphi'$. 
\end{proof}

\begin{prop}
\label{strong convergence of rooftops}
If $\varphi_i \to \varphi, \varphi'_j \to \varphi'$ strongly in $\E^1 (X, L)$ and $\varphi_i \wedge \varphi'_j, \varphi \wedge \varphi'$ exists in $\E^1 (X, L)$, then $\varphi_i \wedge \varphi'_j \to \varphi \wedge \varphi'$ strongly in $\E^1 (X, L)$. 
\end{prop}

\begin{proof}
By (\ref{d1 comparison for rooftop}), we have 
\[ d_1 (\varphi_i \wedge \varphi_j', \varphi \wedge \varphi') \le d_1 (\varphi_i \wedge \varphi_j', \varphi_i \wedge \varphi') + d_1 (\varphi_i \wedge \varphi', \varphi \wedge \varphi') \le d_1 (\varphi_j', \varphi') + d_1 (\varphi_i, \varphi). \]
The claim follows by the equivalence of $d_1$-convergence and strong convergence proved in \cite{BJ4}. 
\end{proof}

In particular, for any strongly convergent net $\varphi_i \to \varphi \in \E^1 (X, L)$ and $\tau_i \to \tau \in \mathbb{R}$, we have strong convergence $\varphi_i \wedge \tau_i \to \varphi \wedge \tau \in \E^1 (X, L)$. 

\subsubsection{A metric structure on the space $\E^{\exp} (X, L)$}
\label{A metric structure on the space Eexp}

Now we introduce the distance $d_{\exp}$ modeled on Luxemburg norm in Orlicz analysis. 
For $\varphi, \varphi' \in \nH (X, L)$, we put 
\begin{equation} 
d_{\exp} (\varphi, \varphi') := \inf \Big{\{} \beta \in (0, \infty) ~\Big{|}~ \int_\mathbb{R} (e^{|t/\beta|} - 1) \DHm_{\varphi, \varphi'} \le 1 \Big{\}}. 
\end{equation}
We note $\beta (e^{|t/\beta|} -1) \le \beta' (e^{|t/\beta'|} -1)$ for $\beta \ge \beta'$ by convexity. 
In particular, we have 
\begin{equation} 
\label{dexp Eexp comparison}
d_{\exp} (\varphi, \varphi') \le \rho^{-1} \max \{ 1, \int_\mathbb{R} (e^{|\rho t|} - 1) \DHm_{\varphi, \varphi'} \} 
\end{equation}
for any $\rho > 0$. 

\begin{prop}
\label{dp dexp}
For $1 \le p < \infty$ and $\varphi, \varphi' \in \nH (X, L)$, we have 
\[ d_p (\varphi, \varphi') \le \lceil p \rceil \cdot d_{\exp} (\varphi, \varphi'). \]
\end{prop}

\begin{proof}
Since $|t/\beta|^p/\lceil p \rceil ! \le |t/\beta|^{\lfloor p \rfloor}/\lfloor p \rfloor! + |t/\beta|^{\lceil p \rceil}/\lceil p \rceil ! \le e^{|t/\beta|} -1$, we have 
\[ \int_{\mathbb{R}} |t/\beta|^p \DHm_{\varphi, \varphi'} \le \lceil p \rceil ! \]
for $\beta > d_{\exp} (\varphi, \varphi')$. 
Thus we get $d_p (\varphi, \varphi') \le (\lceil p \rceil !)^{1/p} \beta \le (\lceil p \rceil !)^{1/\lfloor p \rfloor} \beta \le \lceil p \rceil \cdot \beta$ by $\lceil p \rceil ! \cdot \lceil p \rceil \le \lceil p \rceil^{\lceil p \rceil}$. 
Taking the limit $\beta \searrow d_{\exp} (\varphi, \varphi')$, we obtain the claim. 
\end{proof}

To see the triangle inequality on $d_{\exp}$, we introduce the following norm on $\mathbb{R}^{N_m}$: 
\[ \chi_m (x_1, \ldots, x_N) := \inf \Big{\{} \beta \in (0, \infty) ~\Big{|}~ \frac{1}{m^n} \sum_{i=1}^{N_m} (e^{|x_i/m\beta|} -1) \le 1 \Big{\}}. \]
Using the convexity of $\Phi_m (t) = e^{|t/m|} -1$, we can easily check that $\chi_m$ defines an $\mathfrak{S}_N$-invariant norm on $\mathbb{R}^{N_m}$. 
This is known as Luxemburg norm on the finite set $\{ 1, \ldots, N_m \}$ with respect to the Young weight $\Phi_m$, which is a basic material in Orlicz analysis. 

\begin{rem}
The condition $\lim_{t \to 0} \Phi_m (t)/t = 0$ usually assumed for the duality on Orlicz space is not used in this article. 
We remark that one may replace the Young weight $e^{|t|} -1$ with $\Phi (t) = e^{|t|} - |t| -1$, of which we have the non-negative convex conjugate $\Psi (t) = (|t|+1) \log (|t|+1) + |t| = O (|t| \log|t|)$. 
This would have an advantage for applying duality of Orlicz spaces. 
We put 
\[ \tilde{d}_{\exp} (\varphi, \varphi') := \inf \Big{\{} \beta \in (0, \infty) ~\Big{|}~ \int_\mathbb{R} (e^{|t/\beta|} -|t/\beta| - 1) \DHm_{\varphi, \varphi'} \le 1 \Big{\}}. \]
By $e^{|t/\beta|} -|t/\beta| - 1 \le e^{|t/\beta|} -1$, we have $\tilde{d}_{\exp} \le d_{\exp}$. 
Conversely, since $|t/\beta|^2/2 \le e^{|t/\beta|} - |t/\beta| -1$, we have 
\[ \int_\mathbb{R} |t/\beta| \DHm_{\varphi, \varphi'} \le \sqrt{(e^L) \int_\mathbb{R} |t/\beta|^2 \DHm_{\varphi, \varphi'}} \le \sqrt{ 2(e^L) \int_\mathbb{R} (e^{|t/\beta|} - |t/\beta| -1) \DHm_{\varphi, \varphi'} }. \]
The convexity implies $e^{|t/r\beta|} - |t/r\beta| -1 \le r^{-1} (e^{|t/\beta|} - |t/\beta| -1)$ for $r \ge 1$, hence we have 
\[ \int_\mathbb{R} (e^{|t/r\beta|} - |t/r\beta| -1) \DHm_{\varphi, \varphi'} \le r^{-1} \int_\mathbb{R} (e^{|t/\beta|} - |t/\beta| -1) \DHm_{\varphi, \varphi'} \le r^{-1} \]
for any $\beta > \tilde{d}_{\exp} (\varphi, \varphi')$. 
In particular, we have 
\[ \int_\mathbb{R} |t/r\beta| \DHm_{\varphi, \varphi'} \le \sqrt{2 (e^L) r^{-1}}, \]
so that we get 
\[ \int_\mathbb{R} (e^{|t/r\beta|} -1) \DHm_{\varphi, \varphi'} \le \sqrt{2 (e^L) r^{-1}} + r^{-1}. \]
Put $r := \max \{ 2, 8 (e^L) \}$, then we obtain 
\[ \int_\mathbb{R} (e^{|t/r\beta|} -1) \DHm_{\varphi, \varphi'} \le 1. \]
This implies $d_{\exp} \le r \beta$. 
Taking the limit $\beta \searrow \tilde{d}_{\exp}$, we get $d_{\exp} \le r \tilde{d}_{\exp}$ for the fixed constant $r$. 
Therefore, $d_{\exp}$ and $\tilde{d}_{\exp}$ will define the same uniform structure. 
%Conversely, for $\varphi_i, \varphi \in \nH (X, L)$ with $\tilde{d} (\varphi_i, \varphi) \to 0$, we have 
%\[ \varlimsup_{i \to \infty} \int_{\mathbb{R}} (e^{|t/\beta|} -|t/\beta| - 1) \DHm_{\varphi_i, \varphi} \le 1 \]
%for any $\beta > 0$. 
%Since $|t/\beta|^2 \le e^{|t/\beta|} -|t/\beta| - 1$, we have $d_2 (\varphi_i, \varphi) \to 0$. 
%It follows that $d_1 (\varphi_i, \varphi) \to 0$, so that we have $\lim_{i \to \infty} \int_{\mathbb{R}} |t/\beta| \DHm_{\varphi_i, \varphi} = 0$ for any $\beta > 0$. 
%Thus we get
%\[ \varlimsup_{i \to \infty} \int_{\mathbb{R}} (e^{|t/\beta|} - 1) \DHm_{\varphi_i, \varphi} \le 1 \]
%for any $\beta > 0$, which shows $d_{\exp} (\varphi_i, \varphi) \to 0$. 
%Therefore, $d_{\exp}$ and $\tilde{d}_{\exp}$ define the same topology. 
\end{rem}

For each $m$, take a codiagonal basis $\bm{s}_m$ of $R_m$ for $\varphi, \varphi'$ and put 
\begin{align*} 
d_{\exp}^m (\| \cdot \|_m^\varphi, \| \cdot \|_m^{\varphi'}) 
&:= \chi_m (\lambda (\| \cdot \|_m^{\varphi}, \| \cdot \|_m^{\varphi'}))
\\
&= \inf \Big{\{} \beta \in (0, \infty) ~\Big{|}~ \frac{1}{m^n} \sum_{i=1}^{N_m} (e^{|(\lambda_i^\varphi (\bm{s}_m) - \lambda_i^{\varphi'} (\bm{s}_m))/m \beta|} -1) \le 1 \Big{\}}. 
\end{align*}
Then thanks to \cite[Theorem 3.1]{BE}, we have the following triangle inequality. 
\begin{equation}
d_{\exp}^m (\| \cdot \|_m^\varphi, \| \cdot \|_m^{\varphi''}) \le d_{\exp}^m (\| \cdot \|_m^\varphi, \| \cdot \|_m^{\varphi'}) + d_{\exp}^m (\| \cdot \|_m^{\varphi'}, \| \cdot \|_m^{\varphi''}). 
\end{equation}

\begin{lem}
For $\varphi, \varphi' \in \nH (X, L)$, we have 
\[ \lim_{m \to \infty} d_{\exp}^m (\| \cdot \|_m^\varphi, \| \cdot \|_m^{\varphi'}) = d_{\exp} (\varphi, \varphi'). \]
\end{lem}

\begin{proof}
Suppose $\beta' > d_{\exp} (\varphi, \varphi')$, then we have 
\[ \int_\mathbb{R} (e^{|t/\beta'|} -1) \DHm_{\varphi, \varphi'} \le 1. \]
It follows that for any $\varepsilon > 0$, there exists $m_\varepsilon$ such that for every $m \ge m_\varepsilon$ we have 
\[ \frac{1}{m^n} \sum_{i=1}^{N_m} (e^{|(\lambda_i^\varphi (\bm{s}_m) - \lambda_i^{\varphi'} (\bm{s}_m))/m \beta'|} -1) \le 1+\varepsilon. \]
Then for $m \ge m_\varepsilon$, we get 
\[ \frac{1}{m^n} \sum_{i=1}^{N_m} (e^{|(\lambda_i^\varphi (\bm{s}_m) - \lambda_i^{\varphi'} (\bm{s}_m))/m (1+\varepsilon) \beta'|} -1) \le \frac{1}{m^n} \sum_{i=1}^{N_m} \frac{1}{1+ \varepsilon} (e^{|(\lambda_i^\varphi (\bm{s}_m) - \lambda_i^{\varphi'} (\bm{s}_m))/m \beta'|} -1) \le 1 \]
by convexity. 
It follows that 
\[ \varlimsup_{i \to \infty} d_{\exp}^m (\| \cdot \|_m^\varphi, \| \cdot \|_m^{\varphi'}) \le (1+ \varepsilon) \beta'. \]
Taking the limits $\varepsilon \searrow 0, \beta' \searrow d_{\exp} (\varphi, \varphi')$, we get 
\[ \varlimsup_{m \to \infty} d_{\exp}^m (\| \cdot \|_m^\varphi, \| \cdot \|_m^{\varphi'}) \le d_{\exp} (\varphi, \varphi'). \]

Next we suppose $\beta < d_{\exp} (\varphi, \varphi')$, then we have 
\[ \int_\mathbb{R} (e^{|t/\beta|} -1) \DHm_{\varphi, \varphi'} \ge 1+ \delta \]
for some $\delta > 0$. 
Taking sufficiently large $m$, we can assume 
\[ m^{-n} \sum_{i=1}^{N_m} (e^{|(\lambda_i^\varphi (\bm{s}_m) - \lambda_i^{\varphi'} (\bm{s}_m))/m \beta|} -1) \ge 1+ \delta/2. \]
Thus we have $\beta < d^m_{\exp} (\varphi, \varphi')$. 
Taking the limit $\beta \nearrow d_{\exp} (\| \cdot \|_m^\varphi, \| \cdot \|_m^{\varphi'})$, we get 
\[ d_{\exp} (\varphi, \varphi') \le \varliminf_{m \to \infty} d^m_{\exp} (\| \cdot \|_m^\varphi, \| \cdot \|_m^{\varphi'}). \]
\end{proof}

\begin{prop}
\label{dexp gives a distance on H}
$d_{\exp}$ is a distance on $\nH (X, L)$. 
\end{prop}

\begin{proof}
Since $\DHm_{\varphi, \varphi'} = \delta_0$ iff $\varphi = \varphi'$, we have $d_{\exp} (\varphi, \varphi') = 0$ iff $\varphi = \varphi'$. 
Since $\DHm_{\varphi, \varphi'} = (t \mapsto -t)_* \DHm_{\varphi', \varphi}$, we have $d_{\exp} (\varphi, \varphi') = d_{\exp} (\varphi', \varphi)$. 
Thanks to the triangle inequality on $d^m_{\exp}$ and the above lemma, we get $d_{\exp} (\varphi, \varphi'') \le d_{\exp} (\varphi, \varphi') + d_{\exp} (\varphi', \varphi'')$.  
\end{proof}

We will extend the distance $d_{\exp}$ to the space $\E^{\exp} (X, L)$ of our interest and show its completeness. 
We follow the steps in \cite{Dar}, where the archimedean case is treated. 
To be precise, the Young weight $\Phi (t) = e^{|t|} -1$ does not satisfy the $\mathcal{W}_p^+$-condition he assumed in his argument, so $\mathcal{E}^{\exp} (X, L)$ is not treated even in the archimedean case, however, we can extend many of his results to such general Young weights as he remarked. 

We make efficient use of the assumption that $E_{\exp} (\varphi_{; \rho}) > - \infty$ \textit{for every $\rho > 0$} in the following lemma. 
This is reminiscent of the fact that a function in the Orlicz space $L^\Phi$, which consists of functions satisfying $\int \Phi (|f|/\beta) d\mu < \infty$ for \textit{some} $\beta > 0$, can be approximated by simple functions with respect to the $L^\Phi$-norm iff $\int \Phi (|f|/\beta) d\mu < \infty$ for \textit{every} $\beta > 0$. 
This lemma is the key tool for the extension of the distance. 

\begin{lem}
\label{decreasing convergence implies Cauchy}
Let $\{ \varphi_i \}_{i \in I} \subset \nH (X, L)$ be a decreasing net pointwisely converging to $\varphi \in \E^{\exp} (X, L)$. 
Then for every $\varepsilon > 0$, there exists $k \in I$ such that $d_{\exp} (\varphi_i, \varphi_j) \le \varepsilon$ for every $i, j \ge k$. 
\end{lem}

\begin{proof}
For $\varphi, \varphi' \in \nH (X, L)$ with $\varphi \le \varphi'$, we have the following estimate 
\begin{equation} 
\label{distance lemma}
\int_{\mathbb{R}} (e^{|t/\beta|} - 1) \DHm_{\varphi, \varphi'} \le \Big{(} \frac{E (\varphi') - E (\varphi)}{4\beta/3} \Big{)}^{1/4} (-e^{\sup \varphi_{;2/\beta}} E_{\exp} (\varphi_{;2/\beta}))^{3/4}. 
\end{equation}
The claim is a consequence of this estimate. 
Indeed, since $0 < - E_{\exp} (\varphi_{i; 2/\beta}) \le - E_{\exp} (\varphi_{;2/\beta}) < \infty$ and $\sup \varphi_i \le \sup \varphi_0 =: c$ (note $c_{;2/\beta} = 2c/\beta$), we get 
\[ \int_{\mathbb{R}} (e^{|t/\beta|} - 1) \DHm_{\varphi_i, \varphi_j} \le \Big{(} \frac{(-e^{2c/\beta} E_{\exp} (\varphi_{;2/\beta}) )^3}{4\beta/3} |E (\varphi_j) - E (\varphi_i)| \Big{)}^{1/4}. \] 
Since we have $\lim_{i \to \infty} E (\varphi_i) = E (\varphi)$, for every $\varepsilon > 0$ we can take large $k_\varepsilon$ so that 
\[ |E (\varphi_j) - E (\varphi_i)| \le \frac{4\varepsilon}{3} (-e^{-2c/\varepsilon} E_{\exp} (\varphi_{;2/\varepsilon}))^{-3} \]
for every $i, j \ge k_\varepsilon$. 
Then by the above inequality, we get 
\[ \int_{\mathbb{R}} (e^{|t/\varepsilon|} - 1) \DHm_{\varphi_i, \varphi_j} \le 1, \]
so that $d_{\exp} (\varphi_i, \varphi_j) \le \varepsilon$ for every $i, j \ge k_\beta$. 

Now we show the estimate (\ref{distance lemma}). 
Since $\varphi \le \varphi'$, we have $\lambda_i^\varphi (\bm{s}) \le \lambda_i^{\varphi'} (\bm{s})$. 
Both sides of the inequality is invariant under the replacement $\varphi \mapsto \varphi - c, \varphi' \mapsto \varphi' -c$, so that we may assume $\varphi' \le 0$. 
It follows that $0 \le \lambda_i^{\varphi'} (\bm{s}) - \lambda_i^\varphi (\bm{s}) \le -\lambda_i^\varphi (\bm{s})$ and $1 \le e^{-\sup \varphi}$. 
Using $e^x - 1 \le x e^x$ for $x \ge 0$ and Cauchy--Schwarz inequality, we compute 
\begin{align*}
\int_\mathbb{R} (e^{|t/\beta|} -1) \DHm_{\varphi, \varphi'} 
&= \lim_{m \to \infty} \frac{1}{m^n} \sum_{i=1}^{N_m} (e^{|(\lambda_i^\varphi (\bm{s}_m) - \lambda_i^{\varphi'} (\bm{s}_m))/\beta m|} -1)
\\
&\le \lim_{m \to \infty} \frac{1}{m^n} \sum_{i=1}^{N_m} |(\lambda_i^\varphi (\bm{s}_m) - \lambda_i^{\varphi'} (\bm{s}_m))/\beta m| \cdot e^{|(\lambda_i^\varphi (\bm{s}_m) - \lambda_i^{\varphi'} (\bm{s}_m))/\beta m|} 
\\
&\le \lim_{m \to \infty} \Big{(} \frac{1}{m^n} \sum_{i=1}^{N_m} |(\lambda_i^\varphi (\bm{s}_m) - \lambda_i^{\varphi'} (\bm{s}_m))/\beta m|^2 \Big{)}^{1/2} 
\\
& \qquad \cdot \lim_{m \to \infty} \Big{(} \frac{1}{m^n} \sum_{i=1}^{N_m} e^{2|(\lambda_i^\varphi (\bm{s}_m) - \lambda_i^{\varphi'} (\bm{s}_m))/\beta m|} \Big{)}^{1/2}
\\
&\le \lim_{m \to \infty} \beta^{-1} \Big{(} \frac{1}{m^n} \sum_{i=1}^{N_m} \lambda_i^{\varphi'} (\bm{s}_m)/m - \frac{1}{m^n} \sum_{i=1}^{N_m} \lambda_i^\varphi (\bm{s}_m)/m \Big{)}^{1/4} 
\\
&\qquad \cdot \lim_{m \to \infty} \Big{(} \frac{1}{m^n} \sum_{i=1}^{N_m} (- \lambda_i^\varphi (\bm{s}_m)/m)^3 \Big{)}^{1/4} 
\\
& \qquad \qquad \cdot \lim_{m \to \infty} \Big{(} \frac{1}{m^n} \sum_{i=1}^{N_m} e^{-2\lambda_i^\varphi (\bm{s}_m) /\beta m} \Big{)}^{1/2}
\\
&= (2^3 \beta/3!)^{-1/4} (E (\varphi') - E (\varphi))^{1/4} (\int_\mathbb{R} -(2t/\beta)^3/3! \DHm_\varphi )^{1/4} (- E_{\exp} (\varphi_{;2/\beta}))^{1/2}
\\
&\le (4\beta/3)^{-1/4} (E (\varphi') - E (\varphi))^{1/4} (- E_{\exp} (\varphi_{;2/\beta}))^{3/4}. 
\end{align*}
\end{proof}

The following proof traces the argument in \cite{Dar}, where the archimedean analogue is studied. 
We exhibit it here for the readers convenience. 

\begin{prop}
\label{dexp well-defined}
Let $\{ \varphi_i \}_{i \in I}, \{ \psi_j \}_{j \in J} \subset \mathcal{H} (X, L)$ be decreasing nets pointwisely converging to $\varphi, \psi \in \E^{\exp} (X, L)$, respectively. 
Then the limit 
\[ \lim_{(i, j) \to \infty} d_{\exp} (\varphi_i, \psi_j) \]
exists as a finite value depending only on $\varphi, \psi$. 

Namely, there exists $\Delta \in \mathbb{R}$ depending only on $\varphi, \psi$ such that for every $\varepsilon > 0$ there exists $(i_\varepsilon, j_\varepsilon) \in I \times J$ such that 
\[ |d_{\exp} (\varphi_i, \psi_j) - \Delta| < \varepsilon \]
for every $i \ge i_0$ and $j \ge j_0$. 
\end{prop}

\begin{proof}
Thanks to the triangle inequality and the above lemma, for any $\varepsilon > 0$ there exists $(i_\varepsilon, j_\varepsilon) \in I \times J$ such that 
\[ |d_{\exp} (\varphi_i, \psi_j) - d_{\exp} (\varphi_k, \psi_l)| \le d_{\exp} (\varphi_i, \varphi_k) + d_{\exp} (\psi_j, \psi_l) \le 2 \varepsilon \]
for every $(i, j), (k, l) \ge (i_\varepsilon, j_\varepsilon)$. 
Thus the net $\{ d_{\exp} (\varphi_i, \psi_j) \}_{(i, j) \in I \times J}$ is Cauchy, so it converges. 

To see the independence, we take other decreasing nets $\{ \varphi'_{i'} \}_{i' \in I'}, \{ \psi'_{j'} \}_{j' \in J'}$ converging to $\varphi, \varphi'$, respectively. 
Put $\Delta := \lim_{(i, j) \to \infty} d_{\exp} (\varphi_i, \psi_j)$ and $\Delta' := \lim_{(i', j') \to \infty} d_{\exp} (\varphi'_{i'}, \psi'_{j'})$. 
For $\beta > 0$, take $(i_\beta, j_\beta) \in I \times J$ and $(i'_\beta, j'_\beta) \in I' \times J'$ so that $|\Delta - d_{\exp} (\varphi_i, \psi_j)| \le \beta$ and $|\Delta' - d_{\exp} (\varphi'_{i'}, \psi'_{j'})| \le \beta$ for every $(i, j) \ge (i_\beta, j_\beta)$ and $(i', j') \ge (i_\beta', j_\beta')$, respectively. 
By the triangle inequality, we have 
\begin{align*} 
|\Delta - \Delta'| 
&\le |d_{\exp} (\varphi_i, \psi_j) - d_{\exp} (\varphi'_{i'}, \psi'_{j'})| + 2 \beta 
\\
&\le d_{\exp} (\varphi_i, \varphi'_{i'}) + d_{\exp} (\psi_j, \psi'_{j'}) + 2\beta 
\end{align*}
for every $(i, i') \ge (i_\beta, i_\beta') \in I \times I'$ and $(j, j') \ge (j_\beta, j_\beta') \in J \times J'$. 
It suffices to show that there exists $(k, k') \ge (i_\beta, i_\beta')$ and $(l, l') \ge (j_\beta, j_\beta')$ such that $d_{\exp} (\varphi_k, \varphi'_{k'}) \le 2\beta$ and $d_{\exp} (\psi_l, \psi_{l'}') \le 2\beta$. 

By Lemma \ref{Dini}, for any $\varepsilon > 0$ and $i \ge i_\beta$, there exists $i'_{i, \varepsilon} \ge i'_\beta$ such that $\varphi_{i'}' \le \varphi_i + \varepsilon$ for every $i' \ge i'_{i, \varepsilon}$. 
Then by (\ref{distance lemma}), we get 
\begin{align*} 
\int_{\mathbb{R}} (e^{|t/\beta|} - 1) \DHm_{\varphi_i + \varepsilon, \varphi'_{i'}} 
&\le \Big{(} \frac{E (\varphi_i + \varepsilon) - E (\varphi'_{i'})}{4\beta/3} \Big{)}^{1/4} (-e^{\sup \varphi_{i;2/\beta}} E_{\exp} (\varphi_{i;2/\beta}))^{3/4}
\\
&\le \Big{(} \frac{E (\varphi_i) - E (\varphi'_{i'}) + \varepsilon (e^L)}{4\beta/3} \Big{)}^{1/4} (-e^{\sup \varphi_{0;2/\beta}} E_{\exp} (\varphi_{;2/\beta}))^{3/4}. 
\end{align*}
Now for $\beta > 0$, put 
\[ \varepsilon_\beta := \min \{ \beta, \frac{\beta}{3 (e^L)} (-e^{\sup \varphi_{0;2/\beta}} E_{\exp} (\varphi_{;2/\beta}))^{-3} \} \]
and take large $k_\beta \ge i_\beta$ so that 
\[ |E (\varphi_{k_\beta}) - E (\varphi)| \le \frac{\beta}{3} (-e^{\sup \varphi_{0;2/\beta}} E_{\exp} (\varphi_{;2/\beta}))^{-3}. \]
Then take large $k_\beta' \ge i'_{k_\beta, \varepsilon_\beta} \ge i'_\beta$ so that 
\[ |E (\varphi) - E (\varphi'_{k_\beta'})| \le \frac{\beta}{3} (-e^{\sup \varphi_{0;2/\beta}} E_{\exp} (\varphi_{;2/\beta}))^{-3}. \]
Then we get 
\[ \int_{\mathbb{R}} (e^{|t/\beta|} - 1) \DHm_{\varphi_i + \varepsilon, \varphi'_{i'}} \le 1, \]
which implies $d_{\exp} (\varphi_{k_\beta} + \varepsilon_\beta, \varphi'_{k_\beta'}) \le \beta$. 
Therefore we get 
\[ d_{\exp} (\varphi_{k_\beta}, \varphi'_{k_\beta'}) \le d_{\exp} (\varphi_{k_\beta}, \varphi_{k_\beta} + \varepsilon_\beta) + d_{\exp} (\varphi_{k_\beta} + \varepsilon_\beta, \varphi'_{k_\beta'}) \le \varepsilon_\beta + \beta \le 2\beta \]
as desired. 
\end{proof}

Now for $\varphi, \psi \in \E^{\exp} (X, L)$, we put 
\begin{equation}
\label{dexp distance definition}
d_{\exp} (\varphi, \psi) := \lim_{(i, j) \to \infty} d_{\exp} (\varphi_i, \psi_j)
\end{equation}
by taking pointwisely convergent decreasing nets $\varphi_i \searrow \varphi$, $\psi_j \searrow \psi$. 
To show that $d_{\exp}$ is a distance on $\E^{\exp} (X, L)$, it suffices to show that $d_{\exp} (\varphi, \varphi') = 0$ iff $\varphi = \varphi'$: the triangle inequality and the reflexivity follows readily from the definition. 
To check this, we note the following. 

\begin{rem} 
\label{dp distance}
By a similar argument, we can also extend $d_p$ to the space 
\[ \E^p (X, L) := \{ \varphi \in \E (X, L) ~|~ \int_\mathbb{R} |t|^p \DHm_\varphi < \infty \}. \]
We have $\E^{\exp} (X, L) \subset \E^p (X, L)$ for every $1 \le p < \infty$. 
The following inequalities are inherited: 
\begin{itemize}
\item $(e^L)^{-1} \cdot d_1 (\varphi, \varphi') \le (e^L)^{-1/p} \cdot d_p (\varphi, \varphi')$ for $\varphi, \varphi' \in \E^p (X, L)$, 

\item $d_r (\varphi, \varphi')^{r (q-p)} \le d_p (\varphi, \varphi')^{p (q-r)} d_q (\varphi, \varphi')^{q (r-p)}$ for $\varphi, \varphi' \in \E^q (X, L)$ and $p \le r \le q$. 

\item $d_p (\varphi, \varphi') \le \lceil p \rceil \cdot d_{\exp} (\varphi, \varphi')$ for $\varphi, \varphi' \in \E^{\exp} (X, L)$. 
\end{itemize}
\end{rem}

It follows that $d_\bullet (\varphi, \varphi') = 0$ for $\varphi, \varphi' \in \E^\bullet (X, L)$ and $\bullet = p, \exp$ implies $d_1 (\varphi, \varphi') = 0$, hence $\varphi = \varphi'$ thanks to Proposition \ref{Ibar and d1}. 
Therefore, we conclude the following. 

\begin{prop}
For $\bullet = p, \exp$, the pseudo-distance $d_\bullet$ defined by (\ref{dexp distance definition}) gives a distance on $\E^\bullet (X, L)$. 
\end{prop}

By the above remark, $d_{\exp}$-convergence implies $d_p$-convergence. 
Conversely, we have the following. 

\begin{prop}
For $\{ \varphi_i \}_{i \in I}, \varphi \in \E^{\exp} (X, L)$, uniform convergence $\varphi_i \to \varphi$ implies $d_{\exp}$-convergence. 
\end{prop}

\begin{proof}
For any $\varepsilon > 0$, there exists sufficiently large $i$ satisfying $\varphi - \varepsilon \le \varphi_i \le \varphi + \varepsilon$. 
Then we get 
\[ d_{\exp} (\varphi_i, \varphi) \le d_{\exp} (\varphi_i, \varphi - \varepsilon) + d_{\exp} (\varphi - \varepsilon, \varphi) \le d_{\exp} (\varphi + \varepsilon, \varphi - \varepsilon) + \varepsilon (e^L) = 3 \varepsilon (e^L). \]
Here we used $d_{\exp} (\varphi, \varphi'') \le d_{\exp} (\varphi, \varphi')$ for $\varphi \ge \varphi' \ge \varphi'' \in \E^{\exp} (X, L)$, which is just an exercise. 
\end{proof}

\subsubsection{The distance $d_{\exp}$ with anchor in $\nH (X, L)$}
\label{dexp with anchor in H}

Following the steps in section \ref{subsection: Moment energy}, we can introduce the relative moment energy $E_\chi (\varphi, \varphi')$ for $\varphi \in \PSH (X, L)$ and $\varphi' \in \nH (X, L)$ with respect to an increasing right continuous function $\chi$: 
\[ E_\chi (\varphi, \varphi') := \inf \Big{\{} \int_\mathbb{R} \chi \DHm_{\tilde{\varphi}, \varphi'} ~\Big{|}~ \varphi \le \tilde{\varphi} \in \nH (X, L) \Big{\}}. \]
It is monotonic and continuous along decreasing nets on the first variable. 
If $\chi$ is concave, we have 
\[ E_\chi (\varphi, \varphi') \ge E_\chi (\varphi_{;2}) + \frac{1}{2} \chi (-2 \sup \varphi) (e^L), \]
so that $E_\chi (\cdot, \varphi')$ is finite on $\E^\chi (X, L)$. 

Similarly, as in section \ref{Duistermaat--Heckman measure and moment energy}, we can define the relative Duistermaat--Heckman measure $\DHm_{\varphi, \varphi'}$ for $\varphi \in \PSH (X, L)$ and $\varphi' \in \nH (X, L)$. 
For $\varphi \in \E (X, L)$, we have $\int_\mathbb{R} \DHm_{\varphi, \varphi'} = (e^L)$ and $E_\chi (\varphi, \varphi') = \int_\mathbb{R} \chi \DHm_{\varphi, \varphi'}$ for increasing right continuous function $\chi$. 
In particular, $\int_\mathbb{R} \chi \DHm_{\varphi, \varphi'}$ is finite on $\E^{\chi} (X, L)$ and we have 
\[ \lim_{i \to \infty} \int_\mathbb{R} \chi \DHm_{\varphi_i, \varphi'} = \int_\mathbb{R} \chi \DHm_{\varphi, \varphi'} \]
for any convergent decreasing net $\varphi_i \searrow \varphi \in \E (X, L)$. 
Slight generally, we can show the same convergence for moderate $\chi$ in the sense of Definition \ref{moderate}. 
In particular, 
\[ \lim_{i \to \infty} \int_\mathbb{R} (e^{|t/\beta|}-1) \DHm_{\varphi_i, \varphi'} = \int_\mathbb{R} (e^{|t/\beta|}-1) \DHm_{\varphi, \varphi'} \]
for any convergent decreasing net $\varphi_i \searrow \varphi \in \E (X, L)$. 

Now we show the following formula on $d_{\exp} (\varphi, \varphi')$ with an anchor $\varphi' \in \nH (X, L)$. 
This helps us to simplify some arguments in the rest of this article. 

\begin{prop}
\label{distance with anchor}
For $\varphi \in \E^{\exp} (X, L)$ and $\varphi' \in \nH (X, L)$, we have 
\[ d_{\exp} (\varphi, \varphi') = \inf \Big{\{} \beta > 0 ~\Big{|}~ \int_\mathbb{R} (e^{|t/\beta|} -1) \DHm_{\varphi, \varphi'} \le 1 \Big{\}} \]
\end{prop}

\begin{proof}
Take a decreasing net $\{ \varphi_i \}_{i \in I} \subset \nH (X, L)$ converging to $\varphi \in \E^{\exp} (X, L)$. 
We have $\lim_{i \to \infty} d_{\exp} (\varphi_i, \varphi') = d_{\exp} (\varphi, \varphi')$ by the definition of the metric. 
It follows that for $\beta > d_{\exp} (\varphi, \varphi')$ we can take $i_\beta$ so that $d_{\exp} (\varphi_i, \varphi') < \beta$ for every $i \ge i_\beta$. 
Then since $\int_\mathbb{R} (e^{|t/\beta|} -1) \DHm_{\varphi_i, \varphi'} \le 1$ for $i \ge i_\beta$, we get 
\[ \int_\mathbb{R} (e^{|t/\beta|} -1) \DHm_{\varphi, \varphi'} = \lim_{i \to \infty} \int_\mathbb{R} (e^{|t/\beta|} -1) \DHm_{\varphi_i, \varphi'} \le 1. \]
Thus we obtain 
\[ d_{\exp} (\varphi, \varphi') \ge \inf \Big{\{} \beta > 0 ~\Big{|}~ \int_\mathbb{R} (e^{|t/\beta|} -1) \DHm_{\varphi, \varphi'} \le 1 \Big{\}}. \]

Conversely, take $\beta' < d_{\exp} (\varphi, \varphi')$. 
Take small $\varepsilon > 0$ so that $(1+ \varepsilon) \beta' < d_{\exp} (\varphi, \varphi')$. 
Then we can find $i_{\beta', \varepsilon}$ so that $(1+\varepsilon) \beta' < d_{\exp} (\varphi_i, \varphi')$ for every $i \ge i_{\beta', \varepsilon}$. 
Recall by convexity we have 
\[ \int_\mathbb{R} (e^{|t/\beta'|} -1) \DHm_{\varphi_i, \varphi'} \ge (1+\varepsilon) \int_\mathbb{R} (e^{|t/(1+ \varepsilon) \beta'|} -1) \DHm_{\varphi_i, \varphi'} > 1+ \varepsilon. \]
It follows that  
\[ \int_\mathbb{R} (e^{|t/\beta'|} -1) \DHm_{\varphi, \varphi'} = \lim_{i \to \infty} \int_\mathbb{R} (e^{|t/\beta'|} -1) \DHm_{\varphi_i, \varphi'} \ge 1+ \varepsilon > 1. \]
Thus for any $\beta' < d_{\exp} (\varphi, \varphi')$, we obtain
\[ \beta' \le \inf \Big{\{} \beta > 0 ~\Big{|}~ \int_\mathbb{R} (e^{|t/\beta|} -1) \DHm_{\varphi, \varphi'} \le 1 \Big{\}}. \]
Taking the limit $\beta' \nearrow d_{\exp} (\varphi, \varphi')$, we obtain the reverse inequality. 
\end{proof}

A similar argument shows the following. 

\begin{lem}
\label{dexp convergence criterion}
For a net $\{ \varphi_i \}_{i \in I}$, $\varphi \in \E^{\exp} (X, L)$ and $\varphi' \in \nH (X, L)$, we have $d_{\exp} (\varphi_i, \varphi') \to d_{\exp} (\varphi, \varphi')$ if 
\[ \int_\mathbb{R} e^{|t/\beta|} \DHm_{\varphi_i, \varphi'} \to \int_\mathbb{R} e^{|t/\beta|} \DHm_{\varphi, \varphi'} \]
for every $\beta > 0$. 
\end{lem}

\begin{proof}
Firstly we show $d_{\exp} (\varphi, \varphi') \le \varliminf_{i \to \infty} d_{\exp} (\varphi_i, \varphi)$. 
Take $\beta < d_{\exp} (\varphi, \varphi')$. 
Then we can take $\delta > 0$ so that 
\[ \int_\mathbb{R} (e^{|t/\beta|} -1) \DHm_{\varphi, \varphi'} \ge 1 + \delta. \]
By the assumption $\int_\mathbb{R} e^{|t/\beta|} \DHm_{\varphi_i, \varphi'} \to \int_\mathbb{R} e^{|t/\beta|} \DHm_{\varphi, \varphi'}$, there exists $i_0$ such that 
\[ \int_\mathbb{R} e^{|t/\beta|} \DHm_{\varphi_i, \varphi'} \ge 1+ \frac{\delta}{2} \]
for every $i \ge i_0$, hence we get $\beta < d_{\exp} (\varphi_i, \varphi')$ for every $i \ge i_0$. 
It follows that $\beta \le \varliminf_{i \to \infty} d_{\exp} (\varphi_i, \varphi')$. 
Taking the limit $\beta \nearrow d_{\exp} (\varphi, \varphi')$, we obtain the desired estimate. 

To see the reverse inequality, take $\beta' > d_{\exp} (\varphi, \varphi')$. 
Since 
\[ \int_\mathbb{R} (e^{|t/\beta'|} -1) \DHm_{\varphi, \varphi'} \le 1, \]
for any $\varepsilon > 0$ we can take $i_\varepsilon$ so that 
\[ \int_\mathbb{R} e^{|t/\beta'|} \DHm_{\varphi_i, \varphi'} \le 1+ \varepsilon \]
for every $i \ge i_{\varepsilon}$. 
Now we again recall 
\[ \int_\mathbb{R} (e^{|t/(1+\varepsilon) \beta'|} -1) \DHm_{\varphi_i, \varphi'} \le \frac{1}{1+\varepsilon} \int_\mathbb{R} (e^{|t/\beta'|} -1) \DHm_{\varphi_i, \varphi'} \le 1 \]
by convexity. 
It follows that $d_{\exp} (\varphi_i, \varphi') \le (1+ \varepsilon) \beta'$ for every $i \ge i_\varepsilon$, hence $\varlimsup_{i \to \infty} d_{\exp} (\varphi_i, \varphi') \le (1+ \varepsilon) \beta'$ for any $\varepsilon > 0$ and $\beta' > d_{\exp} (\varphi, \varphi')$. 
Thus we obtain $\varlimsup_{i \to \infty} d_{\exp} (\varphi_i, \varphi) \le d_{\exp} (\varphi, \varphi')$. 
\end{proof}

\subsubsection{Intermediates}
\label{Intermediates}

We show the density of $\nH (X, L) \subset \E^{\exp} (X, L)$. 

\begin{prop}
\label{decreasing convergence is dexp convergent}
For a decreasing net $\{ \varphi_i \}_{i \in I} \subset \E^{\exp} (X, L)$ pointwisely converging to $\varphi \in \E^{\exp} (X, L)$, we have 
\[ d_{\exp} (\varphi_i, \varphi) \to 0. \]
In particular, $\nH (X, L)$ is dense in $\E^{\exp} (X, L)$ with respect to $d_{\exp}$-topology. 
\end{prop}

\begin{proof}
Suppose firstly $\varphi_i \in \nH (X, L)$. 
By Lemma \ref{decreasing convergence implies Cauchy}, for any $\varepsilon > 0$ there exists $k_\varepsilon$ such that $d_{\exp} (\varphi_i, \varphi_j) < \varepsilon$ for all $i, j \ge k_\varepsilon$. 
It follows that $d_{\exp} (\varphi_i, \varphi) = \lim_{j \to \infty} d_{\exp} (\varphi_i, \varphi_j) \le \varepsilon$ for $i \ge k_\varepsilon$. 
Thus we get the claim in this case. 
In particular, $\nH (X, L)$ is dense in $\E^{\exp} (X, L)$. 

Now we study the general case $\varphi_i \in \E^{\exp} (X, L)$. 
As we already noted in the beginning of section \ref{dexp with anchor in H}, we have 
\[ \int_\mathbb{R} e^{|t/\beta|} \DHm_{\varphi_i, \varphi'} \to \int_\mathbb{R} e^{|t/\beta|} \DHm_{\varphi, \varphi'} \]
for any $\beta > 0$, $\varphi' \in \nH (X, L)$ and convergent decreasing sequence $\varphi_i \searrow \varphi \in \E^{\exp} (X, L)$. 
By Lemma \ref{dexp convergence criterion}, we get $d_{\exp} (\varphi_i, \varphi') \to d_{\exp} (\varphi, \varphi')$. 
By the above argument, we can take $\varphi' \in \nH (X, L)$ so that $d_{\exp} (\varphi, \varphi') \le \varepsilon$ for any $\varepsilon > 0$. 

It follows that 
\[ \varlimsup d_{\exp} (\varphi_i, \varphi) \le \lim d_{\exp} (\varphi_i, \varphi') + d_{\exp} (\varphi', \varphi) = 2 d_{\exp} (\varphi, \varphi') \le 2 \varepsilon \]
for any $\varepsilon > 0$. 
\end{proof}

The following is a refinement of Proposition \ref{dexp well-defined}. 

\begin{cor}
Let $\{ \varphi_i \}_{i \in I}, \{ \psi_j \}_{j \in J} \subset \E^{\exp} (X, L)$ decreasing nets pointwisely converging to $\varphi, \psi \in \E^{\exp} (X, L)$, respectively. 
Then we have 
\[ \lim_{(i, j) \to \infty} d_{\exp} (\varphi_i, \psi_j) = d_{\exp} (\varphi, \psi). \]
\end{cor}

\begin{proof}
By the triangle inequality, we have 
\[ |d_{\exp} (\varphi_i, \psi_j) - d_{\exp} (\varphi, \psi)| \le d_{\exp} (\varphi_i, \varphi) + d_{\exp} (\psi_j, \psi), \]
so that the claim follows from the above proposition. 
\end{proof}

The following will be used in the proposition below and in the proof of the continuity of the exponential moment energy $E_{\exp}$. 

\begin{prop}
\label{dexp Cauchy implies Eexp bounded}
For any $\rho > 0$, the exponential moment energy $E_{\exp} (\varphi_{;\rho})$ is bounded on any $d_{\exp}$-Cauchy net. 
\end{prop}

\begin{proof}
We firstly note for $\varphi \in \E^{\exp} (X, L)$ and $\varphi' \in \nH (X, L)$ 
\[ \int_{\mathbb{R}} (e^{|t/\beta|} -1) \DHm_\varphi \le \frac{1}{2} \int_{\mathbb{R}} (e^{2 |t/\beta|} -1) \DHm_{\varphi, \varphi'} + \frac{1}{2} \int_{\mathbb{R}} (e^{2 |t/\beta|} -1) \DHm_{\varphi'} \]
by convexity and the continuity along decreasing nets $\varphi_i \searrow \varphi$. 
%Since both sides of the inequality is continuous along decreasing nets $\varphi_i \searrow \varphi$ as long as $\varphi \in \E (X, L)$, it suffices to show the case $\varphi \in \nH (X, L)$. 
%In this case, we compute 
%\begin{align*} 
%\int_{\mathbb{R}} (e^{|t/\beta|} -1) \DHm_\varphi 
%&= \lim_{m \to \infty} \frac{1}{m^n} \sum_{i=1}^{N_m} (e^{|\lambda^\varphi_i/m\beta|} -1) 
%\\
%&\le \lim_{m \to \infty} \frac{1}{m^n} \sum_{i=1}^{N_m} (\frac{1}{2} e^{|2 (\lambda^\varphi_i - \lambda^{\varphi'}_i)/m\beta|} + \frac{1}{2} e^{|2 \lambda^{\varphi'}_i/m\beta|} -1)
%\\
%&= \frac{1}{2} \int_{\mathbb{R}} (e^{2 |t/\beta|} -1) \DHm_{\varphi, \varphi'} + \frac{1}{2} \int_{\mathbb{R}} (e^{2 |t/\beta|} -1), 
%\end{align*}
%using the convexity. 

Secondly, for $\varphi, \varphi' \in \E^{\exp} (X, L)$ with $d_{\exp} (\varphi, \varphi') < \beta/2$, we note 
\[ \int_{\mathbb{R}} (e^{|t/\beta|} -1) \DHm_\varphi \le \frac{1}{2} + \frac{1}{2} \int_{\mathbb{R}} (e^{2 |t/\beta|} -1) \DHm_{\varphi'}. \]
Both sides are continuous along decreasing nets and the assumption $d_{\exp} (\varphi, \varphi') < \beta/2$ is stable for $d_{\exp}$-small perturbation, so we may assume $\varphi' \in \nH (X, L)$, thanks to the above proposition. 
Since $d_{\exp} (\varphi, \varphi') < \beta/2$, we have 
\[ \int_{\mathbb{R}} (e^{2|t/\beta|} -1) \DHm_{\varphi, \varphi'} \le 1 \]
by Proposition \ref{distance with anchor}. 
This shows 
\begin{align*} 
\int_{\mathbb{R}} (e^{|t/\beta|} -1) \DHm_{\varphi} 
&\le \frac{1}{2} \int_{\mathbb{R}} (e^{2|t/\beta|} -1) \DHm_{\varphi, \varphi'} + \frac{1}{2} \int_{\mathbb{R}} (e^{2 |t/\beta|} -1) \DHm_{\varphi'} 
\\
&\le \frac{1}{2} + \frac{1}{2} \int_{\mathbb{R}} (e^{2 |t/\beta|} -1) \DHm_{\varphi'}
\end{align*}
as desired. 

Now let $\{ \varphi_i \}_{i \in I} \subset \E^{\exp} (X, L)$ be a Cauchy net. 
Take $\beta = \rho^{-1} > 0$. 
Since $0 \le -E_{\exp} (\varphi_{i; \rho}) = \int_{\mathbb{R}} e^{-t/\beta} \DHm_{\varphi_i} \le \int_{\mathbb{R}} (e^{|t/\beta|} -1) \DHm_{\varphi_i} + (e^L)$, it suffices to bound $\int_{\mathbb{R}} (e^{|t/\beta|} -1) \DHm_{\varphi_i}$. 
Take $k_\beta$ so that $d_{\exp} (\varphi_i, \varphi_j) < \beta/2$ for $i, j \ge k_\beta$. 
Then by the above argument, we get 
\begin{align*}
\int_{\mathbb{R}} (e^{|t/\beta|} -1) \DHm_{\varphi_i} \le \frac{1}{2} + \frac{1}{2} \int_{\mathbb{R}} (e^{2 |t/\beta|} -1) \DHm_{\varphi_{k_\beta}}
\end{align*}
for $i \ge k_\beta$, which shows the boundedness. 
\end{proof}

\begin{quest}
Does there exist a $d_{\exp}$-bounded set with unbounded $E_{\exp} (\varphi_{; \rho})$? 
\end{quest}

\begin{rem}
We have the following reverse implication: for any $\rho > 0$, the subset 
\[ \{ \varphi \in \E^{\exp} (X, L) ~|~ \sup \varphi \le C, ~ E_{\exp} (\varphi_{; \rho}) \ge - C' \} \] 
is $d_{\exp}$-bounded. 
Indeed, since 
\[ \int_\mathbb{R} (e^{|\rho t|} -1) \DHm_\varphi \le - E_{\exp} (\varphi_{; \rho}) + (e^L) (e^{\rho \sup \varphi} - 1) \le C' + (e^L) (e^{\rho C} -1), \]
we get a bound on $d_{\exp} (\varphi, 0)$ by the inequality (\ref{dexp Eexp comparison}). 
\end{rem}

We will use the following in the proof of the completeness of $\E^{\exp} (X, L)$. 

\begin{prop}
\label{convergence of decreasing sequence}
A decreasing net $\{ \varphi_i \}_{i \in I} \subset \E^{\exp} (X, L)$ has a limit $\varphi \in \E^{\exp} (X, L)$ in $d_{\exp}$-topology if and only if $E_{\exp} (\varphi_{i; \rho})$ is bounded for every $\rho > 0$. 
This is the case in particular when $\varphi_i$ is $d_{\exp}$-Cauchy.  
\end{prop}

\begin{proof}
Suppose $\{ \varphi_i \}_{i \in I}$ is a decreasing net with bounded $E_{\exp} (\varphi_{i; \rho})$. 
By the above remark, it is $d_{\exp}$-bounded. 
Put $c := \sup \varphi_0$. 
Since $c \ge \varphi_i$, we have $d_{\exp} (c, \varphi_i) \ge d_{\exp} (c, \sup \varphi_i) = |c - \sup \varphi_i|/\log (1+ (e^L)^{-1})$, so that $\varphi_i (v_{\mathrm{triv}}) = \sup \varphi_i$ is bounded from below. 
It follows that $\varphi_i$ is pointwisely convergent to a limit $\varphi \in \mathrm{PSH} (X, L)$. 

Now since $\varphi_i$ decreasingly converges to $\varphi$, we have $E_{\exp} (\varphi_{i; \rho}) \to E_{\exp} (\varphi_{;\rho})$. 
Since $E_{\exp} (\varphi_{i ;\rho})$ is bounded for every $\rho > 0$, the limit $E_{\exp} (\varphi_{;\rho})$ is finite for every $\rho > 0$, which shows $\varphi \in \E^{\exp} (X, L)$. 
Now $d_{\exp}$-convergence follows from Proposition \ref{decreasing convergence is dexp convergent}. 
\end{proof}

\subsubsection{Completeness}
\label{Completeness}

Now we assume the continuity of envelopes holds for $(X, L)$ (see section \ref{continuity of envelopes}) and show the completeness of $(\E^{\exp} (X, L), d_{\exp})$. 
As we observed in section \ref{existence of rooftop}, the rooftop $\varphi \wedge \varphi'$ exist under the continuity of envelopes. 

\begin{prop}
For $\varphi, \varphi' \in \E^{\exp} (X, L)$, we have 
\[ E_{\exp} (\varphi) + E_{\exp} (\varphi') \le E_{\exp} (\varphi \wedge \varphi') \le \min \{ E_{\exp} (\varphi), E_{\exp} (\varphi') \}. \]
In particular, for $\varphi, \varphi' \in \E^{\exp} (X, L)$, we have $\varphi \wedge \varphi' \in \E^{\exp} (X, L)$. 
\end{prop}

\begin{proof}
As $E_{\exp}$ is continuous along decreasing nets, we may assume $\varphi, \varphi' \in \nH (X, L)$. 
Take a codiagonal basis $\bm{s}$ of $R_m$ for $\varphi, \varphi'$. 
Since $\lambda^{\varphi \wedge \varphi'}_i (\bm{s}) = \min \{ \lambda^\varphi_i (\bm{s}), \lambda^{\varphi'}_i (\bm{s}) \}$, we have 
\[ \max \{ \sum_{i=1}^{N_m} e^{-\lambda^{\varphi}_i (\bm{s})}, \sum_{i=1}^{N_m} e^{-\lambda^{\varphi'}_i (\bm{s})} \} \le \sum_{i=1}^{N_m} e^{-\lambda^{\varphi \wedge \varphi'}_i (\bm{s})} \le \sum_{i=1}^{N_m} e^{- \lambda^\varphi_i (\bm{s})} + \sum_{i=1}^{N_m} e^{-\lambda^{\varphi'}_i (\bm{s})}. \]
Taking the limit $m \to \infty$, we obtain the claim. 
\end{proof}

\begin{lem}
For $\varphi, \varphi' \in \E^{\exp} (X, L)$, we have 
\[ \max \{ d_{\exp} (\varphi, \varphi \wedge \varphi'), d_{\exp} (\varphi', \varphi \wedge \varphi') \} \le d_{\exp} (\varphi, \varphi'). \]
\end{lem}

\begin{proof}
As $d_{\exp}$ is continuous along decreasing nets, we may assume $\varphi, \varphi' \in \nH (X, L)$ (see also section \ref{existence of rooftop}). 
For each $m$, take a basis $\bm{s}_m$ of $R_m$ codiagonal for $\varphi, \varphi'$. 
Since it is diagonal also for $\varphi \wedge \varphi'$ and $\lambda_i^{\varphi \wedge \varphi'} (\bm{s}) = \min \{ \lambda_i^\varphi (\bm{s}), \lambda_i^{\varphi'} (\bm{s}) \}$, we have 
\begin{align*}
\int_\mathbb{R} (e^{|t/\beta|} -1) \DHm_{\varphi, \varphi \wedge \varphi'} 
&= \lim_{m \to \infty} \frac{1}{m^n} \sum_{i=1}^{N_m} (e^{|(\lambda_i^\varphi (\bm{s}_m) - \lambda_i^{\varphi \wedge \varphi'} (\bm{s}_m))/\beta m|} -1)
\\
&= \lim_{m \to \infty} \frac{1}{m^n} \sum_{i=1}^{N_m} (e^{|\max \{ 0, \lambda_i^\varphi (\bm{s}_m) - \lambda_i^{\varphi'} (\bm{s}_m) \}/\beta m|} -1)
\\
&\le \lim_{m \to \infty} \frac{1}{m^n} \sum_{i=1}^{N_m} (e^{|(\lambda_i^\varphi (\bm{s}_m) - \lambda_i^{\varphi'} (\bm{s}_m) )/\beta m|} -1)
\\
&= \int_\mathbb{R} (e^{|t/\beta|} -1) \DHm_{\varphi, \varphi'}. 
\end{align*}
Thus we get 
\begin{align*} 
d_{\exp} (\varphi, \varphi') 
&= \inf \Big{\{} \beta ~\Big{|}~ \int_\mathbb{R} (e^{|t/\beta|} -1) \DHm_{\varphi, \varphi'} \le 1 \Big{\}} 
\\
&\ge \inf \Big{\{} \beta ~\Big{|}~ \int_\mathbb{R} (e^{|t/\beta|} -1) \DHm_{\varphi, \varphi \wedge \varphi'} \le 1 \Big{\}} = d_{\exp} (\varphi, \varphi \wedge \varphi')
\end{align*}
for $\varphi, \varphi' \in \nH (X, L)$. 
The general case follows from the continuity along decreasing nets. 
\end{proof}

\begin{lem}
\label{dexp comparison for rooftop}
For $\varphi, \varphi', \varphi'' \in \E^{\exp} (X, L)$, we have 
\[ d_{\exp} (\varphi \wedge \varphi', \varphi \wedge \varphi'') \le 2 d_{\exp} (\varphi', \varphi''). \]
When $\varphi'' \le \varphi'$, we have 
\begin{equation}
\label{distance of wedge}
d_{\exp} (\varphi \wedge \varphi', \varphi \wedge \varphi'') \le d_{\exp} (\varphi', \varphi''). 
\end{equation}
\end{lem}

\begin{proof}
We firstly show the latter claim. 
We may assume $\varphi, \varphi', \varphi'' \in \nH (X, L)$. 
Let $\bm{s}'$ be a basis which is codiagonal for $\varphi \wedge \varphi', \varphi''$ and is well ordered with respect to $(\varphi', \varphi'')$ and $\bm{s}$ be a basis which is codiagonal for $\varphi', \varphi''$ and is well ordered with respect to $(\varphi', \varphi'')$. 
Since $\varphi \wedge \varphi'' = (\varphi \wedge \varphi') \wedge \varphi''$ by $\varphi'' \le \varphi'$, $\bm{s}'$ is also diagonal with respect to $\varphi \wedge \varphi''$. 
Since $\varphi'' \le \varphi'$, we have $\lambda_i^{\varphi''} (\bm{s}') \le \lambda_i^{\varphi'} (\bm{s}')$ and $\lambda_i^{\varphi \wedge \varphi''} (\bm{s}') \le \lambda_i^{\varphi \wedge \varphi'} (\bm{s}')$, so that we get 
\begin{align*}
0 
&\le \lambda_i^{\varphi \wedge \varphi'} (\bm{s}') - \lambda_i^{\varphi \wedge \varphi''} (\bm{s}') 
\\
&= \min \{ \lambda_i^{\varphi} (\bm{s}'), \lambda_i^{\varphi'} (\bm{s}') \} - \min \{ \lambda_i^{\varphi} (\bm{s}'), \lambda_i^{\varphi''} (\bm{s}') \}
\\
&\le \lambda_i^{\varphi'} (\bm{s}') - \lambda_i^{\varphi''} (\bm{s}'). 
\end{align*}
(When $\lambda_i^\varphi (\bm{s}') \le \lambda_i^{\varphi''} (\bm{s}')$, we have $\lambda_i^\varphi (\bm{s}') \le \lambda_i^{\varphi'} (\bm{s}')$, so that $\lambda_i^{\varphi \wedge \varphi'} (\bm{s}') - \lambda_i^{\varphi \wedge \varphi''} (\bm{s}') = 0 \le \lambda_i^{\varphi'} (\bm{s}') - \lambda_i^{\varphi''} (\bm{s}')$ in this case. )
By Lemma \ref{diagonal lemma}, we have 
\[ \lambda_i^{\varphi'} (\bm{s}') - \lambda_i^{\varphi''} (\bm{s}') \le \lambda_i^{\varphi'} (\bm{s}) - \lambda_i^{\varphi''} (\bm{s}). \]
Thus we can compute 
\begin{align*} 
\int_\mathbb{R} (e^{|t/\beta|} -1) \DHm_{\varphi \wedge \varphi', \varphi \wedge \varphi''} 
&= \lim_{m \to \infty} \frac{1}{m^n} \sum_{i=1}^{N_m} (e^{|(\lambda_i^{\varphi \wedge \varphi'} (\bm{s}_m') - \lambda_i^{\varphi \wedge \varphi''} (\bm{s}_m'))/\beta m|} -1)
\\
&\le \lim_{m \to \infty} \frac{1}{m^n} \sum_{i=1}^{N_m} (e^{|(\lambda_i^{\varphi'} (\bm{s}_m) - \lambda_i^{\varphi''} (\bm{s}_m))/\beta m|} -1)
\\
&= \int_\mathbb{R} (e^{|t/\beta|} - 1) \DHm_{\varphi', \varphi''}. 
\end{align*}
Thus we get 
\begin{align*} 
d_{\exp} (\varphi \wedge \varphi', \varphi \wedge \varphi'') 
&= \inf \Big{\{} \beta ~\Big{|}~ \int_\mathbb{R} (e^{|t/\beta|} -1) \DHm_{\varphi \wedge \varphi', \varphi \wedge \varphi''} \le 1 \Big{\}} 
\\
&\le \inf \Big{\{} \beta ~\Big{|}~ \int_\mathbb{R} (e^{|t/\beta|} -1) \DHm_{\varphi', \varphi''} \le 1 \Big{\}} = d_{\exp} (\varphi', \varphi''). 
\end{align*}

Now we deal with the general case. 
Since $\varphi', \varphi'' \ge \varphi' \wedge \varphi''$, we compute 
\begin{align*}
d_{\exp} (\varphi \wedge \varphi', \varphi \wedge \varphi'') 
&\le d_{\exp} (\varphi \wedge \varphi', \varphi \wedge \varphi' \wedge \varphi'') + d_{\exp} (\varphi \wedge \varphi' \wedge \varphi'', \varphi \wedge \varphi'') 
\\
&\le d_{\exp} (\varphi', \varphi' \wedge \varphi'') + d_{\exp} (\varphi'', \varphi' \wedge \varphi'')
\\
&\le 2 d_{\exp} (\varphi', \varphi''),
\end{align*}
using (\ref{distance of wedge}) and the above lemma. 
\end{proof}

In the proof of the completeness, we consider the limit of an increasing sequence $\varphi_i \in \E^{\exp} (X, L)$. 
To ensure the existence of limit, we must assume the continuity of envelopes. 

\begin{prop}
\label{convergence of increasing sequence}
Assume the continuity of envelopes holds for $(X, L)$. 
If $\{ \varphi_i \}_{i \in I}$ is a $d_{\exp}$-bounded increasing net in $\E^{\exp} (X, L)$, then we have a limit $\varphi \in \E^{\exp} (X, L)$ in $d_{\exp}$-topology. 
\end{prop}

\begin{proof}
By Proposition \ref{Ibar and d1} and the boundedness of $d_1 \le d_{\exp}$, $\sup \varphi_i$ is bounded. 
By \cite[Proposition 4.48, Theorem 9.5]{BJ3} that $\varphi_i$ converges to some $\varphi \in \E^1 (X, L)$ in the strong topology. 
Since $\varphi \ge \varphi_i \in \E^{\exp} (X, L)$, we have $\varphi \in \E^{\exp} (X, L)$. 
It suffices to show $d_{\exp} (\varphi_i, \varphi) \to 0$. 
For $\beta > 0$, take $\varepsilon > 0$ so that 
\[ \Big{(} \frac{\varepsilon}{4\beta/3} \Big{)}^{1/4} (- e^{\sup \varphi_{; 2/\beta}} E_{\exp} (\varphi_{0; 2/\beta}))^{3/4} \le 1. \]
Take $\varphi' \in \nH (X, L)$ so that $\varphi \le \varphi'$ and $d_{\exp} (\varphi, \varphi') \le \varepsilon/2$. 
Then we have $E (\varphi') - E (\varphi) = d_1 (\varphi', \varphi) \le d_{\exp} (\varphi, \varphi') \le \varepsilon/2$. 
It follows that we can take $i_\varepsilon$ so that $E (\varphi') - E (\varphi_i) = (E (\varphi') - E (\varphi)) + (E (\varphi) - E (\varphi_i)) \le \varepsilon$ for $i \ge i_\varepsilon$. 
Since $\varphi_i \le \varphi'$, we have 
\begin{align*} 
\int_\mathbb{R} (e^{|t/\beta|} -1) \DHm_{\varphi_i, \varphi'} 
&\le \Big{(} \frac{E (\varphi') - E (\varphi_i)}{4\beta/3} \Big{)}^{1/4} (-e^{\sup \varphi_i} E_{\exp} (\varphi_{i; 2/\beta}))^{3/4} 
\\
&\le \Big{(} \frac{\varepsilon}{4\beta/3} \Big{)}^{1/4} (- e^{\sup \varphi_{; 2/\beta}} E_{\exp} (\varphi_{0; 2/\beta}))^{3/4} \le 1
\end{align*}
by (\ref{distance lemma}). 
It follows that $d_{\exp} (\varphi_i, \varphi) \le d_{\exp} (\varphi_i, \varphi') + d_{\exp} (\varphi', \varphi) \le \beta + \varepsilon/2$ for $i \ge i_\varepsilon$. 
Taking the limits $i \to \infty$ and $\varepsilon \to 0$, we get $\varlimsup d_{\exp} (\varphi_i, \varphi) \le \beta$ for every $\beta > 0$, so that $\varlimsup d_{\exp} (\varphi_i, \varphi) = 0$. 
\end{proof}

Now we show the completeness of the metric space $(\E^{\exp} (X, L), d_{\exp})$. 
The proof of the completeness of the archimedean $\mathcal{E}^p (X, L)$ in \cite{Dar} adapts to our framework. 

\begin{thm}
Assume the continuity of envelopes holds for $(X, L)$. 
Then the metric space $(\E^{\exp} (X, L), d_{\exp})$ is complete. 
\end{thm}

\begin{proof}
Take a Cauchy sequence $\{ \varphi_i \}_{i \in \mathbb{N}}$ in $\E^{\exp} (X, L)$. 
It suffices to show there exists a limit $\varphi \in \E^{\exp} (X, L)$ in the metric topology. 
If there exists a subsequence $\{ \varphi_j \}_j$ converging to some $\varphi \in \E^{\exp} (X, l)$ in the metric topology, then since $d_{\exp} (\varphi_i, \varphi) \le d_{\exp} (\varphi_i, \varphi_j) + d_{\exp} (\varphi_j, \varphi)$, the original sequence $\varphi_i$ converges to $\varphi$. 
Thus we may assume $d_{\exp} (\varphi_i, \varphi_{i+1}) < 1/2^{i+1}$ by replacing the original sequence with a subsequence. 
We construct the limit $\varphi \in \E^{\exp} (X, L)$ by two steps. 

We put $\hat{\varphi}_i^p := \varphi_i \wedge \dotsb \wedge \varphi_{i+p} \in \E^{\exp} (X, L)$. 
Since $\hat{\varphi}_i^{p+1} = \hat{\varphi}_i^p \wedge \varphi_{i+p+1} \le \hat{\varphi}_i^p$, $\{ \hat{\varphi}_i^p \}_{p=0}^\infty$ is a decreasing sequence for each $i$. 
We compute 
\[ d_{\exp} (\hat{\varphi}_i^p, \hat{\varphi}_i^{p+1}) = d_{\exp} (\hat{\varphi}_i^p \wedge \varphi_{i+p}, \hat{\varphi}_i^p \wedge \varphi_{i+p+1}) \le 2 d_{\exp} (\varphi_{i+p}, \varphi_{i+p+1}) \le 1/2^{i+p}. \]
Then since 
\[ d_{\exp} (\varphi_i, \hat{\varphi}_i^p) \le \sum_{q=0}^{p-1} d_{\exp} (\varphi_i^q, \varphi_i^{q+1}) \le 1/2^i, \]
the sequence $\{ \hat{\varphi}_i^p \}_{p=0}^\infty$ is a $d_{\exp}$-Cauchy decreasing sequence, hence it has a limit $\hat{\varphi}_i \in \E^{\exp} (X, L)$ in the metric topology by Proposition \ref{convergence of decreasing sequence}. 

Since $\hat{\varphi}_i^{p+1} = \varphi_i \wedge \hat{\varphi}_{i+1}^p \le \hat{\varphi}_{i+1}^p$, we have $\hat{\varphi}_i \le \hat{\varphi}_{i+1}$. 
We compute 
\begin{align*}
d_{\exp} (\hat{\varphi}_i, \hat{\varphi}_{i+1}) = \lim_{p \to \infty} d_{\exp} (\hat{\varphi}_i^{p+1}, \hat{\varphi}_{i+1}^p) 
&= \lim_{p \to \infty} d_{\exp} (\varphi_i \wedge \hat{\varphi}_{i+1}^p, \varphi_{i+1} \wedge \hat{\varphi}_{i+1}^p)
\\
&\le 2 d_{\exp} (\varphi_i, \varphi_{i+1}) \le 1/2^i. 
\end{align*}
This shows $d_{\exp} (\hat{\varphi}_0, \hat{\varphi}_i) \le 1$, so that $\hat{\varphi}_i$ is a $d_{\exp}$-bounded increasing sequence, hence it admits a limit $\varphi \in \E^{\exp} (X, L)$ in the metric topology by Proposition \ref{convergence of increasing sequence}, under the assumption that the continuity of envelopes holds. 

Finally, we compute 
\begin{align*} 
d_{\exp} (\varphi, \varphi_i) 
&\le d_{\exp} (\varphi, \hat{\varphi}_i) + d_{\exp} (\hat{\varphi}_i, \varphi_i) 
\\
&\le d_{\exp} (\varphi, \hat{\varphi}_i) + \lim_{p \to \infty} d_{\exp} (\hat{\varphi}_i^p, \varphi_i) \le 1/2^{i-1} + 1/2^i, 
\end{align*}
so that we get $d_{\exp} (\varphi, \varphi_i) \to 0$. 
\end{proof}

\subsection{The non-archimedean $\mu$-entropy}
\label{non-archimedean mu-entropy}

Now we are ready to extend the non-archimedean $\mu$-entropy $\NAmu^\lambda = \NAmu + \lambda \bm{\check{\sigma}}: \nH (X, L) \to \mathbb{R}$ to $\E^{\exp} (X, L)$. 
As we observed in section \ref{Towards non-archimedean formalism: moment measure of test configuration}, we can write $\NAmu$ as 
\[ \NAmu (\varphi) = - 2\pi \frac{\int_{X^{\mathrm{NA}}} A_X \int e^{-t} \mathcal{D}_\varphi + E_{\exp}^{K_X} (\varphi)}{\iint_{X^{\mathrm{NA}}} e^{-t} \mathcal{D}_\varphi} \]
for $\varphi \in \nH (X, L)$. 
To extend the functional, we would construct an extension of $E_{\exp}^M$ on $\nH (X, L)$ to $\E^{\exp} (X, L)$ and show the continuity of $\int_\mathbb{R} p (t) e^{-t} \DHm_\varphi$ (for $p (t) = 1, n-t$), $\int e^{-t} \mathcal{D}_\varphi$ and $E_{\exp}^M (\varphi)$ on $\varphi \in \E^{\exp} (X, L)$ with respect to the $d_{\exp}$-topology. 
The key tools are tomographic expressions of these functionals. 

\subsubsection{The continuity of Duistermaat--Heckman measure with respect to $d_1$-topology}
\label{The continuity of Duistermaat--Heckman measure with respect to d1-topology}

We firstly observe the weak continuity of Duistermaat--Heckman measure with respect to the strong topology. 

\begin{lem}
\label{limit of DH measure for d1 convergent sequence}
Let $\chi$ be a continuous function on $\mathbb{R}$ which has left bounded support, i.e. $\chi (t) =0$ for every $t \ll 0$. 
Then we have $\int_\mathbb{R} \chi \DHm_{\varphi_i} \to \int_\mathbb{R} \chi \DHm_\varphi$ if $\varphi_i \to \varphi \in \E^1 (X, L)$ in $d_1$. 
\end{lem}

\begin{proof}
Assume firstly $\chi$ is locally Lipschitz. 
Then for each $T \in \mathbb{R}$, we have a constant $C_T$ satisfying $|\chi (t) - \chi (s)| \le C_T |t-s|$ for every $t, s \le T$ as $\chi$. 
For $\varphi, \varphi' \in \nH (X, L)$, we compute 
\begin{align*} 
|\int_\mathbb{R} \chi \DHm_\varphi - \int_\mathbb{R} \chi \DHm_{\varphi'}| 
&\le \lim_{m \to \infty} \frac{1}{m^n} \sum_{i=1}^{N_m} |\chi (\lambda^\varphi_i (\bm{s})/m) - \chi (\lambda^{\varphi'}_i (\bm{s})/m)| 
\\
&\le C_{\max \{ \varphi (v_{\mathrm{triv}}), \varphi' (v_{\mathrm{triv}}) \}} \lim_{m \to \infty} \frac{1}{m^n} \sum_{i=1}^{N_m} |\lambda^\varphi_i (\bm{s})/m - \lambda^{\varphi'}_i (\bm{s})/m|
\\
&= C_{\max \{ \varphi (v_{\mathrm{triv}}), \varphi' (v_{\mathrm{triv}}) \}} d_1 (\varphi, \varphi'). 
\end{align*}
For $\varphi, \varphi' \in \E^1 (X, L)$, taking convergent decreasing nets $\varphi_i \searrow \varphi, \varphi_i' \searrow \varphi'$ so that $\varphi_i, \varphi_i' \in \nH (X, L)$ and passing to the limit, we get 
\[ |\int_\mathbb{R} \chi \DHm_\varphi - \int_\mathbb{R} \chi \DHm_{\varphi'}| \le C_{\max \{ \varphi (v_{\mathrm{triv}}), \varphi' (v_{\mathrm{triv}}) \}} d_1 (\varphi, \varphi'). \]

For a weak convergent $\varphi_i \to \varphi$, we have $\varphi_i (v_{\mathrm{triv}}) \to \varphi (v_{\mathrm{triv}})$, so the constants $C_{\max \{ \varphi_i (v_{\mathrm{triv}}), \varphi (v_{\mathrm{triv}}) \}}$ are bounded. 
Then by the above inequality, $d_1$-convergence of $\varphi_i \to \varphi$ implies $\int_\mathbb{R} \chi \DHm_{\varphi_i} \to \int_\mathbb{R} \chi \DHm_\varphi$. 

Now assume $\chi$ is just continuous. 
For $\varepsilon > 0$, take a locally Lipschitz function $\tilde{\chi}$ with left bounded support so that $|\chi - \tilde{\chi}| < \varepsilon$ on $(-\infty, \max \{ \varphi (v_{\mathrm{triv}}), \varphi' (v_{\mathrm{triv}}) \}]$. 
We note such $\tilde{\chi}$ exists by Weierstrass approximation. 
Then there exists $i_\varepsilon$ such that 
\[ |\int_\mathbb{R} \chi \DHm_{\varphi_i} - \int_\mathbb{R} \chi \DHm_\varphi| \le 2\varepsilon (e^L) + |\int_\mathbb{R} \tilde{\chi} \DHm_{\varphi_i} - \int_\mathbb{R} \tilde{\chi} \DHm_\varphi| \le 3 \varepsilon. \]
Taking the limit $\varepsilon \searrow 0$, we obtain the claim. 
\end{proof}
 
\begin{prop}
\label{lower semi-continuity of norms in d1 topology}
Let $\chi$ be a non-negative continuous function on $\mathbb{R}$. 
If $\varphi_i \to \varphi \in \E^1 (X, L)$ in $d_1$, then we have 
\[ \int_\mathbb{R} \chi \DHm_\varphi \le \varliminf_{i \to \infty} \int_\mathbb{R} \chi \DHm_{\varphi_i}. \]
\end{prop}

\begin{proof}
Take continuous cut off functions $\beta_j: \mathbb{R} \to [0,1]$ so that $\beta_j = 0$ on $(-\infty, -j]$ and $\beta = 1$ on $[-j+1, \infty)$. 
By the monotone convergence theorem, we have 
\[ \lim_{i \to \infty} \int_\mathbb{R} \beta_j \chi \DHm_\varphi = \int_\mathbb{R} \chi \DHm_\varphi. \]
On the other hand, by the above lemma, we have 
\[ \int_\mathbb{R} \beta_j \chi \DHm_\varphi = \lim_{i \to \infty} \int_\mathbb{R} \beta_j \chi \DHm_{\varphi_i} \le \lim_{i \to \infty} \int_\mathbb{R} \chi \DHm_{\varphi_i}. \]
Thus we get 
\[ \int_\mathbb{R} \chi \DHm_\varphi \le \lim_{i \to \infty} \int_\mathbb{R} \chi \DHm_{\varphi_i}. \]
\end{proof}

\subsubsection{Key estimates}

We firstly recall fundamental estimates established in \cite{BJ3}. 

\begin{prop}[Lemma 5.28 in \cite{BJ3} (cf. Lemma 3.23 in \cite{BJ1})]
\label{BJ estimate}
There exists a positive constant $C_n$ depending only on the dimension $n$ of $X$ such that 
\[ \Big{|} \int_{X^{\mathrm{NA}}} (\psi - \psi (v_{\mathrm{triv}})) \mathrm{MA} (\varphi) \Big{|} \le C_n d_1 (\varphi, 0)^{\frac{1}{2}} \max \{ d_1 (\varphi, 0), d_1 (\psi, 0) \}^{\frac{1}{2}} \]
for every $\varphi, \psi \in \E^1 (X, L)$. 
\end{prop}

\begin{proof}
We put $\varphi', \psi' = 0$ in \cite[Lemma 5.28]{BJ3} and then apply Proposition \ref{Ibar and d1}. 
We note $\psi (v_{\mathrm{triv}}) \int_{X^{\mathrm{NA}}} \mathrm{MA} (\varphi) = \int_{X^{\mathrm{NA}}} \psi \mathrm{MA} (0)$. 
\end{proof}

The left hand side is invariant when replacing $\varphi \mapsto \varphi +c$, $\psi \mapsto \psi + c'$, so we actually have 
\[ \Big{|} \int_{X^{\mathrm{NA}}} (\psi - \psi (v_{\mathrm{triv}})) \mathrm{MA} (\varphi) \Big{|} \le C_n \underline{d}_1 (\varphi, 0)^{\frac{1}{2}} \max \{ \underline{d}_1 (\varphi, 0), \underline{d}_1 (\psi, 0) \}^{\frac{1}{2}} \]
for 
\[ \underline{d} (\varphi, \psi) := \inf_{c \in \mathbb{R}} d_1 (\varphi + c, \psi). \]

The following is also a consequence of \cite[Lemma 5.28]{BJ3}

\begin{prop}
\label{double limit}
Let $\{ \varphi_i \}_{i \in I}, \{ \psi_i \}_{i \in I} \subset \E^1 (X, L)$ be nets converging strongly to $\varphi, \psi \in \E^1 (X, L)$, respectively. 
Then we have 
\[ \lim_{i \to \infty} \int_{X^{\mathrm{NA}}} \psi_i \mathrm{MA} (\varphi_i) = \int_{X^{\mathrm{NA}}} \psi \mathrm{MA} (\varphi). \]
We also have 
\[ \lim_{i \to \infty} \int_{X^{\mathrm{NA}}} g_i \mathrm{MA} (\varphi_i) = \int_{X^{\mathrm{NA}}} g \mathrm{MA} (\varphi) \]
for a uniform convergent net of continuous functions $\{ g_i \}_{i \in I} \in C^0 (X^{\mathrm{NA}})$. 
\end{prop}

\begin{proof}
By the assumption, we have $I (\varphi_i) \to I (\varphi)$ and $I (\psi_i) \to I (\psi)$, so we may assume $I (\varphi_i), I (\psi_i) \le C$ for a uniform constant $C > 0$ by replacing $I$ with $\{ i \ge i_0 \}$ if necessary. 
Then by \cite[Lemma 5.28]{BJ3}, we obtain 
\begin{align*} 
\Big{|} \int_{X^{\mathrm{NA}}} \psi \mathrm{MA} (\varphi) 
&- \int_{X^{\mathrm{NA}}} \psi_i \mathrm{MA} (\varphi_i) \Big{|} 
\\
&\le \Big{|} \int_{X^{\mathrm{NA}}} \psi (\mathrm{MA} (\varphi) - \mathrm{MA} (\varphi_i)) \Big{|} + (e^L) |\sup \psi - \sup \psi_i|
\\
&\quad + \Big{|} \int_{X^{\mathrm{NA}}} (\psi - \psi_i) (\mathrm{MA} (\varphi_i) - \mathrm{MA} (\varphi_{\mathrm{triv}})) \Big{|} 
\\
&\le C_n I (\psi)^{\frac{1}{2^n}} \max \{ I (\varphi), I (\varphi_i), I (\psi) \}^{\frac{1}{2} - \frac{1}{2^n}} I (\varphi, \varphi_i)^{\frac{1}{2}} 
\\
&\quad + (e^L) |\sup \psi - \sup \psi_i|
\\
&\qquad + C_n I (\varphi_i)^{\frac{1}{2}} \max \{ I (\varphi_i), I (\psi), I (\psi_i) \}^{\frac{1}{2} - \frac{1}{2^n}} I (\psi, \psi_i)^{\frac{1}{2^n}}
\\
&\le C' I (\varphi, \varphi_i)^{\frac{1}{2}} + (e^L) |\sup \psi - \sup \psi_i| + C'' I (\psi, \psi_i)^{\frac{1}{2^n}} 
\end{align*}
for uniform constants $C', C'' > 0$. 
This estimate proves the first claim. 
The latter claim follows from the first claim and the fact \cite[Theorem 2.2, Corollary 2.11]{BJ3} that for any $\varepsilon$ there exists $\psi, \psi' \in \nH (X, L)$ such that $|g - (\psi - \psi')| \le \varepsilon$ (cf. \cite[Corollary 2.8]{BJ1}). 
\end{proof}

Let $L_0, \ldots, L_n$ be ample $\mathbb{Q}$-line bundles. 
The \textit{energy paring} $(L_0, \varphi_0) \dotsb (L_n, \varphi_n) \in \mathbb{R}$ for $\varphi_i \in \E^1 (X, L_i)$ is constructed in \cite{BJ3}. 
In this article, we are interested in the following three cases: 
\begin{align*} 
\frac{1}{(n+1)!} (L, \varphi)^{\cdot n+1} 
&= E (\varphi) = \int_\mathbb{R} t \DHm_\varphi, 
\\ 
\frac{1}{n!} (M, 0) \cdot (L, \varphi)^{\cdot n} 
&= \frac{d}{dt}\Big{|}_{t=0} \frac{1}{(n+1)!} (L+tM, \varphi)^{\cdot n+1}, 
\\
\frac{1}{n!} (0, \psi) \cdot (L, \varphi)^{\cdot n} 
&= \int_{X^{\mathrm{NA}}} \psi \mathrm{MA} (\varphi). 
\end{align*}
When $M$ is not ample, we put 
\[ \frac{1}{n!} (M, 0) \cdot (L, \varphi)^{\cdot n} := \frac{1}{n!} (M_1, 0) \cdot (L, \varphi)^{\cdot n} - \frac{1}{n!} (M_2, 0) \cdot (L, \varphi)^{\cdot n}, \]
using ample $\mathbb{Q}$-line bundles $M_1, M_2$ with $M = M_1 - M_2$. 

\begin{thm}[Theorem 5.32 in \cite{BJ3}]
\label{BJ estimate 2}
Let $M$ be a $\mathbb{Q}$-line bundle and $L$ be ample $\mathbb{Q}$-line bundles on $X$. 
Take ample $M_1, M_2$ and $\theta \ge 1$ so that $\theta^{-1} L \le M_1, M_2 \le \theta L$. 
Then there exists a positive constant $C_n$ depending only on the dimension $n$ of $X$ such that 
\[ |(M, 0) \cdot (L, \varphi)^{\cdot n} - (M, 0) \cdot (L, \varphi')^{\cdot n}| \le C_n \theta^{n^2} d_1 (\varphi, \varphi')^{\frac{1}{2^{3n-2}}} (d_1 (\varphi, 0) + d_1 (\varphi', 0))^{1 - \frac{1}{2^{3n-2}}} \]
for every $\varphi, \varphi' \in \E^{\exp} (X, L)$. 
Here $d_1$ denotes the $d_1$-distance on $\E^1 (X, L)$. 
In particular, we have 
\[ \lim_{i \to \infty} (M, 0) \cdot (L, \varphi_i)^{\cdot n} = (M, 0) \cdot (L, \varphi)^{\cdot n} \] 
when a net $\{ \varphi_i \}_{i \in I} \subset \E^1 (X, L)$ converges strongly to $\varphi \in \E^1 (X, L)$.  
\end{thm}

\begin{proof}
We firstly compute 
\[ (M, 0) \cdot (L, \varphi)^{\cdot n} = (M, 0) \cdot (L, \varphi - \sup \varphi)^{\cdot n} + n (M, L^{\cdot n-1}) \cdot \sup \varphi. \]
Then by \cite[Theorem 5.32, Proposition 5.26]{BJ3}, we have 
\begin{align*} 
|(M, 0) \cdot (L, \varphi)^{\cdot n} 
&- (M, 0) \cdot (L, \varphi')^{\cdot n}| 
\\
&\le |(M_1, 0) \cdot (L, \varphi - \sup \varphi)^{\cdot n} - (M_1, 0) \cdot (L, \varphi' - \sup \varphi')^{\cdot n}|
\\
&\quad + |(M_2, 0) \cdot (L, \varphi - \sup \varphi)^{\cdot n} - (M_2, 0) \cdot (L, \varphi' - \sup \varphi')^{\cdot n}| 
\\
&\qquad + n ((M_1, L^{\cdot n-1}) + (M_2, L^{\cdot n-1})) \cdot |\sup \varphi - \sup \varphi'| 
\\
&\le C_n \theta^{n^2} I (\varphi, \varphi')^{\frac{1}{2^{3n-2}}} \max \{ I (\varphi), I (\varphi') \}^{1 - \frac{1}{2^{3n-2}}} 
\\
&\quad+ 2 n \theta (e^L) \cdot |\sup \varphi - \sup \varphi'|^{\frac{1}{2^{3n-2}}} (|\sup \varphi| + |\sup \varphi'|)^{1- \frac{1}{2^{3n-2}}}
\\
&\le \max \{ C_n, 2 n \} \theta^{n^2} \bar{I} (\varphi, \varphi')^{\frac{1}{2^{3n-2}}} (\bar{I} (\varphi, 0) + \bar{I} (\varphi', 0) )^{1 - \frac{1}{2^{3n-2}}} 
\end{align*}
for some constant $C_n > 0$ depending only on $n$. 
Here the last inequality follows by $a^{1/p} b^{1/q} + c^{1/p} d^{1/q} \le (a+c)^{1/p} (b +d)^{1/q}$ for $a, b, c, d \ge 0$ and $p, q \ge 1$ with $1/p+ 1/q = 1$. 
By Proposition \ref{Ibar and d1}, we conclude the proof. 
\end{proof}

The following estimate is crucial in our argument. 
We note 
\[ d_1 (\varphi \wedge \tau - \tau, 0) = -E (\varphi \wedge \tau - \tau) = - \frac{(L, \varphi)^{\cdot n+1}}{(n+1)!} \]
as $\DHm_{\varphi \wedge \tau - \tau}$ is supported on $(-\infty, 0]$ by $\varphi \wedge \tau - \tau \le 0$. 

\begin{lem}
\label{exponential domination}
For $\varphi \in \E^1 (X, L)$ and $\alpha > 0$, we have 
\[ d_1 (\varphi \wedge \tau - \tau, 0) \le \min \Big{\{} \frac{1}{e \alpha} \int_\mathbb{R} e^{-\alpha t} \DHm_\varphi \cdot e^{\alpha \tau}, (e^L) |\tau| + d_1 (\varphi, 0) \Big{\}} \]
for every $\tau \in \mathbb{R}$. 

As a consequence, for $\varphi \in \E^{\exp} (X, L)$ and for any $\rho > \varepsilon > 0$, we have a positive constant $C > 0$ depending boundedly on $\log (\rho- \varepsilon), \log (\rho + \varepsilon), \log (e^L), d_1 (\varphi, 0)$ and $\int_\mathbb{R} e^{- (\rho + \varepsilon)t} \DHm_\varphi = -E_{\exp} (\varphi_{; \rho+ \varepsilon})$ such that 
\[ e^{-\rho \tau} d_1 (\varphi \wedge \tau - \tau, 0) \le C e^{-\varepsilon |\tau|} \]
for every $\tau \in \mathbb{R}$. 
\end{lem}

\begin{proof}
Using $(\tau - t) e^{\alpha t} \le (e \alpha)^{-1} e^{\alpha \tau}$ for $t \le \tau$, we compute 
\begin{align*} 
d_1 (\varphi \wedge \tau - \tau, 0) 
&= \int_\mathbb{R} |t| \DHm_{\varphi \wedge \tau - \tau} = \int_{(-\infty, \tau)} (\tau -t) e^{\alpha t} e^{-\alpha t} \DHm_\varphi 
\\
&\le (e \alpha)^{-1} e^{\alpha \tau} \int_\mathbb{R} e^{-\alpha t} \DHm_\varphi. 
\end{align*}
On the other hand, we compute 
\[ d_1 (\varphi \wedge \tau - \tau, 0) = \int_{(-\infty, \tau)} |t- \tau| \DHm_\varphi \le |\tau| \int_\mathbb{R} \DHm_\varphi + \int_\mathbb{R} |t| \DHm_\varphi. \]
Thus we get the first claim. 

We can check the second claim as follows. 
Since $\rho - \varepsilon > 0$, we have a constant $C'$ depending boundedly on $\log (\rho- \varepsilon), \log (e^L)$ and $d_1 (\varphi, 0)$ such that 
\[ e^{-(\rho - \varepsilon) \tau} \Big{(} (e^L) |\tau| + d_1 (\varphi, 0) \Big{)} \le C' \]
on $\tau \ge 0$. 
Explicitly, we may take $C' = \frac{(e^L)}{\rho - \varepsilon} e^{-1 + d_1 (\varphi, 0) \cdot (\rho -\varepsilon)/(e^L)}$. 
Then by the first claim, we get 
\begin{align*} 
e^{-\rho \tau} d_1 (\varphi \wedge \tau - \tau, 0) 
&\le \min \Big{\{} \frac{\int_\mathbb{R} e^{-(\rho + \varepsilon) t} \DHm_\varphi}{e (\rho + \varepsilon)} e^{\varepsilon \tau}, C' e^{-\varepsilon \tau} \Big{\}} 
\\
&\le \max \Big{\{} \frac{\int_\mathbb{R} e^{-(\rho + \varepsilon) t} \DHm_\varphi}{e (\rho + \varepsilon)}, C' \Big{\}} \cdot e^{-\varepsilon |\tau|},
\end{align*}
which shows the second claim. 
\end{proof}

\begin{rem}
We show $d_{\exp} (\varphi \wedge \tau - \tau, 0) \to 0$ as an independent interest. 
We note 
\[ \int_\mathbb{R} (e^{|t/\beta|} -1) \DHm_{\varphi \wedge \tau - \tau} = \int_{(-\infty, \tau)} (e^{-(t -\tau)/\beta} -1) \DHm_\varphi \]
for $\varphi \in \E^{\exp} (X, L)$. 
Indeed, since both sides are continuous along decreasing nets $\varphi_i \searrow \varphi$, it suffices to check the case $\varphi \in \nH (X, L)$, which is confirmed in \cite[Proposition 3.4]{BJ2}: $\DHm_{\varphi \wedge \tau - \tau} = (t' \mapsto t' - \tau)_* \DHm_{\varphi \wedge \tau} = (t' \mapsto t' - \tau)_* (t \mapsto \min \{ t, \tau \})_* \DHm_\varphi = (t \mapsto \min \{ t, \tau \} - \tau)_* \DHm_\varphi$. 
Then since 
\[ \int_{(-\infty, \tau)} (e^{-(t -\tau)/\beta} -1) \DHm_\varphi \le e^{\tau/\beta} \int_\mathbb{R} e^{-t/\beta} \DHm_\varphi - \int_{(-\infty, \tau)} \DHm_\varphi \to 0 \]
as $\tau \to -\infty$, we get $\int_\mathbb{R} (e^{|t/\beta|} -1) \DHm_{\varphi \wedge \tau - \tau} \to 0$ for any $\beta > 0$. 
Then by Lemma \ref{dexp convergence criterion}, we get $d_{\exp} (\varphi \wedge \tau - \tau, 0) \to 0$ as $\tau \to -\infty$. 
\end{rem}

\subsubsection{Continuity of exponential moment energy}
\label{Continuity of exponential moment energy}

We make use of the following tomographic expression to show the continuity. 

\begin{prop}
\label{tomographic expression of Eexp}
If $\varphi \in \E^1 (X, L)$ has finite $E_{\exp} (\varphi_{; 1+ \varepsilon})$ for some $\varepsilon > 0$, we have 
\begin{gather*}
E_{\exp} (\varphi) = \int_\mathbb{R} \frac{(L, \varphi \wedge \tau - \tau)^{\cdot n+1}}{(n+1)!} e^{-\tau} d \tau. 
\end{gather*}
\end{prop}

\begin{proof}
We note $\frac{(L, \varphi \wedge \tau - \tau)^{\cdot n+1}}{(n+1)!}$ is continuous on $\tau$, so it is measurable. 
Suppose $\frac{(L, \varphi \wedge \tau - \tau)^{\cdot n+1}}{(n+1)!} w (\tau)$ is integrable with respect to $d\tau$, then using Fubini--Tonelli theorem, we compute 
\begin{align*}
\int_\mathbb{R} \frac{(L, \varphi \wedge \tau - \tau)^{\cdot n+1}}{(n+1)!} w (\tau) d\tau 
&= \int_\mathbb{R} d\tau ~w (\tau) \int_\mathbb{R} 1_{(-\infty, \tau)} (t) (t -\tau) \DHm_\varphi (t)
%\\
%&= \int_{\mathbb{R} \times \mathbb{R}} 1_{\{ (\tau, t) \in \mathbb{R}^2 ~|~ t < \tau \}} (\tau, t) w (\tau) (t-\tau) d\tau \otimes \DHm_\varphi (t)
\\
&= \int_\mathbb{R} \DHm_\varphi (t) \int_\mathbb{R} 1_{(t, \infty)} (\tau) w (\tau) (t - \tau) d\tau
\\
&= \int_\mathbb{R} \DHm_\varphi (t) \int_{(t, \infty)} w (\tau) (t - \tau) d\tau, 
\end{align*}
using 
\begin{align*}
\frac{(L, \varphi \wedge \tau - \tau)^{\cdot n+1}}{(n+1)!}
&= \int_\mathbb{R} t \DHm_{\varphi \wedge \tau - \tau} = \int_{(-\infty, \tau)} (t -\tau) \DHm_\varphi (t). 
\end{align*}

For $w (\tau) = e^{-\tau}$, we have $\int_{(t, \infty)} e^{-\tau} (t - \tau) d\tau = - e^{-t}$, so we get 
\[ \int_\mathbb{R} \frac{(L, \varphi \wedge \tau - \tau)^{\cdot n+1}}{(n+1)!} e^{-\tau} d\tau = -\int_\mathbb{R} e^{-t} \DHm_\varphi (t) \]
if $\frac{(L, \varphi \wedge \tau - \tau)^{\cdot n+1}}{(n+1)!} e^{-\tau}$ is integrable. 
It suffices to check $\frac{(L, \varphi \wedge \tau - \tau)^{\cdot n+1}}{(n+1)!} e^{-\tau}$ is integrable. 
This follows by the second claim of Lemma \ref{exponential domination}: for $\varphi \in \E^{\exp} (X, L)$, we have 
\begin{align*} 
|\frac{(L, \varphi \wedge \tau - \tau)^{\cdot n+1}}{(n+1)!} e^{-\tau}| 
%&\le e^{-\tau} \int_\mathbb{R} |t| \DHm_{\varphi \wedge \tau -\tau} 
%\\
&=  e^{-\tau} d_1 (\varphi \wedge \tau - \tau, 0) \le C_\varepsilon e^{-\varepsilon |\tau|}. 
\end{align*}
\end{proof}

\begin{cor}
\label{strong Eexp boundedness implies Eexp convergence}
Suppose a sequence $\varphi_i \in \E^1 (X, L)$ converges to $\varphi \in \E^1 (X, L)$ in $d_1$ and has bounded $E_{\exp} (\varphi_{i ; 1+ \varepsilon})$ for some $\varepsilon > 0$, then $E_{\exp} (\varphi_i) \to E_{\exp} (\varphi)$. 
\end{cor}

\begin{proof}
By the assumption, we have a uniform estimate
\[ |\frac{(L, \varphi_i \wedge \tau - \tau)^{\cdot n+1}}{(n+1)!} e^{-\tau}| \le C_\varepsilon e^{-\varepsilon |\tau|} \]
thanks to Lemma \ref{exponential domination}. 
Since $\varphi_i \to \varphi$ in $d_1$, we have $\varphi_i \wedge \tau - \tau \to \varphi \wedge \tau - \tau$ in $d_1$ for each $\tau$, so that we get $\frac{(L, \varphi_i \wedge \tau - \tau)^{\cdot n+1}}{(n+1)!} e^{-\tau} \to \frac{(L, \varphi \wedge \tau - \tau)^{\cdot n+1}}{(n+1)!} e^{-\tau}$ for every $\tau \in \mathbb{R}$ by $(L, \varphi)^{\cdot n+1}/(n+1)! = E (\varphi)$. 
Then the claim follows by the dominated convergence theorem. 
\end{proof}

\begin{thm}
\label{dexp convergence implies Eexp convergence}
For every $\rho > 0$, $E_{\exp} (\varphi_{; \rho})$ is continuous on $\E^{\exp} (X, L)$ with respect to $d_{\exp}$-topology. 
\end{thm}

\begin{proof}
This is a consequence of the above corollary and Proposition \ref{dexp Cauchy implies Eexp bounded}. 
\end{proof}

\subsubsection{$E_{\exp}$-topology}
\label{Eexp-topology}

\begin{defin}
The \textit{$E_{\exp}$-topology} on $\E^{\exp} (X, L)$ is the coarsest refinement of the strong topology (i.e. $d_1$-topology) which makes $E_{\exp} (\varphi_{; \rho})$ continuous for every $\rho >0$. 
\end{defin}

Thanks to Theorem \ref{dexp convergence implies Eexp convergence}, a functional on $\E^{\exp} (X, L)$ is continuous (resp. usc, lsc) with respect to $d_{\exp}$-topology if it is continuous (resp. usc, lsc) with respect to $E_{\exp}$-topology. 
Thanks to the following, to check the continuity of a functional with respect to $E_{\exp}$-topology, it suffices to check the sequential continuity. 

\begin{prop}
\label{Eexp topology is first countable}
The $E_{\exp}$-topology on $\E^{\exp} (X, L)$ is first countable. 
\end{prop}

\begin{proof}
Let us introduce the $E_{\exp}^{\mathbb{N}}$-topology on $\E^{\exp} (X, L)$ as the coarsest refinement of the strong topology which makes $E_{\exp} (\varphi_{; n})$ continuous for every $n \in \mathbb{N}_+$. 
By Corollary \ref{strong Eexp boundedness implies Eexp convergence}, $E_{\exp} (\varphi_{; \rho})$ for $\rho \in \mathbb{R}_+$ is sequentially continuous with respect to the $E_{\exp}^{\mathbb{N}}$-topology. 
Thus it suffices to show the first countability of the $E_{\exp}^{\mathbb{N}}$-topology, in which case we get the equivalence of the $E_{\exp}^{\mathbb{N}}$-topology and $E_{\exp}$-topology. 

Now observe that the $E_{\exp}^{\mathbb{N}}$-topology is equivalent to the induced topology via 
\[ \E^{\exp} (X, L) \hookrightarrow \E^1 (X, L) \times \mathbb{R}^{\mathbb{N}_+}: \varphi \mapsto (\varphi, E_{\exp} (\varphi_{;1}), E_{\exp} (\varphi_{; 2}), \ldots ). \]
Note $\E^1 (X, L) \times \mathbb{R}^{\mathbb{N}_+}$ is first countable as it is a countable product of first countable spaces. 
Thus the subspace $\E^{\exp} (X, L)$ is also first countable, which shows the first countability of the $E_{\exp}^{\mathbb{N}}$-topology as desired. 
\end{proof}

There is a Fr\'echet type distance compatible with the $E_{\exp}$-topology: 
\begin{equation} 
d_{E_{\exp}} (\varphi, \varphi') = d_1 (\varphi, \varphi') + \int_0^\infty \frac{|E_{\exp} (\varphi_{;\rho}) - E_{\exp} (\varphi'_{;\rho})|}{1+ |E_{\exp} (\varphi_{;\rho}) - E_{\exp} (\varphi'_{;\rho})|} e^{-\rho} d\rho. 
\end{equation}
Indeed, if a sequence $\varphi_i$ converges to $\varphi$ in $E_{\exp}$-topology, then $d_1 (\varphi_i, \varphi) \to 0$ and $E_{\exp} (\varphi_{i ; \rho}) \to E_{\exp} (\varphi_{; \rho})$ for every $\rho > 0$, so that $d_{E_{\exp}} (\varphi_i, \varphi) \to 0$ by the dominated convergence theorem. 
Thus $d_{E_{\exp}}$-closed set is $E_{\exp}$-closed thanks to the first countability of $E_{\exp}$-topology. 
Conversely, if a sequence $\varphi_i$ converges to $\varphi$ in $d_{E_{\exp}}$-topology, then $\varphi_i$ converges to $\varphi$ in $d_1$ and $f_i (\rho) = E_{\exp} (\varphi_{i ; \rho})$ converges in measure to $f (\rho) = E_{\exp} (\varphi_{; \rho})$ with respect to the measure $dm = e^{-\rho} d\rho$. 
Namely, $m (\{ \rho \in (0, \infty) ~|~ |f_i (\rho) - f (\rho)| > \varepsilon \}) \to 0$ for every $\varepsilon > 0$. 
It follows that there exists a subsequence $f_{i_j}$ which converges to $f$ almost everywhere. 
By Corollary \ref{strong Eexp boundedness implies Eexp convergence}, $E_{\exp} (\varphi_{i_j; \rho})$ converges to $E_{\exp} (\varphi_{; \rho})$ for every $\rho > 0$. 
Therefore, $d_{E_{\exp}}$-convergent sequence has $E_{\exp}$-convergent subsequence. 
Now let $F \subset \E^{\exp} (X, L)$ be a $E_{\exp}$-closed set. 
Take a sequence $\{ \varphi_i \} \subset F$ which converges to $\varphi \in \E^{\exp} (X, L)$ with respect to the $d_{E_{\exp}}$-topology. 
Since we have an $E_{\exp}$-convergent subsequence $\{ \varphi_{i_j} \} \subset \{ \varphi_i \}$ converging to $\varphi$, $\varphi$ is in $F$ by the $E_{\exp}$-closedness, which shows $F$ is $d_{E_{\exp}}$-closed. 

\begin{quest}
As we saw in Corollary \ref{strong Eexp boundedness implies Eexp convergence}, $E_{\exp}$-convergence of $\varphi_i \to \varphi$ is equivalent to $d_1$-convergence with bounded $E_{\exp} (\varphi_{i; \rho'})$ for every $\rho' > 0$. 
This looks like rather weak condition. 
Is $E_{\exp}$-topology strictly weaker than $d_{\exp}$-topology? 
\end{quest}

\begin{prop}
Let $\chi$ be a $C^2$-function on $\mathbb{R}$ satisfying $|\chi'' (\tau)| \le C e^{-\alpha \tau}$ on $(-\infty, \tau_0]$ for some $\tau_0 \in \mathbb{R}$ and $C, \alpha > 0$. 
If $\varphi_i \to \varphi \in \E^{\exp} (X, L)$ in $E_{\exp}$-topology, then $\int_\mathbb{R} \chi \DHm_{\varphi_i} \to \int_\mathbb{R} \chi \DHm_\varphi$. 
\end{prop}

\begin{proof}
Take a smooth cut-off function $\beta: \mathbb{R} \to [0,1]$ so that $\beta = 0$ on $[\tau_0, \infty)$ and $\beta = 1$ on $(-\infty, \tau_0-1]$. 
Then for $w = (\beta \chi)''$ we have $|w (\tau)| \le C e^{-\alpha \tau}$ on $\mathbb{R}$ by taking larger $C$ if necessary. 
Thanks to this and Lemma \ref{exponential domination}, $\frac{(L, \varphi \wedge \tau - \tau)^{\cdot n+1}}{(n+1)!} w (\tau)$ is integrable. 
Thus we compute 
\[ \int_\mathbb{R} \frac{(L, \varphi \wedge \tau - \tau)^{\cdot n+1}}{(n+1)!} w (\tau) d\tau = \int_\mathbb{R} \beta \chi \DHm_\varphi \]
by $\int_{(t, \infty)} w (\tau) (t - \tau) d\tau = (\beta \chi) (\tau)$, as in the proof of Proposition \ref{tomographic expression of Eexp}. 
Again by Lemma \ref{exponential domination}, convergence in $E_{\exp}$-topology implies a uniform estimate 
\[ |\frac{(L, \varphi_i \wedge \tau - \tau)^{\cdot n+1}}{(n+1)!} w (\tau)| \le C_\varepsilon e^{-\varepsilon |\tau|} \]
and pointwise convergence $(L, \varphi_i \wedge \tau - \tau)^{\cdot n+1} \to (L, \varphi \wedge \tau - \tau)^{\cdot n+1}$. 
Then by the dominated convergence theorem, we get the continuity of $\int_\mathbb{R} \beta \chi \DHm_\varphi$. 
As for $\int_\mathbb{R} (1-\beta) \chi \DHm_\varphi$, we can apply Lemma \ref{limit of DH measure for d1 convergent sequence}. 
\end{proof}

Since $d_p \le \lceil p \rceil \cdot d_{\exp}$, $d_{\exp}$-convergence implies $d_p$-convergence. 
We still have $d_p$-convergence in $E_{\exp}$-convergence. 

\begin{prop}
\label{Eexp convergence implies dp convergence}
$E_{\exp}$-convergence implies $d_p$-convergence for every $1 \le p < \infty$. 
\end{prop}

\begin{proof}
By Remark \ref{dp distance}, we have  
\[ d_p (\varphi_i, \varphi)^{p (q-1)} \le d_1 (\varphi_i, \varphi)^{q-p} d_q (\varphi_i, \varphi)^{q(p-1)} \le d_1 (\varphi_i, \varphi)^{q-p} (d_q (\varphi_i, 0) + d_q (\varphi, 0))^{q (p-1)}. \]
Thus it suffices to bound $d_q (\varphi_i, 0)^q = \int |t|^q \DHm_{\varphi_i}$ uniformly along $E_{\exp}$-convergent sequence. 
By the above proposition, we have $\int |t|^q \DHm_{\varphi_i} \to \int |t|^q \DHm_\varphi$, so it is uniformly bounded. 
\end{proof}

Thus $E_{\exp}$-topology is equivalent to the coarsest refinement of $d_p$-topology which makes $E_{\exp} (\varphi_{; \rho})$ continuous for every $\rho > 0$.  

\begin{quest}
Is $E_{\exp}$-topology equivalent to the coarsest refinement of the \textit{weak topology} inherited from $\PSH (X, L)$ which makes $E_{\exp} (\varphi_{; \rho})$ continuous for every $\rho > 0$? 
\end{quest}

We also obtain the following, which will be related to $E_{\exp}^L (\varphi) + \int_{X^{\mathrm{NA}}} \varphi \int e^{-t} \mathcal{D}_\varphi$. 

\begin{prop}
\label{sigma continuity}
The functional $\bm{\check{\sigma}}: \E^{\exp} (X, L) \to \mathbb{R}$ defined by 
\begin{align}
\bm{\check{\sigma}} (\varphi) 
&:= \frac{\iint_{X^{\mathrm{NA}}} (n-t) e^{-t} \mathcal{D}_\varphi}{\iint_{X^{\mathrm{NA}}} e^{-t} \mathcal{D}_\varphi} - \log \iint_{X^{\mathrm{NA}}} e^{-t} \mathcal{D}_\varphi
\\ \notag
&=\frac{\int_\mathbb{R} (n-t) e^{-t} \DHm_\varphi}{\int_\mathbb{R} e^{-t} \DHm_\varphi} - \log \int_\mathbb{R} e^{-t} \DHm_\varphi
\end{align}
is continuous with respect to $E_{\exp}/d_{\exp}$-topology. 
\end{prop}

\begin{rem}
\label{sigma bounded}
Putting $d\mu_\varphi = \frac{1}{\int_\mathbb{R} e^{-t} \DHm_\varphi} e^{-t} \DHm_\varphi$ and $d\nu_\varphi = \frac{1}{\int_\mathbb{R} \DHm_\varphi} \DHm_\varphi$, we get a lower bound 
\[ \bm{\check{\sigma}} (\varphi) = n + \int_\mathbb{R} \frac{d\mu_\varphi}{d\nu_\varphi} \log \frac{d\mu_\varphi}{d\nu_\varphi} d\nu_\varphi - \log \int_\mathbb{R} \DHm_\varphi \ge n - \log (e^L) = \bm{\check{\sigma}} (0) \]
by Jensen's inequality. 
\end{rem}

\subsubsection{Continuous extension of the functional $E_{\exp}^M$}
\label{Continuous extension of the functional EexpM}

We recall 
\[ E_{\exp}^M (\varphi_{(\mathcal{X}, \mathcal{L}; \rho)}) =  -\big{(} (M. e^L) - \rho (\tilde{M}_{\mathbb{P}^1, \mathbb{G}_m}. e^{\tilde{\mathcal{L}}_{\mathbb{G}_m}}; \rho) \big{)} \]
for $\varphi_{(\mathcal{X}, \mathcal{L}; \rho)} \in \nH (X, L)$. 
We note 
\[ E_{\exp}^M (\varphi_{(\mathcal{X}, \mathcal{L}; \rho)}) = \frac{d}{ds}\Big{|}_{s=0} E_{\exp} (L+sM, \varphi_{(\mathcal{X}, \mathcal{L}; \rho)}). \]

To extend $E_{\exp}^M$ continuously to $\E^{\exp} (X, L)$, we make use of the following tomographic expression. 

\begin{thm}
\label{EexpM continuity}
For $\varphi \in \nH (X, L)$, we have 
\[ E_{\exp}^M (\varphi) = \int_\mathbb{R} \frac{(M, 0) \cdot (L, \varphi \wedge \tau - \tau)^{\cdot n}}{n!} e^{-\tau} d \tau. \]
For general $\varphi \in \E^{\exp} (X, L)$, we regard this formula as the definition of $E_{\exp}^M (\varphi)$. 
Then $E_{\exp}^M (\varphi)$ is finite and the functional $E_{\exp}^M: \E^{\exp} (X, L) \to \mathbb{R}$ is continuous with respect to $E_{\exp}/d_{\exp}$-topology. 
\end{thm}

\begin{proof}
For $\varphi \in \nH (X, L)$ and each $\tau \in \mathbb{R}$, we have 
\[ \frac{(M, 0) \cdot (L, \varphi \wedge \tau -\tau)^{\cdot n}}{n!} = \frac{d}{ds}\Big{|}_{s=0} \frac{(L +sM, \varphi \wedge \tau - \tau)^{\cdot n+1}}{(n+1)!}. \]
To obtain the claim, we would compute as 
\[ \int_\mathbb{R} \frac{d}{ds}\Big{|}_{s=0} \frac{(L +sM, \varphi \wedge \tau - \tau)^{\cdot n+1}}{(n+1)!} e^{-\tau} d\tau = \frac{d}{ds}\Big{|}_{s=0} \int_\mathbb{R} \frac{(L +sM, \varphi \wedge \tau - \tau)^{\cdot n+1}}{(n+1)!} e^{-\tau} d\tau. \]
This is valid if we have a uniform estimate 
\[ | (M, 0) \cdot (L +sM, \varphi \wedge \tau - \tau)^{\cdot n} e^{-\tau} | \le C_\varepsilon e^{-\varepsilon |\tau|} \]
for every small $s$. 
We put $L_s := L +sM$. 

Now take ample $M', M''$ and $\theta \ge 1$ so that $M = M' -M''$ and $\theta^{-1}L < M', M'' < \theta L$. 
Then $\theta^{-1} L_s \le M', M'' \le \theta L_s$ for small $s$. 
It follows by Theorem \ref{BJ estimate 2} that 
\[ |(M, 0) \cdot (L_s, \varphi \wedge \tau - \tau)^{\cdot n}| \le C_n \theta^{n^2} d_1 (\varphi \wedge \tau - \tau, 0). \]
Then by Lemma \ref{exponential domination}, we get 
\[ |(M, 0) \cdot (L_s, \varphi \wedge \tau -\tau)^{\cdot n} e^{-\tau}| \le C_\varepsilon e^{-\varepsilon |\tau|} \]
for $\varphi \in \E^{\exp} (X, L)$ and small $s$ as desired. 
In particular, the right hand side is finite for $\varphi \in \E^{\exp} (X, L)$. 

It suffices to see the continuity
\[ \int_\mathbb{R} (M, 0) \cdot (L, \varphi_i \wedge \tau - \tau)^{\cdot n} e^{-\tau} d \tau \to  \int_\mathbb{R} (M, 0) \cdot (L, \varphi \wedge \tau - \tau)^{\cdot n} e^{-\tau} d \tau \]
for $\varphi_i\to \varphi \in \E^{\exp} (X, L)$ in $E_{\exp}/d_{\exp}$-topology. 
By the proof of Lemma \ref{exponential domination}, the constant $C_\varepsilon$ in the above estimate can be taken uniformly for $\varphi_i$. 
On the other hand, again by Theomre \ref{BJ estimate 2}, we have  
\[ (M, 0) \cdot (L, \varphi_i \wedge \tau - \tau)^{\cdot n} \to (M, 0) \cdot (L, \varphi \wedge \tau - \tau)^{\cdot n} \]
for every $\tau \in \mathbb{R}$.  
Thus we can apply the dominated convergence theorem to the above limit and get the desired continuity. 
\end{proof}

\subsubsection{Tomographic expression of exponential moment measure}

As for the exponential moment measure $\int e^{-t} \mathcal{D}_\varphi$, we have the following tomographic expression. 

\begin{prop}
\label{tomographic expression of exp moment measure}
For $\varphi \in \E^{\exp} (X, L)$ and $\psi \in \E^1 (X, L)$, the measurable function $e^{-\tau} \int_{X^{\mathrm{NA}}} (\psi - \psi (v_{\mathrm{triv}})) \mathrm{MA} (\varphi \wedge \tau)$ is integrable with respect to $d\tau$ and we have 
\[ \int_{X^{\mathrm{NA}}} \psi \int e^{-t} \mathcal{D}_\varphi = \int_\mathbb{R} d\tau ~e^{-\tau}  \int_{X^{\mathrm{NA}}} (\psi-\psi (v_{\mathrm{triv}})) \mathrm{MA} (\varphi \wedge \tau) + \psi (v_{\mathrm{triv}}) \int_\mathbb{R} e^{-\tau} \DHm_\varphi (\tau). \]
Moreover, for fixed $n, (e^L)$ and $\varepsilon < 1$, we have a positive constant $C_\varepsilon$ depending boundedly on $d_1 (\varphi, 0), \int_\mathbb{R} e^{-(2+2 \varepsilon) t} \DHm_\varphi$ and $d_1 (\psi, 0)$ satisfying 
\[ \Big{|} e^{-\tau} \int_{X^{\mathrm{NA}}} (\psi - \psi (v_{\mathrm{triv}})) \mathrm{MA} (\varphi \wedge \tau) \Big{|} \le C_\varepsilon e^{-\varepsilon |\tau|}. \]
\end{prop}

The claim includes that $\psi \in \E^1 (X, L)$ is integrable with respect to $\int e^{-t} \mathcal{D}_\varphi$. 

\begin{proof}
We firstly show the uniform estimate 
\begin{equation} 
\label{exponential decay}
\Big{|} e^{-\tau} \int_{X^{\mathrm{NA}}} (\psi - \psi (v_{\mathrm{triv}})) \mathrm{MA} (\varphi \wedge \tau) \Big{|} \le C_\varepsilon e^{-\varepsilon |\tau|}, 
\end{equation}
which in particular shows that the function is integrable with respect to $d\tau$. 
By Proposition \ref{BJ estimate}, we have 
\begin{align*} 
\Big{|} \int_{X^{\mathrm{NA}}} (\psi - \psi (v_{\mathrm{triv}})) \mathrm{MA} (\varphi \wedge \tau) \Big{|} 
&= \Big{|} \int_{X^{\mathrm{NA}}} (\psi - \psi (v_{\mathrm{triv}})) \mathrm{MA} (\varphi \wedge \tau - \tau) \Big{|} 
\\
&\le C_n \max \{ d_1 (\psi, 0)^{1/2} d_1 (\varphi \wedge \tau - \tau, 0)^{1/2}, d_1 (\varphi \wedge \tau - \tau, 0) \}. 
\end{align*}
By Lemma \ref{exponential domination}, we have 
\begin{gather*} 
e^{-\tau} d_1 (\varphi \wedge \tau - \tau, 0) \le C_\varepsilon' e^{- \varepsilon |\tau|}, 
\\
e^{-2\tau} d_1 (\varphi \wedge \tau - \tau, 0) \le C_\varepsilon'' e^{-2 \varepsilon |\tau|} 
\end{gather*}
by some $C_\varepsilon', C_\varepsilon'' > 0$ depending boundedly on $d_1 (\varphi, 0), \int_\mathbb{R} e^{-(2+ 2 \varepsilon) \tau} \DHm_\varphi$. 
Here we note $\int_\mathbb{R} e^{-(1+ \varepsilon) \tau} \DHm_\varphi$ is bounded when $d_1 (\varphi, 0), \int_\mathbb{R} e^{-(2+ 2 \varepsilon) \tau} \DHm_\varphi$ is bounded. 
Thus we get 
\[ \Big{|} e^{-\tau} \int_{X^{\mathrm{NA}}} (\psi - \psi (v_{\mathrm{triv}})) \mathrm{MA} (\varphi \wedge \tau) \Big{|} \le C_n \max \{ (C''_\varepsilon)^{1/2} d_1 (\psi, 0)^{1/2}, C_\varepsilon' \} \cdot e^{- \varepsilon |\tau|} \]
as desired. 

Now we show the claim. 
Since the claim obviously holds for constant $\psi$, we may assume $\psi (v_{\mathrm{triv}}) = \sup \psi = 0$ by the linearity. 
Moreover, we may assume $\psi \in \nH (X, L)$ as both sides of the equality are continuous along convergent decreasing \textit{sequence} $\psi_i \searrow \psi$ (countable regularization). 
Indeed, the left hand side is continuous by the monotone convergence theorem (see Proposition \ref{monotone convergence for net}). 
As for the right hand side, the dominated convergence theorem applies to the first term since we have 
\[ \int_{X^{\mathrm{NA}}} (\psi_i - \psi_i (v_{\mathrm{triv}})) \mathrm{MA} (\varphi \wedge \tau) \to \int_{X^{\mathrm{NA}}} (\psi - \psi (v_{\mathrm{triv}})) \mathrm{MA} (\varphi \wedge \tau) \]
for each $\tau \in \mathbb{R}$ and a uniform bound 
\[ \Big{|} e^{-\tau} \int_{X^{\mathrm{NA}}} (\psi_i - \psi_i (v_{\mathrm{triv}})) \mathrm{MA} (\varphi \wedge \tau) \Big{|} \le C_\varepsilon e^{-\varepsilon |\tau|} \] 
due to the uniform boundedness of $d_1 (\psi_i, 0) \to d_1 (\psi, 0)$. 
The second term is continuous as we already proved. 

Let $\chi_i$ be a compactly supported smooth function which takes $1$ on the interval $[-i, i]$, $0$ on the complement of $(-i-1, i+1)$ and monotonic on the intervals $(-i-1, -i), (i, i+1)$ with $|\chi_i'| \le 1+ \delta$. 
Since $\chi_i \nearrow 1_\mathbb{R}$ and $\psi$ is continuous, we have 
\[ \int_{X^{\mathrm{NA}}} \psi \int e^{-t} \mathcal{D}_\varphi = \lim_{i \to \infty} \int_{X^{\mathrm{NA}}} \psi \int \chi_i e^{-t} \mathcal{D}_\varphi. \]

By Proposition \ref{moment measure via distribution}, we have 
\begin{align*}
\int_{X^{\mathrm{NA}}} \psi \int \chi_i e^{-t} \mathcal{D}_\varphi 
&= \int_\mathbb{R} d\tau ~ \chi_i (\tau) e^{-\tau} \frac{d}{d\tau} \int_{X^{\mathrm{NA}}} \psi \mathrm{MA} (\varphi \wedge \tau) 
\\
&= - \int_\mathbb{R} d\tau ~ \chi_i' (\tau) e^{-\tau} \int_{X^{\mathrm{NA}}} \psi \mathrm{MA} (\varphi \wedge \tau) 
\\
&\qquad + \int_\mathbb{R} d\tau ~  \chi_i (\tau) e^{-\tau} \int_{X^{\mathrm{NA}}} \psi \mathrm{MA} (\varphi \wedge \tau) . 
\end{align*}

The first term will vanish as we have 
\[ |\int_\mathbb{R} d \tau ~ \chi_i' (\tau) e^{-\tau} \int_{X^{\mathrm{NA}}} \psi \mathrm{MA} (\varphi \wedge \tau)| \le (1+\delta) \sup_{\tau \in [-i-1, -i] \cup [i, i+1]} e^{-\tau} \int_{X^{\mathrm{NA}}} \psi \mathrm{MA} (\varphi \wedge \tau) \to 0 \]
for $i \to \infty$ thanks to the above estimate. 

As for the second term, since 
\[ |\chi_i (\tau) e^{-\tau} \int_{X^{\mathrm{NA}}} \psi \mathrm{MA} (\varphi \wedge \tau)| \nearrow |e^{-\tau} \int_{X^{\mathrm{NA}}} \psi \mathrm{MA} (\varphi \wedge \tau)| \le C_\varepsilon e^{-\varepsilon |\tau|} \] 
as $i \to \infty$, we get 
\[ \int_\mathbb{R} d\tau ~ \chi_i (\tau) e^{-\tau} \int_{X^{\mathrm{NA}}} \psi \mathrm{MA} (\varphi \wedge \tau) \to \int_\mathbb{R} d\tau~ e^{-\tau} \int_{X^{\mathrm{NA}}} \psi \mathrm{MA} (\varphi \wedge \tau) \]
by the dominated convergence theorem, which shows the proposition.  
\end{proof}

\begin{quest}
Does the claim hold also for $\psi \in C^0 (X^{\mathrm{NA}})$? 
\end{quest}

\begin{cor}
The measure $\int e^{-t} \mathcal{D}_\varphi$ has finite energy $E^\vee$ and hence does not charge pluripolar sets. 
\end{cor}

\begin{proof}
We in particular have $\E^1 \subset L^1 (\int e^{-t} \mathcal{D}_\varphi)$ by the above proposition, so the claim follows by \cite[Theorem 6.23]{BJ3}. 
\end{proof}

%\begin{quest}
%The claim holds for $\psi \in C^0 (X^{\mathrm{NA}})$? 
%\end{quest}

\subsubsection{Continuity of exponential moment measure}
\label{Continuity of exponential moment measure}

\begin{thm}
\label{exponential moment measure convergence}
Suppose a sequence $\{ \varphi_i \}_{i \in \mathbb{N}} \subset \E^{\exp} (X, L)$ converges to $\varphi \in \E^{\exp} (X, L)$ in $E_{\exp}/d_{\exp}$-topology. 
\begin{enumerate}
\item If a sequence $\{ g_i \}_{i \in \mathbb{N}} \subset C^0 (X^{\mathrm{NA}})$ uniformly converges to $g \in C^0 (X^{\mathrm{NA}})$, then 
\[ \int_{X^{\mathrm{NA}}} g_i \int e^{-t} \mathcal{D}_{\varphi_i} \to \int_{X^{\mathrm{NA}}} g \int e^{-t} \mathcal{D}_\varphi. \]

\item If a sequence $\{ \psi_i \}_{i \in \mathbb{N}} \subset \E^1 (X, L)$ strongly converges to $\psi \in \E^1 (X, L)$, then 
\[ \int_{X^{\mathrm{NA}}} \psi_i \int e^{-t} \mathcal{D}_{\varphi_i} \to \int_{X^{\mathrm{NA}}} \psi \int e^{-t} \mathcal{D}_\varphi. \]
\end{enumerate}
\end{thm}

\begin{proof}

(1) We assume $g_i \to g \in C^0 (X^{\mathrm{NA}})$ uniformly. 
For any $\varepsilon > 0$, we can take $i_\varepsilon$ and two psh functions $\psi_1, \psi_2 \in \nH (X)$ so that 
\[ \sup |g - (\psi_1 - \psi_2)|, \sup |g_i - (\psi_1 - \psi_2)| \le \varepsilon \] 
for $i \ge i_\varepsilon$. 
Then since 
\begin{align*} 
\Big{|} \int_{X^{\mathrm{NA}}} 
&g \int e^{-t} \mathcal{D}_\varphi - \int_{X^{\mathrm{NA}}} g_i \int e^{-t} \mathcal{D}_{\varphi_i} \Big{|} 
\\
&\le \varepsilon \int_\mathbb{R} e^{-t} \DHm_\varphi + \varepsilon \int_\mathbb{R} e^{-t} \DHm_{\varphi_i} + \Big{|} \int_{X^{\mathrm{NA}}} \psi_1 \int e^{-t} \mathcal{D}_\varphi - \int_{X^{\mathrm{NA}}} \psi_1 \int e^{-t} \mathcal{D}_{\varphi_i} \Big{|} 
\\
&\qquad \qquad+ \Big{|} \int_{X^{\mathrm{NA}}} \psi_2 \int e^{-t} \mathcal{D}_\varphi - \int_{X^{\mathrm{NA}}} \psi_2 \int e^{-t} \mathcal{D}_{\varphi_i} \Big{|}
\end{align*}
and 
\[ \int_\mathbb{R} e^{-t} \DHm_{\varphi_i} \to \int_\mathbb{R} e^{-t} \DHm_\varphi \] 
by $\varphi_i \to \varphi$ in $d_{\exp}$, the claim follows from the case $g_i = g = \psi \in \nH (X)$. 
Therefore, it suffices to show the second claim. 

(2) We assume $\psi_i \to \psi \in \E^1 (X, L)$ in $d_1$. 
Since $\psi_i \to \psi$ in the weak topology of $\PSH (X, L)$, i.e. $\psi_i (v) \to \psi (v)$ for every quasi-monomial valuation $v$, we in particular have $\psi_i (v_{\mathrm{triv}}) \to \psi (v_{\mathrm{triv}})$. 
Meanwhile, we have $\int_\mathbb{R} e^{-\tau} \DHm_{\varphi_i} (\tau) \to \int_\mathbb{R} e^{-\tau} \DHm_{\varphi} (\tau)$. 
Thus, thanks to Proposition \ref{tomographic expression of exp moment measure}, it suffices to show 
\[ \int_\mathbb{R} d\tau ~e^{-\tau} \int_{X^{\mathrm{NA}}} (\psi_i-\psi_i (v_{\mathrm{triv}})) \mathrm{MA} (\varphi_i \wedge \tau) \to \int_\mathbb{R} d\tau ~e^{-\tau} \int_{X^{\mathrm{NA}}} (\psi-\psi (v_{\mathrm{triv}})) \mathrm{MA} (\varphi \wedge \tau). \]

Since $\varphi_i \wedge \tau \to \varphi \wedge \tau$ and $\psi_i \to \psi$ in $d_1$, we have 
\[ \int_{X^{\mathrm{NA}}} (\psi_i-\psi_i (v_{\mathrm{triv}})) \mathrm{MA} (\varphi_i \wedge \tau) \to \int_{X^{\mathrm{NA}}} (\psi-\psi (v_{\mathrm{triv}})) \mathrm{MA} (\varphi \wedge \tau) \]
for each $\tau \in \mathbb{R}$. 
Thus the integrands are pointwiesely convergent. 

On the other hand, since $\varphi_i \to \varphi$ in $E_{\exp}/d_{\exp}$ and $\psi_i \to \psi$ in $d_1$, we have a uniform bound on $d_1 (\varphi_i, 0), \int_\mathbb{R} e^{-(2+2\varepsilon) t} \DHm_{\varphi_i}, d_1 (\psi_i, 0)$. 
Then by (the proof of) Proposition \ref{tomographic expression of exp moment measure}, we have a uniform constant $C_\varepsilon$ such that 
\[ \Big{|} e^{-\tau} \int_{X^{\mathrm{NA}}} (\psi_i-\psi_i (v_{\mathrm{triv}})) \mathrm{MA} (\varphi_i \wedge \tau) \Big{|} \le C_\varepsilon e^{-\varepsilon |\tau|} \]
for every $i$ and $\tau \in \mathbb{R}$. 
Now the desired convergence follows from the dominated convergence theorem. 
\end{proof}

\subsubsection{Non-archimedean $\mu$-entropy}
\label{Non-archimedean mu-entropy}

We firstly note the following formula. 

\begin{prop}
\label{energy pairing formula}
For $\varphi \in \E^{\exp} (X, L)$, we have 
\[ -\int_\mathbb{R} (n-t) e^{-t} \DHm_\varphi = E^L_{\exp} (\varphi) + \int_{X^{\mathrm{NA}}} \varphi \int e^{-t} \mathcal{D}_\varphi. \]
\end{prop}

\begin{proof}
We can easily check the claim for $\varphi_{(\mathcal{X}, \mathcal{L})} \in \nH (X, L)$ as we have
\begin{align*} 
\int_\mathbb{R} (n-\rho t) e^{-\rho t} \DHm_\varphi
&= (L.e^L)- \rho (\bar{\mathcal{L}}. e^{\bar{\mathcal{L}}}; \rho) 
\\
&= \big{(} (L.e^L)- \rho (\tilde{L}_{\mathbb{A}^1}. e^{\bar{\mathcal{L}}}; \rho) \big{)} - \rho ((\bar{\mathcal{L}} - \tilde{L}_{\mathbb{A}^1}). e^{\bar{\mathcal{L}}}; \rho)
\end{align*} 
and $\varphi (\rho. v_E) = \rho \frac{\mathrm{ord}_E (\tilde{\mathcal{L}} - \tilde{L}_{\mathbb{A}^1})}{\mathrm{ord}_E \tilde{\mathcal{X}}_0}$. 
The general case follows by the continuity of the functionals with respect to $E_{\exp}/d_{\exp}$-topology. 
\end{proof}

\begin{cor}
\label{sigma formula}
For $\varphi \in \E^{\exp} (X, L)$, we have 
\[ \bm{\check{\sigma}} (\varphi) = - \frac{ E^L_{\exp} (\varphi) + \int_{X^{\mathrm{NA}}} \varphi \int e^{-t} \mathcal{D}_\varphi }{ \iint_{X^{\mathrm{NA}}} e^{-t} \mathcal{D}_\varphi } - \log \iint_{X^{\mathrm{NA}}} e^{-t} \mathcal{D}_\varphi. \]
\end{cor}

Now we introduce the non-archimedean $\mu$-entropy for general $\varphi \in \E^{\exp} (X, L)$. 
To ensure the lsc extension of the log discrepancy $A_X$, we assume $X$ is klt in what follows. 

\begin{defin}[The non-archimedean $\mu$-entropy]
Let $X$ be a klt variety. 
For $\varphi \in \E^{\exp} (X, L)$, we put 
\begin{align*} 
\NAmu (\varphi) 
&:= - 2\pi \frac{\int_{X^{\mathrm{NA}}} A_X \int e^{-t} \mathcal{D}_\varphi + E_{\exp}^{K_X} (\varphi)}{\iint_{X^{\mathrm{NA}}} e^{-t} \mathcal{D}_\varphi},
\\
\NAmu^\lambda (\varphi) 
&:= \NAmu (\varphi) + \lambda \bm{\check{\sigma}} (\varphi) 
\\ 
&= - \frac{\int_{X^{\mathrm{NA}}} (2\pi A_X + \lambda \varphi) \int e^{-t} \mathcal{D}_\varphi + E_{\exp}^{2\pi K_X +\lambda L} (\varphi)}{\iint_{X^{\mathrm{NA}}} e^{-t} \mathcal{D}_\varphi} - \lambda \log \iint_{X^{\mathrm{NA}}} e^{-t} \mathcal{D}_\varphi.  
\end{align*}
\end{defin}

From the observation in the beginning of this section, this $\NAmu^\lambda$ extends the non-archimedean $\mu$-entropy for test configurations defined in section \ref{non-archimedean mu-entropy of test configuration}. 

For any lsc function $f$, we have $f  = \sup \{ g ~|~ f \ge g \in C^0 \}$, so that we have $\varliminf_{i \to \infty} \int f d\mu_i \ge \int f d\mu$ for any weakly convergent net $\mu_i \to \mu$ of Radon measures. 
It follows from what we proved that $\NAmu^\lambda$ gives an upper semi-continuous function on $\mathcal{E}^{\exp}_{\mathrm{NA}} (X, L)$ with respect to $E_{\exp}/d_{\exp}$-topology. 
Therefore we have proved all the results in Theorem \ref{NAmu entropy extension}. 

\subsubsection{Maximizing non-archimedean $\mu$-entropy}
\label{maximizing non-archimedean mu-entropy}

To reorganize Theorem \ref{characteristic mu maximization implies muK-semistability} and Theorem \ref{NAmu maximizer} on the characteristic $\mu$-entropy $\cmu^\lambda$ in the non-archimedean setup, we would compare $\NAmu^\lambda (\varphi)$ and $\cmu^\lambda (\mathcal{F}_\varphi)$ for $\varphi \in \pcH (X, L)$. 

For $\varphi = \varphi_{(\mathcal{X}, \mathcal{L})} \in \nH (X, L)$, we have $\mathcal{F}_\varphi = \mathcal{F}_{(\mathcal{X}_d, \mathcal{L}_d; d^{-1})}$ for sufficiently divisible $d$. 
Then since the central fibre of $\mathcal{X}_d$ is reduced, we obtain $\NAmu^\lambda (\varphi) = \NAmu^\lambda (\mathcal{X}_d, \mathcal{L}_d; d^{-1}) = \cmu^\lambda (\mathcal{X}_d, \mathcal{L}_d; d^{-1}) = \cmu^\lambda (\mathcal{F}_\varphi)$. 

For $\varphi \in \pcH (X, L)$, the associated filtration $\mathcal{F}_\varphi$ is finitely generated by Theorem \ref{finite generation of Fphi}. 
Then there exists a polyhedral configuration $(\mathcal{X}/B_\sigma, \mathcal{L}; \zeta)$ such that $\mathcal{F}_{(\mathcal{X}, \mathcal{L}; \zeta)} = \mathcal{F}_\varphi$ for $\zeta \in \sigma^\circ$. 
We recall the central fibre of $(\mathcal{X}/B_\sigma, \mathcal{L}; \zeta)$ is reduced by Proposition \ref{Krull envelope has reduced central fibre}. 
For general $\xi \in \sigma$, we have the associated non-archimedean metric $\varphi_\xi := \varphi_{(\mathcal{X}, \mathcal{L}; \xi)} \in \pcH (X, L)$. 
We have $\mathcal{F}_{\varphi_\xi} = \mathcal{F}_{(\mathcal{X}, \mathcal{L}; \xi)}$ by Proposition \ref{reduced central fibre implies homogeneity}. 
We recall $\cmu^\lambda (\mathcal{F}_{(\mathcal{X}, \mathcal{L}; \xi)}) = \cmu^\lambda (\mathcal{X}, \mathcal{L}; \xi)$ is continuous on $\xi \in \sigma$. 
On the other hand, we already know $\NAmu^\lambda (\varphi_\eta) = \cmu^\lambda (\mathcal{F}_{(\mathcal{X}, \mathcal{L}; \eta)})$ for rational $\eta \in \sigma \cap N_{\mathbb{Q}}$. 
Thus to see $\NAmu^\lambda (\varphi) = \cmu^\lambda (\mathcal{F}_\varphi)$ for general $\varphi \in \pcH (X, L)$, it suffices to check $\NAmu^\lambda (\varphi_\xi)$ is continuous on $\xi \in \sigma$. 
We already know the upper semi-continuity, so at least $\NAmu^\lambda (\varphi) \ge \cmu^\lambda (\mathcal{F}_\varphi)$ for $\varphi \in \pcH (X, L)$. 

The continuity of $\sigma \to \mathbb{R}: \xi \mapsto \NAmu^\lambda (\varphi_\xi)$ can be reduced to Question \ref{continuity of entropy along polyhedral configuration} below as follows. 
For $\eta \in \sigma^\circ \cap N$, the pullback $(\mathcal{X}_\eta, \mathcal{L}_\eta)$ along $\mathbb{A}^1 \to B_\sigma$ gives a \textit{normal} test configuration thanks to the reducedness of the central fibre. 
It follows that for each irreducible component $E \subset \mathcal{X}_o$ and $\eta \in \sigma^\circ \cap N$, we can assign a valuation $v_{E, \eta}$ on $X$ so that 
\[ \mathrm{MA} (\varphi_\eta) = \sum_{E \subset \mathcal{X}_o} (E. \mathcal{L}^{\cdot n}). \delta_{v_{E, \eta}}. \]
For $\eta \in N_{\mathbb{Q}}$, we put $v_{E, \eta} := \rho^{-1}. v_{E, \rho. \eta}$ by taking sufficiently divisible $\rho \in \mathbb{N}_+$. 
Since $\varphi_\xi \in \pcH (X, L)$ is continuous on $\xi \in \sigma^\circ$ with respect to the uniform topology, we have $\mathrm{MA} (\varphi_{\xi_i}) \to \mathrm{MA} (\varphi_\xi)$ for $\xi_i \to \xi \in \sigma^\circ$. 
This gives a continuous extension $\sigma \to X^{\mathrm{lin}}: \xi \mapsto v_{E, \xi}$ for each $E \subset \mathcal{X}_o$. 

Now by the continuity we can easily see the moment measure of $\varphi_\xi$ is given by 
\[ \int e^{-t} \mathcal{D}_{\varphi_\xi} = \sum_{E \subset \mathcal{X}_o} \mathrm{ord}_E \mathcal{X}_0 \int_\mathbb{R} e^{-t} \DHm_{(E, \mathcal{L}|_E; \xi)} .\delta_{v_{E, \xi}}. \]
Thus we obtain the following expression 
\[ \int_{X^{\mathrm{NA}}} A_X \int e^{-t} \mathcal{D}_{\varphi_\xi} = \sum_{E \subset \mathcal{X}_o} \mathrm{ord}_E \mathcal{X}_0 \int_\mathbb{R} e^{-t} \DHm_{(E, \mathcal{L}|_E; \xi)} A_X (v_{E, \xi}). \]
This is the only part in $\NAmu^\lambda (\varphi_\xi)$ which may cause discontinuity. 
Since $\DHm_{(E, \mathcal{L}|_E; \xi)}$ is continuous on $\xi$ (consider $(X', L') := (E, \mathcal{L}|_E)$, then we have $\DHm_{\varphi'_\xi} = \DHm_{(E, \mathcal{L}|_E; \xi)}$ for the associated metric $\varphi'_\xi$ on $(X', L')$ is continuous on $\xi$), we can reduce the problem to the following question. 

\begin{quest}
\label{continuity of entropy along polyhedral configuration}
For a polyhedral configuration $(\mathcal{X}/B_\sigma, \mathcal{L})$ with reduced central fibre and an irreducible component $E \subset \mathcal{X}_o$, the log discrepancy $A_X (v_{E, \xi})$ is continuous on $\xi \in \sigma^\circ$? 
\end{quest}

%\begin{quest}
%\label{continuity for polyhedral family}
%Let $(\mathcal{X}, \mathcal{L})$ be a $\sigma'$-configuration with reduced central fibre. 
%Then for each $E \subset \mathcal{X}_o$, can we find a fan $\Sigma$ supporting on $\sigma'$ and a collection of snc divisors $\{ (Y_\sigma, D_\sigma) \to X \}_{\sigma \in \Sigma}$ over $X$ satisfying the following?: for each $\sigma \in \Sigma$, the family $\{ v_{E, \xi} \}_{\xi \in \sigma \cap N_\mathbb{Q}}$ is in $\mathrm{QM} (Y_\sigma, D_\sigma)$ and extends to a continuous map $\sigma \to \mathrm{QM} (Y_\sigma, D_\sigma)$. 
%\end{quest}

We can check this for proper vectors. 

\begin{prop}
Let $(X, L)$ be a polarized normal variety with a torus $T$ action. 
Then there is a fan $\Sigma$ on $\mathfrak{t}$ and a collection of snc divisors $\{ (Y_\sigma, D_\sigma) \to X \}_{\sigma \in \Sigma}$ over $X$ satisfying the following: for each $\sigma \in \Sigma$, the family $\{ v_{X, \xi} \}_{\xi \in \sigma}$ associated to the polyhedral configuration $(X_\sigma, L_\sigma)$ in Example \ref{product sigma-configuration} factors through $\mathrm{QM} (Y_\sigma, D_\sigma) \subset X^{\mathrm{lin}}$ continuously. 
\end{prop}

\begin{proof}
Firstly consider a torus $T$ action on $(\mathbb{C}P^n, \mathcal{O} (1))$. 
We may assume the action is given by $(z_0: \ldots: z_n). t = (\chi^{\mu_0} (t) z_0: \ldots: \chi^{\mu_n} (t) z_n)$. 
Then since 
\[ (z_0: \ldots: z_n). \exp t \xi = (e^{t \langle \mu_0, \xi \rangle} z_0: \ldots: e^{t \langle \mu_n, \xi \rangle} z_n), \]
the zero set $Z (\xi)$ of the vector field $\xi$ is the union of linear subspaces 
\[ Z_\lambda (\xi) := \{ (z_0: \ldots : z_n) \in \mathbb{C}P^n ~|~ z_i = 0 \text{ if } \langle \mu_i, \xi \rangle \neq \lambda \}. \]
We put 
\begin{align*} 
W_\lambda (\xi) 
&:= \{ z \in \mathbb{C}P^n ~|~ \lim_{t \to -\infty} z. \exp t \xi \in Z_\lambda (\xi) \}
\\
&= \{ (z_0: \ldots : z_n) \in \mathbb{C}P^n ~|~ z_i = 0 \text{ if } \langle \mu_i, \xi \rangle < \lambda, z_j \neq 0 \text{ for some } j \text{ with } \langle \mu_j, \xi \rangle = \lambda \}. 
\end{align*}
Then a $T$-invariant closed subset $F$ intersects with $Z_\lambda (\xi)$ if and only if it intersects with $W_\lambda (\xi)$. 

The function $\varphi (\xi) = \min_j \{ \langle \mu_j, \xi \rangle \}$ is a concave piecewise linear function. 
Take a fan $\Sigma$ on $\mathfrak{t}$ so that $\varphi|_\sigma$ is linear for every $\sigma \in \Sigma$. 
Then $Z_\sigma := Z_{\varphi (\xi)} (\xi)$ is independent of the choice of $\xi \in \sigma^\circ$ and the complement $\mathbb{C}P^n \setminus W_\sigma$ of $W_\sigma := W_{\varphi (\xi)} (\xi)$ is a proper linear subspace. 

For a $T$-equivariant birational proper morphism $X' \to X$ and $\xi \in N$, we have $v_{X', \xi} = v_{X, \xi}$, so we may assume $X$ is smooth. 
For each $\xi \in \mathfrak{t}$, the zero set $Z_X (\xi)$ of the vector field is the fixed point set of the torus $T_\xi := \overline{\exp \mathbb{R} \xi}_\mathbb{C}$, so that it is a union of connected smooth subvarieties. 
Embed $X$ into $\mathbb{C}P^{N-1} = \mathbb{P} (H^0 (X, L^{\otimes m})^\vee)$ by the linear system $|mL|$. 
Then we have $Z_X (\xi) = X \cap Z (\xi)$. 
By the reducedness, $X$ is not contained in any proper linear subspace $\mathbb{P}W \subsetneq \mathbb{C}P^{N-1}$, so that $X \cap Z_{\varphi (\xi)} (\xi)$ is a non-empty connected component of $Z_X (\xi)$. 
By the above argument, we have a fan $\Sigma$ on $\mathfrak{t}$ such that $Z_{X, \sigma} := X \cap Z_{\varphi (\xi)} (\xi)$ is independent of the choice of $\xi \in \sigma^\circ$. 
For each $\sigma$, we denote by $T_\sigma$ the torus associated to $\mathbb{R} \sigma \subset \mathfrak{t}$. 
By the construction, $Z_{X, \sigma}$ is a connected component of the fixed point set of $T_\sigma$. 

In the following, we fix $\sigma \in \Sigma$. 
Consider the weight decomposition of the normal bundle $N Z_{X, \sigma} = \bigoplus_{\mu \in M_\xi} N_\mu Z_{X, \sigma}$ with respect to the $T_\sigma$ action. 
By the construction, we have $\langle \mu, \xi \rangle > 0$ for every $\xi \in \sigma^\circ$ and $\mu \in M_\sigma$ with $N_\mu Z_{X, \sigma} \neq 0$. 
Take a $T_\sigma$-equivariant \'etale morphism $U \to T_x X$ from a $T_\sigma$-invariant Zariski open neighbourhood $U$ of a point $x \in Z_{X, \sigma}$ (cf. \cite{Sum, Dre}). 
Then using a basis of $T_x X$ compatible with the weight decomposition $T_x Z \oplus \bigoplus_{\mu \in M_\sigma} N_{\mu, x} Z$, we get a parameter system $z^1, \ldots, z^n \in \mathcal{O}_{X, x}$ satisfying $z^i. t = \chi^{\mu_i} (t) z^i$. 

We expand $f \in \mathcal{O}_{X, x}$ as $\sum_\nu a_\nu z^\nu$. 
Then for $\xi \in \sigma^\circ \cap N$, the $\mathbb{G}_m$-invariant extension $\bar{f}$ to $X \times \mathbb{A}^1 \circlearrowleft_\xi \mathbb{G}_m$ can be written as 
\[ \bar{f} = \sum_\nu a_\nu z^\nu \varpi^{\sum_{i=1}^n \langle \mu_i, \xi \rangle \nu_i}. \]
It follows that 
\[ v_{X, \xi} (f) = \min \{ \sum_{i=1}^n \langle \mu_i, \xi \rangle \nu_i ~|~ a_\nu \neq 0 \}. \]
This formula gives a continuous extension $\{ v_{X, \xi} \in \mathrm{QM}_x (X, \overline{\{ z_1 \dotsb z_n = 0 \}}) \}_{\xi \in \sigma}$, which shows the claim. 
\end{proof}

\begin{cor}
\label{continuity of log discrepancy for product configuration}
For any torus action $(X, L) \circlearrowleft T$, the functional $\mathfrak{t} \to \mathbb{R}: \xi \mapsto A_X (v_\xi)$ is continuous. 
\end{cor}

\begin{proof}
This is a consequence of the above proposition and the definition of log discrepancy (cf. \cite{JM}). 
\end{proof}

Now by the above remark, we obtain the following. 

\begin{prop}
\label{comparison of NAmu and chmu}
For any proper vector $\xi$, we have $\NAmu^\lambda (\varphi_\xi) = \cmu^\lambda (X, L; \xi)$. 
In particular, $\NAmu^\lambda (\varphi_\xi)$ is continuous on $\xi \in \mathfrak{t}$. 
\end{prop}

Thus we get the following reformulation. 

\begin{cor}
If $\NAmu^\lambda$ is maximized by $\varphi_\xi$, then $(X, L)$ is $\check{\mu}^\lambda_\xi$K-semistable. 
\end{cor}

\begin{proof}
We have 
\[ \cmu^\lambda (\mathcal{F}_\varphi) \le \NAmu^\lambda (\varphi) \le \NAmu^\lambda (\varphi_\xi) = \cmu^\lambda (X, L; \xi) \] 
for $\varphi \in \nH (X, L)$, so we can apply Theorem \ref{characteristic mu maximization implies muK-semistability}. 
\end{proof}

The following is a refinement of \cite[Proposition 3.14]{Ino2}. 

\begin{prop}
\label{large limit of mu-entropy}
Suppose $X$ is klt. 
For $\varphi = \varphi_{(\mathcal{X}, \mathcal{L}; \xi)} \in \nH^\mathbb{R} (X, L)$, we have 
\[ \lim_{\rho \to \infty} \rho^{-1} \bm{\check{\sigma}} (\varphi_{;\rho}) = 0, \quad \lim_{\rho \to \infty} \rho^{-1} \NAmu (\varphi_{;\rho}) < 0. \]
When the central fibre $\mathcal{X}_o = \mathcal{X}_o (\varphi)$ is irreducible, we explicitly have 
\[ \lim_{\rho \to \infty} \rho^{-1} \NAmu (\varphi_{;\rho}) = -2\pi A_X (v_{\mathcal{X}_o, \xi}). \]
\end{prop}

\begin{proof}
Let $(\mathcal{X}, \mathcal{L}; \xi)$ be a polyhedral configuration with reduced fibre corresponding to the filtration $\mathcal{F}_\varphi$. 
We may normalize $\varphi$ so that $\inf \varphi = \inf \operatorname{\mathrm{supp}} \DHm_\varphi = \inf \DHm_{(\mathcal{X}, \mathcal{L}; \xi)} = 0$. 
We recall 
\[ \frac{\int e^{-\rho t} \mathcal{D}_\varphi }{\iint_{X^{\mathrm{NA}}} e^{-\rho t} \mathcal{D}_\varphi } = \sum_{E \subset \mathcal{X}_o} \mathrm{ord}_E \mathcal{X}_0 \frac{\int_\mathbb{R} e^{-\rho t} \DHm_{(E, \mathcal{L}|_E; \xi)}}{\int_\mathbb{R} e^{-\rho t} \DHm_{(\mathcal{X}, \mathcal{L}; \xi)}} .\delta_{\rho. v_{E, \xi}}. \]

Similarly as \cite[Theorem 5.10]{BHJ1} (cf. \cite[Theorem 5.7]{GGK}), $\DHm_{(E, \mathcal{L}|_E; \xi)}$ is either Dirac mass or absolutely continuous with respect to the Lebesgue measure which has piecewise polynomial density of degree at most $\dim X -1$. 
In particular, for small $\varepsilon > 0$, we can write $\DHm_{(\mathcal{X}, \mathcal{L}; \xi)}|_{[0, \varepsilon)} = f (t) t^k dt$ for some positive continuous function $f: [0, \varepsilon) \to (0, \infty)$ and $k \ge 0$. 
Since $\DHm_{(\mathcal{X}, \mathcal{L}; \xi)}|_{[0, \varepsilon)}$ is the sum of $\mathrm{ord}_E \mathcal{X}_o \cdot \DHm_{(E, \mathcal{L}|_E; \xi)}$, we have either $\DHm_{(E, \mathcal{L}|_E; \xi)}|_{[0, \varepsilon)} = 0$ or $\DHm_{(E, \mathcal{L}|_E; \xi)}|_{[0, \varepsilon)} = f_E (t) t^{k_E} dt$ for some $k_E \ge k$ and some positive continuous function $f_E$. 
By Proposition \ref{infimum of DH measure}, the latter case happens only when $\varphi (v_{E, \xi}) = 0$. 
Again since $\DHm_{(\mathcal{X}, \mathcal{L}; \xi)}|_{[0, \varepsilon)}$ is the sum of $\mathrm{ord}_E \mathcal{X}_o \cdot \DHm_{(E, \mathcal{L}|_E; \xi)}$, there actually exists one such $E$ with $k_E = k$. 

By easy calculus, we get 
\begin{gather*} 
\rho^{k+1} \int_\mathbb{R} e^{-\rho t} \DHm_{(\mathcal{X}, \mathcal{L}; \xi)} \to k! \cdot f (0) 
\\
\rho^{k_E+1} \int_\mathbb{R} e^{-\rho t} \DHm_{(E, \mathcal{L}|_E; \xi)} \to 
\begin{cases} 
0
\\
k_E ! \cdot f_E (0) 
\end{cases}
\end{gather*}
as $\rho \to \infty$. 
We put 
\[ c_{E, \xi} := \lim_{\rho \to \infty} \mathrm{ord}_E \mathcal{X}_o \cdot \frac{\int_\mathbb{R} e^{-\rho t} \DHm_{(E, \mathcal{L}|_E; \xi)}}{\int_\mathbb{R} e^{-\rho t} \DHm_{(\mathcal{X}, \mathcal{L}; \xi)}} = 
\begin{cases}
\mathrm{ord}_E \mathcal{X}_o \cdot f_E (0) /f (0) 
& k_E = k
\\
0 
& \text{otherwise}
\end{cases}. \]
We note $c_{E, \xi} > 0$ and $\sum_{E \subset \mathcal{X}_o} c_{E, \xi} = 1$. 

Now we compute 
\[ \rho^{-1} \frac{\int_{X^{\mathrm{NA}}} A_X \int e^{-t} \mathcal{D}_{\varphi_{;\rho}} }{\iint_{X^{\mathrm{NA}}} e^{- t} \mathcal{D}_{\varphi_{;\rho}} } = \frac{\int_{X^{\mathrm{NA}}} A_X \int e^{-\rho t} \mathcal{D}_\varphi}{\iint_{X^{\mathrm{NA}}} e^{-\rho t} \mathcal{D}_\varphi} \to \sum_{E \subset \mathcal{X}_o} c_{E, \xi} A_X (v_{E, \xi}) > 0. \]
Similarly, 
\[ \rho^{-1} \frac{\int_{X^{\mathrm{NA}}} \varphi_{;\rho} \int e^{-t} \mathcal{D}_{\varphi_{;\rho}}}{\iint_{X^{\mathrm{NA}}} e^{-t} \mathcal{D}_{\varphi_{;\rho}} } = \frac{\int_{X^{\mathrm{NA}}} \varphi \int e^{-\rho t} \mathcal{D}_\varphi }{\iint_{X^{\mathrm{NA}}} e^{-\rho t} \mathcal{D}_\varphi} \to \sum_{E \subset \mathcal{X}_o} c_{E, \xi} \varphi (v_{E, \xi}) = 0 \]
as $c_{E, \xi} = 0$ unless $\varphi (v_{E, \xi}) = 0$. 

As for $E_{\exp}^M$, we firstly compute 
\begin{align*} 
E_{\exp}^M (\varphi_{; \rho}) 
&= \int_\mathbb{R} \frac{(M, 0) \cdot (L, \varphi_{;\rho} \wedge \tau - \tau)^{\cdot n}}{n!} e^{-\tau} d\tau 
\\
&= \int_\mathbb{R} \frac{(M, 0) \cdot (L, (\varphi \wedge \rho^{-1} \tau - \rho^{-1} \tau)_{;\rho})^{\cdot n}}{n!} e^{-\tau} d\tau 
\\
&= \rho \int_\mathbb{R} \frac{(M, 0) \cdot (L, \varphi \wedge \rho^{-1} \tau - \rho^{-1} \tau)^{\cdot n}}{n!} e^{-\tau} d\tau 
\\
&= \rho^2 \int_\mathbb{R} \frac{(M, 0) \cdot (L, \varphi \wedge \sigma - \sigma)^{\cdot n}}{n!} e^{- \rho \sigma} d\sigma. 
\end{align*}
Similarly, 
\begin{align*} 
\iint_{X^{\mathrm{NA}}} e^{-t} \mathcal{D}_{\varphi_{; \rho}} 
&= -\int_\mathbb{R} \frac{(L, \varphi_{;\rho} \wedge \tau - \tau)^{\cdot n+1}}{(n+1)!} e^{-\tau} d\tau 
%\\
%&= -\int_\mathbb{R} \frac{(L, (\varphi \wedge \rho^{-1} \tau - \rho^{-1} \tau)_{;\rho})^{\cdot n+1}}{(n+1)!} e^{-\tau} d\tau 
%\\
%&= - \rho \int_\mathbb{R} \frac{(L, \varphi \wedge \rho^{-1} \tau - \rho^{-1} \tau)^{\cdot n+1}}{(n+1)!} e^{-\tau} d\tau 
\\
&= - \rho^2 \int_\mathbb{R} \frac{(L, \varphi \wedge \sigma - \sigma)^{\cdot n+1}}{(n+1)!} e^{- \rho \sigma} d\sigma. 
\end{align*}
As in the proof of Theorem \ref{EexpM continuity}, we have 
\begin{align*} 
|(M, 0) \cdot (L, \varphi \wedge \tau - \tau)^{\cdot n}| 
&\le C (n+1)! d_1 (\varphi \wedge \tau -\tau, 0) 
\\
&= C (n+1)! \int_\mathbb{R} (-t) \DHm_{\varphi \wedge \tau - \tau} = - C (L, \varphi \wedge \tau -\tau)^{\cdot n+1} 
\end{align*}
for a constant $C$ independent of $\varphi, \tau$. 
It follows that
\begin{align*} 
\Big{|} \frac{E_{\exp}^M (\varphi_{;\rho}) }{\iint_{X^{\mathrm{NA}}} e^{-t} \mathcal{D}_{\varphi_{; \rho}} } \Big{|} 
&= (n+1) \frac{\int_\mathbb{R} |(M, 0) \cdot (L, \varphi \wedge \sigma - \sigma)^{\cdot n}| e^{-\rho \sigma} d\sigma }{- \int_\mathbb{R} (L, \varphi \wedge \sigma - \sigma)^{\cdot n+1} e^{-\rho \sigma} d\sigma} 
\\
&\le (n+1) C,
\end{align*}
so that we have 
\[ \rho^{-1} \frac{E_{\exp}^M (\varphi_{;\rho}) }{\iint_{X^{\mathrm{NA}}} e^{-t} \mathcal{D}_{\varphi_{; \rho}} } \to 0. \]

Finally, we have 
\[ \rho^{-1} \log \iint_{X^{\mathrm{NA}}} e^{-t} \mathcal{D}_{\varphi_{;\rho}} = O (\rho^{-1} \log \rho) \to 0 \]
\end{proof}

Now we can show the following, which completes the proof of Theorem \ref{forall exists}. 

\begin{cor}
\label{properness of mu-entropy}
When $X$ is klt, the functional $\cmu^\lambda(X, L; \bullet): \mathfrak{t} \to \mathbb{R}$ is proper for any torus action $(X, L) \circlearrowleft T$. 
In particular, it admits a maximizer. 
\end{cor}

\begin{proof}
By the computation in the above proposition, we have 
\[ \NAmu^\lambda (\varphi_\xi) = -2\pi A_X (v_\xi) - \frac{E^{2\pi K_X + \lambda L}_{\exp} (\varphi_\xi)}{\iint_{X^{\mathrm{NA}}} e^{-t} \mathcal{D}_{\varphi_\xi}} - \lambda \log (e^L) - \lambda \log \frac{\int_\mathbb{R} e^{-t} \DHm_{\varphi_\xi}}{\int_\mathbb{R} \DHm_{\varphi_\xi}}. \]
For a norm $\| \|$ on $\mathfrak{t}$, consider the continuous function $f_\rho (\xi) := \rho^{-1} \NAmu^\lambda (\varphi_{\rho. \xi})$ on the unit sphere $S$. 
By H\"older's inequality, 
\[ g_\rho (\xi) := \rho^{-1} \log \frac{\int_\mathbb{R} e^{-t} \DHm_{\varphi_{\rho. \xi} }}{\int_\mathbb{R} \DHm_{\varphi_{\rho. \xi} }} = \rho^{-1} \log \frac{\int_\mathbb{R} e^{-\rho t} \DHm_{\varphi_{\xi} }}{\int_\mathbb{R} \DHm_{\varphi_{\xi} }} \]
is monotonically increasing, so that $g_\rho$ uniformly converges to $0$ on $S$ by Dini's lemma. 
On the other hand, $\frac{E^{2\pi K_X + \lambda L}_{\exp} (\varphi_\xi)}{\iint_{X^{\mathrm{NA}}} e^{-t} \mathcal{D}_{\varphi_\xi}} + \lambda \log (e^L)$ is bounded on $\mathfrak{t}$, so that 
\[ h_\rho (\xi) := \rho^{-1} ( \frac{E^{2\pi K_X + \lambda L}_{\exp} (\varphi_\xi)}{\iint_{X^{\mathrm{NA}}} e^{-t} \mathcal{D}_{\varphi_\xi}} + \lambda \log (e^L)) \]
uniformly converges to $0$ on $S$. 

Since $X$ is klt, $\inf_{\xi \in S} A_X (\xi)$ is positive. 
Thus for any $0 < \delta < 2\pi \inf_{\xi \in S} A_X (\xi)$, we can take large $\rho_\delta$ so that 
\[ f_\rho \le - 2\pi \inf_{\xi \in S} A_X (\xi) - h_\rho - \lambda g_\rho \le - \delta \]
for every $\rho \le \rho_\delta$. 
It follows that $\NAmu^\lambda (\varphi_\xi) \le -\delta \| \xi \|$ for $\rho \ge \rho_\delta$. 
Thus any unbounded sequence $\{ \xi_i \} \subset \mathfrak{t}$ has unbounded $\NAmu^\lambda$, which shows the properness. 
\end{proof}

Question \ref{continuity of entropy along polyhedral configuration} would be proved for general polyhedral configuration in coming \cite{BJ5}. 
At the moment, let us assume $\NAmu^\lambda (\varphi) = \cmu^\lambda (\mathcal{F}_\varphi)$ in the following reformulation to clarify our status. 

\begin{prop}
If $\NAmu^\lambda$ is maximized by $\varphi \in \pcH (X, L)$ and $\NAmu^\lambda (\varphi) = \cmu^\lambda (\mathcal{F}_\varphi)$, then the central fibre $(\mathcal{X}_o (\varphi), \mathcal{L}_o (\varphi)) = \Proj \mathcal{R}_o (\mathcal{F}_\varphi)$ is $\mu^\lambda$K-semistable with respect to the proper vector $\xi^\varphi_o$ on $(\mathcal{X}_o (\varphi), \mathcal{L}_o (\varphi))$ induced by the filtration $\mathcal{F}_\varphi$. 
\end{prop}

\subsubsection{Odaka's theorem in $\mu$-entropy formalism}
\label{Odaka's theorem in mu-entropy formalism}

Here we observe the non-archimedean $\mu$-entropy maximization for Calabi--Yau variety and canonically polarized variety. 
The results can be regarded as a reformulation of Odaka's theorem \cite{Oda1}, which is a K-stability/non-archimedean counterpart of Aubin--Calabi--Yau theorem. 
Note we assume $X$ is klt to ensure the lsc extension of the log discrepancy $A_X$. 
The proofs work also for log canonical varieties as soon as $A_X$ on $X^{\mathrm{NA}}$ makes sense. 

\begin{prop}
If the trivial metric $\varphi_{\mathrm{triv}} = 0$ maximizes $\NAmu^\lambda$ on $\E^{\exp} (X, L)$, then $\NAmu^{\lambda'}$ is maximized by $\varphi_{\mathrm{triv}}$ for $\lambda' \le \lambda$. 
\end{prop}

\begin{proof}
Recall $\bm{\check{\sigma}} (\varphi) \ge \bm{\check{\sigma}} (\varphi_{\mathrm{triv}})$. 
The claim follows by 
\[ \NAmu^{\lambda'} (\varphi) = \NAmu^\lambda (\varphi) - (\lambda - \lambda') \bm{\check{\sigma}} (\varphi) \le \NAmu^\lambda (\varphi_{\mathrm{triv}}) - (\lambda - \lambda') \bm{\check{\sigma}} (\varphi_{\mathrm{triv}}). \]
\end{proof}

\begin{cor}
Suppose $K_X \equiv_{\mathbb{Q}} 0$ and $\lambda \le 0$, then the trivial metric $\varphi_{\mathrm{triv}} = 0$ maximizes $\NAmu^\lambda$ on $\E^{\exp} (X, L)$. 
\end{cor}

\begin{proof}
Since $X$ is log canonical, we compute 
\[ \NAmu (\varphi) = - 2\pi \frac{\int_{X^{\mathrm{NA}}} A_X \int e^{-t} \mathcal{D}_\varphi}{\iint_{X^{\mathrm{NA}}} e^{-t} \mathcal{D}_\varphi} \le 0 = \NAmu (\varphi_{\mathrm{triv}}), \]
so $\varphi_{\mathrm{triv}}$ maximizes $\NAmu$. 
By the above proposition, $\varphi_{\mathrm{triv}}$ maximizes $\NAmu^\lambda$ for $\lambda \le 0$. 
\end{proof}

\begin{lem}
For $\varphi \in \E^1 (X, L)$, we have 
\[ \int_\mathbb{R} t \DHm_\varphi - \int_{X^{\mathrm{NA}}} \varphi \mathrm{MA} (\varphi) \ge 0. \]
\end{lem}

\begin{proof}
Since $\int_\mathbb{R} t \DHm_\varphi = E (\varphi)$, this is nothing but \cite[Proposition 5.26]{BJ3}. 
\end{proof}

\begin{prop}
Suppose $K_X$ is ample and $K_X = L$ and $\lambda \le 0$, then the trivial metric $\varphi_{\mathrm{triv}} = 0$ maximizes $\NAmu^\lambda$ on $\E^{\exp} (X, L)$. 
\end{prop}

\begin{proof}
We have $E^{K_X}_{\exp} = E^L_{\exp}$ and $\NAmu (\varphi_{\mathrm{triv}}) = 2\pi n$. 
Since $X$ is log canonical, we have $\NAmu (\varphi) \le -2\pi \frac{E^L_{\exp} (\varphi)}{\int_\mathbb{R} e^{-t} \DHm_\varphi}$. 
Thus to see $\NAmu (\varphi) \le \NAmu (\varphi_{\mathrm{triv}})$, it suffices to show 
\[ E^L_{\exp} (\varphi) +n \int_\mathbb{R} e^{-t} \DHm_\varphi \ge 0. \]
By Proposition \ref{energy pairing formula}, we have 
\[ E^L_{\exp} (\varphi) +n \int_\mathbb{R} e^{-t} \DHm_\varphi = \int_\mathbb{R} t e^{-t} \DHm_\varphi - \int_{X^{\mathrm{NA}}} \varphi \int e^{-t} \mathcal{D}_\varphi. \]
To reduce the problem to the above lemma, we consider another expression. 
Firstly we note 
\begin{align*} 
(n+1) E_{\exp} (\varphi) 
&= \int_\mathbb{R} \frac{(L, \varphi \wedge \tau - \tau)^{\cdot n+1}}{n!} e^{-\tau} d\tau 
\\
&= \int_\mathbb{R} \frac{(L, 0) \cdot (L, \varphi \wedge \tau - \tau)^{\cdot n}}{n!} e^{-\tau} d\tau + \int_\mathbb{R} \frac{(0, \varphi \wedge \tau - \tau) \cdot (L, \varphi \wedge \tau - \tau)^{\cdot n}}{n!} e^{-\tau} d\tau 
\\
&= E^L_{\exp} (\varphi) + \int_\mathbb{R} d\tau e^{-\tau} \int_{X^{\mathrm{NA}}} (\varphi \wedge \tau - \tau) \mathrm{MA} (\varphi \wedge \tau - \tau).  
\end{align*}
Then since $E_{\exp} (\varphi) = - \int_\mathbb{R} e^{-t} \DHm_\varphi$, we have 
\begin{align*} 
E^L_{\exp} (\varphi) +n \int_\mathbb{R} e^{-t} \DHm_\varphi 
&= E^L_{\exp} (\varphi) - n E_{\exp} (\varphi) 
\\
&= E_{\exp} (\varphi) - \int_\mathbb{R} d\tau e^{-\tau} \int_{X^{\mathrm{NA}}} (\varphi \wedge \tau - \tau) \mathrm{MA} (\varphi \wedge \tau - \tau) 
\\
&= \int_\mathbb{R} d\tau e^{-\tau} \Big{(} \int_\mathbb{R} t \DHm_{\varphi \wedge \tau - \tau} - \int_{X^{\mathrm{NA}}} (\varphi \wedge \tau - \tau) \mathrm{MA} (\varphi \wedge \tau - \tau) \Big{)}. 
\end{align*}
Thanks to the above lemma, the integrand
\[ \int_\mathbb{R} t \DHm_{\varphi \wedge \tau - \tau} - \int_{X^{\mathrm{NA}}} (\varphi \wedge \tau - \tau) \mathrm{MA} (\varphi \wedge \tau - \tau) \]
is non-negative, hence we get the desired inequality. 
\end{proof}

The reverse direction is also shown in \cite{Oda2}: K-semistable ``almost $\mathbb{Q}$-Gorenstein'' scheme $X$ admits only semi log canonical singularities. 
In our non-archimedean formalism, we must assume $X$ is log canonical to ensure the lsc extension of the log discrepancy $A_X$, so the reverse direction for non-archimedean $\mu$-entropy does not make sense. 
As for the characteristic $\mu$-entropy, we do not need to assume $X$ is log canonical nor irreducible. 
If the trivial configuration maximizes the characteristic $\mu$-entropy, $X$ is K-semistable, hence $X$ has only semi log canonical singularities by Odaka's result. 

\subsection{Relation to other works}
\label{Relation to other works}

\subsubsection{Relation to $H$-entropy}
\label{Relation to H-entropy}

Let $(X, L) = (X, -K_X)$ be a $\mathbb{Q}$-Fano variety. 
We recall 
\[ \check{H}_{\mathrm{NA}} (\varphi) := - \inf_{x \in X^{\mathrm{qm}}} (A_X (x) + \varphi (x)) - \log \iint_{X^{\mathrm{NA}}} e^{-t} \mathcal{D}_\varphi \]
for $\varphi \in \E^{\exp} (X, L)$. 
For each $x \in X^{\mathrm{qm}}$, $\varphi \mapsto A_X (x) + \varphi (x)$ is upper semi-continuous with respect to the weak topology on $\PSH (X, L)$, so $\inf_{x \in X^{\mathrm{qm}}} (A_X (x) + \varphi (x))$ is upper semi-continuous. 
It follows that $\check{H}_{\mathrm{NA}}$ is lower semi-continuous with respect to $E_{\exp}/d_{\exp}$-topology. 
We can also easily see the continuity along decreasing nets. 
Actually, it would be shown in coming \cite{BJ5} that this functional is continuous with respect to $d_1$-topology, and hence $\check{H}_{\mathrm{NA}}$ is continuous with respect to $E_{\exp}/d_{\exp}$-topology, but we do not use this fact. 

\begin{prop}
\label{comparison of NAmu and H}
For $\varphi \in \E^{\exp} (X, L)$, we have 
\[ \NAmu^{2\pi} (\varphi) \le 2\pi \check{H}_{\mathrm{NA}} (\varphi). \]
\end{prop}

\begin{proof}
Since $L = -K_X$, we have 
\[ \NAmu^{2\pi} (\varphi) = - 2\pi \frac{\int_{X^{\mathrm{NA}}} (A_X + \varphi) \int e^{-t} \mathcal{D}_\varphi}{\iint_{X^{\mathrm{NA}}} e^{-t} \mathcal{D}_\varphi} - 2\pi \log \iint_{X^{\mathrm{NA}}} e^{-t} \mathcal{D}_\varphi. \]
Thus it suffices to show 
\[ \frac{\int_{X^{\mathrm{NA}}} (A_X + \varphi) \int e^{-t} \mathcal{D}_\varphi}{\iint_{X^{\mathrm{NA}}} e^{-t} \mathcal{D}_\varphi} \ge \inf_{x \in X^{\mathrm{qm}}} (A_X (x) + \varphi (x)). \]
This is shown in \cite{BJ2} for general Radon measure: since $A_X \circ p_\mathcal{X} \nearrow A_X$ and $\varphi \circ p_\mathcal{X} \searrow \varphi$ pointwisely on $X^{\mathrm{NA}}$ (cf. \cite[Theorem 5.29]{BJ1}, \cite[Theorem 2.1]{BJ2}) and the image of $p_\mathcal{X}$ is in $X^{\mathrm{qm}}$, we have 
\begin{align*} 
\int_{X^{\mathrm{NA}}} (A_X + \varphi) d\mu 
&= \lim_{\mathcal{X} \in \mathrm{SNC} (X)} \int_{X^{\mathrm{NA}}} (A_X +\varphi) \circ p_\mathcal{X} d\mu 
\\
&\ge \inf_{x \in X^{\mathrm{qm}}} (A_X (x) + \varphi (x)) \cdot \int_X d\mu
\end{align*}
by the monotone convergence theorem for net (see Propsition \ref{monotone convergence for net}). 
\end{proof}

\begin{prop}
Let $(\mathcal{X}, \mathcal{L})$ be a weakly special degeneration, i.e. a normal $\mathbb{Q}$-Gorenstein test configuration such that $\mathcal{L} = -K_{\mathcal{X}/\mathbb{C}}$ is relatively ample, the pair $(\mathcal{X}, \mathcal{X}_0)$ is log canonical and the central fibre $\mathcal{X}_0$ is reduced and irreducible. 
Then we have $\NAmu^{2\pi} (\varphi_{(\mathcal{X}, \mathcal{L}; \rho)}) = 2\pi \check{H}_{\mathrm{NA}} (\varphi_{(\mathcal{X}, \mathcal{L}; \rho)})$. 
\end{prop}

We recall a special degeneration is a normal $\mathbb{Q}$-Gorenstein test configuration such that $\mathcal{L} = -K_{\mathcal{X}/\mathbb{C}}$ is relatively ample and $\mathcal{X}_0$ is klt. 
As noted in \cite[Lemma 2.2]{Berm}, $\mathbb{Q}$-Gorenstein assumption automatically follows by the assumption that $\mathcal{X}_0$ is normal. 
Any special degeneration is weakly special by the inversion of adjunction. 

\begin{proof}
By \cite[Proposition 7.29]{BHJ1}, we have 
\[ \inf_{x \in X^{\mathrm{qm}}} (A_X (x) + \varphi (x)) = A_X (v_{\mathcal{X}_0}) + \varphi (v_{\mathcal{X}_0}) = \frac{\int_{X^{\mathrm{NA}}} (A_X + \varphi) \int e^{-t} \mathcal{D}_\varphi}{\iint e^{-t} \mathcal{D}_\varphi}, \]
which shows the claim. 
\end{proof}

\begin{quest}
Since $\NAmu^{2\pi} - 2\pi \check{H}_{\mathrm{NA}} \le 0$ is usc, the subset 
\[ \{ \varphi \in \E^{\exp} (X, L) ~|~ \NAmu^{2\pi} (\varphi) = 2\pi \check{H}_{\mathrm{NA}} (\varphi) \} \] 
is closed with respect to $E_{\exp}$-topology. 
Then what is the closure of the subset $\{ \varphi_{(\mathcal{X}, \mathcal{L}; \rho)} ~|~ (\mathcal{X}, \mathcal{L}): \text{ weakly special degeneration } \}$? 
At least, $\mathrm{MA} (\varphi)$ is supported on a point of $X^{\mathrm{lin}}$ for $\varphi$ in the closure, so the closure is not the entire $\E^{\exp} (X, L)$. 
\end{quest}

Now using the crucial existence result \cite{BLXZ}, we show the following. 

\begin{thm}
Let $(X, L) = (X, -K_X)$ be a $\mathbb{Q}$-Fano variety. 
Then we have 
\[ \sup_{\varphi \in \E^{\exp} (X, L)} \NAmu^{2\pi} (\varphi) = \sup_{\varphi \in \E^{\exp} (X, L)} 2\pi \check{H}_{\mathrm{NA}} (\varphi). \]
The maximum is attained by some $\varphi_\mathcal{F} \in \pcH (X, L)$ associated to a finitely generated filtration $\mathcal{F}$ with $\mu^{2\pi}$K-semistable $\mathbb{Q}$-Fano central fibre. 
\end{thm}

\begin{proof}
Since $\check{H}_{\mathrm{NA}}$ is lower semi-continuous and $\nH$ is dense in $\E^{\exp} (X, L)$, we have 
\[ \sup_{\varphi \in \E^{\exp} (X, L)} \check{H}_{\mathrm{NA}} (\varphi) = \sup_{\varphi \in \nH (X, L)} \check{H}_{\mathrm{NA}} (\varphi). \]

In \cite{HL2}, the continuity of envelopes is implicitly assumed in the definition of $\check{H}_{\mathrm{NA}} (\mathcal{F})$ for general filtration $\mathcal{F}$: the non-archimedean psh metric $\varphi_{\mathcal{F}}$ for general filtration $\mathcal{F}$ is defined \textit{under the continuity of envelopes} (see section \ref{Filtration associated to continuous psh metric}). 
In order to clarify that the existence of maximizer of $\check{H}_{\mathrm{NA}}$ does not rely on the assumption on the continuity of envelopes, we recall some arguments in \cite{HL2}. 
For a valuation of linear growth $v \in X^{\mathrm{lin}}$, we put 
\[ \check{\beta} (v) := - A_X (v) - \log \int_\mathbb{R} e^{-t} \nu_\infty (\mathcal{F}_v). \]
This is well-defined without assuming the continuity of envelopes. 
Under the continuity of envelopes, this is related to $\check{H}_{\mathrm{NA}}$ in the following way: thanks to Theorem \ref{NApsh associated to valuation of linear growth}, for any valuation of linear growth $v \in X^{\mathrm{lin}}$, the filtration $\mathcal{F}_v$ defines a non-archimedean psh metric $\varphi_v = \varphi_{\mathcal{F}_v} \in C^0 \cap \PSH (X, L)$, for which we have 
\[ \check{\beta} (v) \le H_{\mathrm{NA}} (\varphi_v). \]
It is shown in \cite[Theorem 4.9]{HL2} there exists a quasi-monomial valuation $v \in X^{\mathrm{qm}} \subset X^{\mathrm{lin}}$ satisfying 
\[ \check{\beta} (v) = \sup_{w \in X^{\mathrm{lin}}} \check{\beta} (w) \ge \sup_{(\mathcal{X}, \mathcal{L})} \check{H}_{\mathrm{NA}} (\varphi_{(\mathcal{X}, \mathcal{L})}), \]
where the last inequality is a consequence of \cite[Theorem 3.4, Lemma 4.2]{HL2}. 
By the above remark, we have 
\[ \check{\beta} (v) \ge \sup_{\varphi \in \E^{\exp} (X, L)} \check{H}_{\mathrm{NA}} (\varphi). \]
Now thanks to \cite{BLXZ}, for the quasi-monomial valuation $v$ maximizing $\check{\beta}$, the filtration $\mathcal{F}_v$ is finitely generated and its central fibre is a $\mathbb{Q}$-Fano variety. 
This result does not rely on the continuity of envelopes. 
Therefore, we obtain the non-archimedean psh metric $\varphi_v = \varphi_{\mathcal{F}_v} \in \pcH (X, L)$ associated to the finitely generated filtration $\mathcal{F}_v$ and conclude 
\[ \check{H}_{\mathrm{NA}} (\varphi_v) = \sup_{\varphi \in \E^{\exp} (X, L)} \check{H}_{\mathrm{NA}} (\varphi), \] 
thanks to the inequality $\sup \check{H}_{\mathrm{NA}} \ge H_{\mathrm{NA}} (\varphi_v) \ge \check{\beta} (v)$. 

Moreover, any $w$ in a small neighbourhood $U$ of $v$ in a suitable cone $\mathrm{QM}_\eta (Y, D) \subset \mathrm{Val} (X)$ gives a finitely generated filtration $\mathcal{F}_w$ with the same central fibre (cf. \cite[Lemma 2.10]{LX}) and $\varphi_w$ corresponds to a weakly special degeneration if $w$ is divisorial by \cite[Theorem 2.24]{LXZ}. 
Since divisorial valuations are dense in $\mathrm{QM}_\eta (Y, D)$, $\varphi_v$ is in the closure of $\{ \varphi_{(\mathcal{X}, \mathcal{L}; \rho)} ~|~ (\mathcal{X}, \mathcal{L}): \text{ weakly special degeneration } \}$, so that we have $\NAmu^{2\pi} (\varphi_v) = 2\pi \check{H}_{\mathrm{NA}} (\varphi_v)$. 
Here we note $U \to C^0 \cap \PSH (X, L): w \mapsto \varphi_w$ is continuous for the sup norm. 
Therefore, we get $\NAmu^{2\pi} \le 2\pi \check{H}_{\mathrm{NA}} \le \NAmu^{2\pi} (\varphi_v)$ and conclude $\varphi_v$ maximizes $\NAmu^{2\pi}$ on $\E^{\exp} (X, L)$. 

To check $\mu^{2\pi}$K-semistability, it suffices to check $\NAmu^{2\pi} (\varphi_v) = \cmu^{2\pi} (\mathcal{X}_o, \mathcal{L}|_{\mathcal{X}_o}; \xi_o^v)$. 
As we remark in the above, we have $\NAmu^{2\pi} (\varphi_w) = 2\pi \check{H}_{\mathrm{NA}} (\varphi_w)$ for $w$ in $U$. 
Then $\NAmu^{2\pi} (\varphi_w)$ is continuous on $w \in U$ by the upper semi-continuity of $\NAmu^{2\pi}$ and the lower semi-continuity of $\check{H}_{\mathrm{NA}}$. 
On the other hand, the filtration $\mathcal{F}_w$ induces a proper vector $\xi^w_o$ on the central fibre $(\mathcal{X}_o, \mathcal{L}|_{\mathcal{X}_o})$. 
As $w_i \to w$, we have $\xi^{w_i}_o \to \xi^w_o$. 
Then $\cmu^{2\pi} (\mathcal{X}_o, \mathcal{L}|_{\mathcal{X}_o}; \xi_o^w)$ is also continuous on $w$. 
Since the central fibre is reduced, we already know $\NAmu^{2\pi} (\varphi_{(\mathcal{X}, \mathcal{L}; \xi_o^w)}) = \cmu^{2\pi} (\mathcal{X}_o, \mathcal{L}|_{\mathcal{X}_o}; \xi^w_o)$ for divisorial $w \in U$. 
Then the desired equality follows by the continuity. 
\end{proof}

\subsubsection{Relation to normalized Donaldson--Futaki invariant}
\label{Relation to normalized Donaldson--Futaki invariant}

In the study \cite{Ino2} of $\mu^\lambda$-cscK metrics, we encounter extremal metrics in the limit $\lambda \to -\infty$. 
Here we observe its non-archimedean counterpart: in the limit $\lambda \to -\infty$ the non-archimedean $\mu$-entropy is related to normalized Donaldson--Futaki invariant which appears in Donaldson--Xia's minimax principle \cite{Don1, Xia}. 
See also \cite{Don2, Sze}, \cite{Der1, Der2, Ino4} and section \ref{Maximizing non-archimedean Calabi energy}. 

Recall the non-archimedean Mabuchi functional is defined on $\E^1 (X, L)$ by 
\begin{equation}
M_{\mathrm{NA}} (\varphi) = \int_{X^{\mathrm{NA}}} A_X \mathrm{MA} (\varphi) + \frac{(K_X, 0) \cdot (L, \varphi)^{\cdot n}}{n!} - \frac{(K_X. e^L)}{(e^L)} \frac{(L, \varphi)^{\cdot n+1}}{(n+1)!}. 
\end{equation}
Putting $b_\varphi := \int_\mathbb{R} t \DHm_\varphi/\int_\mathbb{R} \DHm_\varphi =  E (\varphi)/(e^L)$ and $\| \bar{\varphi} \| := (\int_\mathbb{R} (t-b_\varphi)^2 \DHm_\varphi)^{1/2}$, we introduce the following functional defined on $\E^2 (X, L)$: 
\begin{equation} 
C_{\mathrm{NA}} (\varphi) := - \frac{1}{(e^L)} \Big{(} 2 \pi M_{\mathrm{NA}} (\varphi) + \frac{1}{2} \| \bar{\varphi} \|^2 \Big{)}. 
\end{equation}
This functional is upper semi-continuous on $\E^2 (X, L)$ with respect to the topology induced from the metric $d_2$ (see section \ref{metric space Eexp}). 

This functional is related to normalized Donaldson--Futaki invariant in the following way. 
We note $C_{\mathrm{NA}} (\varphi_{; \rho})$ is a quadratic function on $\rho \ge 0$: 
\begin{align*} 
C_{\mathrm{NA}} (\varphi_{; \rho}) 
&= - \frac{1}{(e^L)} \Big{(} 2 \pi M_{\mathrm{NA}} (\varphi) \cdot \rho + \frac{1}{2} \| \bar{\varphi} \|^2 \cdot \rho^2 \Big{)}
\\
&= - \frac{\| \bar{\varphi} \|^2}{2 (e^L)} \Big{(} \rho + \frac{2\pi M_{\mathrm{NA}} (\varphi)}{\| \bar{\varphi} \|^2} \Big{)}^2 + \frac{2 \pi^2}{(e^L)} \frac{M_{\mathrm{NA}} (\varphi)^2}{ \| \bar{\varphi} \|^2}. 
\end{align*}
Thus we have 
\[ \sup_{\rho \ge 0} C_{\mathrm{NA}} (\varphi_{; \rho}) =
\begin{cases} 
0
& M_{\mathrm{NA}} (\varphi) > 0
\\
\frac{2 \pi^2}{(e^L)} \frac{M_{\mathrm{NA}} (\varphi)^2}{\| \bar{\varphi} \|^2} = \frac{\| \bar{\varphi}_{; \rho_{\max}} \|^2}{2 (e^L)}
& M_{\mathrm{NA}} (\varphi) \le 0
\end{cases}, \]
where the maximum are attained at $\rho_{\max} = 0$ and $\rho_{\max} = - 2\pi M_{\mathrm{NA}} (\varphi) / \| \bar{\varphi} \|^2$, respectively. 
The right hand side is nothing but the normalized Donaldson--Futaki invariant. 
As observed in \cite{Don1, Ino4} (see also \cite{Der1}), we have 
\[ \sup_{\varphi \in \nH (X, L)} C_{\mathrm{NA}} (\varphi) \le \inf_{\omega_\phi \in \mathcal{H} (X, L)} C (\omega_\phi) \]
for smooth $(X, L)$, which can be made into the equality by completing the domain of the functionals \cite{Xia} in a suitable way. 

Now as in \cite{Ino2, Ino4}, let us observe the extremal limit $\lambda \to -\infty$ of the non-archimedean $\mu$-entropy. 
As for $\bm{\check{\sigma}}$, we can easily see the following. 

\begin{prop}
For $\varphi \in \E^{\exp} (X, L)$, we have 
\begin{align*} 
\frac{d}{d\rho}\Big{|}_{\rho=+0} \bm{\check{\sigma}} (\varphi_{;\rho}) 
&= \lim_{\rho \to + 0} \rho^{-1} (\bm{\check{\sigma}} (\varphi_{; \rho}) - \bm{\check{\sigma}} (\varphi_{\mathrm{triv}})) 
= 0, 
\\
\frac{d^2}{d\rho^2}\Big{|}_{\rho=+0} \bm{\check{\sigma}} (\varphi_{;\rho}) 
&= 2 \lim_{\rho \to + 0} \rho^{-2} (\bm{\check{\sigma}} (\varphi_{; \rho}) - \bm{\check{\sigma}} (\varphi_{\mathrm{triv}})) 
= \frac{\| \bar{\varphi} \|^2}{(e^L)}. 
\end{align*}
\end{prop}

\begin{proof}
Recall 
\[ \bm{\check{\sigma}} (\varphi_{;\rho}) = \frac{\int_\mathbb{R} (n-\rho t) e^{-\rho t} \DHm_\varphi }{\int_\mathbb{R} e^{-\rho t} \DHm_\varphi} - \log \int_\mathbb{R} e^{-\rho t} \DHm_\varphi. \]
We compute 
\begin{align*}
\frac{d}{d\rho} \bm{\check{\sigma}} (\varphi_{;\rho}) 
%&= \frac{\int_\mathbb{R} (-t - t (n-\rho. t)) e^{-\rho. t} \DHm_\varphi \cdot \int_\mathbb{R} e^{-\rho. t} \DHm_\varphi - \int_\mathbb{R} (n-\rho. t) e^{-\rho. t} \DHm_\varphi \cdot \int_\mathbb{R} (-t) e^{-\rho. t} \DHm_\varphi}{(\int_\mathbb{R} e^{-\rho. t} \DHm_\varphi)^2}
%\\
%&\qquad - \frac{\int_\mathbb{R} (-t) e^{-\rho. t} \DHm_\varphi}{\int_\mathbb{R} e^{-\rho. t} \DHm_\varphi}
%\\
&=\rho \frac{\int_\mathbb{R} t^2 e^{-\rho t} \DHm_\varphi \cdot \int_\mathbb{R} e^{-\rho t} \DHm_\varphi -  (\int_\mathbb{R} t e^{-\rho t} \DHm_\varphi)^2}{(\int_\mathbb{R} e^{-\rho t} \DHm_\varphi)^2}
\end{align*}
%\begin{align*}
%\frac{d}{d\rho}\Big{|}_{\rho=+0} \bm{\check{\sigma}} (\varphi_{;\rho}) 
%&= - \frac{\int_{X^{\mathrm{NA}}} \varphi \mathrm{MA} (\varphi) + (L, 0) \cdot (L, \varphi)^{\cdot n}/n!}{(L^{\cdot n})/n!} 
%\\
%&\qquad + \frac{(L^{\cdot n})/(n-1)! \cdot (L, \varphi)^{\cdot n+1}/(n+1)!}{\Big{(} (L^{\cdot n})/n! \Big{)}^2} - \frac{- (L, \varphi)^{\cdot n+1}/(n+1)!}{(L^{\cdot n})/n!}
%\\
%&= - \frac{(L, \varphi)^{\cdot n+1}/n! - n (L, \varphi)^{\cdot n+1}/(n+1)! - (L, \varphi)^{\cdot n+1}/(n+1)!}{(L^{\cdot n})/n!} = 0. 
%\end{align*}
and 
\begin{align*}
\frac{d^2}{d\rho^2}\Big{|}_{\rho=0} \bm{\check{\sigma}} (\varphi_{;\rho}) 
&= \lim_{\rho \to 0} \rho^{-1} \frac{d}{d\rho} \bm{\check{\sigma}} (\varphi_{;\rho})
\\
&= \frac{\int_\mathbb{R} t^2 \DHm_\varphi \cdot \int_\mathbb{R} \DHm_\varphi -  (\int_\mathbb{R} t \DHm_\varphi)^2}{(\int_\mathbb{R} \DHm_\varphi)^2}
= \frac{\| \bar{\varphi} \|^2}{(e^L)}.
\end{align*}
\end{proof}

To see the behavior of $\NAmu$, we prepare some computations. 

\begin{lem}
\label{limit computation lemma}
Let $f$ be a right continuous function on $\mathbb{R}$ such that $f (\sigma) e^{-\rho \sigma}$ is integrable for every small $\rho > 0$ and $f (\sigma) = c$ for $\sigma \gg 0$. 
Then we have 
\[ \lim_{\rho \to 0} \rho \int_\mathbb{R} f e^{-\rho \sigma} d\sigma = c \]
and 
\[ \lim_{\rho \to 0} \rho^2 \int_\mathbb{R} f \sigma e^{-\rho \sigma} d\sigma = c \]
\end{lem}

\begin{proof}
For 
\[ \chi (\sigma) = 
\begin{cases} 
0 
& \sigma < \tau
\\
\sigma - \tau 
& \tau \le \sigma \le \tau + c 
\\
c
& \sigma > \tau + c
\end{cases}, \]
we compute 
\[ \rho \int_\mathbb{R} \chi e^{-\rho \sigma} d\sigma 
%= \rho \Big{(} \int_\tau^{\tau+c} (\sigma - \tau) e^{-\rho \sigma} d\sigma + c \int_{\tau+c}^\infty e^{-\rho \sigma} d\sigma \Big{)}
%= \rho \Big{(} [- \rho^{-1} (\sigma - \tau) e^{-\rho \sigma} - \rho^{-2} e^{-\rho \sigma}]_\tau^{\tau+c} + c [-\rho^{-1} e^{-\rho \sigma}]_{\tau+c}^\infty \Big{)}
%= \rho \Big{(} -\rho^{-1} c e^{-\rho (\tau +c)} - \rho^{-2} (e^{-\rho (\tau+c)} - e^{-\rho \tau}) + c \rho^{-1} e^{-\rho (\tau+c)} \Big{)}
=  - \rho^{-1} (e^{-\rho (\tau+c)} - e^{-\rho \tau}) \to c \]
as $\rho \to 0$. 
Similarly, 
\[ \rho^2 \int_\mathbb{R} \chi \sigma e^{-\rho \sigma} d\sigma = - ((\tau +c) e^{-\rho (\tau+c)} - \tau e^{-\rho \tau}) - 2 \rho^{-1} (e^{-\rho (\tau + c)} - e^{-\rho \tau}) \to c. \]
On the other hand, since $f - \chi$ and $(f - \chi) \sigma$ has left bounded support, $|f - \chi| \max \{ e^{-\rho_0 \sigma}, 1 \}$ and $|(f - \chi) \sigma| \max \{ e^{-\rho_0 \sigma}, 1 \}$ are integrable. 
Thus we have uniform bounds 
\[ \Big{|} \int_\mathbb{R} (f - \chi) e^{-\rho \sigma} d\sigma \Big{|} \le \int_\mathbb{R} |f - \chi| \max \{ e^{-\rho_0 \sigma}, 1 \} d\sigma < \infty, \]
\[ \Big{|} \int_\mathbb{R} (f - \chi) \sigma e^{-\rho \sigma} d\sigma \Big{|} \le \int_\mathbb{R} |(f - \chi) \sigma| \max \{ e^{-\rho_0 \sigma}, 1 \} d\sigma < \infty. \]
It follows that 
\[ \lim_{\rho \to 0} \rho \int_\mathbb{R} f e^{-\rho \sigma} d\sigma = c + \lim_{\rho \to 0} \rho \int_\mathbb{R} (f - \chi) e^{-\rho \sigma} d\sigma = c, \]
\[ \lim_{\rho \to 0} \rho^2 \int_\mathbb{R} f \sigma e^{-\rho \sigma} d\sigma = c + \lim_{\rho \to 0} \rho^2 \int_\mathbb{R} (f - \chi) \sigma e^{-\rho \sigma} d\sigma = c. \]
\end{proof}

\begin{prop}
For $\varphi \in \E^{\exp} (X, L)$ and $\psi \in \E^1 (X, L) \cup C^0 (X^{\mathrm{NA}})$, we have 
\[ \frac{d}{d\rho}\Big{|}_{\rho=+0} \int_{X^{\mathrm{NA}}} \psi_{;\rho} \int e^{-t} \mathcal{D}_{\varphi_{; \rho}} = \lim_{\rho \to +0} \int_{X^{\mathrm{NA}}} \psi \int e^{-\rho t} \mathcal{D}_\varphi = \int_{X^{\mathrm{NA}}} \psi \mathrm{MA} (\varphi). \]

For $\varphi \in \PSH^{\mathrm{bdd}} (X, L)$, we have 
\[ \frac{d}{d\rho}\Big{|}_{\rho=+0} \int_{X^{\mathrm{NA}}} A_X \int e^{-t} \mathcal{D}_{\varphi_{; \rho}} = \lim_{\rho \to +0} \int_{X^{\mathrm{NA}}} A_X \int e^{-\rho t} \mathcal{D}_\varphi = \int_{X^{\mathrm{NA}}} A_X \mathrm{MA} (\varphi). \]
\end{prop}

\begin{proof}
For $\psi \in \E^1 (X, L)$, using the above lemma, we compute 
\begin{align*} 
\int_{X^{\mathrm{NA}}} \psi \int e^{-\rho t} \mathcal{D}_\varphi 
&= \rho^{-1} \int_{X^{\mathrm{NA}}} \psi_{; \rho} \int e^{-t} \mathcal{D}_{\varphi_{; \rho}} 
\\
&=\rho^{-1} \Big{(} \int_\mathbb{R} d\tau e^{-\tau} \int_{X^{\mathrm{NA}}} (\psi_{;\rho} - \psi_{;\rho} (v_{\mathrm{triv}})) \mathrm{MA} (\varphi_{;\rho} \wedge \tau) + \psi_{;\rho} (v_{\mathrm{triv}}) \int_\mathbb{R} e^{-\tau} \DHm_{\varphi_{;\rho}} \Big{)}
\\
&= \rho \int_\mathbb{R} d\sigma e^{-\rho \sigma} \int_{X^{\mathrm{NA}}} (\psi - \psi (v_{\mathrm{triv}})) \mathrm{MA} (\varphi \wedge \sigma) + \psi (v_{\mathrm{triv}}) \int_\mathbb{R} e^{-\rho \tau} \DHm_\varphi
\\
&\to \int_{X^{\mathrm{NA}}} \psi \mathrm{MA} (\varphi)
\end{align*}
as $\rho \to 0$. 
Since $\lim_{\rho \to +0} \int_{X^{\mathrm{NA}}} \psi_{; \rho} \int e^{-t} \mathcal{D}_{\varphi_{; \rho}} = 0$, we obtain the first line for $\psi \in \E^1 (X, L)$. 
By uniform approximation, we can also check the same convergence for $\psi \in C^0 (X^{\mathrm{NA}})$. 
In particular, the measure $\int e^{-\rho t} \mathcal{D}_\varphi$ converges weakly to $\mathrm{MA} (\varphi)$ as $\rho \to 0$. 

We similarly have $\lim_{\rho \to +0} \int_{X^{\mathrm{NA}}} A_X \int e^{-t} \mathcal{D}_{\varphi_{; \rho}} = 0$. 
We compute 
\[ \rho^{-1} \int_{X^{\mathrm{NA}}} A_X \int e^{-t} \mathcal{D}_{\varphi_{; \rho}} = \rho^{-1} \int_{X^{\mathrm{NA}}} A_X (\rho. v) \int e^{-\rho t} \mathcal{D}_\varphi = \int_{X^{\mathrm{NA}}} A_X \int e^{-\rho t} \mathcal{D}_\varphi. \]
Since $A_X$ is lsc, we get 
\[ \varliminf_{\rho \to 0} \int_{X^{\mathrm{NA}}} A_X \int e^{-\rho t} \mathcal{D}_\varphi \ge \int_{X^{\mathrm{NA}}} A_X \mathrm{MA} (\varphi). \]
On the other hand, for $\varphi \in \PSH^{\mathrm{bdd}} (X, L)$, we have 
\[ \int e^{-\rho t} \mathcal{D}_\varphi \le e^{-\rho \inf \varphi} \mathrm{MA} (\varphi). \]
Indeed, we can directly check this for $\varphi = \varphi_{(\mathcal{X}, \mathcal{L})} \in \nH (X, L)$ by 
\[ \frac{\int_\mathbb{R} e^{-\rho t} \DHm_{(E, \mathcal{L}|_E)}}{\int_\mathbb{R} \DHm_{(E, \mathcal{L}|_E)}} \le e^{-\rho \inf \varphi}. \]
The general case follows by passing to the limit: $\varliminf_{i \to \infty} \inf \varphi_i \ge \inf \varphi$ for any convergent decreasing net $\varphi_i \searrow \varphi \in \PSH^{\mathrm{bdd}} (X, L)$. 
It follows that we have 
\[ \varlimsup_{\rho \to 0} \int_{X^{\mathrm{NA}}} A_X \int e^{-\rho t} \mathcal{D}_\varphi \le \varlimsup_{\rho \to 0} e^{-\rho \inf \varphi} \int_{X^{\mathrm{NA}}} A_X \mathrm{MA} (\varphi) = \int_{X^{\mathrm{NA}}} A_X \mathrm{MA} (\varphi) \]
for $\varphi \in \PSH^{\mathrm{bdd}} (X, L)$, which shows the claim. 
\end{proof}

We speculate the latter equality holds for general $\varphi \in \E^{\exp} (X, L)$ with finite $\int_{X^{\mathrm{NA}}} A_X \int e^{-t} \mathcal{D}_\varphi < \infty$. 

\begin{prop}
For $\varphi \in \PSH^{\mathrm{bdd}} (X, L)$, we have 
\[ \frac{d}{d\rho}\Big{|}_{\rho=+0} \NAmu (\varphi_{;\rho}) = \lim_{\rho \to + 0} \rho^{-1} (\NAmu (\varphi_{; \rho}) - \NAmu (\varphi_{\mathrm{triv}})) 
= - \frac{2\pi}{(e^L)} M_{\mathrm{NA}} (\varphi). \]
\end{prop}

\begin{proof}
Recall 
\begin{align*} 
\NAmu (\varphi_{;\rho}) 
&= - 2\pi \frac{\int_{X^{\mathrm{NA}}} A_X \int e^{-t} \mathcal{D}_{\varphi_{;\rho}} + E_{\exp}^{K_X} (\varphi_{;\rho})}{\int_\mathbb{R} e^{-\rho t} \DHm_\varphi}.
\end{align*}

For $\varphi \in \E^{\exp} (X, L)$, we show 
\[ \frac{d}{d\rho}\Big{|}_{\rho=+0} E^M (\varphi_{;\rho}) = \frac{(M, 0) \cdot (L, \varphi)^{\cdot n}}{n!}. \]
From the proof of Proposition \ref{large limit of mu-entropy}, we recall 
\begin{align*} 
E_{\exp}^M (\varphi_{; \rho}) 
&= \int_\mathbb{R} \frac{(M, 0) \cdot (L, \varphi_{;\rho} \wedge \tau - \tau)^{\cdot n}}{n!} e^{-\tau} d\tau 
%\\
%&= \int_\mathbb{R} \frac{(M, 0) \cdot (L, (\varphi \wedge \rho^{-1} \tau - \rho^{-1} \tau)_{;\rho})^{\cdot n}}{n!} e^{-\tau} d\tau 
%\\
%&= \rho \int_\mathbb{R} \frac{(M, 0) \cdot (L, \varphi \wedge \rho^{-1} \tau - \rho^{-1} \tau)^{\cdot n}}{n!} e^{-\tau} d\tau 
\\
&= \rho^2 \int_\mathbb{R} \frac{(M, 0) \cdot (L, \varphi \wedge \sigma - \sigma)^{\cdot n}}{n!} e^{- \rho \sigma} d\sigma. 
\end{align*}
Again by Lemma \ref{limit computation lemma}, we compute 
\begin{align*} 
\frac{d}{d\rho}\Big{|}_{\rho=+0} E_{\exp}^M (\varphi_{;\rho}) 
&= \lim_{\rho \to +0} 2 \rho \int_\mathbb{R} \frac{(M, 0) \cdot (L, \varphi \wedge \sigma - \sigma)^{\cdot n}}{n!} e^{- \rho \sigma} d\sigma 
\\
&\qquad- \lim_{\rho \to +0} \rho^2 \int_\mathbb{R} \frac{(M, 0) \cdot (L, \varphi \wedge \sigma - \sigma)^{\cdot n}}{n!} \sigma e^{- \rho \sigma} d\sigma 
\\
&= \frac{(M, 0) \cdot (L, \varphi)^{\cdot n}}{n!}. 
\end{align*}

By Leibniz rule, we compute 
\begin{align*} 
\frac{d}{d\rho}\Big{|}_{\rho=+0} \NAmu (\varphi_{;\rho}) 
&= - 2\pi \frac{\int_{X^{\mathrm{NA}}} A_X \mathrm{MA} (\varphi) + (K_X, 0) \cdot (L, \varphi)^{\cdot n}/n!}{\int_\mathbb{R} \DHm_\varphi} + 2\pi (K_X. e^L) \frac{\int_\mathbb{R} t \DHm_\varphi}{( \int_\mathbb{R} \DHm_\varphi )^2}
\\
&= - \frac{2\pi}{(e^L)} M_{\mathrm{NA}} (\varphi). 
\end{align*}
\end{proof}

For $\varphi \in \nH (X, L)$, we can also compute these derivatives by the equivariant intersection formulae. 

\begin{cor}
For $\varphi \in \PSH^{\mathrm{bdd}} (X, L)$, we have 
\[ \frac{d}{d\rho}\Big{|}_{\rho=0} \NAmu^\lambda (\varphi_{;\rho}) = - \frac{2\pi}{(e^L)} M_{\mathrm{NA}} (\varphi). \]
For $\varphi \in \E^2 (X, L)$, we have 
\[ \frac{d}{d\rho}\Big{|}_{\rho=0} C_{\mathrm{NA}} (\varphi_{;\rho}) = - \frac{2\pi}{(e^L)} M_{\mathrm{NA}} (\varphi). \]
\end{cor}

%Computation via equivariant intersection
%\[ \lim_{\rho \to 0} \rho^{-1} (\NAmu (\varphi_{; \rho}) - \frac{(K_X. e^L)}{(e^L)}) = \frac{d}{d\rho}\Big{|}_{\rho=0} \NAmu (\varphi_{;\rho}) = - \frac{(K_{\bar{\mathcal{X}}/\mathbb{P}^1}. \bar{\mathcal{L}}^{\cdot n})}{n! (e^L)} - \frac{(K_X. e^L)}{(e^L)} \frac{(\bar{\mathcal{L}}^{\cdot n+1})}{(n+1)! (e^L)}  \]
%
%\begin{align*} 
%\lim_{\rho \to 0} \rho^{-2} (\bm{\check{\sigma}} (\varphi_{;\rho}) - n + \log (e^L)) 
%&= \lim_{\rho \to 0} \rho^{-1} (\frac{\int_\mathbb{R} (-t) e^{-\rho t} \DHm_\varphi}{\int_\mathbb{R} e^{-\rho t} \DHm_\varphi} + \frac{(\bar{\mathcal{L}}^{\cdot n+1})}{(n+1)! (e^L)})
%\\
%&\qquad - \lim_{\rho \to 0} \rho^{-2} (\rho \frac{(\bar{\mathcal{L}}^{\cdot n+1})}{(n+1)! (e^L)} + \log \int_\mathbb{R} e^{-\rho t} \DHm_\varphi)
%\\
%&= \frac{d}{d\rho}\Big{|}_{\rho = 0} \frac{\int_\mathbb{R} (-t) e^{-\rho t} \DHm_\varphi}{\int_\mathbb{R} e^{-\rho t} \DHm_\varphi} - \frac{1}{2} \frac{d^2}{d\rho^2}\Big{|}_{\rho=0} \log \int_\mathbb{R} e^{-\rho t} \DHm_\varphi
%\\
%&= -\frac{1}{2} \Big{(} \frac{(\bar{\mathcal{L}}_{\mathbb{G}_m}^{\cdot n+2}; 1)}{(n+2)! (e^L)} + (\frac{(\bar{\mathcal{L}}^{\cdot n+1})}{(n+1)! (e^L)})^2 \Big{)}
%\end{align*}

Now we find the following as an expansion of \cite{Ino2}. 

\begin{thm}
For $\varphi \in \PSH^{\mathrm{bdd}} (X, L)$, we have 
\[ \lim_{\rho \to +0} \rho^{-1} (\NAmu^{-\rho^{-1}} (\varphi_{; \rho}) - \NAmu^{-\rho^{-1}} (0)) = C_{\mathrm{NA}} (\varphi). \]
\end{thm}

\begin{proof}
We compute 
\begin{align*} 
\lim_{\rho \to +0} \rho^{-1} (\NAmu^{-\rho^{-1}} (\varphi_{; \rho}) - \NAmu^{-\rho^{-1}} (0)) 
&= \frac{d}{d\rho}\Big{|}_{\rho=+0} \NAmu (\varphi_{;\rho}) - \frac{1}{2} \frac{d^2}{d\rho^2}\Big{|}_{\rho=+0} \bm{\check{\sigma}} (\varphi_{;\rho}) 
\\
&=  - 2\pi (e^L)^{-1} M_{\mathrm{NA}} (\varphi) - \frac{1}{2} (e^L)^{-1} \| \bar{\varphi} \|^2
\end{align*}
\end{proof}

If Conjecture \ref{properness conjecture} holds, then by Proposition \ref{maximizer under properness conjecture} we have a maximizer $\varphi^\rho_{\mathrm{opt}}$ of $\NAmu^{-\rho^{-1}}$ for $\rho > 0$ with $\sup \varphi^\rho_{\mathrm{opt}} = 0$. 
Then the rescaling $\tilde{\varphi}^\rho_{\mathrm{opt}} := \varphi^\rho_{\mathrm{opt}; \rho^{-1}}$ gives a maximizer of the normalized functional $\rho^{-1} (\NAmu^{-\rho^{-1}} (\bullet_{; \rho}) - \NAmu^{-\rho^{-1}} (0))$. 
The author speculates $\tilde{\varphi}^\rho_{\mathrm{opt}}$ converges to a limit $\varphi_{\mathrm{ext}}$ which maximizes $C_{\mathrm{NA}}$. 

\subsubsection{Maximizing non-archimedean Calabi energy}
\label{Maximizing non-archimedean Calabi energy}

Here we discuss the maximization problem for $C_{\mathrm{NA}}$, referring to various fundamental conjectures. 
It is studied in \cite{Xia} that for smooth $(X, L)$ the normalized Donaldson--Futaki invariant extends to the space $\mathcal{R}^2 (X, \omega)$ of geodesic rays in $\mathcal{E}^2 (X, L)$ in an (archimedean) analytic way: we put 
\begin{align*} 
\bm{C}_{\mathrm{ray}} (\ell) 
&:= - \frac{1}{(e^L)} \Big{(} 2 \pi \bm{M}^\infty (\ell) + \frac{1}{2} \|\ell \|^2 \Big{)},
\\ 
\bm{M}^\infty (\ell)
&:= \lim_{t \to \infty} t^{-1} M (\ell_t),
\\
\| \ell \|
&:= \Big{(} \int_X |\dot{\ell}_t|^2 \frac{\omega_{\ell_t}^n}{n!} \Big{)}^{1/2}
\end{align*}
for a geodesic ray $\ell$ in $\mathcal{E}^2 (X, L)$ normalized by $\int_X \dot{\ell}_t \omega_{\ell_t}^n/n! = 0$. 
Here $M$ denotes the Mabuchi functional. 
Similarly as $C_{\mathrm{NA}}$, we have 
\[ \sup_{\rho \ge 0} \bm{C}_{\mathrm{ray}} (\ell_{;\rho}) = 
\begin{cases}
0 
& \bm{M}^\infty (\ell) > 0 
\\
\frac{2\pi^2}{(e^L)} \frac{\bm{M}^\infty (\ell)^2}{\| \ell \|^2} 
& \bm{M}^\infty (\ell) \le 0
\end{cases}, \]
where $\ell_{;\rho}$ denotes the rescaled geodesic $(\ell_{; \rho})_t = \ell_{\rho t}$. 
Then the main theorem of \cite{Xia} shows that $\bm{C}_{\mathrm{ray}}$ admits a maximizing geodesic ray, and the following minimax principle holds for the Calabi energy 
\begin{equation} 
\sup_{\ell \in \mathcal{R}^2 (X, \omega)} \bm{C}_{\mathrm{ray}} (\ell) = \inf_{\omega_\phi \in \mathcal{E}^2 (X, L)} C (\omega_\phi). 
\end{equation}
See also \cite{His} for a similar equality for Ding version of the Calabi energy. 

On the other hand, it is shown by \cite{Li2} that the maximizing geodesic ray is maximal in the sense of \cite{BBJ} (since it destabilizes the Mabuchi energy) and hence is indeed subordinate to a non-archimedean psh metric $\varphi \in \E^2 (X, L)$. 
So we actually have 
\[ \sup_{\varphi \in \E^2 (X, L)} \bm{C}_{\mathrm{ray}} (\ell_\varphi) = \inf_{\omega_\phi \in \mathcal{E}^2 (X, L)} C (\omega_\phi). \]

Therefore, to conclude the existence of a maximizer $\varphi \in \E^2 (X, L)$ of $C_{\mathrm{NA}}$, it suffices to compare $\bm{C}_{\mathrm{ray}} (\ell_\varphi)$ and $C_{\mathrm{NA}} (\varphi)$. 
By \cite[Theorem 1.7]{Li2}, we know $\bm{M}^\infty (\ell_\varphi) \ge M_{\mathrm{NA}} (\varphi)$, so that at least we have 
\[ C_{\mathrm{NA}} (\varphi) \ge \bm{C}_{\mathrm{ray}} (\ell_\varphi). \]
The reverse inequality for maximizing $\varphi \in \E^2 (X, L)$ can be reduced to one of the following conjectures. 

\begin{conj}[Minimax conjecture for Calabi energy]
\label{Minimax conjecture for Calabi energy}
Suppose $X$ has only klt singularities. 
Then we have 
\[ \sup_{\varphi \in \E^2 (X, L)} C_{\mathrm{NA}} (\varphi) = \inf_{\omega_\phi \in \mathcal{E}^2 (X, L)} C (\omega_\phi) \]
with an appropriate definition of the right hand side. 
\end{conj}

As for smooth $X$, this conjecture can be reduced to the following inequality from what we observed. 
\[ \sup_{\varphi \in \E^2 (X, L)} C_{\mathrm{NA}} (\varphi) \le \inf_{\omega_\phi \in \mathcal{E}^2 (X, L)} C (\omega_\phi) \]
We recall Donaldson's inequality \cite{Don1} states 
\begin{equation} 
\label{Donaldson inequality}
\sup_{\varphi \in \nH (X, L)} C_{\mathrm{NA}} (\varphi) \le \inf_{\omega_\phi \in \mathcal{H} (X, L)} C (\omega_\phi). 
\end{equation}
If we can replace the right hand side with $\inf_{\omega_\phi \in \mathcal{E}^2 (X, L)} C (\omega_\phi)$, then this conjecture can be reduced to another Conjecture \ref{regularization of entropy}. 
This replacement is related to the smoothness of Calabi flow as remarked in \cite[Remark 4.2]{Xia}. 

The following conjecture provides a more direct way to show $C_{\mathrm{NA}} (\varphi) = \bm{C}_{\mathrm{ray}} (\ell_\varphi)$. 
As observed in \cite{Li2}, this conjecture also follows from Conjecture \ref{regularization of entropy}. 

\begin{conj}[Slope formula for entropy \cite{Li2}]
\label{slope formula for entropy}
The slope of the entropy 
\[ \lim_{t \to \infty} t^{-1} \int_X \log \frac{\mathrm{MA} (\ell_{\varphi, t})}{\mathrm{MA} (\ell_{\varphi, 0})} \mathrm{MA} (\ell_{\varphi, t}) \] 
along the maximal geodesic ray $\{ \ell_{\varphi, t} \}_{t \in [0,\infty)}$ subordinate to $\varphi \in \E^2 (X, L)$ (cf. \cite{BBJ}) is equal to the non-archimedean entropy $\int_{X^{\mathrm{NA}}} A_X \mathrm{MA} (\varphi)$. 
\end{conj}

The following conjecture provides yet another approach for the maximization problem on $C_{\mathrm{NA}}$, which does not relying on Xia's existence result unlike the above approaches. 
This approach is more close to what we proposed for the maximization problem on the non-archimedean $\mu$-entropy. 

\begin{conj}[Properness of Calabi energy]
\label{properness conjecture for Calabi energy}
Assume $(X, L)$ is klt. 
Consider $d_p$-topology for $1 \le p < 2$ on $\E^2 (X, L)$. 
Then the subset 
\[ \{ \varphi \in \E^2 (X, L) ~|~ E (\varphi) = 0,~ C_{\mathrm{NA}} (\varphi) \ge C \} \]
is compact in $d_p$-topology. 
\end{conj}

We note the subset is not compact in $d_2$-topology as we can see in the following toric example. 
This is the reason we consider $E_{\exp}$-topology in analogous Conjecture \ref{properness conjecture} rather than stronger $d_{\exp}$-topology. 

\begin{eg}
Consider a polarized toric normal variety $(X, L)$ and the associated toric polytope $P$. 
As explained in section \ref{NAmu via Legendre dual}, for a lower semi-continuous convex function $q$ on $P$, we can assign a non-archimedean psh metric $\varphi_q$ on $(X, L)$. 
Thanks to \cite[Proposition 6.3]{Li2}, we have 
\[ C_{\mathrm{NA}} (\varphi_q) = - \frac{2\pi}{\int_P d\mu} \Big{(} \int_{\partial P} q d\sigma + \frac{(K_X. e^L)}{(e^L)} \int_P q d\mu \Big{)} - \frac{1}{2 \int_P d\mu} \int_P (q- \bar{q})^2 d\mu. \]

Now consider the toric polytope $P = [0,1]^2$ and the following sequence of convex functions on $P$: 
\[ q_n = \max \{ 0, n - n^2 (x+ y) \} - \frac{1}{6n}. \]
We can compute $\int_{\partial P} q_n d\sigma = 1 - 2/3n$, $\int_P q_n d\mu = 0$ and $\int_P q_n^2 d\mu = 1/12 - (1/6n)^2$. 
It follows that $C_{\mathrm{NA}} (\varphi_{q_n}) = - 2\pi (1- 2/3n) - (1/12 - (1/6n)^2)/2 \ge -2\pi - 1/24$ is bounded and $E (\varphi_{q_n}) = 0$. 
Since $d_1 (\varphi_{q_n}, 0) = \int_P |q_n| d\mu \le 1/3n$, the corresponding sequence of non-archimedean psh metrics $\varphi_{q_n}$ converges to the trivial metric $\varphi_{\mathrm{triv}} = 0$ in $d_1$-topology (or even $d_p$-topology for $p < 2$), but it contains no subsequence converging in $d_2$-topology as $d_2 (\varphi_{q_n}, 0) \ge 1/18$. 
\end{eg}

This conjecture implies the existence of maximizers. 

\begin{prop}
Assuming Conjecture \ref{properness conjecture for Calabi energy}, there exists a maximizer $\varphi_{\mathrm{ext}} \in \E^2 (X, L)$ of $C_{\mathrm{NA}}$. 
\end{prop}

\begin{proof}
For a $d_1$-convergent sequence $\varphi_i \to \varphi \in \E^2 (X, L)$, we have $\varliminf_{i \to \infty} \| \bar{\varphi}_i \| \le \| \bar{\varphi} \|$ by Proposition \ref{lower semi-continuity of norms in d1 topology}, so that 
\[ \varlimsup_{i \to \infty} C_{\mathrm{NA}} (\varphi_i) \le C_{\mathrm{NA}} (\varphi). \]

Take a sequence $\varphi_i \in \E^2 (X, L)$ so that $C_{\mathrm{NA}} (\varphi_i) \nearrow \sup C_{\mathrm{NA}}$. 
By Conjecture \ref{properness conjecture for Calabi energy}, we have a $d_1$-convergent subsequence $\varphi_j \to \varphi$. 
Then we get 
\[ \sup C_{\mathrm{NA}} = \varlimsup_{j \to \infty} C_{\mathrm{NA}} (\varphi_j) \le C_{\mathrm{NA}} (\varphi) \le \sup C_{\mathrm{NA}}. \]
Thus the limit $\varphi$ attains the maximum. 
\end{proof}

In view of the following proposition, this conjecture is analogous to the following fact: any $L^p$-bounded almost everywhere convergent sequence of measurable functions on a finite measure space converges in $L^q$-topology for any $q < p$. 
Recall this is a consequence of Egorov's theorem. 

\begin{prop}
Suppose $C_{\mathrm{NA}}$ is bounded from above on $\E^2 (X, L)$, and the continuity of envelopes holds for $(X, L)$. 
Then for any $C \in \mathbb{R}$, the subset 
\[ \{ \varphi \in \E^2 (X, L) ~|~ E (\varphi) = 0,~ C_{\mathrm{NA}} (\varphi) \ge C \} \]
is $d_2$-bounded and relatively weakly compact. 
\end{prop}

\begin{proof}
Take $\varphi$ in the subset. 
If $M_{\mathrm{NA}} (\varphi) > 0$, we have 
\[ -C \ge -C_{\mathrm{NA}} (\varphi) \ge \frac{\| \bar{\varphi} \|^2}{2(e^L)}. \]
If $M_{\mathrm{NA}} (\varphi) \le 0$, we have 
\[ \sup_{\rho \ge 0} C_{\mathrm{NA}} (\varphi_{;\rho}) = \frac{2\pi^2}{(e^L)} \frac{M_{\mathrm{NA}} (\varphi)^2}{\| \bar{\varphi} \|^2}. \]
Since $C_{\mathrm{NA}}$ is bounded from above on $\E^2 (X, L)$, we have a constant $C' \ge \max \{ 0, C \}$ such that 
\[ -2\pi M_{\mathrm{NA}} (\varphi) \le \sqrt{2 (e^L) C'} \cdot \| \bar{\varphi} \|. \]
It follows that 
\begin{align*} 
-C \ge -C_{\mathrm{NA}} (\varphi) 
&\ge \frac{1}{(e^L)} \Big{(} - \sqrt{2 (e^L) C'} \cdot \| \bar{\varphi} \| + \frac{1}{2} \| \bar{\varphi} \|^2 \Big{)} 
\\
&= \frac{1}{2 (e^L)} \Big{(} \| \bar{\varphi} \| - \sqrt{2 (e^L) C'} \Big{)}^2 - C', 
\end{align*}
hence we get 
\[ \sqrt{2 (e^L) (C' - C)} + \sqrt{2 (e^L) C'} \ge \| \bar{\varphi} \|. \]
Therefore, $d_2 (0, \varphi) = \| \bar{\varphi} \|$ is uniformly bounded on the subset in either cases. 

Since $\bar{I} (\varphi, 0) \le C_n d_1 (0, \varphi) \le C_n d_2 (0, \varphi)$ by \cite[Lemma 5.5]{BJ4}, $\sup \varphi$ is uniformly bounded on the subset. 
Therefore, under the continuity of envelopes, any sequence in the subset contains a weakly convergent subnet by \cite[Corollary 4.58]{BJ3} (cf. \cite[Corollary 6.5]{BJ1}). 
At the moment, we do not know if $\lim_{i \to \infty} E (\varphi_i) = E (\varphi)$ for the $d_2$-bounded weakly convergent net $\varphi_i \to \varphi$. 
\end{proof}

When $X$ is smooth, we know the boundedness of $C_{\mathrm{NA}}$ on $\nH (X, L)$ thanks to Donaldson's inequality (\ref{Donaldson inequality}). 
The boundedness on $\E^2 (X, L)$, which is the assumption of the proposition, holds especially when the following conjecture holds. 
The conjecture is confirmed for $T$-invariant metrics on toric varieties (cf. \cite[Proposition 6.3]{Li2}). 

\begin{conj}[Regularization of entropy \cite{BJ2, Li2}]
\label{regularization of entropy}
For any $\varphi \in \E^1 (X, L)$, there exists a sequence $\{ \varphi_i \}_{i \in \mathbb{N}} \subset \nH (X, L)$ converging to $\varphi$ in the strong topology such that 
\[ \lim_{i \to \infty} \int_{X^{\mathrm{NA}}} A_X \mathrm{MA} (\varphi_i) = \int_{X^{\mathrm{NA}}} A_X \mathrm{MA} (\varphi). \]
\end{conj}

On the other hand, we know the continuity of envelopes holds for smooth $X$, so that the assumptions of the above proposition are valid for smooth $X$ under this regularization conjecture. 
In conjunction with this conjecture, we can reduce Conjecture \ref{properness conjecture for Calabi energy} for smooth $X$ to the following more simple conjecture. 
The existence of maximizers of $C_{\mathrm{NA}}$ is related such fundamental conjectures in the non-archimedean pluripotential theory. 

\begin{conj}
\label{Egorov type estimate conjecture}
Every $d_2$-bounded weakly convergent sequence in $\E^2 (X, L)$ is $d_1$-convergent. 
\end{conj}

We note $d_2$-bounded $d_1$-convergent sequence is $d_p$-convergent for $1 \le p < 2$ by Lebesgue interpolation as in the proof of Proposition \ref{Eexp convergence implies dp convergence}: 
\[ d_p (\varphi_i, \varphi)^p \le d_1 (\varphi_i, \varphi)^{2-p} (d_2 (\varphi_i, 0) + d_2 (\varphi, 0))^{2 (p-1)}. \]

\subsubsection{Relation to normalized Ding invariant}

Again we consider a $\mathbb{Q}$-Fano variety $(X, L) = (X, -K_X)$. 
The minimax equality for a Ricci potential version of the Calabi energy is studied in \cite{His}: 
\[ \sup_{\varphi \in \nH (X, L)} R_{\mathrm{NA}} (\varphi) = \inf_{\omega_\phi \in \mathcal{H} (X, L)} R (\omega_\phi). \]

Here we put 
\[ R (\omega_\phi) := \frac{1}{(e^L)} \int_X (e^{h_\phi} - 1)^2 \omega_\phi^n/n! \]
using the Ricci potential $\mathrm{Ric} (\omega_\phi) - 2\pi \omega_\phi = \sqrt{-1} \partial \bar{\partial} h_\phi$ normalized by $\int_X e^{h_\phi} \omega_\phi^n/n! = (e^L)$. 
The critical points of this functional are Mabuchi soliton (cf. \cite{Mab}).
The existence of Mabuchi soliton is studied in \cite{His2}. 

Similarly as $C_{\mathrm{NA}}$, we put 
\[  R_{\mathrm{NA}} (\varphi) := - \frac{1}{(e^L)} \Big{(} 2\pi D_{\mathrm{NA}} (\varphi) + \frac{1}{2} \| \bar{\varphi} \|^2 \Big{)}, \]
using the non-archimedean Ding functional 
\[ D_{\mathrm{NA}} (\varphi) = (e^L) \inf_{x \in X^{\mathrm{qm}}} (A_X (x) + \varphi (x)) - E (\varphi). \]
We note 
\begin{align*} 
R_{\mathrm{NA}} (\varphi_{; \rho}) 
&= - \frac{1}{(e^L)} \Big{(} 2\pi D_{\mathrm{NA}} (\varphi) \cdot \rho + \frac{1}{2} \| \bar{\varphi} \|^2 \cdot \rho^2 \Big{)} 
\\
&= - \frac{\| \bar{\varphi} \|^2}{2 (e^L)} \Big{(} \rho + 2\pi D_{\mathrm{NA}} (\varphi) \Big{)}^2 + \frac{2\pi^2}{(e^L)} \frac{D_{\mathrm{NA}} (\varphi)^2}{\| \bar{\varphi} \|^2}, 
\end{align*}
so that we have 
\[ \sup_{\rho \ge 0} R_{\mathrm{NA}} (\varphi_{; \rho}) = 
\begin{cases}
0
& D_{\mathrm{NA}} (\varphi) > 0
\\
\frac{2\pi^2}{(e^L)} \frac{D_{\mathrm{NA}} (\varphi)^2}{\| \bar{\varphi} \|^2}
& D_{\mathrm{NA}} (\varphi) \le 0
\end{cases}.
 \]
 
Compared to the Calabi flow, we have a smooth solution to the gradient flow of $R$ by \cite{CHT}, so that we can restrict things to $\nH (X, L)$ and $\mathcal{H} (X, L)$, rather than $\E^2 (X, L)$ and $\mathcal{E}^2 (X, L)$. 

Now the non-archimedean Ding functional is known to be continuous along decreasing net. 
This is an alternative of Conjecture \ref{regularization of entropy} for $R_{\mathrm{NA}}$. 
It follows that $R_{\mathrm{NA}}$ is bounded from above on $\E^2 (X, L)$ when $X$ is smooth. 
Thus we can reduce the maximization problem for $R_{\mathrm{NA}}$ to Conjecture \ref{Egorov type estimate conjecture}, following what we observed for $C_{\mathrm{NA}}$.

\section{Appendix: Toric illustration}

\subsection{Non-archimedean $\mu$-entropy via Legendre dual}
\label{NAmu via Legendre dual}

\subsubsection{Legendre transform of $T$-invariant non-archimedean metric}

We firstly review \cite[Appendix B]{BJ3}. 

Let $(X, L)$ be a polarized normal toric variety and $P \subset \mathfrak{t}^\vee$ be the moment polytope: 
\[ P := l^{-1} \overline{\{ \mu \in M ~|~ H^0 (X, lL)_\mu \neq 0 \}}^{\mathrm{conv}}, \]
where $l \in \mathbb{N}_+$ is taken so that $l L$ is a line bundle. 
On $P$, we consider the restriction of Lebesgue measure $d\mu$ normalized by the lattice $N$. 
(Since $|\det A| = 1$ for $A \in GL (N)$, the measure is uniquely determined. )

On the boundary $\partial P$, we consider the measure $d\sigma$ which is characterized on each face $\partial P \cap \{ \mu \in \mathfrak{t} ~|~ \langle \mu, \eta \rangle - \lambda = 0 \}$ by 
\[ d\sigma (B) = d\mu_{\mathrm{Ker} \eta} (B - \mu_0) \]
for a Borel subset $B$ of the face. 
Here $\mu_0 \in \mathfrak{t}$ is a point with $\langle \mu_0, \eta \rangle -\lambda = 0$ and $d\mu_{\mathrm{Ker} \eta}$ is the Lebesgue measure on $\mathrm{Ker} \eta$ normalized by the lattice $N \cap \mathrm{Ker} \eta$. 

Consider a normal toric test configuration $(\mathcal{X}, \mathcal{L})$ of $(X, L)$. 
Take sufficiently large $c > 0$ so that the line bundle $\bar{\mathcal{L}}_c := \bar{\mathcal{L}} + c. [\mathcal{X}_0]$ over $\bar{\mathcal{X}}$ is ample. 
Then the moment polytope $Q_c$ associated to the polarized toric $T \times \mathbb{G}_m$-variety $(\bar{\mathcal{X}}, \bar{\mathcal{L}}_c)$ can be written as 
\[ Q_c = \{ (\mu, t) \in \mathfrak{t}^\vee \times \mathbb{R} ~|~ \mu \in P, ~0 \le t \le - q (\mu) + c \}, \] 
using a continuous convex function $q_{(\mathcal{X}, \mathcal{L})}: P \to \mathbb{R}$ of the form 
\[ q_{(\mathcal{X}, \mathcal{L})} (\mu) = \max_E \{ \langle \mu, \eta_E \rangle - \lambda_E \} \] 
for some $\eta_E \in (\mathrm{ord}_E \mathcal{X}_0)^{-1} N$ and $\lambda_E \in \mathbb{Q}$ assigned to each irreducible component $E$ of $\mathcal{X}_0$ (cf. \cite[Proposition 4.1.1]{CLS}). 
This $q_{(\mathcal{X}, \mathcal{L})}$ is independent of the choice of $c > 0$. 
By putting $q_{(\mathcal{X}, \mathcal{L})} (\mu) = + \infty$ for $\mu \notin P$, we can regard $q_{(\mathcal{X}, \mathcal{L})}$ as a lower semi-continuous convex function on $\mathfrak{t}^\vee$. 

Recall the Legendre dual of a function $q: \mathfrak{t}^\vee \to [-\infty, \infty]$ is a function $q^*: \mathfrak{t} \to [-\infty, \infty]$ given by 
\[ q^* (\xi) := \sup \{ \langle \mu, \xi \rangle - q (\mu) ~|~ \mu \in \mathfrak{t}^\vee \}. \]
Since $\{ \langle \mu, \xi \rangle - q (\mu) \}_{\mu \in \mathfrak{t}^\vee}$ is a family of linear functions, the supremum gives a lower semi-continuous convex function on $\mathfrak{t}$. 
We have $(q^*)^* = q$ if and only if $q$ is a lower semi-continuous convex function. 

\begin{prop}
We have 
\[ q_{(\mathcal{X}, \mathcal{L})}^* (\xi) = \varphi_{(\mathcal{X}, \mathcal{L})} (v_{-\xi}) + \sup_{\mu \in P} \langle \mu, \xi \rangle. \]
In particular, $\varphi_{(\mathcal{X}, \mathcal{L})} (v_{-\xi}) + \sup_{\mu \in P} \langle \mu, \xi \rangle$ is a convex function on $\mathfrak{t}$. 
\end{prop}

\begin{proof}
We recall $H^0 (X, L^{\otimes m}) = \bigoplus_{\mu \in mP \cap M} H^0 (X, L^{\otimes m})_\mu$ and $H^0 (X, L^{\otimes m})_\mu = \mathbb{C}. s_\mu$. 
The basis $\{ s_\mu \}$ is diagonal with respect to $\| \cdot \|^{(\mathcal{X}, \mathcal{L})}$, so we have 
\begin{align*} 
\varphi_{(\mathcal{X}, \mathcal{L})} (v_{-\xi}) 
&= \frac{1}{m} \sup_{\mu \in mP \cap M} \{ - v_{-\xi} (s_\mu) - \log \| s_\mu \|^{(\mathcal{X}, \mathcal{L})} \} 
\\
&= \frac{1}{m} \sup_{\mu \in mP \cap M} \{ - (\langle \mu, -\xi \rangle - m \sigma_{\min, m} (\mathcal{F}_{-\xi})) - \log \| s_\mu \|^{(\mathcal{X}, \mathcal{L})} \}
\\
&= \sup_{\mu \in P \cap m^{-1} M} \{ \langle \mu, \xi \rangle - \frac{1}{m} \log \| s_{m \mu} \|^{(\mathcal{X}, \mathcal{L})} \} - \sup_{\mu \in P} \langle \mu, \xi \rangle, 
\end{align*}
where the last equality follows by $\sigma_{\min, m} (\mathcal{F}_{-\xi}) = \inf \supp \DHm_{\mathcal{F}_{-\xi}} = \inf_{\mu \in P} \langle \mu, -\xi \rangle = - \sup_{\mu \in P} \langle \mu, \xi \rangle$. 

It suffices to show 
\[ q_{(\mathcal{X}, \mathcal{L})}^* = \sup_{\mu \in P \cap m^{-1} M} \{ \langle \mu, \xi \rangle - \frac{1}{m} \log \| s_{m \mu} \|^{(\mathcal{X}, \mathcal{L})} \} \]
for sufficiently divisible $m$. 
Since 
\begin{align*} 
- \frac{1}{m} \log \| s_{m \mu} \|^{(\mathcal{X}, \mathcal{L})} 
&= \frac{1}{m} \sup \{ \lambda \in \mathbb{Z} ~|~ \varpi^{-\lambda}. \bar{s}_{m\mu} \in H^0 (\mathcal{X}, \mathcal{L}^{\otimes m}) \} 
\\
&= \sup \{ \lambda \in m^{-1} \mathbb{Z} ~|~ (\mu, \lambda + c) \in Q_c \cup P \times (-\infty, 0) \}
\\
&= \sup \{ \lambda \in m^{-1} \mathbb{Z} ~|~ \lambda \le - q (\mu) \}, 
\end{align*}
we have 
\[ q_{(\mathcal{X}, \mathcal{L})}^* \ge \sup_{\mu \in P \cap m^{-1} M} \{ \langle \mu, \xi \rangle - \frac{1}{m} \log \| s_{m \mu} \|^{(\mathcal{X}, \mathcal{L})} \}. \]
The equality for sufficiently divisible $m$ is a consequence of the finite expression of $q = \max_E \{ \langle \mu, \eta_E \rangle - \lambda_E \}$. 
\end{proof}

Let $\nH (X, L)_T$ denote the set of non-archimedean psh metrics associated to toric test configurations. 
We call a non-archimedean psh metric $\varphi$ on $(X, L)$ \textit{$T$-invariant} if it is the limit of some decreasing net $\{ \varphi_i \}_{i \in I} \subset \nH (X, L)_T$. 
We denote by $\PSH (X, L)_T$ the set of $T$-invariant non-archimedean psh metrics. 
For $\varphi \in \PSH (X, L)_T$, we put 
\[ f_\varphi (\xi) := \varphi (v_{-\xi}) + \sup_{\mu \in P} \langle \mu, \xi \rangle. \]
Since $v_\xi$ is quasi-monomial, each $f_\varphi (\xi)$ is finite. 
For a convergent decreasing net $\varphi_i \searrow \varphi$ in $\PSH (X, L)_T$, we have $f_{\varphi_i} \searrow f_\varphi$. 
Thanks to the above proposition, $f_\varphi$ is convex for $\varphi \in \nH (X, L)$, which implies the convexity of $f_\varphi$ for general $\varphi \in \PSH (X, L)_T$. 
Therefore, $f_\varphi$ is finite valued convex function on $\mathfrak{t}$, in particular, it is continuous. 

Similarly, we put $q_\varphi := f_\varphi^*: \mathfrak{t}^\vee \to (-\infty, \infty]$ for $\varphi \in \PSH (X, L)_T$. 
Again by the above proposition, we have $q_{\varphi_{(\mathcal{X}, \mathcal{L})}} = q_{(\mathcal{X}, \mathcal{L})}$. 
For $\varphi_i \searrow \varphi$, we have $q_{\varphi_i} \nearrow q_\varphi$. 

\begin{prop}
\label{filtration via Legendre dual}
For $\varphi \in C^0 \cap \PSH (X, L)_T$, we have 
\[ \mathcal{F}_\varphi^\lambda R_m = \{ s \in R_m ~|~ \lambda \le - m q_\varphi (\mu/m) \text{ if } s_\mu \neq 0 \}. \]
\end{prop}

\begin{proof}
Since $\mathcal{F}_\varphi$ is $T$-invariant, we have $\mathcal{F}_\varphi = \bigcap_{\xi \in \mathfrak{t}} \mathcal{F}_{v_\xi} [\varphi (v_\xi)]$. 
It follows that $s \in \mathcal{F}_\varphi^\lambda R_m$ iff 
\[ v_\xi (s_\mu) + m \varphi (v_\xi) \ge \lambda \]
for every $\xi \in \mathfrak{t}$ and $\mu$ with $s_\mu \neq 0$. 
We compute 
\[ v_\xi (s) + m \varphi (v_\xi) = \inf \{ \langle \mu, \xi \rangle ~|~ s_\mu \neq 0 \} + m \sup_{\mu \in P} \langle \mu, -\xi \rangle + m \varphi (v_\xi) = \inf \{ \langle \mu, \xi \rangle ~|~ s_\mu \neq 0 \} + m f_\varphi (-\xi). \]
On the other hand, by the convex duality, we have 
\[ m q_\varphi (\mu/m) = \sup \{ \langle \mu, \xi \rangle - m f_\varphi (\xi) ~|~ \xi \in \mathfrak{t} \} = - \inf \{ v_{-\xi} (s_\mu) + m \varphi (v_{-\xi}) ~|~ \xi \in \mathfrak{t} \} \]
for $0 \neq s_\mu \in H^0 (X, L^{\otimes m})_\mu$. 
This proves the claim. 
\end{proof}

\begin{prop}
For $\varphi \in \PSH (X, L)_T$, we have 
\[ q_{\varphi \wedge \tau} = \max \{ q_\varphi, -\tau \} \]
\end{prop}

\begin{proof}
We note $q_\tau = - \tau$ for $\tau \in \mathbb{R}$. 
By the above proposition, for $\varphi \in C^0 \cap \PSH (X, L)_T$, we have 
\begin{align*} 
\{ s \in R_m ~|~ \lambda \le - q_{\varphi \wedge \tau} (\mu/m) \text{ if } s_\mu \neq 0 \} 
&= \mathcal{F}_{\varphi \wedge \tau}^\lambda R_m = (\mathcal{F}_\varphi^\lambda \cap \mathcal{F}_\tau^\lambda) R_m 
\\
&= \{ s \in R_m ~|~ \lambda \le - \max \{ q_\varphi (\mu/m), -\tau \} \text{ if } s_\mu \neq 0 \}. 
\end{align*}
This shows the claim for $\varphi \in C^0 \cap \PSH (X, L)_T$. 
The general case follows by the limit argument $\varphi_i \searrow \varphi$. 
\end{proof}

\begin{prop}
\label{relative DH measure via Legendre dual}
For $\varphi \in \PSH (X, L)_T$ and $\varphi' \in \nH (X, L)_T$, we have 
\[ \DHm_{\varphi, \varphi'} = ((q_{\varphi'}-q_\varphi)^* d\mu)|_\mathbb{R}. \]
Here we note $(q_{\varphi'}-q_\varphi)^* d\mu$ is a measure on $[-\infty, \infty)$ and may charge $\{ - \infty \}$, in which case we have $\int_\mathbb{R} \DHm_{\varphi, \varphi'} < (e^L)$. 
\end{prop}

\begin{proof}
We firstly assume $\varphi \in \nH (X, L)_T$. 
Take a basis $\{ s_{m, \mu} \}_{\mu \in mP \cap M}$ of $H^0 (X, L^{\otimes m}) = \bigoplus_{\mu \in mP \cap M} H^0 (X, L^{\otimes m})_\mu$ so that $H^0 (X, L^{\otimes m})_\mu = \mathbb{C}. s_\mu$. 
Then it is codiagonal with respect to $\varphi, \varphi'$, so we have 
\[ \DHm_{\varphi, \varphi'} = \lim_{m \to \infty} \frac{1}{m^n} \sum_{\mu \in mP \cap M} \delta_{\lambda^\varphi (s_{m, \mu})/m - \lambda^{\varphi'} (s_{m, \mu})/m}. \] 
By Proposition \ref{filtration via Legendre dual}, we have $\lambda^\varphi (s_{m, \mu}) = - m q_\varphi (\mu/m)$, so that we get 
\[ \DHm_{\varphi, \varphi'} = \lim_{m \to \infty} \frac{1}{m^n} \sum_{\mu \in P \cap m^{-1} M} \delta_{- q_\varphi (\mu) + q_{\varphi'} (\mu)}. \]
By the Riemannian integral, the limit can be identified with $((q_{\varphi'}-q_\varphi)^* d\mu)|_\mathbb{R}$. 

Now consider the general case $\varphi \in \PSH (X, L)_T$. 
It suffices to check 
\[ \int_{[\tau, \infty)} \DHm_{\varphi, \varphi'} = \int_{[\tau, \infty)} (q_{\varphi'} -q_\varphi)^* d\mu = \mu (\{ q_{\varphi'} -q_\varphi \ge \tau \}) \]
for general $\varphi \in \PSH (X, L)_T$. 
The left hand side is continuous along decreasing nets $\varphi_i \searrow \varphi$ by our construction of the Duistermaat--Heckman measure. 
On the other hand, the right hand side is continuous as $F_i := \{ q_{\varphi'} - q_{\varphi_i} \ge \tau \}$ gives an decreasing net of closed sets for which we have $\bigcap_i F_i = \{ q_{\varphi'} - q_\varphi \ge \tau \}$. 
\end{proof}

It follows that for $\varphi \in \PSH (X, L)_T$ we have $\varphi \in \E (X, L)$ iff $\mu (\{ q_\varphi = \infty \}) = 0$, and $\varphi \in \E^1 (X, L)$ iff $\int_P q_\varphi d\mu < \infty$. 
In particular, for $\varphi \in \E (X, L)_T = \E (X, L) \cap \PSH (X, L)_T$, $q_\varphi$ is continuous on $P^\circ$. 
If $q_\varphi$ is finite valued on $P$, then it is continuous on $P$ by \cite{GKR}. 

\subsubsection{$d_{\exp}$-metric via Legendre dual}

We observe the Legendre dual of 
\[ \E^{\exp} (X, L)_T := \E^{\exp} (X, L) \cap \PSH (X, L)_T. \]
We consider 
\[ \mathrm{Conv}^{\exp} (P) := \{ q: P \to (-\infty, \infty] ~|~ q \text{ is lsc convex and } \int_P e^{\rho q} d\mu < \infty \text{ for } \forall \rho > 0 \}. \]
Since $q$ is finite valued on a dense subset, it is automatically continuous on the interior $P^\circ$. 
By the convexity and the lower semi-continuity, for $\mu_o \in P^\circ$ and $\mu \in \partial P$, the convex function $q ((1-t) \mu_o + t\mu)$ is continuous on $t \in [0,1]$ (possibly $+\infty$ at $t=1$), hence $q$ is uniquely determined by $q|_{P^\circ}$. 
We consider the following distance on $\mathrm{Conv}^{\exp} (P)$: 
\[ d^{\exp} (q, q') := \inf \{ \beta > 0 ~|~ \int_P (e^{|q-q'|/\beta} -1) d\mu \le 1 \}. \]

\begin{prop}
\label{completeness of the Legendre dual}
$(\mathrm{Conv}^{\exp} (P), d^{\exp})$ is a complete metric space. 
If $q_i \to q$ in $d^{\exp}$, then $q_i$ converges uniformly to $q$ on every compact set $K \subset P^\circ$ and we have $q \le \varliminf q_i$ everywhere on $P$. 
\end{prop}

\[ \mathrm{Conv}^{\exp} (P) \to [0, \infty]: q \mapsto \lim_{t \to 1} \int_{\partial P} e^{q ((1-t) \mu_o +t \mu)} d\sigma \]

\begin{proof}
Let $q_i$ be a Cauchy sequence. 
By the completeness of the small Orlicz space $L^{\exp} (P) = \{ f: P \to [-\infty, \infty] ~|~ \int_P e^{\rho f} d\mu < \infty \text{ for } \forall \rho > 0 \}$, we have a Lebesgue measurable function $\tilde{q}: P \to [-\infty, \infty]$ such that $q_i \to \tilde{q}$ in the $L^{\exp}$-norm. 
Since $q_i$ converges to $\tilde{q}$ pointwisely on a dense subset, there exists a convex function $q^\circ$ on $P^\circ$ such that $q_i$ converges to $q^\circ$ uniformly on any compact set $K \Subset P^\circ$ by a general argument on convex function. 
Fix a point $\mu_o \in P^\circ$ and put $q (\mu) := \lim_{t \to 1} q^\circ ((1-t) \mu_o + t \mu)$ for $\mu \in P$. 
As we have $q|_{P^\circ} = q^\circ \stackrel{\text{a.e.}}{=} \tilde{q}$, $q_i \to q$ in the $L^{\exp}$-norm. 
In particular, we get $q \in \mathrm{Conv}^{\exp} (P)$ and hence $q_i \to q$ in $d_{\exp}$. 

We show that $q$ is lower semi-continuous. 
The argument mimics \cite{Don2}. 
Take a convergent sequence $\mu_i \to \mu \in P$. 
Put $f_i (t) := q ((1-t) \mu_o + t \mu_i)$ and $f (t) := q ((1-t) \mu_o + t \mu)$, then by the continuity of $q|_{P^\circ}$, $f_i$ converges to $f$ pointwisely on $[0,1)$. 
Since
\[ f_i (t) \le (1-t) q (\mu_o) + t q (\mu_i) \]
we have 
\[ f (t) \le (1-t) q (\mu_o) + t \varliminf_{i \to \infty} q (\mu_i) \]
for $t \in [0,1)$. 
Taking the limit $t \to 1$, we get
\[ q (\mu) = f (1) \le \varliminf_{i \to \infty} q (\mu_i). \]

The rest claim is $q \le \varliminf_{i \to \infty} q_i$. 
We have
\[ q_i ((1-t) \mu_o +t \mu) \le (1-t) q_i (\mu_o) + t q_i (\mu) \] 
Since $q_i ((1-t) \mu_o +t \mu) \to q ((1-t) \mu_o +t \mu)$ for $t \in [0, 1)$, we have 
\[ q ((1-t) \mu_o +t \mu) \le (1-t) q (\mu_o) + t \varliminf_{i \to \infty} q_i (\mu) \]
for $t \in [0,1)$. 
Taking the limit $t \to 1$, we get 
\[ q (\mu) \le \varliminf_{i \to \infty} q_i (\mu). \]
\end{proof}

For $\varphi \in \E^{\exp} (X, L)_T$, we have 
\[ E_{\exp} (\varphi_{; \rho}) = \int_\mathbb{R} e^{-\rho t} \DHm_\varphi = \int_P e^{\rho q_\varphi} d\mu, \]
so that $q_\varphi \in \mathrm{Conv}^{\exp} (P)$. 

\begin{prop}
The map 
\[ (\E^{\exp} (X, L)_T, d_{\exp}) \to (\mathrm{Conv}^{\exp} (P), d^{\exp}): \varphi \mapsto q_\varphi \]
is an isometry. 
Assuming the continuity of envelopes, the map is bijective. 
\end{prop}

\begin{proof}
By Proposition \ref{relative DH measure via Legendre dual}, for $\varphi \in \E^{\exp} (X, L)_T$ and $\varphi' \in \nH (X, L)_T$, we have 
\[ \int_\mathbb{R} (e^{|t|/\beta} -1) \DHm_{\varphi, \varphi'} = \int_P (e^{|q_\varphi - q_{\varphi'}|/\beta} -1) d\mu. \]
Then by Proposition \ref{distance with anchor}, we get 
\[ d_{\exp} (\varphi, \varphi') = d^{\exp} (q_\varphi, q_{\varphi'}). \]
To see this for general $\varphi' \in \E^{\exp} (X, L)_T$, take a decreasing net $\varphi'_i \subset \nH (X, L)_T$ so that $\varphi'_i \searrow \varphi'$, then the monotone convergence theorem shows $d^{\exp} (q_{\varphi'_i}, q_{\varphi'}) \to 0$. 
It follows that $d^{\exp} (q_\varphi, q_{\varphi'_i}) \to d^{\exp} (q_\varphi, q_{\varphi'})$. 
Since we already know $d_{\exp} (\varphi, \varphi'_i) \to d_{\exp} (\varphi, \varphi')$, this shows the above map is an isometry. 

To see the surjectivity, we want to assign $\varphi_q \in \E^{\exp} (X, L)_T$ to $q \in \mathrm{Conv}^{\exp} (P)$. 
Since $q = \sup \{ \ell: P \to \mathbb{R} ~|~ q \ge \ell: \text{ rational affine function } \}$, we can find an increasing net of convex functions $q_i$ on $P$ of the form $q_i = q_{\varphi_i}$ for $\varphi_i \in \nH (X, L)_T$ so that $q_i \nearrow q$. 
It is not evident if $\varphi_i$ gives a decreasing net, so we instead use the completeness of $\E^{\exp} (X, L)_T$ as follows. 
By the monotone convergence theorem, we have $d_{\exp} (q_i, q) \to 0$. 
Since $d_{\exp} (\varphi_i, \varphi_j) = d^{\exp} (q_i, q_j)$, $\{ \varphi_i \} \subset \E^{\exp} (X, L)_T$ gives a Cauchy net. 
Under the continuity of envelopes, $\E^{\exp} (X, L)_T \subset \E^{\exp} (X, L)$ is complete, so that the Cauchy net converges to some $\varphi \in \E^{\exp} (X, L)_T$. 
Since 
\[ d^{\exp} (q, q_\varphi) \le d^{\exp} (q, q_i) + d^{\exp} (q_i, q_\varphi) = d^{\exp} (q, q_i) + d_{\exp} (\varphi_i, \varphi) \to 0, \]
we have $q_\varphi = q$. 
\end{proof}

\subsubsection{Non-archimedean $\mu$-entropy via Legendre dual}

For $q \in \mathrm{Conv}^{\exp} (P)$, we put 
\begin{align}
\NAmu^* (q) 
&:= -2\pi \frac{\int_{\partial P} e^q d\sigma}{\int_P e^q d\mu}, 
\\
\bm{\check{\sigma}}^* (q) 
&:= \frac{\int_P (n+q) e^q d\mu}{\int_P e^q d\mu} - \log \int_P e^q d \mu, 
\\
\NAmu^{*, \lambda} (q) 
&:= \NAmu^* (q) + \lambda \bm{\check{\sigma}}^* (q). 
\end{align}

\begin{prop}
\label{Toric formula}
For a normal test configuration $(\mathcal{X}, \mathcal{L})$ and $\rho > 0$, we have 
\[ \NAmu^\lambda (\varphi_{(\mathcal{X}, \mathcal{L}; \rho)}) = \NAmu^{*, \lambda} (\rho q_{(\mathcal{X}, \mathcal{L})}). \]
\end{prop}

\begin{proof}
We firstly note the canonical divisor formula for toric variety $K_{\bar{\mathcal{X}}} = - \sum_F D_F$, where $D_F$ denote the prime divisor corresponding to a face $F \subset \partial Q$ (cf. \cite[Theorem 8.2.3]{CLS}). 
Since $D_F$ is $T \times \mathbb{G}_m$-invariant, we have a $T \times \mathbb{G}_m$-equivariant Chow class $- \sum_F [D_F]^{T \times \mathbb{G}_m}$. 
The above formula indeed gives the equality of $T \times \mathbb{G}_m$-equivariant Chow classes $K_{\bar{\mathcal{X}}}^{T \times \mathbb{G}_m} = - \sum_F [D_F]^{T \times \mathbb{G}_m}$ by \cite[Theorem 13.3.1]{CLS}. 

Let $u_F = (u_F^N, a_F) \in N \times \mathbb{Z}$ denote the minimal generator of the normal ray of the face $F$. 
Since we have $\pi^* K_{\mathbb{P}^1}^{\mathbb{G}_m} = -[\mathcal{X}_0]^{\mathbb{G}_m} - [\mathcal{X}_\infty] = -\sum_F |a_F| [D_F]^{\mathbb{G}_m}$ (cf. \cite[Proposition 4.1.1]{CLS}) with $\mathcal{X}_\infty = \pi^{-1} (0:1) = X$, we have 
\begin{align*}
 K_{\bar{\mathcal{X}}/\mathbb{P}^1}^{\log, T \times \mathbb{G}_m} 
 &= K_{\bar{\mathcal{X}}/\mathbb{P}^1}^{T \times \mathbb{G}_m} + ([\mathcal{X}^{\mathrm{red}}_0]^{\mathbb{G}_m} - [\mathcal{X}_0]^{\mathbb{G}_m}) 
 \\
 &= -\sum_F (1- |a_F|) [D_F]^{\mathbb{G}_m} + \sum_{F, a_F \neq 0} (1- |a_F|) [D_F]^{\mathbb{G}_m} = - \sum_{F, a_F = 0} [D_F]^{\mathbb{G}_m}. 
\end{align*}
Thus we get 
\[ (K_{\bar{\mathcal{X}}/\mathbb{P}^1}^{\log}. e^{\bar{\mathcal{L}}}; \rho) = -\sum_{F, a_F = 0} \int_F e^{-\rho t} d\sigma \otimes dt. \]

%It suffices to show the following equalities 
%\begin{gather*}
%(e^{\bar{\mathcal{L}}}; \rho) = \int_P \frac{1- e^{\rho q (\mu)}}{\rho} d\mu, 
%\\
%(\bar{\mathcal{L}}. e^{\bar{\mathcal{L}}}; \rho) = \int_P \frac{1+ \rho q (\mu) e^{\rho q (\mu)}}{\rho} d\mu, 
%\\
%(K_{\bar{\mathcal{X}}/\mathbb{P}^1}^{\log}. e^{\bar{\mathcal{L}}}; \rho) = \int_{\partial P} \frac{1- e^{\rho q (\mu)}}{\rho} d\mu. 
%\end{gather*}
%
%\[ [t e^{-\rho t} + \frac{1}{\rho} e^{-\rho t} ]_0^{-q (\mu)} = - \frac{1}{\rho} (1 + \rho q (\mu) e^{\rho q (\mu)} - e^{\rho q (\mu)}) \]

Now we compute 
\begin{gather*} 
(e^{\bar{\mathcal{L}}}; \rho) = \int_Q e^{-\rho t} d\mu \otimes dt = \int_P d\mu \int_0^{-q(\mu)} e^{-\rho t} dt = \int_P \frac{1- e^{\rho q (\mu)}}{\rho} d\mu, 
\\ 
(\bar{\mathcal{L}}. e^{\bar{\mathcal{L}}}; \rho) = \int_Q ((n+1)-\rho t) e^{-\rho t} d\mu \otimes dt = \int_P d\mu \int_0^{-q (\mu)} (-\rho t) e^{-\rho t} dt = \int_P \frac{n- (n- \rho q (\mu)) e^{\rho q (\mu)}}{\rho} d\mu, 
\\
(K_{\bar{\mathcal{X}}/\mathbb{P}^1}^{\log}. e^{\bar{\mathcal{L}}}; \rho) = \sum_{F, a_F = 0} \int_F e^{-\rho t} d\sigma \otimes dt = \int_{\partial P} d\sigma \int_0^{-q (\mu)} e^{-\rho t} dt = \int_{\partial P} \frac{1- e^{\rho q (\mu)}}{\rho} d\mu, 
\end{gather*}
which shows the claim in conjunction with Lemma \ref{Localization}. 
\end{proof}

\begin{prop}
The functional $\NAmu^{*, \lambda}$ is upper semi-continuous with respect to the $d^{\exp}$-topology on $\mathrm{Conv} (P)$ and is continuous along increasing nets $q_i \nearrow q$. 
\end{prop}

\begin{proof}
We can check the continuity of $\bm{\check{\sigma}}^*$ in the same way as $\bm{\check{\sigma}}$. 
As for $\NAmu^*$, it suffices to check the lower semi-continuity of the map 
\[ \mathrm{Conv}^{\exp} (P) \to [0, \infty]: q \mapsto \int_{\partial P} e^q d\sigma. \]
Suppose $q_i \to q$ in $d_{\exp}$. 
By Proposition \ref{completeness of the Legendre dual}, we have $q \le \varliminf q_i$, so we get 
\[ \int_{\partial P} e^q d\sigma \le \int_{\partial P} \varliminf e^{q_i} d\sigma \le \varliminf \int_{\partial P} e^{q_i} d\sigma \]
by Fatou's lemma. 

If $q_i \nearrow q$, then by the monotone convergence theorem 
\[ \lim \int_{\partial P} e^{q_i} d\sigma = \int_{\partial P} e^q d\sigma, \]
which proves the continuity along $q_i \nearrow q$. 
\end{proof}

In particular, $\NAmu^{*, \lambda} (q_\varphi)$ is continuous along decreasing sequence $\varphi_i \searrow \varphi$. 
By the upper semi-continuity of $\NAmu^\lambda (\varphi)$, we get 
\[ \NAmu^\lambda (\varphi) \ge \NAmu^{*, \lambda} (q_\varphi) \]
for $\varphi \in \E^{\exp} (X, L)_T$. 
The equality holds under Conjecture \ref{Regularization of exponential entropy} for $\E^{\exp} (X, L)_T$. 

\begin{quest}
Can we prove the existence of maximizers of $\NAmu^{*, \lambda}$ for $\lambda \le 0$? 
Can we find a polytope $P$ for which we can show the existence of a piecewise affine maximizer? 
\end{quest}

\subsection{Illustrations}

Here we compute the $\mu$-entropy for proper vectors in explicit examples and give illustrations of the graphs. 
We compute integration of exponential by localization to vertices (cf. \cite[Theorem 13.5.2]{CLS}: let $P \subset \mathbb{R}^n$ be a simple polytope, i.e. for each vertex $v$, the cone $C_v = \mathrm{Cone} (P-v)$ is spanned by precisely $n$-vectors $\mu_{v, 1}, \ldots, \mu_{v, n} \in M$. 
Assume $\mu_{v, i}$ is primitive, i.e. $d^{-1} \mu_{v, i} \notin M$ for every integer $d \ge 2$. 
Then we have 
\[ \int_P e^{\langle \mu, \xi \rangle} d\mu = (-1)^n \sum_{v: \text{ vertex }} \frac{e^{\langle v, \xi \rangle} \cdot [M: \mathbb{Z} \mu_{v, 1} + \dotsb + \mathbb{Z} \mu_{v, n}]}{\prod_{i=1}^n \langle \mu_{v, i}, \xi \rangle}. \]
Here we can compute the index by 
\[ \mathrm{index} (v) := [M: \mathbb{Z} \mu_{v, 1} + \dotsb + \mathbb{Z} \mu_{v, n}] = |\det (a_{ij})| \]
if $\mu_{v, i} = \sum_{j=1}^n a_{ij} \mu_j$ for a basis $(\mu_1, \ldots, \mu_n)$ of $M$. 

\subsubsection{K\"ahler classes on the two points blowing-up of $\mathbb{C}P^2$}

For $0 < \delta < 3/2$, consider the polytope $P_\delta \subset \mathbb{R}^2$ given by the convex hull of the following five vertices: 
\[ v_1 = (-1, -1), ~~ v_2 = (2 -\delta, -1), ~~ v_3 = (2-\delta, -1+\delta), ~~ v_4 = (-1+ \delta, 2-\delta), ~~ v_5 = (-1, 2 - \delta). \]
\begin{center}
\begin{tikzpicture}
\filldraw[very thick, fill=cyan] (-1,-1)--(0.8,-1)--(0.8,0.2)--(0.2,0.8)--(-1,0.8)--cycle node[above = 20pt, left = 30pt] {};
\end{tikzpicture}
\end{center}

The associated toric variety is the two points blowing-up $X = \mathbb{C}P^2 \# 2 \overline{\mathbb{C}P^2}$ of the projective space. 
The associated polarization is $L_\delta = \delta (-K_X) + (1-\delta) \beta^* (-K_{\mathbb{C}P^2})$, where $\beta: X \to \mathbb{C}P^2$ is the blowing-up morphism. 

For a generic $\eta_\epsilon = (1, \epsilon)$ ($\epsilon \neq 0, 1$), we will compute 
\[ \int_{P_\delta} e^{\langle \mu, x \eta_\epsilon \rangle} d\mu = x^{-2} \sum_{v: \mathrm{vertex}} \frac{e^{x \langle v, \eta_\epsilon \rangle}}{\langle \mu_{v, +}, \eta_\epsilon \rangle \langle \mu_{v, -}, \eta_\epsilon \rangle} \]
and 
\[ \int_{\partial P_\delta} e^{\langle \mu, x \eta_\epsilon \rangle} d\sigma = - x^{-1} \sum_{v: \mathrm{vertex}} \frac{e^{x \langle v, \eta_\epsilon \rangle} \cdot \langle \mu_{v, +} + \mu_{v, -}, \eta_\epsilon \rangle}{\langle \mu_{v, +}, \eta_\epsilon \rangle \langle \mu_{v, -}, \eta_\epsilon \rangle}. \]

Substituting
\begin{align*}
\mu_{v_1, +} = (1,0), \quad \mu_{v_1, -} = (0,1), 
\\
\mu_{v_2, +} = (0, 1), \quad \mu_{v_2, -} = (-1, 0), 
\\
\mu_{v_3, +} = (-1, 1), \quad \mu_{v_3, -} = (0,-1), 
\\
\mu_{v_4, +} = (-1, 0), \quad \mu_{v_4, -} = (1, -1), 
\\
\mu_{v_5, +} = (0, -1), \quad \mu_{v_5, -} = (1, 0), 
\end{align*}
we get 
\begin{align*}
\int_{P_\delta} e^{\langle \mu, x \eta_\epsilon \rangle} d\mu 
&= x^{-2} \Big{(} \epsilon^{-1} e^{- (1+\epsilon) x} - \epsilon^{-1} e^{(2 - \delta -\epsilon) x} + \frac{1}{\epsilon (1-\epsilon)} e^{(2 -\delta -\epsilon + \delta \epsilon) x} 
\\
&\qquad - \frac{1}{1-\epsilon} e^{(-1+\delta + 2 \epsilon - \delta \epsilon) x} -\epsilon^{-1} e^{(-1 + 2\epsilon -\delta \epsilon) x} \Big{)}
\end{align*}
and 
\begin{align*}
\int_{\partial P_\delta} e^{\langle \mu, x \eta_\epsilon \rangle} d\sigma
&= -x^{-1} \Big{(} \frac{1+\epsilon}{\epsilon} e^{- (1+\epsilon) x} + \frac{1-\epsilon}{\epsilon} e^{(2 - \delta -\epsilon) x} - \frac{1}{\epsilon (1-\epsilon)} e^{(2 -\delta -\epsilon + \delta \epsilon) x} 
\\
&\qquad + \frac{\epsilon}{1-\epsilon} e^{(-1+\delta + 2 \epsilon - \delta \epsilon) x} - \frac{1-\epsilon}{\epsilon} e^{(-1 + 2\epsilon -\delta \epsilon) x} \Big{)}. 
\end{align*}

%We also obtain 
%\begin{align*} 
%\int_{P_\delta} e^{\langle \mu, (x,y) \rangle} d\mu 
%&= (x y)^{-1} e^{-x-y} + (xy)^{-1} e^{(2-\delta)x-y} + \frac{1}{y(x-y)} e^{(2-\delta)x -(1-\delta) y} 
%\\
%&\qquad- \frac{1}{x (x-y)} e^{(-1+\delta)x +(2-\delta)y} - (xy)^{-1} e^{-x (2-\delta)y}. 
%\end{align*}
%and 
%\begin{align*}
%\int_{\partial P_\delta} e^{\langle \mu, (x,y) \rangle} d\sigma
%&= - \Big{(} \frac{x+y}{xy} e^{- x- y} + \frac{x-y}{xy} e^{(2 - \delta) x -y} - \frac{x}{y (x-y)} e^{(2 -\delta) x +(-1 + \delta) y} 
%\\
%&\qquad + \frac{y}{x(x-y)} e^{(-1+\delta) x + (2  - \delta ) y} - \frac{x-y}{xy} e^{-x + (2 -\delta ) y} \Big{)}. 
%\end{align*}

Taking the limit $\epsilon \to 1$, we get 
\begin{align*} 
\int_{P_\delta} e^{\langle \mu, x \eta_1 \rangle} d\mu 
%&= \lim_{\epsilon \to 1} x^{-2} \Big{(} \epsilon^{-1} e^{- (1+\epsilon) x} - \epsilon^{-1} e^{(2 - \delta -\epsilon) x} + \frac{1}{\epsilon (1-\epsilon)} e^{(2 -\delta -\epsilon + \delta \epsilon) x} 
%\\
%&\qquad - \frac{1}{1-\epsilon} e^{(-1+\delta + 2 \epsilon - \delta \epsilon) x} -\epsilon^{-1} e^{(-1 + 2\epsilon -\delta \epsilon) x} \Big{)}
%\\
&=x^{-2} \Big{(} e^{-2 x} -e^{(1-\delta)x} + (1+ (3-2\delta)x) e^x - e^{(1-\delta)x} \Big{)}
\end{align*}
and 
\begin{align*} 
\int_{\partial P_\delta} e^{\langle \mu, x \eta_1 \rangle} d\sigma 
%&= -\lim_{\epsilon \to 1} x^{-1} \Big{(} \frac{1+\epsilon}{\epsilon} e^{- (1+\epsilon) x} + \frac{1-\epsilon}{\epsilon} e^{(2 - \delta -\epsilon) x} - \frac{1}{\epsilon (1-\epsilon)} e^{(2 -\delta -\epsilon + \delta \epsilon) x} 
%\\
%&\qquad + \frac{\epsilon}{1-\epsilon} e^{(-1+\delta + 2 \epsilon - \delta \epsilon) x} - \frac{1-\epsilon}{\epsilon} e^{(-1 + 2\epsilon -\delta \epsilon) x} \Big{)}
%\\
&= -x^{-1} \Big{(} 2 e^{-2 x} - (2+ (3-2\delta)x ) e^x \Big{)}. 
\end{align*}

The following is the graph of $- \frac{1}{2\pi} \NAmu (x. \eta_1)$ for $\delta =1$. 
The minimizer is not $x = 0$, which implies there is no cscK metrics while there exists a K\"ahler--Ricci soliton (= $\mu$-cscK metric). 

\begin{figure}[h]
\includegraphics[width=6cm]{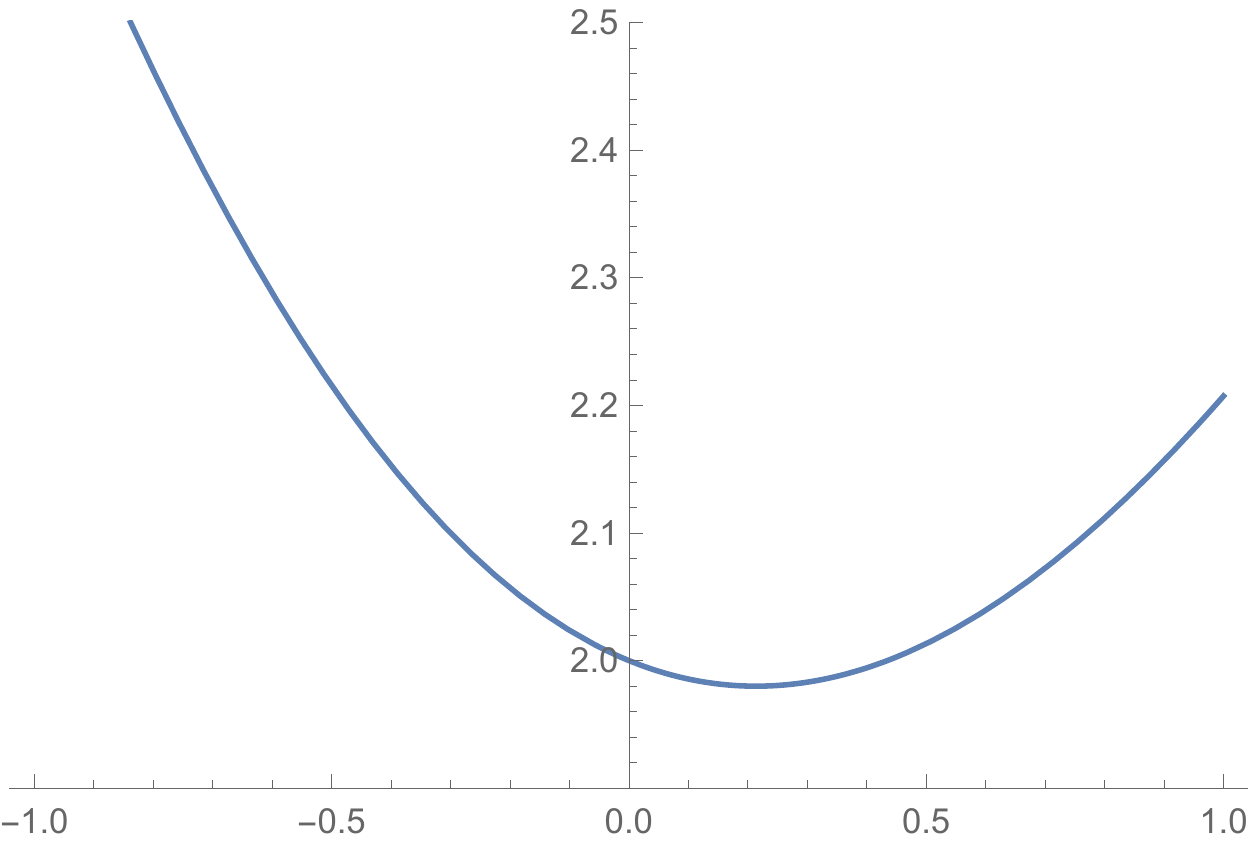}
\end{figure}

\subsubsection{Donaldson's example}

We recall Donaldson's example of toric orbifolds which are not K-semistable while its Futaki invariant vanishes for equivariant product configurations (cf. \cite{Don2}). 
The polytope $P_n$ is defined as the convex hull of the following nine vertices: 
\begin{gather*} 
v_1^+ = (1,0), \quad v_1^- = (0,1), \quad v_1 = v_1 = r_n (1,1), 
\\
v_2^+ = (4, 0) + (-1,1), \quad v_2^- = (4, 0) + (-1, 0), \quad v_2 = (4, 0) + r_n (-2, 1), 
\\
v_3^+ = (0,4) + (0,-1), \quad v_3^- = (0,4) + (1,-1), \quad v_3 = (0, 4) + r_n (1, -2),  
\end{gather*}
where $r_n := (n-2)/(3n-5)$. 

\begin{center}
\includegraphics[width=8cm]{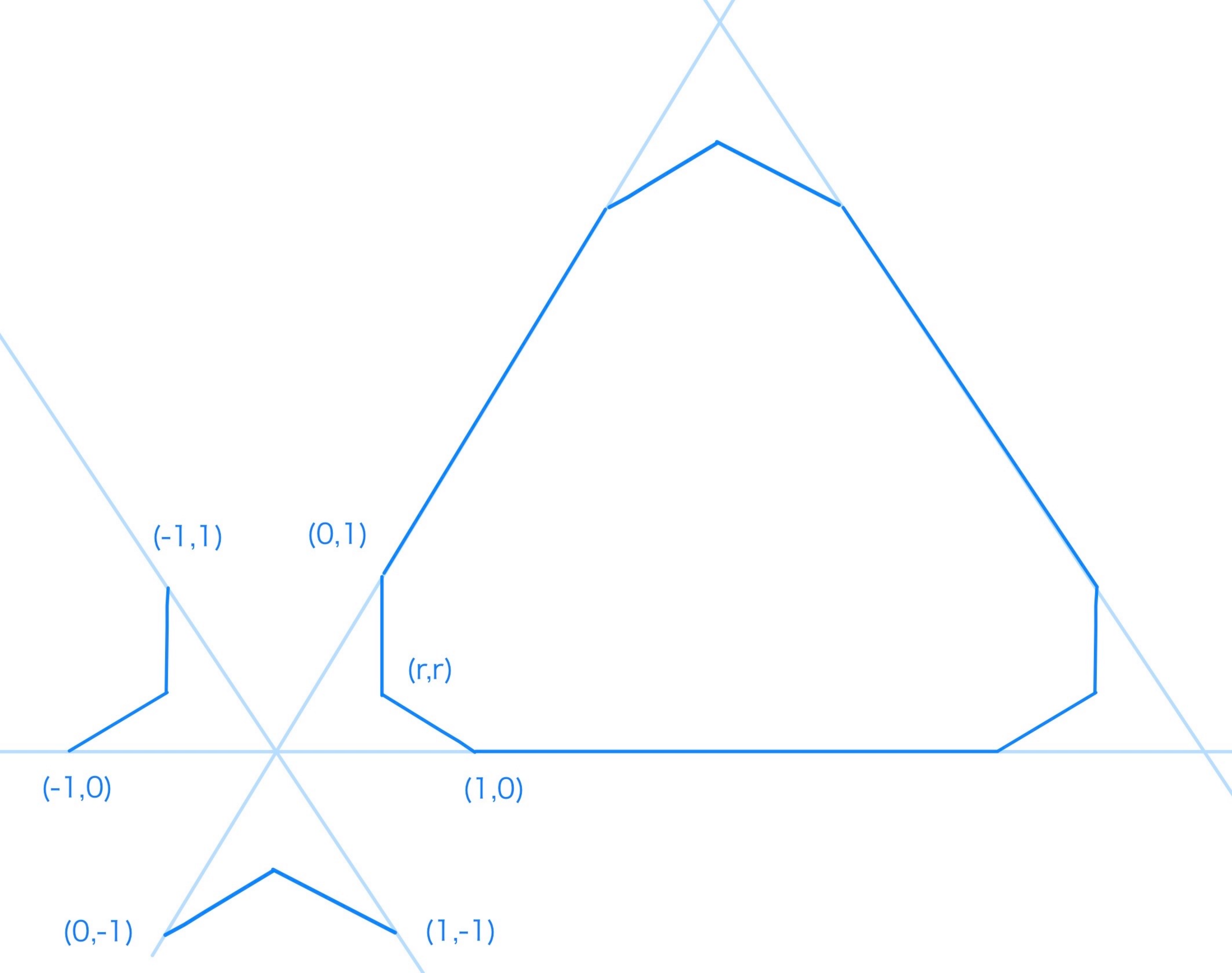}
\end{center}

For $r_5 = 3/10$, we have 
\begin{gather*} 
v_1^+ = (1,0), \quad v_1^- = (0,1), \quad v_1 = (3/10, 3/10) , 
\\
v_2^+ = (3, 1), \quad v_2^- = (3, 0), \quad v_2 = (34/10, 3/10), 
\\
v_3^+ = (0,3), \quad v_3^- = (1,3), \quad v_3 = (3/10, 34/10) 
\end{gather*}
and
\begin{gather*} 
\mu_{v_1^+, +} = (1, 0), \quad \mu_{v_1^+, -} = (-7, 3), 
\\
\mu_{v_1^-, +} = (3, -7), \quad \mu_{v_1^-, -} = (0, 1), 
\\
\mu_{v_1, +} = (7, -3), \quad \mu_{v_1, -} = (-3, 7), 
\\[4pt]
\mu_{v_2^+, +} = (-1, 1), \quad \mu_{v_2^+, -} = (4, -7), 
\\
\mu_{v_2^-, +} = (4, 3), \quad \mu_{v_2^-, -} = (-1, 0), 
\\
\mu_{v_2, +} = (-4, 7), \quad \mu_{v_2, -} = (-4, -3), 
\\[4pt]
\mu_{v_3^+, +} = (0,-1), \quad \mu_{v_3^+, -} = (3, 4), 
\\
\mu_{v_3^-, +} = (-7, 4), \quad \mu_{v_3^-, -} = (1, -1), 
\\
\mu_{v_3, +} = (-3, -4),  \quad \mu_{v_3, -} = (7, -4). 
\end{gather*}
In particular, we have 
\[ \mathrm{index} (v_i^+) = 3, \quad \mathrm{index} (v_i^-) = 3, \quad \mathrm{index} (v_i) = 40 \]
for $i= 1, 2, 3$. 

For generic $\eta_\epsilon = (1, \epsilon)$, we will compute 
\[ \int_P e^{\langle \mu, x \eta_\epsilon \rangle} d\mu = x^{-2} \sum_{v: \mathrm{vertex}} \frac{e^{x \langle \mu, \eta_\epsilon \rangle} \cdot \mathrm{index} (v)}{\langle \mu_{v, +}, \eta_\epsilon \rangle \langle \mu_{v, -}, \eta_\epsilon \rangle} \]
and 
\[ \int_{\partial P} e^{\langle \mu, x \eta_\epsilon \rangle} d\sigma = - x^{-1} \sum_{v: \mathrm{vertex}} \frac{e^{x \langle \mu, \eta_\epsilon \rangle} \cdot \langle \mu_{v, +} + \mu_{v, -}, \eta_\epsilon \rangle}{\langle \mu_{v, +}, \eta_\epsilon \rangle \langle \mu_{v, -}, \eta_\epsilon \rangle}. \]

We have
\begin{gather*} 
\langle v_1^+, \eta \rangle = 1, \quad \langle v_1^-, \eta \rangle = \epsilon, \quad \langle v_1, \eta \rangle= 3/10 + 3/10 \cdot \epsilon, 
\\
\langle v_2^+, \eta \rangle = 3 + \epsilon, \quad \langle v_2^-, \eta \rangle = 3, \quad \langle v_2, \eta \rangle = 34/10 + 3/10 \cdot \epsilon, 
\\
\langle v_3^+, \eta \rangle = 3 \epsilon, \quad \langle v_3^-, \eta \rangle = 1 + 3\epsilon, \quad \langle v_3, \eta \rangle = 3/10 + 34/10 \cdot \epsilon,  
\end{gather*}
and 
\begin{gather*} 
\langle \mu_{v_1^+, +}, \eta \rangle = 1, \quad \langle \mu_{v_1^+, -}, \eta \rangle = -7 + 3\epsilon, 
\\
\langle \mu_{v_1^-, +}, \eta \rangle = 3 - 7\epsilon, \quad \langle \mu_{v_1^-, -}, \eta \rangle = \epsilon, 
\\
\langle \mu_{v_1, +}, \eta \rangle = 7 -3 \epsilon, \quad \langle \mu_{v_1, -}, \eta \rangle = -3 + 7 \epsilon, 
\\
\langle \mu_{v_2^+, +}, \eta \rangle = -1 + \epsilon, \quad \langle \mu_{v_2^+, -}, \eta \rangle = 4 - 7 \epsilon
\\
\langle \mu_{v_2^-, +}, \eta \rangle = 4 + 3 \epsilon, \quad \langle \mu_{v_2^-, -}, \eta \rangle = -1
\\
\langle \mu_{v_2, +}, \eta \rangle = -4 + 7 \epsilon, \quad \langle \mu_{v_2, -}, \eta \rangle = -4 -3 \epsilon
\\
\langle \mu_{v_3^+, +}, \eta \rangle = - \epsilon, \quad \langle \mu_{v_3^+, -}, \eta \rangle = 3 + 4 \epsilon
\\
\langle \mu_{v_3^-, +}, \eta \rangle = -7 + 4\epsilon, \quad \langle \mu_{v_3^-, -}, \eta \rangle = 1 -\epsilon
\\
\langle \mu_{v_3, +}, \eta \rangle = -3 - 4 \epsilon,  \quad \langle \mu_{v_3, -}, \eta \rangle = 7 - 4 \epsilon. 
\end{gather*}

It follows that 
\begin{align*}
\int_P e^{\langle \mu, x \eta \rangle} d\mu 
&= x^{-2} \Big{(} \frac{3}{-7 + 3 \epsilon} e^x + \frac{3}{(3 - 7\epsilon) \epsilon} e^{\epsilon x} + \frac{40}{(7 -3 \epsilon) (-3+ 7 \epsilon)} e^{(3/10 + 3/10 \cdot \epsilon) x} 
\\
&\qquad - \frac{3}{(1-\epsilon) (4 - 7 \epsilon)} e^{(3+ \epsilon) x} - \frac{3}{4 +3 \epsilon} e^{3x} + \frac{40}{(4 - 7 \epsilon)(4 + 3\epsilon)} e^{(34/10 + 3/10 \cdot \epsilon) x} 
\\
&\qquad - \frac{3}{\epsilon(3+4 \epsilon)} e^{3\epsilon x} - \frac{3}{(7-4\epsilon)(1-\epsilon)} e^{(1+ 3\epsilon) x} - \frac{40}{(3+ 4\epsilon)(7 - 4 \epsilon)} e^{(3/10 + 34/10 \cdot \epsilon) x} \Big{)}
\end{align*}
and 
\begin{align*}
\int_{\partial P} e^{\langle \mu, x \eta \rangle} d\sigma 
&= -x^{-1} \Big{(} \frac{-6+3\epsilon}{-7 + 3 \epsilon} e^x + \frac{3 - 6 \epsilon}{(3 - 7\epsilon) \epsilon} e^{\epsilon x} + \frac{4 + 4\epsilon}{(7 -3 \epsilon) (-3+ 7 \epsilon)} e^{(3/10 + 3/10 \cdot \epsilon) x} 
\\
&\qquad - \frac{3-6\epsilon}{(1-\epsilon) (4 - 7 \epsilon)} e^{(3+ \epsilon) x} - \frac{3+3\epsilon}{4 +3 \epsilon} e^{3x} + \frac{-8+4\epsilon}{(4 - 7 \epsilon)(4 + 3\epsilon)} e^{(34/10 + 3/10 \cdot \epsilon) x} 
\\
&\qquad - \frac{3+3\epsilon}{\epsilon(3+4 \epsilon)} e^{3\epsilon x} - \frac{-6+3\epsilon}{(7-4\epsilon)(1-\epsilon)} e^{(1+ 3\epsilon) x} - \frac{4-8\epsilon}{(3+ 4\epsilon)(7 - 4 \epsilon)} e^{(3/10 + 34/10 \cdot \epsilon) x} \Big{)}. 
\end{align*}

Taking the limit $\epsilon \to 0$, we get 
\begin{align*}
\int_P e^{\langle \mu, x \eta_0 \rangle}
%&= x^{-2} \Big{(} - \frac{6}{7} e^x - \frac{80}{21} e^{3/10 \cdot x} - \frac{3}{2} e^{3 x} + \frac{5}{2} e^{34/10 \cdot x} 
%\\
%&\qquad + \lim_{\epsilon \to 0} (\frac{3}{(3 - 7\epsilon) \epsilon} e^{\epsilon x} - \frac{3}{\epsilon (3+4 \epsilon)} e^{3\epsilon x}) \Big{)}
%\\
&= x^{-2} \Big{(} - \frac{6}{7} e^x - \frac{80}{21} e^{3/10 \cdot x} - \frac{3}{2} e^{3 x} + \frac{5}{2} e^{34/10 \cdot x} + \frac{11}{3} -2x \Big{)}
\end{align*}
and 
\begin{align*}
\int_{\partial P} e^{\langle \mu, x \eta_0 \rangle} d\sigma
%&= -x^{-1} \Big{(} \frac{12}{7} e^x - \frac{8}{21} e^{3/10 \cdot x} - \frac{3}{2} e^{3 x} - \frac{1}{2} e^{34/10 \cdot x} 
%\\
%&\qquad + \lim_{\epsilon \to 0} (\frac{3 - 6 \epsilon}{(3 - 7\epsilon) \epsilon} e^{\epsilon x} - \frac{3+3\epsilon}{\epsilon(3+4 \epsilon)} e^{3\epsilon x}) \Big{)}
%\\
&= -x^{-1} \Big{(} \frac{12}{7} e^x - \frac{8}{21} e^{3/10 \cdot x} - \frac{3}{2} e^{3 x} - \frac{1}{2} e^{34/10 \cdot x} + \frac{2}{3} -2x \Big{)}. 
\end{align*}

The following are the graphs of $- \frac{1}{2\pi} \NAmu (x. \eta_0)$ and its derivative $- \frac{1}{2\pi} \frac{d}{dx} \NAmu (x. \eta_0) = -\frac{1}{2\pi} \Fut_{x. \eta_0} (\eta_0)$, respectively. 
This illustration shows the $\mu$-entropy for proper vectors is not concave even for $\lambda = 0$. 
The author speculates the $\mu$-entorpy is quasi-concave for $\lambda \le 0$. 
\begin{figure}[h]
\begin{minipage}[b]{0.45\linewidth}
\centering
\includegraphics[width=5.5cm]{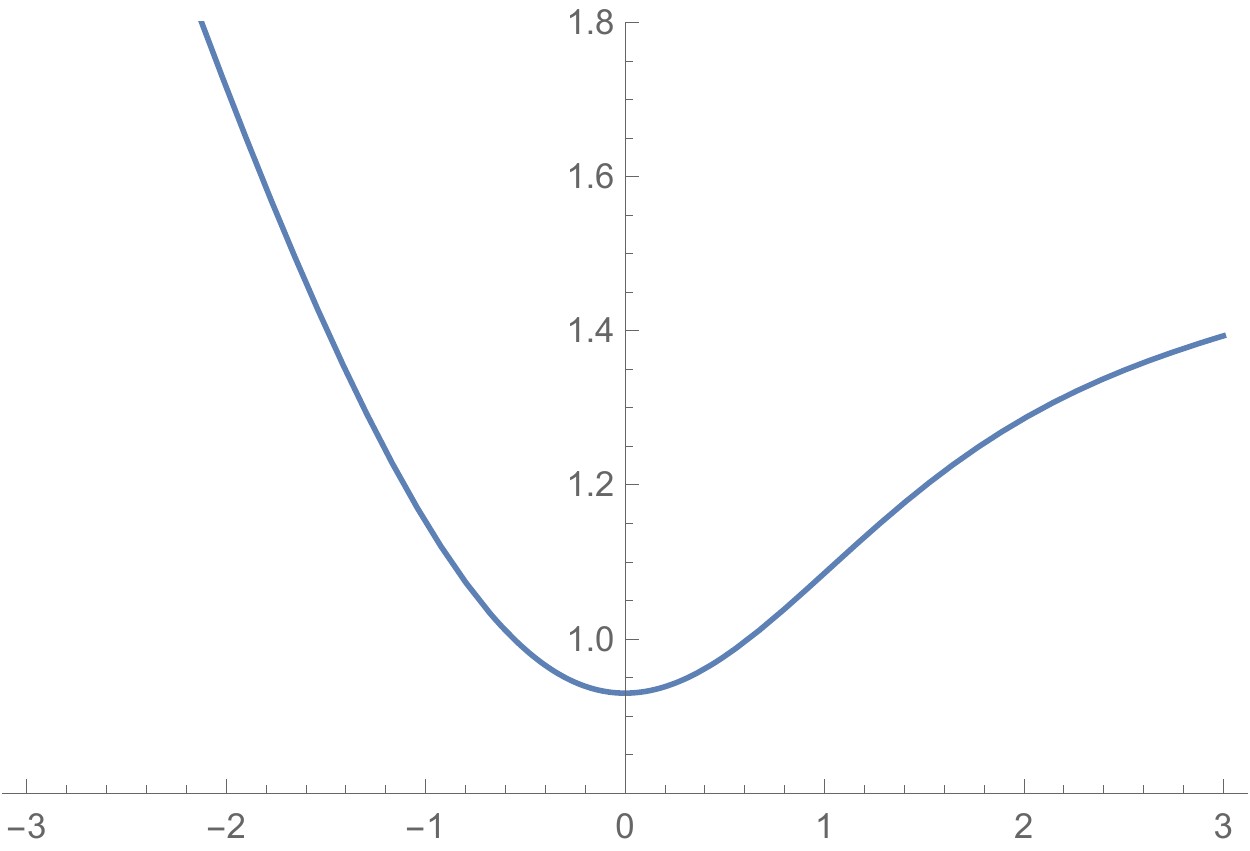}
\end{minipage}
\begin{minipage}[b]{0.45\linewidth}
\centering
\includegraphics[width=5.5cm]{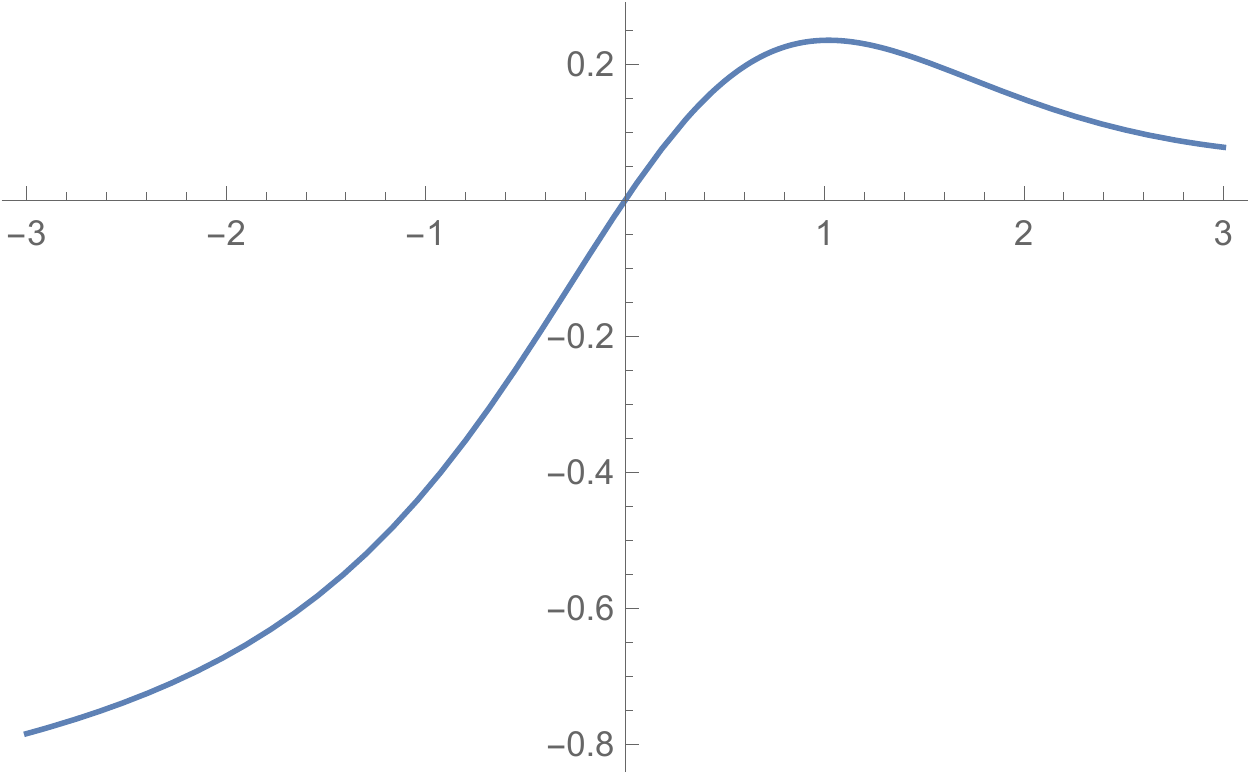}
\end{minipage}
\end{figure}

\end{document}